\documentclass[a4paper,10pt]{amsart}
\usepackage{amsfonts,amssymb,amsmath,german,style_1,amsthm,mathrsfs,framed,verbatim,wrapfig,subfig,tikz,geometry,array, xcolor,comment,accents}
\usepackage[T1]{fontenc}
\usepackage[utf8x]{inputenc}
\usepackage{mathtools}
\usepackage[greek,english]{babel}
\newtheorem{lemma}{Lemma}[section]
\newtheorem{proposition}{Proposition}[section]
\newtheorem{theorem}{Theorem}[section]
\theoremstyle{definition}

\theoremstyle{definition}
\newtheorem{remark}{Remark}

\allowdisplaybreaks

\begin{document}

\hyphenation{Chris-to-dou-lou}

\title{Shock Development in Spherical Symmetry}
\author[D.~Christodoulou]{Demetrios Christodoulou}
\author[A.~Lisibach]{Andr\'e Lisibach}
\address{Demetrios Christodoulou\\Department of Mathematics\\ETH Zurich}
\address{Andr\'e Lisibach\\Department of Physics\\ETH Zurich\\}
\email{lisibach@itp.phys.ethz.ch}
\thanks{Work supported by ERC Advanced Grant 246574 Partial Differential Equations of Classical Physics.}
\begin{abstract}
  The general problem of shock formation in three space dimensions was solved by D.~Christodoulou in \cite{ch2007}. In this work also a complete description of the maximal development of the initial data is provided. This description sets up the problem of continuing the solution beyond the point where the solution ceases to be regular. This problem is called the shock development problem. It belongs to the category of free boundary problems but in addition has singular initial data because of the behavior of the solution at the blowup surface. The present work delivers the solution to this problem in the case of spherical symmetry for a barotropic fluid. A complete description of the singularities associated to the development of shocks in terms of smooth functions is given.
\end{abstract}
\maketitle

\tableofcontents
\section{Introduction}

\subsection{Overview}
The Euler equations are a set of nonlinear hyperbolic partial differential equations. Physically they represent the conservation of energy, momentum and mass. It is well known that, given smooth initial data, solutions of equations of this type can blow up in finite time. In the case of the Euler equations the gradients of the solution become infinite. The mechanism of the blowup is called \textit{formation of a shock} and has first been studied in one space dimension by Riemann in 1858 \cite{Rie}. The general problem of shock formation in three space dimensions for a fluid with an arbitrary equation of state was solved by Christodoulou in the monograph \cite{ch2007}. In this work also a complete description of the maximal development of the initial data is provided. This description properly sets up the problem of continuing the solution beyond the point where the solution ceases to be regular. This problem is called the \textit{shock development problem} and is stated in the epilogue of \cite{ch2007}. It belongs to the category of free boundary problems but possesses the additional difficulty of having singular data due to the behavior of the solution at the blowup surface. The present work gives the solution to this problem in the physically important case of spherical symmetry for a fluid with barotropic equation of state. The result is a step in understanding the development of shocks in fluids. It provides the basis on which the continuation, interaction and breakdown of shocks in spherical symmetry can be studied. Furthermore, the mathematical tools invented to deal with the problem will be of importance in studying solutions to nonlinear hyperbolic equations beyond shock formation.

\subsection{Shock Development}
The general problem of shock formation in a relativistic fluid has been studied in the monograph \cite{ch2007} by Christodoulou. This work is in the framework of special relativity. The theorems in this monograph give a detailed picture of shock formation in 3-dimensional fluids. In particular a detailed description is given of the geometry of the boundary of the maximal development of the initial data and of the behavior of the solution at this boundary. The notion of maximal development in this context is not that relative to the background Minkowski metric $\eta_{\mu\nu}$, but rather the one relative to the acoustical metric $g_{\mu\nu}$. This is a Lorentzian metric, the null cones of which are the sound cones. In the monograph it is shown that the boundary of the maximal development in the 'acoustical' sense (relative to $g$) consists of a regular part and a singular part. Each component of the regular part $\underline{C}$ is an incoming characteristic (relative to $g$) hypersurface which has a singular past boundary. The singular part of the boundary is the locus of points where the density of foliations by outgoing characteristic (relative to $g$) hypersurfaces blows up. It is the union $\partial_-\mathcal{B}\cup \mathcal{B}$, where each component of $\partial_-\mathcal{B}$ is a smooth embedded surface in Minkowski spacetime, the tangent plane to which at each point is contained in the exterior of the sound cone at that point. On the other hand, each component of $\mathcal{B}$ is a smooth embedded hypersurface in Minkowski spacetime, the tangent hyperplane to which at each point is contained in the exterior of the sound cone at that point, with the exception of a single generator of the sound cone, which lies on the hyperplane itself. The past boundary of a component of $\mathcal{B}$ is the corresponding component of $\partial_-\mathcal{B}$. The latter is at the same time the past boundary of a component of $\underline{C}$. This is the surface where a shock begins to form. The maximal development in the case of spherical symmetry is shown in figure \ref{max_dev}. In spherical symmetry a component of $\partial_-\mathcal{B}$ corresponds to a sphere and therefore to a point in the $t$-$r$-plane, the cusp point, which we denote by $O$.

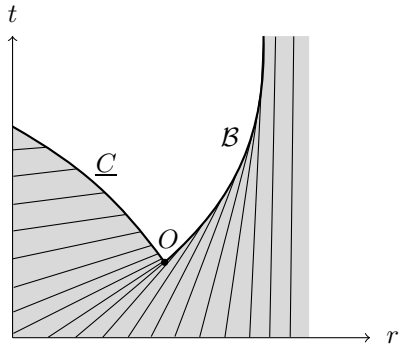
\begin{figure}[h!]
\begin{center}
\begin{tikzpicture}
\node (o) at (0,0) {};
\node (b) at (1.3,3) {};
\node (e) at (-2,1.8) {};
\path [fill=gray!30] (o.center) to [bend right=25] (b.center) to (1.9,3) to (1.9,-1) to (-2,-1) to (-2,1.8) to [bend left=12] (o.center);
\draw [line width=0.8pt] (o.center) to [bend right=25] (b.center);
\draw [line width=0.8pt] (o.center) to [bend right=12] (e.center);
\node at (0.0496,0.311) {$O$};
\draw (o.center) circle (0.4mm);
\fill (o.center) circle (0.4mm);
\node at (0.858,1.6617) {$\mathcal{B}$};
\draw [->] (-2,-1) -- (2.7,-1);
\draw [->] (-2,-1) -- (-2,3);
\node at (3,-1) {$r$};
\node at (-2,3.3) {$t$};
\node at (-0.7689,1.2851) {$\underline{C}$};

\draw (-1.1706,-0.9945) -- (-0.0006,-0.0003);
\draw (0.2504,0.2357) -- (-0.8141,-0.9945);
\draw (0.5015,0.5068) -- (-0.4777,-0.9895);
\draw (-0.0067,0.0181) -- (-1.5321,-0.9945);
\draw (-0.0308,0.0499) -- (-1.9891,-0.9427);

\draw (-0.081,0.1152) -- (-1.9891,-0.6581);
\draw (-0.1663,0.2256) -- (-1.9844,-0.3334);
\draw (0.7325,0.8081) -- (-0.1764,-0.9845);
\draw (0.9484,1.1847) -- (0.1349,-0.9845);
\draw (1.1493,1.6768) -- (0.4613,-0.9895);
\draw (1.2647,2.2743) -- (0.7927,-0.9945);
\draw (1.2999,2.9623) -- (1.1181,-0.993);
\draw (-0.3019,0.4071) -- (-1.9961,0.0462);
\draw (-0.5058,0.6344) -- (-1.9891,0.4164);
\draw (-0.784,0.9236) -- (-1.9891,0.7679);
\draw (-1.1455,1.2399) -- (-1.9891,1.1405);
\draw (-1.5472,1.5262) -- (-1.9941,1.491);
\draw (1.4591,2.9929) -- (1.3817,-0.9895);
\draw (1.6981,2.9964) -- (1.6514,-0.9895);
\end{tikzpicture}\end{center}
\caption[Maximal development in spherical symmetry.]{The maximal development in a neighborhood of a blowup point in spherical symmetry. $\underline{C}$ denotes the regular part of the boundary of the maximal development. $\underline{C}$ is incoming characteristic and originates at the cusp point $O$. The cusp point $O$ corresponds to $\partial_-\mathcal{B}$. $\mathcal{B}$ denotes the singular part of the boundary of the maximal development. The family of outgoing characteristic curves is drawn as straight lines for simplification.}
\label{max_dev}
\end{figure}

Now the maximal development in the acoustical sense, or 'maximal classical solution', is the physical solution of the problem up to $\underline{C}\cup \partial_-\mathcal{B}$, but not up to $\mathcal{B}$. In the last part of the monograph the problem of the physical continuation of the solution is set up as the \textit{shock development problem}. This is a free boundary problem associated to each component of $\partial_-\mathcal{B}$. In this problem one is required to construct a hypersurface of discontinuity $\mathcal{K}$, the shock, lying in the past of the corresponding component of $\mathcal{B}$ but having the same past boundary as the latter, namely the given component of $\partial_-\mathcal{B}$, the tangent hyperplanes to $\mathcal{K}$ and $\mathcal{B}$ coinciding along $\partial_-\mathcal{B}$. Moreover, one is required to construct a solution of the differential conservation laws in the domain in Minkowski spacetime bounded in the past by $\underline{C}\cup \mathcal{K}$, agreeing with the maximal classical solution on $\underline{C}\cup\partial_-\mathcal{B}$, while having jumps across $\mathcal{K}$ relative to the data induced on $\mathcal{K}$ by the maximal classical solution. For reasons which will be made clear below we call this solution \textit{state behind} while the solution in the maximal development we call \textit{state ahead}. The jumps across $\mathcal{K}$ have to satisfy the jump conditions which follow from the integral form of the conservation laws (the relativistic form of the Rankine-Hugoniot jump conditions). Finally, $\mathcal{K}$ is required to be spacelike relative to the acoustical metric $g$ induced by the maximal classical solution, which holds in the past of $\mathcal{K}$, and timelike relative to the new solution, which holds in the future of $\mathcal{K}$ (the last condition is equivalent to the condition that the jump in entropy is positive). The maximal classical solution thus provides the boundary conditions on $\underline{C}\cup \partial_-\mathcal{B}$, as well as a barrier at $\mathcal{B}$. The situation in spherical symmetry is shown in figure \ref{all}.

In the present work the shock development problem is solved in the case of spherical symmetry and under the assumption that the fluid is described by a barotropic equation of state. The presence of spherical symmetry represents an important physical case, also from the point of view of applications, and reduces the problem to one on the $t$-$r$-plane, where $t$ denotes Minkowski time and $r$ denotes the radial coordinate. The assumption of a barotropic equation of state is appropriate for liquids and also for a radiation gas. The fluid being barotropic the energy-momentum conservation law decouples from the particle conservation law. The system of partial differential equations reduces to an inhomogeneous system with two unknowns. One of the key concepts used to deal with the system of equations are the \textit{Riemann Invariants} $\alpha$, $\beta$ of the principal part of the system of equations. The equations are reformulated in terms of characteristic coordinates $(u,v)$. These coordinates are defined by the outgoing and incoming null rays with respect to the acoustical metric $g$, $u$ being constant along outgoing null rays and $v$ being constant along incoming null rays. In addition the coordinates are set up such that the shock $\mathcal{K}$ is given by $u=v$. The system of equations for the time and radial coordinates $(t,r)$ in terms of $(u,v)$ is the \textit{Hodograph system}. The Hodograph system together with the system of equations for the Riemann Invariants is a non-linear four by four system. This system is then solved using a double iteration consisting of an inner and an outer iteration. In the outer iteration the position of the free boundary in the $t$-$r$-plane is iterated, providing, through the jump conditions, in each step the boundary conditions for a fixed boundary problem. The equations being non-linear, this fixed boundary problem is then solved using again an iteration, the inner iteration. The solution of the fixed boundary problem then allows to set the position of the free boundary in the $t$-$r$-plane for the next iterate. This is accomplished as follows. The solution of the fixed boundary problem provides the values of $r$ and $t$ in terms of the characteristic coordinates along the shock $\mathcal{K}$, i.e.~$r(v,v)$, $t(v,v)$. In the formation problem the solution in the maximal development (denoted by $(\cdot)^\ast$) is given in terms of the acoustical coordinates $(t,w)$, where $w$ is a function which is constant along outgoing characteristic hypersurfaces. Now $r(v,v)$ is set equal to the radial coordinate given by the acoustical coordinates, i.e.~$r^\ast(t,w)$, when $t=t(v,v)$ is substituted, i.e.~$r(v,v)=r^\ast(t(v,v),w)$. This equation is called the \textit{identification equation} since it identifies the radial coordinates of events in spacetime, with respect to the solution of the fixed boundary problem and with respect to the solution in the maximal development, along the shock $\mathcal{K}$. It plays a very important role, and the study of it is at the heart of the solution to the problem. The identification equation has to be solved for $w$ in terms of $v$ in order to be able to apply the jump conditions and in order to compute the boundary data for the next iterate in the outer iteration. This is not possible offhandedly. Only after correctly guessing the asymptotic form of the solution as we approach the sphere $\partial_-\mathcal{B}$, can the identification equation be reduced to an equation which is solvable for $w$ in terms of $v$. The iteration then yields the local existence of a continuously differentiable solution to the shock development problem. Also uniqueness of this solution is proven. Finally it is proven that the solution is, away from the shock $\mathcal{K}$, smooth.

\begin{figure}[h!]
\begin{center}
\begin{tikzpicture}
\node (o) at (0,0) {};
\node (b) at (1.3,3) {};
\node (e) at (-2,1.8) {};
\node (a) at (1.3923,1.5412) {};
\node (c) at (-1.1525,1.2389) {};
\node (x) at (-1.1525,1.2319) {};
%\path [fill=gray!30] (o.center) to [bend right=25] (b.center) to (1.9,3) to (1.9,-1) to (-2,-1) to (-2,1.8) to [bend left=12] (o.center);
\path [fill=gray!30] (o.center) to [bend right=12] (a.center) to (1.3888,1.5412) node (v2) {} to (1.3923,-0.9895) to (-2,-1) to (-2,1.8) to [bend left=12] (o.center);

\path [fill=gray!60] (o.center) to [bend right=12] (a.center) to (x.center) to [bend left=9] (o.center);%XXXXX
\draw [line width=0.8pt] (o.center) to [bend right=12] (e.center);
\draw [line width=0.8pt] (o.center) to [bend right=12] (a.center);
%\node at (-0.2316,3.3479) {$O$};
%\draw (o.center) circle (0.4mm);
%\fill (o.center) circle (0.4mm);
\node at (1.5095,1.7802) {$\mathcal{K}$};
\node at (1.0549,2.7048) {$\mathcal{B}$};
\draw [->] (-2,-1) -- (2.7,-1);
\draw [->] (-2,-1) -- (-2,3);
\node at (3,-1) {$r$};
\node at (-2,3.3) {$t$};
\node at (-1.2931,1.7381) {$\underline{C}$};

\draw [dashed,line width=0.4pt] (o.center) to [bend right=25] (b.center);

\draw (-1.1706,-0.9945) -- (0.0021,-0.0036);
\draw (0.1339,0.0966) -- (-0.8141,-0.9945);
\draw (0.3413,0.2635) -- (-0.4777,-0.9895);
\draw (0.0004,-0.0053) -- (-1.5321,-0.9945);
\draw (0.1597,0.1118) -- (-1.9891,-0.9427);
\draw (0.3378,0.2594) -- (-1.9891,-0.6581);
\draw (0.5464,0.4375) -- (-1.9844,-0.3334);
\draw (0.5452,0.434) -- (-0.1764,-0.9845);
\draw (0.7297,0.6221) -- (0.1349,-0.9845);
\draw (0.9143,0.8359) -- (0.4613,-0.9895);
\draw (1.083,1.0374) -- (0.7927,-0.9945);
\draw (1.2189,1.2413) -- (1.1181,-0.993);
\draw (0.7315,0.6227) -- (-1.9961,0.0462);
\draw (0.9213,0.8336) -- (-1.9891,0.3368);
%\draw (-0.784,0.9236) -- (-1.9891,0.7679);
%\draw (-1.1455,1.2399) -- (-1.9891,1.1405) node (v1) {};
\draw (-1.5472,1.5262) -- (-1.9941,1.491);
%\draw (1.4591,2.9929) -- (1.3817,-0.9895);
%\draw (1.8024,2.7855) -- (1.6514,-0.9895);
\draw (1.3888,1.5377) -- (-1.9891,1.1405);
\draw (1.2165,1.246) -- (-1.9867,0.8218);
\draw (1.0736,1.0374) -- (-1.9891,0.5817);
\end{tikzpicture}
\end{center}
\caption[State behind and state ahead.]{The state behind in dark shade and the state ahead in light shade, separated by $\underline{C}$ and the shock $\mathcal{K}$ where outgoing characteristics meet. The family of outgoing characteristic curves is drawn as straight lines for simplification.}
\label{all}
\end{figure}
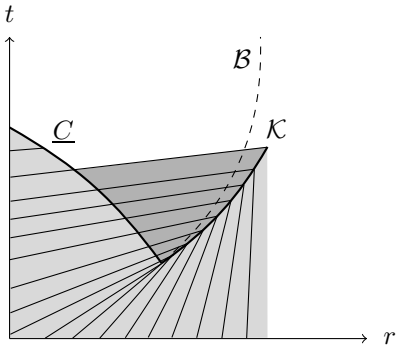

The problem is solved in the framework of special relativity. Nevertheless, no special care is needed to extract information on the non-relativistic limit. This is due to the fact that the non-relativistic limit is a regular limit, obtained by letting the speed of light in conventional units tend to infinity, while keeping the sound speed fixed.

%%% Local Variables: 
%%% mode: latex
%%% TeX-master: "master"
%%% End: 

\section{Relativistic Fluids}
In the present section we introduce the model of a perfect fluid in special relativity. Then we discuss the associated jump conditions. Finally we restrict ourselves to the barotropic case. Most of the material in this section can be found in the first chapter and the epilogue of \cite{ch2007} and in the first section of \cite{I}.
\subsection{Relativistic Perfect Fluids}
The motion of a perfect fluid in special relativity is described by a future-directed unit time-like vector field $u$ and two positive functions $n$ and $s$, the number of particles per unit volume (in the local rest frame of the fluid) and the entropy per particle, respectively. Let us denote the Minkowski metric by $\eta$. The conditions on the velocity $u$ are then
\begin{align}
  \label{eq:1}
  \eta(u,u)=-1, \qquad u^0>0.
\end{align}
The mechanical properties of the fluid are specified once we give the equation of state, which expresses the mass-energy density $\rho$ as a function of $n$ and $s$
\begin{align}
  \label{eq:2}
  \rho=\rho(n,s).
\end{align}
Let $e=\rho /n$ be the energy per particle. According to the first law of thermodynamics we have
\begin{align}
  \label{eq:3}
  de=-pdv+\theta ds,
\end{align}
where $p$ is the pressure, $v=1/n$ the volume per particle and $\theta$ the temperature. We have
\begin{align}
  \label{eq:4}
  p=n\pp{\rho}{n}-\rho,\qquad \theta=\frac{1}{n}\pp{\rho}{s}.
\end{align}
The functions $\rho$, $p$, $\theta$ are assumed to be positive. The equations of motion for a perfect fluid are given by the particle conservation law and the energy-momentum conservation law, i.e.
\begin{align}
  \label{eq:5}
  \nabla_\mu I^\mu&=0,\\
  \label{eq:6}
  \nabla_\nu T^{\mu\nu}&=0,
\end{align}
where $T$ and $I$ are the energy-momentum-stress tensor and the particle current, respectively, given by
\begin{align}
  \label{eq:7}
  T^{\mu\nu}=(\rho+p)u^\mu u^\nu+p(\eta^{-1})^{\mu\nu},\qquad I^\mu=nu^\mu.
\end{align}
The component of \eqref{eq:6} along $u$ is the energy equation
\begin{align}
  \label{eq:8}
  u^\mu\nabla_\mu\rho+(\rho+p)\nabla_\mu u^\mu=0.
\end{align}
Using \eqref{eq:5} in \eqref{eq:8} together with \eqref{eq:4} we deduce
\begin{align}
  \label{eq:9}
  u^\mu\nabla_\mu s=0,
\end{align}
i.e.~modulo the particle conservation law, the energy equation is equivalent to the entropy being constant along the flow lines. Nevertheless the equivalence of the energy and entropy conservation only holds for $C^1$ solutions. Let $\Pi^\mu_\nu\coloneqq\delta^\mu_\nu +u^\mu u_\nu$ denote the projection onto the local simultaneous space of the fluid. The projection of \eqref{eq:6} is the momentum conservation law
\begin{align}
  \label{eq:10}
  (\rho+p)u^\nu\nabla_\nu u^\mu+\Pi^{\mu\nu}\nabla_\nu p=0.
\end{align}

The symbol $\sigma_\xi$ of the Eulerian system \eqref{eq:5}, \eqref{eq:9}, \eqref{eq:10} at a given covector $\xi$ is the linear operator on the space of variations $(\dot{n},\dot{s},\dot{u})$ whose components are
\begin{align}
  \label{eq:11}
  u^\mu\xi_\mu\dot{n}+n\xi_\mu\dot{u}^\mu,\\
  \label{eq:12}
  u^\mu\xi_\mu\dot{s},\\
  \label{eq:13}
  (\rho+p)u^\nu\xi_\nu\dot{u}^\mu+\Pi^{\mu\nu}\xi_\nu\dot{p}.
\end{align}
We note that
\begin{align}
  \label{eq:14}
  \dot{p}=\pp{p}{n}\dot{n}+\pp{p}{s}\dot{s}.
\end{align}
The characteristic subset of $T_x^\ast M$, that is the set of covectors $\xi$ such that the null space of $\sigma_\xi$ is nontrivial, consists of the hyperplane $P_x^\ast$:
\begin{align}
  \label{eq:15}
  H_P\coloneqq u^\mu\xi_\mu=0
\end{align}
and the cone $C_x^\ast$:
\begin{align}
  \label{eq:16}
  H_C\coloneqq\frac{1}{2}\Big((u^\mu\xi_\mu)^2-\eta^2\Pi^{\mu\nu}\xi_\mu\xi_\nu\Big)=0,
\end{align}
where $\eta$ is the sound speed
\begin{align}
  \label{eq:17}
  \eta^2\coloneqq\left(\pp{p}{\rho}\right)_s.
\end{align}
We assume that the equation of state satisfies the basic requirement
\begin{align}
  \label{eq:18}
  0<\eta^2.
\end{align}
The characteristic subset of $T_xM$ corresponding to $P_x^\ast$, i.e., the set of vectors $\dot{x}\in T_xM$ of the form
\begin{align}
  \label{eq:19}
  \dot{x}^\mu=\pp{H_P}{\xi_\mu}=u^\mu,\qquad \xi\in P_x^\ast,
\end{align}
is simply the vector $u(x)$, while the characteristic subset of $T_xM$ corresponding to $C_x^\ast$, i.e., the set of vectors $\dot{x}\in T_xM$ of the form
\begin{align}
  \label{eq:20}
  \dot{x}^\mu=\pp{H_C}{\xi_\mu}=u^\nu\xi_\nu u^\mu-\eta^2\Pi^{\mu\nu}\xi_\nu,\qquad \xi\in C_x^\ast,
\end{align}
is the sound cone $C_x$:
\begin{align}
  \label{eq:21}
  (\eta^2u_\mu u_\nu-\Pi_{\mu\nu})\dot{x}^\mu\dot{x}^\nu=0.
\end{align}
We define the acoustical metric $g_{\mu\nu}$ by
\begin{align}
  \label{eq:22}
  g_{\mu\nu}\coloneqq\eta_{\mu\nu}+(1-\eta^2)u_\mu u_\nu.
\end{align}
$C_x$ is then given by
\begin{align}
  \label{eq:23}
  g_{\mu\nu}\dot{x}^\mu\dot{x}^\nu=0.
\end{align}
We assume that the equation of state satisfies the basic requirement
\begin{align}
  \label{eq:24}
  \eta^2<1,
\end{align}
which is equivalent to the condition that the sound cone is contained within the light cone. For $\xi\in P_x^\ast$ the null space of $\sigma_\xi$ consists of the variations satisfying
\begin{align}
  \label{eq:25}
  \dot{p}=0,\qquad \xi_\mu \dot{u}^\mu=0
\end{align}
(the isobaric vorticity waves). For $\xi\in C_x^\ast$ the null space of $\sigma_\xi$ consists of the variations satisfying
\begin{align}
  \label{eq:26}
  \dot{s}=0,\qquad \dot{u}^\mu=-\frac{\Pi^{\mu\nu}\xi_\nu\dot{p}}{(\rho+p)u^\nu\xi_\nu}
\end{align}
(the adiabatic sound waves). We note that the inverse acoustical metric is given by
\begin{align}
  \label{eq:27}
  (g^{-1})^{\mu\nu}=(\eta^{-1})^{\mu\nu}-\left(\frac{1}{\eta^2}-1\right)u^\mu u^\nu.
\end{align}

We define the one form $\beta$ by
\begin{align}
  \label{eq:28}
  \beta_\mu\coloneqq -hu_\mu,
\end{align}
where $h$ is the enthalpy per particle given by
\begin{align}
  \label{eq:29}
  h\coloneqq e+pv=\frac{\rho+p}{n}.
\end{align}
We deduce
\begin{align}
  \label{eq:30}
  (\mathcal{L}_u\beta)_\mu&=-u^\nu\nabla_\nu(hu_\mu)\notag\\
&=\frac{h}{\rho+p}\nabla_\mu p+u_\mu u^\nu\left(\frac{h}{\rho+p}\nabla_\nu p-\nabla_\nu h\right),
\end{align}
where for the first equality we used the definition of the Lie derivative together with the first of \eqref{eq:1} while for the second equality we used \eqref{eq:10}. By \eqref{eq:29} in conjunction with \eqref{eq:3} the expression in the last parenthesis is equal to $-\theta \nabla_\nu s$. Therefore, by \eqref{eq:9}, the last term vanishes and we have
\begin{align}
  \label{eq:31}
  \mathcal{L}_u\beta=dh-\theta ds.
\end{align}

We define the vorticity two form by
\begin{align}
  \label{eq:32}
  \omega\coloneqq d\beta.
\end{align}
Let us denote by $i_X$ contraction from the left by $X$. From \eqref{eq:28} we deduce
\begin{align}
  \label{eq:33}
  i_u\beta=h.
\end{align}
Since for any exterior differential form $\vartheta$ it holds that $\mathcal{L}_X\vartheta=i_Xd\vartheta+di_X\vartheta$, we obtain from \eqref{eq:31}
\begin{align}
  \label{eq:34}
  i_u\omega=-\theta ds.
\end{align}
We conclude that the equations of motion \eqref{eq:5}, \eqref{eq:6} are equivalent to the system
\begin{align}
  \label{eq:35}
  \nabla_\mu I^\mu&=0,\\
  \label{eq:36}
  u^\mu\nabla_\mu s&=0,\\
  \label{eq:37}
  i_u\omega&=-\theta ds.
\end{align}
In fact \eqref{eq:36} follows from \eqref{eq:37}.

\subsection{Jump Conditions}
It is well known that the solution of the equations \eqref{eq:5}, \eqref{eq:6}, in general, develop discontinuities. Let $\mathcal{K}$ be a hypersurface of discontinuity, i.e.~a $C^1$ hypersurface $\mathcal{K}$ with a neighborhood $\mathcal{U}$ such that $T^{\mu\nu}$ and $I^\mu$ are continuous in the closure of each connected component of the complement of $\mathcal{K}$ in $\mathcal{U}$ but are not continuous across $\mathcal{K}$. Let $N_\mu$ be a covector at $x\in\mathcal{K}$, the null space of which is the tangent space of $\mathcal{K}$ at $x$
\begin{align}
  \label{eq:38}
  T_x\mathcal{K}=\{X^\mu\in T_xM:N_\mu X^\mu=0\}.
\end{align}
Then, denoting by $\jump{\cdot}$ the jump across $\mathcal{K}$ at $x$, we have the jump conditions
\begin{align}
  \label{eq:39}
  \jump{T^{\mu\nu}}N_\nu&=0,\\
  \label{eq:40}
  \jump{I^\mu}N_\mu&=0.
\end{align}
These follow from the integral form of the conservation laws \eqref{eq:5}, \eqref{eq:6}. Consider the 3-form $I_{\alpha\beta\gamma}^\ast$ dual to $I^\mu$, that is,
\begin{align}
  \label{eq:41}
  I_{\alpha\beta\gamma}^\ast=I^\mu \varepsilon_{\mu\alpha\beta\gamma},
\end{align}
where $\varepsilon_{\mu\alpha\beta\gamma}$ is the volume 4-form of the Minkowski metric $\eta$. In terms of $I^\ast$ equation \eqref{eq:5} becomes
\begin{align}
  \label{eq:42}
  dI^\ast=0.
\end{align}

Also, given any vector field $X$, we can define the vector field
\begin{align}
  \label{eq:43}
  P^\mu\coloneqq\eta_{\alpha\beta}X^\alpha T^{\beta\mu}.
\end{align}
By virtue of \eqref{eq:6}, $P$ satisfies
\begin{align}
  \label{eq:44}
  \nabla_\mu P^\mu=\frac{1}{2}\pi_{\mu\nu}T^{\mu\nu},
\end{align}
where
\begin{align}
  \label{eq:45}
  \pi_{\mu\nu}=\mathcal{L}_X\eta_{\mu\nu}.
\end{align}
In terms of the 3-form $P^\ast$ dual to $P$
\begin{align}
  \label{eq:46}
  P^\ast_{\alpha\beta\gamma}=P^\mu \varepsilon_{\mu\alpha\beta\gamma},
\end{align}
equation \eqref{eq:44} reads
\begin{align}
  \label{eq:47}
  dP^\ast=\frac{1}{2}(\pi\cdot T)\varepsilon.
\end{align}

Consider now an arbitrary point $x\in\mathcal{K}$ and let $\mathcal{U}$ be a neighborhood of $x$ in Minkowski spacetime. We denote $\mathcal{W}=\mathcal{K}\cap \mathcal{U}$. Let $Y$ be a vector field without critical points in some larger neighborhood $\mathcal{U}_0\supset \mathcal{U}$ and transversal to $\mathcal{K}$. Let $\mathcal{L}_\delta(y)$ denote the segment of the integral curve of $Y$ through $y\in\mathcal{W}$ corresponding to the parameter interval $(-\delta,\delta)$
\begin{align}
  \label{eq:48}
  \mathcal{L}_\delta(y)\coloneqq\{F_s(y):s\in(-\delta,\delta)\},
\end{align}
where $F_s$ is the flow generated by $Y$. We then define the neighborhood $\mathcal{V}_\delta$ of $x$ in Minkowski spacetime by
\begin{align}
  \label{eq:49}
  \mathcal{V}_\delta\coloneqq\bigcup_{y\in\mathcal{W}}\mathcal{L}_\delta(y).
\end{align}
Integrating equations \eqref{eq:42}, \eqref{eq:47} in $\mathcal{V}_\delta$ and applying Stokes' theorem we obtain
\begin{align}
  \label{eq:50}
  \int_{\partial\mathcal{V}_\delta}I^\ast&=0,\\
  \label{eq:51}
  \int_{\partial\mathcal{V}_\delta}P^\ast&=\int_{\mathcal{V}_\delta}\frac{1}{2}(\pi\cdot T)\varepsilon.
\end{align}
Now the boundary of $\mathcal{V}_\delta$ consists of the hypersurfaces
\begin{align}
  \label{eq:52}
  \mathcal{W}_\delta\coloneqq\{F_\delta(y):y\in \mathcal{W}\},\qquad \mathcal{W}_{-\delta}\coloneqq\{F_{-\delta}(y):y\in \mathcal{W}\},
\end{align}
together with the lateral hypersurface
\begin{align}
  \label{eq:53}
  \bigcup_{y\in\partial\mathcal{W}}\mathcal{L}_\delta(y).
\end{align}
Since this lateral component and $\mathcal{V}_\delta$ are bounded in measure by a constant multiple of $\delta$, we take the limit $\delta\rightarrow 0$ in \eqref{eq:50}, \eqref{eq:51} to obtain
\begin{align}
  \label{eq:54}
  \int_{\mathcal{W}}\jump{I^\ast}&=0,\\
  \label{eq:55}
  \int_{\mathcal{W}}\jump{P^\ast}&=0.
\end{align}
That these are valid for any neighborhood $\mathcal{W}$ of $x$ in $\mathcal{K}$ implies that the corresponding 3-forms induced on $\mathcal{K}$ from the two sides coincide at $x$, or, equivalently, that
\begin{align}
  \label{eq:56}
  \jump{I^\mu}N_\mu=0,\qquad \jump{P^\mu}N_\mu=0.
\end{align}
The first of these equations coincides with \eqref{eq:40}, while the second, for four vector fields $X$ constituting at $x$ a basis for $T_xM$, implies \eqref{eq:39}.

\subsection{Determinism and Entropy Condition\label{sec:det_ent_cond}}
By virtue of \eqref{eq:24} only time-like hypersurfaces of discontinuity can arise. Since $T_x\mathcal{K}$ is time-like, the normal vector $N^\mu=(\eta^{-1})^{\mu\nu}N_\nu$ is space-like and we can normalize it to have unit magnitude
\begin{align}
  \label{eq:57}
  \eta_{\mu\nu}N^\mu N^\nu=1.
\end{align}
We must still determine the orientation of $\mathcal{K}$. Let $N^\mu$ point from one side of $T_x\mathcal{K}$, which we label $+$ and which we say is \textit{behind} $T_x\mathcal{K}$, to the other side of $T_x\mathcal{K}$, which we label $-$ and which we say is \textit{ahead} of $T_x\mathcal{K}$. Then for any quantity $q$ we have $\jump{q}=q_+-q_-$. If we define
\begin{align}
  \label{eq:58}
  u_\perp\coloneqq -\eta(u,N),
\end{align}
then the jump condition \eqref{eq:40} reads
\begin{align}
  \label{eq:59}
  n_+u_{\perp +}=n_-u_{\perp -}=:f,
\end{align}
where the quantity $f$ is called \textit{particle flux}. If $f\neq 0$, the discontinuity is called a \textit{shock}. In this case we choose the orientation of $N$ such that $f>0$, that is, the fluid particles cross the hypersurface of discontinuity $\mathcal{K}$ from the state ahead to the state behind (see figure \ref{orientation}).
%% put here a figure
If $f=0$ the discontinuity is called a \textit{contact discontinuity} and in this case the orientation of $N^\mu$ is merely conventional.

\begin{figure}[h!]
\begin{center}
\begin{tikzpicture}

\node (o) at (-3,0) {};
\node (o) at (-1.1864,-0.491) {};
%\node (o) at (-1.5,-1) {};
\node (b) at (1.0736,3.009) {};
%\node (e) at (-1,3) {};
\path [fill=gray!30] (o.center) to [bend right=25] (b.center) to (2.2,3) to (2.2,-0.5) to (o.center);
\path [fill=gray!60] (o.center) to [bend right=25] (b.center) to (-2,3) to (-2,-0.5) to (o.center);
\draw [line width=0.8pt] (o.center) to [bend right=25] (b.center);

\node at (1.2798,2.4045) {$\mathcal{K}$};
\node at (-0.41,3.5421) {$\textrm{behind}$};
\node at (3,1.2) {$\textrm{ahead}$};

\draw   plot[smooth, tension=.7] coordinates {(0.329,0.8989) (0.141,2.0273) (-0.1659,2.9964)};
\draw  plot[smooth, tension=.7] coordinates {(0.6752,1.5246) (0.5487,2.2593) (0.3312,2.9964)};
\draw  plot[smooth, tension=.7] coordinates {(-0.0647,0.4121) (-0.2784,1.735) (-0.7032,3.0014)};
\draw  plot[smooth, tension=.7] coordinates {(-0.5124,-0.0064) (-0.7026,1.4087) (-1.2807,2.9964)};
\draw  plot[smooth, tension=.7] coordinates {(-0.5124,-0.0064) (-0.3283,-0.2786) (-0.1736,-0.4948)};
\draw  plot[smooth, tension=.7] coordinates {(-0.0647,0.4121) (0.3683,-0.3036) (0.5065,-0.5035)};
\draw  plot[smooth, tension=.7] coordinates {(0.329,0.8972) (0.7456,0.1362) (1.2095,-0.4895)};
\draw  plot[smooth, tension=.7] coordinates {(0.6752,1.5246) (1.23,0.4706) (1.9406,-0.4895)};

\draw [->] (-2,-.5) -- (2.7,-.5);
\draw [->] (-2,-0.5) -- (-2,3.5);
\node at (3,-.5) {$x$};
\node at (-2,3.8) {$t$};

\end{tikzpicture}
\end{center}
\caption[Shock orientation]{Four fluid flow lines crossing the shock $\mathcal{K}$ from the state ahead to the state behind.}
\label{orientation}
\end{figure}
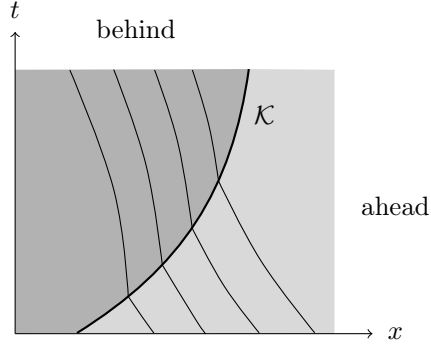

In terms of $v$, the volume per particle, we have
\begin{align}
  \label{eq:60}
  u_{\perp -}=fv_-,\qquad u_{\perp +}=fv_+.
\end{align}
The jump condition \eqref{eq:39} reads
\begin{align}
  \label{eq:61}
  (\rho_++p_+)u_+ u_{\perp +}-p_+N=(\rho_-+p_-)u_- u_{\perp -}-p_-N.
\end{align}
Substituting \eqref{eq:60} into \eqref{eq:61} the latter reduces to
\begin{align}
  \label{eq:62}
  fh_+u_+-p_+N=fh_-u_--p_-N,
\end{align}
where we used \eqref{eq:29}. According to \eqref{eq:62} the vectors $u_+$, $u_-$ and $N$ all lie in the same timelike plane. Taking the $\eta$-inner product of \eqref{eq:62} with $N$ we obtain
\begin{align}
  \label{eq:63}
  fh_+u_{\perp +}+p_+=fh_-u_{\perp -}+p_-.
\end{align}
Substituting from \eqref{eq:60} this becomes
\begin{align}
  \label{eq:64}
  p_+-p_-=-f^2(h_+v_+-h_-v_-).
\end{align}
On the other, taking the $\eta$-inner product of each side of \eqref{eq:62} with itself we obtain
\begin{align}
  \label{eq:65}
  p_+^2-p_-^2=f^2(h_+^2-h_-^2-2h_+v_+p_++2h_-v_-p_-).
\end{align}
Equations \eqref{eq:64} and \eqref{eq:65} together imply whenever $f\neq 0$, as is the case for a shock, the following relation
\begin{align}
  \label{eq:66}
  h_+^2-h_-^2=(p_+-p_-)(h_+v_++h_-v_-).
\end{align}
This is the relativistic \textit{Hugoniot relation}, first derived by A.~Taub \cite{taub}. We note that in the case of a contact discontinuity ($f=0$) \eqref{eq:62} reduces to $p_+=p_-$.

The only shock discontinuities which arise naturally are those which are supersonic relative to the state ahead and subsonic relative to the state behind. We call this the \textit{determinism condition}. The condition that $\mathcal{K}$ is supersonic relative to the state ahead means that, for each $x\in\mathcal{K}$, $N_\mu$ is a time-like covector relative to $g^{-1}_-$, i.e.
\begin{align}
  \label{eq:67}
  (g^{-1}_-)^{\mu\nu}N_\mu N_\nu<0,
\end{align}
while the condition that $\mathcal{K}$ is subsonic relative to the state behind means that, for each $x\in\mathcal{K}$, $N_\mu$ is a space-like covector relative to $g_+^{-1}$, i.e.
\begin{align}
  \label{eq:68}
  (g_+^{-1})^{\mu\nu}N_\mu N_\nu>0.
\end{align}
In view of \eqref{eq:27}, conditions \eqref{eq:67} and \eqref{eq:68} are
\begin{align}
  \label{eq:69}
  u_{\perp -}>\frac{\eta_-}{\sqrt{1-\eta_-^2}},\qquad u_{\perp +}<\frac{\eta_+}{\sqrt{1-\eta_+^2}}.
\end{align}
Substituting from \eqref{eq:60}, these become
\begin{align}
  \label{eq:70}
  f>\frac{\eta_-/v_-}{\sqrt{1-\eta_-^2}},\qquad f<\frac{\eta_+/v_+}{\sqrt{1-\eta_+^2}}.
\end{align}
We conclude that the determinism condition reduces to
\begin{align}
  \label{eq:71}
  \frac{\eta_-/v_-}{\sqrt{1-\eta_-^2}}<\frac{\eta_+/v_+}{\sqrt{1-\eta_+^2}}.
\end{align}
The determinism condition is illustrated in figure \ref{determinism_condition}.

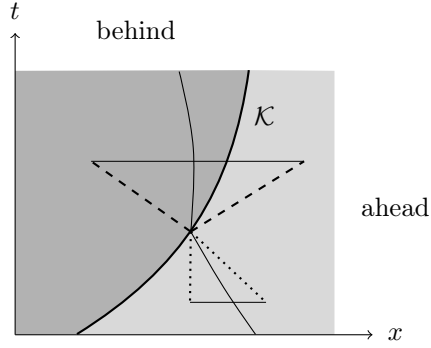
\begin{figure}[h!]
\begin{center}
\begin{tikzpicture}

\node (o) at (-3,0) {};
\node (o) at (-1.1864,-0.491) {};
\node (b) at (1.0736,3.009) {};
\node (x) at (0.3097,0.8673) {};
\path [fill=gray!30] (o.center) to [bend right=25] (b.center) to (2.2,3) to (2.2,-0.5) to (o.center);
\path [fill=gray!60] (o.center) to [bend right=25] (b.center) to (-2,3) to (-2,-0.5) to (o.center);
\draw [line width=0.8pt] (o.center) to [bend right=25] (b.center);
\node at (1.2798,2.4045) {$\mathcal{K}$};
\node at (-0.41,3.5421) {$\textrm{behind}$};
\node at (3,1.2) {$\textrm{ahead}$};
\draw [->] (-2,-.5) -- (2.7,-.5);
\draw [->] (-2,-0.5) -- (-2,3.5);
\node at (3,-.5) {$x$};
\node at (-2,3.8) {$t$};

\draw [dashed,line width=0.8pt] (x.center) to (1.8,1.8);
\draw [dashed, line width=0.8pt] (x.center) to (-1,1.8);
\draw [line width=0.2pt] (-1,1.8) to (1.8,1.8);

\draw [dotted,line width=0.8pt] (x.center) to (1.3,-0.07);
\draw [dotted,line width=0.8pt] (x.center) to (0.3,-0.07);
\draw [line width=0.2pt] (0.3,-0.07) to (1.3,-0.07);

\draw  plot[smooth, tension=.7] coordinates {(x) (0.3432,1.971) (0.1574,2.9903)};

\draw  plot[smooth, tension=.7] coordinates {(x) (0.7449,0.1231) (1.1617,-0.4945)};
\end{tikzpicture}
\end{center}
\caption[Determinism condition]{Illustration of the determinism condition: Sound cones at a point on $\mathcal{K}$. Dotted: Backward sound cone with respect to the state ahead. Dashed: Forward sound cone with respect to the state behind. Thin solid line: fluid flow line.}
\label{determinism_condition}
\end{figure}

We now look at the entropy condition which is
\begin{align}
  \label{eq:72}
  [s]=s_+-s_->0.
\end{align}
In the following we will show the equivalence of the entropy condition to the determinism condition. Since (recall \eqref{eq:3}, \eqref{eq:29})
\begin{align}
  \label{eq:73}
  dh&=de+d(pv)\notag\\
  &=vdp+\theta ds,
\end{align}
the expansion of $\jump{h}=h_+-h_-$ in powers of $\jump{p}$ and $\jump{s}$ is
\begin{align}
  \label{eq:74}
  \jump{h}=v_-\jump{p}+\frac{1}{2}\left(\pp{v}{p}\right)_-\jump{p}^2+\frac{1}{6}\left(\ppp{v}{p}\right)_-\jump{p}^3+\theta_-\jump{s}+\Landau(\jump{s}^2)+\Landau(\jump{p}^4)+\Landau(\jump{p}\jump{s}).
\end{align}
Hence
\begin{align}
  \label{eq:75}
  h_+^2-h_-^2&=2h_-v_-\jump{p}+\left\{h_-\left(\pp{v}{p}\right)_-+v_-^2\right\}\jump{p}^2\notag\\
&\qquad +\left\{\frac{h_-}{3}\left(\ppp{v}{p}\right)_-+v_-\left(\pp{v}{p}\right)_-\right\}\jump{p}^3+2h_-\theta_-\jump{s}+\Landau(\jump{s}^2)+\Landau(\jump{p}^4)+\Landau(\jump{p}\jump{s}).
\end{align}
Also $\jump{v}=v_+-v_-$ is expanded as
\begin{align}
  \label{eq:76}
  \jump{v}=\left(\pp{v}{p}\right)_-\jump{p}+\frac{1}{2}\left(\ppp{v}{p}\right)_-\jump{p}^2+\Landau(\jump{p}^3)+\Landau(\jump{s}).
\end{align}
Hence
\begin{align}
  \label{eq:77}
  (h_+v_++h_-v_-)(p_+-p_-)&=2h_-v_-\jump{p}+\left\{h_-\left(\pp{v}{p}\right)_-+v_-^2\right\}\jump{p}^2\notag\\
&\qquad +\left\{\frac{h_-}{2}\left(\ppp{v}{p}\right)_-+\frac{3v_-}{2}\left(\pp{v}{p}\right)_-\right\}\jump{p}^3+\Landau(\jump{p}^4)+\Landau(\jump{p}\jump{s}).
\end{align}
Comparing \eqref{eq:75} and \eqref{eq:77} with the Hugoniot relation \eqref{eq:66} we conclude
\begin{align}
  \label{eq:78}
  \jump{s}=\frac{1}{12\theta_-h_-}\left\{h_-\left(\ppp{v}{p}\right)_-+3v_-\left(\pp{v}{p}\right)_-\right\}\jump{p}^3+\Landau(\jump{p}^4).
\end{align}

Consider next the condition \eqref{eq:71}. Defining the quantity
\begin{align}
  \label{eq:79}
  q\coloneqq\left(\frac{1}{\eta^2}-1\right)v^2,
\end{align}
the condition \eqref{eq:71} is seen to be equivalent to
\begin{align}
  \label{eq:80}
  \jump{q}<0.
\end{align}
From \eqref{eq:4}, \eqref{eq:17} we have
\begin{align}
  \label{eq:81}
  \frac{1}{\eta^2}=\pp{\rho}{p}=\frac{(\rho+p)}{n}\pp{n}{p}=-\frac{(\rho+p)}{v}\pp{v}{p}.
\end{align}
Hence, in view of \eqref{eq:29} and $v=1/n$, $q$ is given by
\begin{align}
  \label{eq:82}
  q=-h\pp{v}{p}-v^2.
\end{align}
We then obtain
\begin{align}
  \label{eq:83}
  \pp{q}{p}=-h\ppp{v}{p}-3v\pp{v}{p}.
\end{align}
In view of the fact that by \eqref{eq:78} $\jump{s}=\Landau(\jump{p}^3)$, we obtain
\begin{align}
  \label{eq:84}
  \jump{q}=-\left\{h_-\left(\ppp{v}{p}\right)_-+3v_-\left(\pp{v}{p}\right)_-\right\}\jump{p}+\Landau(\jump{p}^2).
\end{align}
Therefore, the condition \eqref{eq:80} is equivalent for suitably small $\jump{p}$ to
\begin{align}
  \label{eq:85}
  \left\{h_-\left(\ppp{v}{p}\right)_-+3v_-\left(\pp{v}{p}\right)_-\right\}\jump{p}>0,
\end{align}
provided that the quantity in the curly bracket is non-zero. This together with \eqref{eq:78} is equivalent to \eqref{eq:72}. We have therefore established, for suitably small $\jump{p}$, the equivalence of the determinism condition \eqref{eq:71} to the entropy condition \eqref{eq:72}.

\begin{remark}
  We will impose the determinism condition in the shock development problem and we will see that this condition is necessary for the solution to be uniquely determined by the data (see the formulation of the shock development problem together with the description of the boundary of the maximal development below).
\end{remark}

\begin{remark}By \eqref{eq:85}, the sign of the coefficient of $\jump{p}$ in \eqref{eq:85} is the same as the sign of $\jump{p}$. Let now $\Sigma$ be defined by
\begin{align}
  \label{eq:734}
  1-h^2\Sigma=\eta^2.
\end{align}
The coefficient of $\jump{p}$ in \eqref{eq:85} can be related to $(d\Sigma/dh)_-$ if the state ahead of the shock is isentropic, as will be the case under consideration. From \eqref{eq:29} we have
\begin{align}
  \label{eq:1193}
  \left(\frac{d\Sigma}{dh}\right)_s=\frac{(d\Sigma/dp)_s}{(dh/dp)_s}=\frac{1}{v}\left(\frac{d\Sigma}{dp}\right)_s,
\end{align}
where we use the subscript $s$ to indicate isentropy. Hence
\begin{align}
  \label{eq:1194}
  vh^2\left(\frac{d\Sigma}{dh}\right)_s=-\left(\frac{d\eta^2}{dp}\right)_s-\frac{2v}{h}(1-\eta^2).
\end{align}
By \eqref{eq:29} we have
\begin{align}
  \label{eq:1195}
  \left(\frac{dh}{dn}\right)_s=\frac{h\eta^2}{n},
\end{align}
which implies
\begin{align}
  \label{eq:1196}
  \frac{1}{\eta^2}=-\frac{h}{v^2}\left(\frac{dv}{dp}\right)_s.
\end{align}
Substituting \eqref{eq:1196} and its derivative with respect to $p$ at constant $s$ in \eqref{eq:1194} we obtain
\begin{align}
  \label{eq:1199}
  -\frac{v^3h^2}{\eta^4}\left(\frac{d\Sigma}{dh}\right)_s=3v\left(\frac{dv}{dp}\right)_s+h\left(\frac{d^2v}{dp^2}\right)_s.
\end{align}
Therefore, if the state ahead is isentropic, the quantity in the curly bracket in \eqref{eq:85} is
\begin{align}
  \label{eq:735}
  h_-\left(\frac{\partial^2v}{\partial p^2}\right)_-+3v_-\left(\pp{v}{p}\right)_-=-\frac{v_-^3h_-^2}{\eta_-^4}\left(\frac{d\Sigma}{dh}\right)_-.
\end{align}
We conclude that the jump in pressure $\jump{p}$ behind the shock is $>0$ or $<0$ according as to whether $(d\Sigma/dh)_-$ is $<0$ or $>0$, at least for suitably small $\jump{p}$.
\end{remark}

\subsection{Barotropic Fluids}
In the barotropic case $p=f(\rho)$ is an increasing function of $\rho$. Therefore,
\begin{align}
  \label{eq:94}
  n\pp{\rho}{n}=p+\rho
\end{align}
is a function of $\rho$, which implies that for a barotropic perfect fluid, $\rho$, and hence also $p$, is a function of the product $\sigma\coloneqq nm$, where $m$ is a function of $s$ alone. In fact
\begin{align}
  \label{eq:95}
  \sigma=nm(s)=\exp\left(\int\frac{d\rho}{\rho+f(\rho)}\right)
\end{align}
and it satisfies
\begin{align}
  \label{eq:96}
  \rho+p=\sigma \frac{d\rho}{d\sigma}.
\end{align}
The positivity of $\theta$ implies that $m$ is a strictly increasing function. Therefore, we can eliminate $s$ in favor of $m$. \eqref{eq:36} becomes
\begin{align}
  \label{eq:97}
  u^\mu\nabla_\mu m=0.
\end{align}

We define
\begin{align}
  \label{eq:98}
  \psi_\mu\coloneqq -\tilde{h}u_\mu,
\end{align}
where
\begin{align}
  \label{eq:99}
  \tilde{h}\coloneqq\frac{h}{m}=\frac{\rho+p}{\sigma}=\frac{d\rho}{d\sigma}.
\end{align}
Comparing \eqref{eq:98} with \eqref{eq:28} we see that
\begin{align}
  \label{eq:1203}
  \beta_\mu=m\psi_\mu.
\end{align}

Defining now
\begin{align}
  \label{eq:100}
  \Omega\coloneqq d\psi,
\end{align}
we obtain
\begin{align}
  \label{eq:101}
  \omega=m\Omega+dm\wedge \psi.
\end{align}
From the second of \eqref{eq:4} we have
\begin{align}
  \label{eq:102}
  \theta ds=\tilde{h}dm.
\end{align}
Therefore
\begin{align}
  \label{eq:103}
  i_u\omega=mi_u\Omega-(\psi\cdot u)dm=mi_u\Omega-\theta ds,
\end{align}
which implies through \eqref{eq:37}
\begin{align}
  \label{eq:104}
  i_u\Omega=0.
\end{align}

From the particle conservation \eqref{eq:35} and the adiabatic condition \eqref{eq:97} we deduce
\begin{align}
  \label{eq:105}
  \nabla_\mu(\sigma u^\mu)=m\nabla_\mu(nu^\mu)+nu^\mu\nabla_\mu m=0,
\end{align}
which, through \eqref{eq:96}, is equivalent to the energy equation \eqref{eq:8}. Therefore, imposing the energy equation \eqref{eq:105} as well as the adiabatic condition \eqref{eq:97} the conservation of particle number follows. We conclude that in the barotropic case the system of equations reduces to the system
\begin{align}
  \label{eq:106}
  \nabla_\mu(\sigma u^\mu)&=0,\\
  \label{eq:107}
  u^\mu\nabla_\mu m&=0,\\
  \label{eq:108}
  i_u\Omega&=0.
\end{align}
The unknowns are $u$, $m$ and $\sigma$. Equation \eqref{eq:107} is decoupled from the other two. We may thus ignore it and consider only the system consisting of \eqref{eq:106}, \eqref{eq:108}.

The irrotational barotropic case is characterized by the existence of a function $\phi$ such that
\begin{align}
  \label{eq:109}
  \psi=d\phi,
\end{align}
which implies
\begin{align}
  \label{eq:110}
  \Omega=0.
\end{align}
Therefore, \eqref{eq:108} is identically satisfied. By \eqref{eq:98},
\begin{align}
  \label{eq:111}
  \psi\cdot X=-\tilde{h}\eta(u,X)>0
\end{align}
whenever $X$ is a future-directed timelike vector. Therefore $\phi$ is a time function. By \eqref{eq:98}, \eqref{eq:109},
\begin{align}
  \label{eq:112}
  H\coloneqq\tilde{h}^2=-(\eta^{-1})^{\mu\nu}\partial_\mu\phi\partial_\nu\phi.
\end{align}
From \eqref{eq:99}
\begin{align}
  \label{eq:113}
  \frac{dH}{d\sigma}=\frac{dH}{d\tilde{h}}\frac{d\tilde{h}}{d\sigma}=\frac{2\tilde{h}}{\sigma}\frac{dp}{d\sigma}=\frac{2H\eta^2}{\sigma},
\end{align}
which implies that $\sigma$ can be expressed as a smooth, strictly increasing function of $H$, i.e.~$\sigma=\sigma(H)$. Defining
\begin{align}
  \label{eq:114}
  G(H)\coloneqq\frac{\sigma(H)}{\sqrt{H}},
\end{align}
equation \eqref{eq:106} becomes
\begin{align}
  \label{eq:115}
  \nabla^\mu(G(H)\partial_\mu\phi)=0,
\end{align}
where $H$ is given by \eqref{eq:112}. Taking into account that (see \eqref{eq:27}, \eqref{eq:98}, \eqref{eq:113})
\begin{align}
  \label{eq:116}
  (g^{-1})^{\mu\nu}=(\eta^{-1})^{\mu\nu}-F\partial^\mu\phi\partial^\nu\phi,
\end{align}
where
\begin{align}
  \label{eq:117}
  F\coloneqq\frac{2}{G}\frac{dG}{dH},
\end{align}
\eqref{eq:115} becomes
\begin{align}
  \label{eq:118}
  (g^{-1})^{\mu\nu}\nabla_\mu\partial_\nu\phi=0.
\end{align}
We note that
\begin{align}
  \label{eq:119}
  \eta^2=\frac{1}{1+HF}.
\end{align}
We also note that in terms of $H$ and $F$, the acoustical metric is given by
\begin{align}
  \label{eq:120}
  g_{\mu\nu}=\eta_{\mu\nu}+\frac{F}{1+HF}\partial_\mu\phi\partial_\nu\phi.
\end{align}

Since \eqref{eq:106} and therefore in the irrotational case \eqref{eq:118} is equivalent to the energy equation \eqref{eq:8}, in the barotropic case we only need to consider the energy-momentum jump conditions \eqref{eq:39}.

%%% Local Variables: 
%%% mode: latex
%%% TeX-master: "./master"
%%% End: 

\section{Setting the Scene\label{chapter_setting}}

\subsection{Nonlinear Wave Equation in Spherical Symmetry}
We choose spherical coordinates $(t,r,\vartheta,\varphi)$. Then $\eta=\textrm{diag}(-1,1,r^2,r^2\sin^2\vartheta)$ and spherically symmetric solutions $\phi=\phi(t,r)$ of equation \eqref{eq:118} satisfy
\begin{align}
  \label{eq:121}
  (g^{-1})^{\mu\nu}\partial_\mu\partial_\nu\phi=-\frac{2}{r}\partial_r\phi,\qquad \textrm{$\mu,\nu=r,t$. }
\end{align}
The radial null vectors $L_\pm$ with respect to the acoustical metric satisfy
\begin{align}
  \label{eq:122}
  g_{\mu\nu}L_\pm^\mu L_\pm^\nu=0,\qquad \textrm{$\mu,\nu=r,t$.}
\end{align}
Using the normalization condition $L^t=1$, we obtain
\begin{align}
  \label{eq:123}
  L_\pm=\partial_t+\frac{v\pm\eta}{1\pm v\eta}\partial_r,
\end{align}
where $\eta$ is the sound speed (see \eqref{eq:17}) and $v$ is the fluid spatial velocity given by
\begin{align}
  \label{eq:124}
  v=\frac{u^r}{u^t}=-\frac{u_r}{u_t}=-\frac{\psi_r}{\psi_t},
\end{align}
where we recall $\psi_\mu=\partial_\mu\phi$. Using the null vectors the inverse acoustical metric can be written as
\begin{align}
  \label{eq:125}
  (g^{-1})^{\mu\nu}=\frac{(g^{-1})^{tt}}{2}(L_+^\mu L_-^\nu+L_-^\mu L_+^\nu).
\end{align}
From \eqref{eq:116} we have
\begin{align}
  \label{eq:126}
  (g^{-1})^{tt}=-1-F(\psi_t^2).
\end{align}
From \eqref{eq:119} in conjunction with \eqref{eq:112} we see that the assumption $\eta^2<1$ is expressed by the condition that $F>0$. The nonlinear wave equation can be written as
\begin{align}
  \label{eq:127}
L_+^\mu L_-^\nu\partial_\mu\partial_\nu\phi=\frac{2}{r(1+F(\psi_t)^2)}\partial_r\phi.
\end{align}

\subsection{Riemann Invariants of the Principal Part}
Keeping only the principal part of \eqref{eq:127} we are left with
\begin{align}
  \label{eq:128}
  L_+^\mu L_-^\nu\partial_\mu\psi_\nu=0.
\end{align}
The Riemann invariants  are defined to be the functions $\alpha(\psi_t,\psi_r)$, $\beta(\psi_t,\psi_r)$ such that
\begin{align}
  \label{eq:129}
  l_{+\mu}\pp{\alpha}{\psi_\mu}=0,\qquad l_{-\mu}\pp{\beta}{\psi_\mu}=0,
\end{align}
where $l_{\pm \mu}$ are the basis 1-forms dual to the basis vector fields $L_{\pm}^\mu$. From \eqref{eq:129} we deduce
\begin{align}
  \label{eq:130}
  \pp{\alpha}{\psi_\mu}=\xi L_-^\mu,\qquad \pp{\beta}{\psi_\mu}=\lambda L_+^\mu
\end{align}
for some functions $\xi$, $\lambda$. Using \eqref{eq:128} we obtain
\begin{align}
  \label{eq:131}
  L_+^\mu\partial_\mu\alpha=L_+^\mu\pp{\alpha}{\psi_\nu}\partial_\mu\psi_\nu=\xi L_+^\mu L_-^\nu \partial_\mu \psi_\nu&=0,\\
\label{eq:132}
L_-^\mu\partial_\mu\beta=L_-^\mu\pp{\beta}{\psi_\nu}\partial_\mu\psi_\nu=\lambda L_-^\mu L_+^\nu \partial_\mu \psi_\nu&=0,
\end{align}
which shows that \eqref{eq:128} is equivalent to the system
\begin{align}
  \label{eq:133}
  L_+\alpha=0,\qquad L_-\beta=0.
\end{align}

We now proceed to determine $\alpha$ and $\beta$. The basis 1-forms dual to the basis vector fields $L_\pm^\mu$ satisfy
\begin{align}
  \label{eq:134}
  l_{+\mu}L_+^\mu=1,\qquad l_{-\mu}L_+^\mu=0,\qquad l_{+\mu}L_-^\mu=0,\qquad l_{-\mu}L_-^\mu=1.
\end{align}
Therefore
\begin{align}
  \label{eq:135}
  l_\pm=\frac{1}{2\eta(1-v^2)}\left[(1\pm v\eta)(\eta\mp v)dt\pm(1-v^2\eta^2)dr\right].
\end{align}
Defining the operators
\begin{align}
  \label{eq:136}
  V_\pm\coloneqq\sum_\mu l_{\pm \mu}\frac{\partial}{\partial \psi_\mu},
\end{align}
\eqref{eq:129} becomes
\begin{align}
  \label{eq:137}
  V_+\alpha=0,\qquad V_-\beta=0.
\end{align}
We introduce the functions $\tilde{h}$, $\zeta$ as coordinates in the positive open cone in the $\psi_t$-$\psi_r$ plane by
\begin{align}
  \label{eq:138}
  \psi_t=\tilde{h}\cosh\zeta,\qquad \psi_r=\tilde{h}\sinh \zeta.
\end{align}
(For $\tilde{h}$ see \eqref{eq:99}, \eqref{eq:112}). Note that by \eqref{eq:124},
\begin{align}
  \label{eq:91}
  v=-\tanh \zeta.
\end{align}
We obtain from \eqref{eq:135}, \eqref{eq:138}
\begin{align}
  \label{eq:139}
  V_\pm=\frac{\cosh\zeta\mp\eta\sinh \zeta}{2\eta \tilde{h}}\left(\eta \tilde{h}\pp{}{\tilde{h}}\pm\pp{}{\zeta}\right).
\end{align}
Defining then the operators
\begin{align}
  \label{eq:140}
  U_\pm\coloneqq\eta \tilde{h}\pp{}{\tilde{h}}\pm\pp{}{\zeta},
\end{align}
\eqref{eq:137} becomes
\begin{align}
  \label{eq:141}
  U_+\alpha=0,\qquad U_-\beta=0.
\end{align}

Let us now define the thermodynamic potential $\tilde{\rho}$ by
\begin{align}
  \label{eq:142}
  \tilde{\rho}\coloneqq\int\frac{d\tilde{h}}{\eta \tilde{h}}.
\end{align}
$\tilde{\rho}$ is defined up to an additional constant. We may fix $\tilde{\rho}$ by setting it equal to zero in the surrounding constant state. Since \eqref{eq:140} takes in terms of $\tilde{\rho}$ the form
\begin{align}
  \label{eq:143}
  U_\pm=\pp{}{\tilde{\rho}}\pm\pp{}{\zeta},
\end{align}
the solutions of \eqref{eq:141} are
\begin{align}
  \label{eq:144}
  \alpha=\tilde{\rho}-\zeta,\qquad\beta=\tilde{\rho}+\zeta,
\end{align}
up to composition on the left with an arbitrary increasing function. Using
\begin{align}
  \label{eq:145}
  \zeta=\arctanh\left(\frac{\psi_r}{\psi_t}\right)=-\arctanh\left(\frac{u^r}{u^t}\right)=-\arctanh\left(v\right)=-\frac{1}{2}\log\left(\frac{1+v}{1-v}\right)
\end{align}
we see that our expressions $\alpha$, $\beta$ agree with (5.16) of \cite{taub} where $\phi$ is in the role of $\tilde{\rho}$ and $u$ is in the role of $v$. $\alpha$ and $\beta$ are the relativistic version of the Riemann invariants introduced in \cite{Rie}.

\subsection{Characteristic System}
In analogy to the equivalence of \eqref{eq:128} and \eqref{eq:133} it follows from \eqref{eq:130} that \eqref{eq:127} is equivalent to the system
\begin{align}
  \label{eq:146}
  L_+\alpha=\frac{2\xi}{r(1+F(\psi_t)^2)}\partial_r\phi,\qquad L_-\beta=\frac{2\lambda}{r(1+F(\psi_t)^2)}\partial_r\phi.
\end{align}
We now proceed to determine $\xi$, $\lambda$ for the choice \eqref{eq:144} of Riemann invariants of the principal part. From $\tilde{h}^2=\psi_t^2-\psi_r^2$ we have
\begin{align}
  \label{eq:147}
  \pp{\tilde{h}}{\psi_t}=\frac{\psi_t}{\tilde{h}}.
\end{align}
From \eqref{eq:145} we obtain
\begin{align}
  \label{eq:148}
  \pp{\zeta}{\psi_t}=-\frac{\psi_r}{\tilde{h}^2}.
\end{align}
Using \eqref{eq:147}, \eqref{eq:148} together with \eqref{eq:142} we obtain
\begin{align}
  \label{eq:149}
  \pp{\alpha}{\psi_t}=\frac{d\tilde{\rho}}{d\tilde{h}}\pp{\tilde{h}}{\psi_t}-\pp{\zeta}{\psi_t}=\frac{1}{H}\left(\frac{\psi_t}{\eta}+\psi_r\right),\\
  \label{eq:150}
  \pp{\beta}{\psi_t}=\frac{d\tilde{\rho}}{d\tilde{h}}\pp{\tilde{h}}{\psi_t}+\pp{\zeta}{\psi_t}=\frac{1}{H}\left(\frac{\psi_t}{\eta}-\psi_r\right).
\end{align}
Using \eqref{eq:130} in the case $\mu=t$ and recalling \eqref{eq:123} we deduce from \eqref{eq:149}, \eqref{eq:150}
\begin{align}
  \label{eq:151}
  \xi=\frac{1}{H}\left(\frac{\psi_t}{\eta}+\psi_r\right),\qquad\lambda=\frac{1}{H}\left(\frac{\psi_t}{\eta}-\psi_r\right).
\end{align}
The system of equations \eqref{eq:146} becomes
\begin{align}
  \label{eq:152}
  L_+\alpha=\frac{2\psi_r}{r\tilde{H}}\left(\frac{\psi_t}{\eta}+\psi_r\right),\qquad L_-\beta=\frac{2\psi_r}{r\tilde{H}}\left(\frac{\psi_t}{\eta}-\psi_r\right),
\end{align}
where
\begin{align}
  \label{eq:153}
  \tilde{H}\coloneqq(1+F(\psi_t)^2)H.
\end{align}

Now we introduce characteristic coordinates $u$, $v$ such that $u=\textrm{const.}$ represents the outgoing and $v=\textrm{const.}$ the incoming characteristic curves. Furthermore, as the characteristic speeds we set $c_\pm\coloneqq L_\pm^r$. It follows that the space-time coordinates $t$, $r$ satisfy
\begin{align}
  \label{eq:154}
  \pp{r}{v}-c_+\pp{t}{v}=0,\qquad \pp{r}{u}-c_-\pp{t}{u}=0.
\end{align}
The system \eqref{eq:152} becomes\footnote{We use
  \begin{align}
    \label{eq:155}
    \pp{}{v}=\pp{t}{v}\pp{}{t}+\pp{r}{v}\pp{}{r}=\pp{t}{v}\left(\pp{}{t}+\frac{\pp{r}{v}}{\pp{t}{v}}\pp{}{r}\right)=\pp{t}{v}\left(\pp{}{t}+c_+\pp{}{r}\right)=\pp{t}{v}L_+.
  \end{align}
Similarly
\begin{align}
  \label{eq:92}
  \pp{}{u}=\pp{t}{u}L_-.
\end{align}

  \label{foot_0}}
\begin{align}
  \label{eq:156}
  \pp{\alpha}{v}=\pp{t}{v}L_+\alpha=\pp{t}{v}\frac{2\psi_r}{r\tilde{H}}\left(\frac{\psi_t}{\eta}+\psi_r\right),\qquad \pp{\beta}{u}=\pp{t}{u}L_-\beta=\pp{t}{u}\frac{2\psi_r}{r\tilde{H}}\left(\frac{\psi_t}{\eta}-\psi_r\right).
\end{align}
Defining
\begin{align}
  \label{eq:157}
  \tilde{A}(\alpha,\beta,r)\coloneqq\frac{2\psi_r}{r\tilde{H}}\left(\frac{\psi_t}{\eta}+\psi_r\right),\qquad \tilde{B}(\alpha,\beta,r)\coloneqq\frac{2\psi_r}{r\tilde{H}}\left(\frac{\psi_t}{\eta}-\psi_r\right),
\end{align}
the characteristic system \eqref{eq:154}, \eqref{eq:156} becomes
\begin{alignat}{3}
  \label{eq:160}
\frac{\p \alpha}{\p v}&=\pp{t}{v}\tilde{A}(\alpha,\beta,r),&\qquad &\frac{\p \beta}{\p u}=\pp{t}{u}\tilde{B}(\alpha,\beta,r),\\
  \label{eq:159}
  \frac{\p r}{\p v}&=\frac{\p t}{\p v}c_+(\alpha,\beta), &&\frac{\p r}{\p u}=\frac{\p t}{\p u}c_-(\alpha,\beta).
\end{alignat}
We note that \eqref{eq:159} is the Hodograph system.
\begin{remark}
The characteristic system is invariant under the conformal map
\begin{align}
  \label{eq:161}
  u\mapsto f(u),\qquad v\mapsto g(v),
\end{align}
where $f$ and $g$ are increasing functions.
\end{remark}
\begin{remark}
  In view of \eqref{eq:112}, \eqref{eq:119}, \eqref{eq:124}, \eqref{eq:153} we can express $\tilde{A}$, $\tilde{B}$ in terms of $r$, the sound speed $\eta$ and the spatial fluid velocity $v$ as
  \begin{align}
    \label{eq:1507}
    \tilde{A}=-\frac{2v\eta}{r(1+v\eta)},\qquad \tilde{B}=-\frac{2v\eta}{r(1-v\eta)}.
  \end{align}
\end{remark}

\subsection{Jump Conditions}
Let $N$ be the unit vector normal to $\mathcal{K}$
\begin{align}
  \label{eq:162}
  N=\frac{1}{\sqrt{1-V^2}}(V\partial_t+\partial_r),
\end{align}
where $V=V(t,r)$ is the shock speed. We define
\begin{align}
  \label{eq:163}
  N'\coloneqq\sqrt{1-V^2}N
\end{align}
and reformulate \eqref{eq:39} as
\begin{align}
  \label{eq:164}
  \jump{T^{\mu\nu}}N'_\nu=0.
\end{align}
In components these are the two jump conditions
\begin{align}
  \label{eq:165}
  -\jump{T^{tt}}V+\jump{T^{tr}}&=0,\\
  \label{eq:166}
  -\jump{T^{rt}}V+\jump{T^{rr}}&=0,
\end{align}
which are equivalent to
\begin{align}
  \label{eq:167}
  V&=\frac{\jump{T^{tr}}}{\jump{T^{tt}}},\\
  \label{eq:168}
  0&=\jump{T^{tt}}\jump{T^{rr}}-\jump{T^{tr}}^2=:J.
\end{align}
Since
\begin{align}
  \label{eq:169}
  u^t=\frac{1}{\sqrt{1-v^2}},\qquad u^r=\frac{v}{\sqrt{1-v^2}},
\end{align}
we obtain from \eqref{eq:7}
\begin{align}
  \label{eq:170}
  T^{tt}=\frac{\rho+p}{1-v^2}-p,\qquad T^{tr}=\frac{(\rho+p)v}{1-v^2},\qquad T^{rr}=\frac{(\rho+p)v^2}{1-v^2}+p.
\end{align}
Using (see \eqref{eq:96}, \eqref{eq:99}, \eqref{eq:112}, \eqref{eq:114}, \eqref{eq:124})
\begin{align}
  \label{eq:171}
  \rho+p=\sigma \tilde{h}=GH=G(\psi_t^2-\psi_r^2)=G\psi_t^2(1-v^2)
\end{align}
the components of the energy-momentum-stress tensor become
\begin{align}
  \label{eq:172}
  T^{tt}=G\psi_t^2-p,\qquad T^{tr}=G\psi_t^2v,\qquad T^{rr}=G\psi_t^2v^2+p.
\end{align}
Let
\begin{align}
  \label{eq:173}
  \gamma\coloneqq\log \tilde{h}-\tilde{\rho}.
\end{align}
\eqref{eq:138} become
\begin{align}
  \label{eq:174}
  \psi_t=\frac{e^\gamma}{2}(e^\beta+e^\alpha),\qquad \psi_r=\frac{e^\gamma}{2}(e^\beta-e^\alpha).
\end{align}
Using (see \eqref{eq:142})
\begin{align}
  \label{eq:175}
  \frac{d\gamma}{d\tilde{\rho}}=\frac{1}{\tilde{h}}\frac{d\tilde{h}}{d\tilde{\rho}}-1=\eta-1,
\end{align}
we get
\begin{align}
  \label{eq:176}
  \pp{\psi_t}{\alpha}=\frac{\psi_t}{2}(\eta+v),\qquad \pp{\psi_t}{\beta}=\frac{\psi_t}{2}(\eta-v),\qquad \pp{\psi_r}{\alpha}=-\frac{\psi_t}{2}(1+v\eta),\qquad\pp{\psi_r}{\beta}=\frac{\psi_t}{2}(1-v\eta).
\end{align}
Now (see \eqref{eq:117}, \eqref{eq:119}, \eqref{eq:142}, \eqref{eq:144})
\begin{align}
  \label{eq:177}
  \pp{G}{\alpha}&=\frac{dG}{dH}\frac{dH}{d\tilde{\rho}}\pp{\tilde{\rho}}{\alpha}\notag\\
&=\frac{1}{2}GFH\eta\notag\\
&=\frac{G}{2\eta}(1-\eta^2).
\end{align}
And similarly
\begin{align}
  \label{eq:178}
  \pp{G}{\beta}=\frac{G}{2\eta}(1-\eta^2).
\end{align}

From \eqref{eq:113}, \eqref{eq:142}, \eqref{eq:144} we have
\begin{align}
  \label{eq:179}
  \pp{p}{\alpha}&=\frac{dp}{dH}\frac{dH}{d\tilde{\rho}}\pp{\tilde{\rho}}{\alpha}\notag\\
&=\frac{1}{2}G\psi_t^2\eta(1-v^2).
\end{align}
And similarly
\begin{align}
  \label{eq:180}
  \pp{p}{\beta}=\frac{1}{2}G\psi_t^2\eta(1-v^2).
\end{align}

From \eqref{eq:124} together with \eqref{eq:176} we obtain
\begin{align}
  \label{eq:181}
  \pp{v}{\alpha}=\frac{1}{2}(1-v^2),\qquad \pp{v}{\beta}=-\frac{1}{2}(1-v^2).
\end{align}
Using \eqref{eq:176}, \eqref{eq:177}, \eqref{eq:178}, \eqref{eq:179}, \eqref{eq:180}, \eqref{eq:181} to compute the partial derivatives of the components of $T$, given by \eqref{eq:172}, we arrive at
\begin{align}
  \label{eq:182}
  \pp{T^{tt}}{\alpha}=\frac{G\psi_t^2}{2\eta}(1+v\eta)^2,\qquad \pp{T^{tr}}{\alpha}=\frac{G\psi_t^2}{2\eta}(v+\eta)(1+v\eta),\qquad \pp{T^{rr}}{\alpha}=\frac{G\psi_t^2}{2\eta}(v+\eta)^2,\\
  \label{eq:183}
  \pp{T^{tt}}{\beta}=\frac{G\psi_t^2}{2\eta}(1-v\eta)^2,\qquad \pp{T^{tr}}{\beta}=\frac{G\psi_t^2}{2\eta}(v-\eta)(1-v\eta),\qquad \pp{T^{rr}}{\beta}=\frac{G\psi_t^2}{2\eta}(v-\eta)^2.
\end{align}
Let us denote $c_\pm\coloneqq L_\pm^r$ (see \eqref{eq:123}). From \eqref{eq:182} we have
\begin{align}
  \label{eq:184}
  \pp{T^{tr}}{\alpha}=c_+\pp{T^{tt}}{\alpha}=\frac{1}{c_+}\pp{T^{rr}}{\alpha},
\end{align}
while from \eqref{eq:183} we have
\begin{align}
  \label{eq:732}
  \pp{T^{tr}}{\beta}=c_-\pp{T^{tt}}{\beta}=\frac{1}{c_-}\pp{T^{rr}}{\beta}.
\end{align}

Let us define
\begin{align}
  \label{eq:186}
  \mu\coloneqq\frac{d\eta}{d\tilde{\rho}}+1-\eta^2.
\end{align}
We note that with $\tilde{\Sigma}$ defined by
\begin{align}
  \label{eq:87}
  1-\tilde{h}^2\Sigma=\eta^2,
\end{align}
we have
\begin{align}
  \label{eq:88}
  \tilde{\Sigma}=m^2\Sigma,
\end{align}
where $\Sigma$ is given by \eqref{eq:734}. So
\begin{align}
  \label{eq:89}
  \left(\frac{d\Sigma}{dh}\right)_s=\frac{1}{m^3}\frac{d\tilde{\Sigma}}{d\tilde{h}}
\end{align}
and
\begin{align}
  \label{eq:90}
  \frac{d\tilde{\Sigma}}{d\tilde{h}}=-\frac{2}{\tilde{h}^3}\mu.
\end{align}

We have the following proposition:
\begin{proposition}\label{proposition_expansion}
  \begin{align}
    \label{eq:185}
    J(\alpha_+,\alpha_-,\beta_+,\beta_-)&=\left(G\psi_t^2(1-v^2)\right)^2\bigg\{\jump{\alpha}\jump{\beta}+\frac{\mu^2}{192\eta^2}\left(\jump{\alpha}^4+\jump{\beta}^4\right)\notag\\
    &\hspace{33mm}+\mathcal{O}\left(\jump{\alpha}^2\jump{\beta}\right)+\mathcal{O}\left(\jump{\alpha}\jump{\beta}^2\right)+\mathcal{O}\left(\jump{\alpha}^5\right)+\mathcal{O}\left(\jump{\beta}^5\right)\bigg\},
  \end{align}
where the coefficients on the right are evaluated at $(\alpha_-,\beta_-)$.
\end{proposition}

\begin{proof}
We prove the Proposition by showing the following statements:
\begin{enumerate}
\item $J$ is symmetric under the interchange of $\alpha$ and $\beta$.
\item
  \begin{align}
    \label{eq:187}
    J(\alpha_-,\alpha_-,\beta_-,\beta_-)&=\pp{J}{\alpha_+}(\alpha_-,\alpha_-,\beta_-,\beta_-)\notag\\
&=\ppp{J}{\alpha_+}(\alpha_-,\alpha_-,\beta_-,\beta_-)\notag\\
&=\frac{\partial^3J}{\partial\alpha_+^3}(\alpha_-,\alpha_-,\beta_-,\beta_-)\notag\\
&=0.
  \end{align}
\item
\begin{align}
    \label{eq:188}
    \pppp{J}{\alpha_+}{\beta_+}(\alpha_-,\alpha_-,\beta_-,\beta_-)=(G\psi_t^2(1-v^2))^2(\alpha_-,\beta_-).
  \end{align}
\item
  \begin{align}
    \label{eq:189}
    \frac{\partial^4J}{\partial\alpha_+^4}(\alpha_-,\alpha_-,\beta_-,\beta_-)=\left(\frac{(G\psi_t^2)^2}{8\eta^2}(1-v^2)^2\mu^2\right)(\alpha_-,\beta_-).
  \end{align}
\end{enumerate}

To check (i) we note that $\rho$ and $p$, like all thermodynamic variables, are functions of $\tilde{\rho}=\frac{1}{2}(\alpha+\beta)$, therefore symmetric under the interchange of $\alpha$ and $\beta$. On the other hand $v=-\psi_r/\psi_t=-\tanh \zeta$ and $\zeta=\frac{1}{2}(\beta-\alpha)$, therefore $v$ is antisymmetric under the interchange of $\alpha$ and $\beta$. It follows from \eqref{eq:168} \eqref{eq:170} that $J$ is symmetric under the interchange of $\alpha$ and $\beta$.

Since $J$ is quadratic in the differences of components of $T$, the first two of (ii) are satisfied. Let
\begin{align}
  \label{eq:190}
  \Delta\coloneqq\pp{T_+^{tt}}{\alpha_+}\pp{T_+^{rr}}{\alpha_+}-\left(\pp{T_+^{tr}}{\alpha_+}\right)^2.
\end{align}
(Where the notation $\partial T_+^{\mu\nu}/\partial \alpha_+=(\partial T^{\mu\nu}/\partial\alpha)(\alpha_+,\beta_+)$ is used). Now, \eqref{eq:182} implies
\begin{align}
  \label{eq:191}
  \Delta=0.
\end{align}
Therefore
\begin{align}
  \label{eq:192}
  \ppp{J}{\alpha_+}=\ppp{T_+^{tt}}{\alpha_+}\jump{T^{rr}}+\jump{T^{tt}}\ppp{T_+^{rr}}{\alpha_+}-2\ppp{T_+^{tr}}{\alpha_+}\jump{T^{tr}},
\end{align}
which implies
\begin{align}
  \label{eq:193}
  \ppp{J}{\alpha_+}(\alpha_-,\alpha_-,\beta_-,\beta_-)=0.
\end{align}
From \eqref{eq:192} we get
\begin{align}
  \label{eq:194}
  \frac{\partial^3J}{\partial\alpha_+^3}=\frac{\partial^3T_+^{tt}}{\partial\alpha_+^3}\jump{T^{rr}}+\jump{T^{tt}}\frac{\partial^3T_+^{rr}}{\partial\alpha_+^3}-2\frac{\partial^3T_+^{tr}}{\partial\alpha_+^3}\jump{T^{tr}}+\pp{\Delta}{\alpha_+},
\end{align}
which, in conjunction with \eqref{eq:191}, implies
\begin{align}
  \label{eq:195}
  \frac{\partial^3J}{\partial\alpha_+^3}(\alpha_-,\alpha_-,\beta_-,\beta_-)=0.
\end{align}

Now we turn to (iii). We have
\begin{align}
  \label{eq:196}
  \pppp{J}{\alpha_+}{\beta_+}(\alpha_-,\alpha_-,\beta_-,\beta_-)=\left(\pp{T^{tt}}{\alpha}\pp{T^{rr}}{\beta}+\pp{T^{tt}}{\beta}\pp{T^{rr}}{\alpha}-2\pp{T^{tr}}{\alpha}\pp{T^{tr}}{\beta}\right)(\alpha_-,\beta_-).
\end{align}
Using \eqref{eq:182}, \eqref{eq:183} we deduce
\begin{align}
  \label{eq:197}
  \pppp{J}{\alpha_+}{\beta_+}(\alpha_-,\alpha_-,\beta_-,\beta_-)=\left(G\psi_t^2(1-v^2)\right)^2(\alpha_-,\beta_-).
\end{align}

Now we turn to (iv). From \eqref{eq:194} we obtain
\begin{align}
  \label{eq:198}
  \frac{\partial^4J}{\partial\alpha_+^4}=\frac{\partial^4T_+^{tt}}{\partial\alpha_+^4}\jump{T^{rr}}+\frac{\partial^3T_+^{tt}}{\partial\alpha_+^3}\pp{T_+^{rr}}{\alpha_+}+\pp{T^{tt}_+}{\alpha_+}\frac{\partial^3T^{rr}_+}{\partial\alpha_+^3}+\jump{T^{tt}}\frac{\partial^4T^{rr}_+}{\partial\alpha_+^4}-2\frac{\partial^4T^{tr}_+}{\partial\alpha_+^4}\jump{T^{tr}}-2\frac{\partial^3T^{tr}_+}{\partial\alpha_+^3}\pp{T^{tr}_+}{\alpha_+}.
\end{align}
From
\begin{align}
  \label{eq:199}
  0=\ppp{\Delta}{\alpha_+}=\frac{\partial^3T_+^{tt}}{\partial\alpha_+^3}\pp{T_+^{rr}}{\alpha_+}+\pp{T_+^{rr}}{\alpha_+}\frac{\partial^3T_+^{rr}}{\partial\alpha_+^3}-2\pp{T_+^{tr}}{\alpha_+}\frac{\partial^3T_+^{tr}}{\partial\alpha_+^3}+2\left(\ppp{T_+^{tt}}{\alpha_+}\ppp{T_+^{rr}}{\alpha_+}-\left(\ppp{T_+^{tr}}{\alpha_+}\right)^2\right),
\end{align}
we deduce
\begin{align}
  \label{eq:200}
  \frac{\partial^4J}{\partial\alpha_+^4}(\alpha_-,\alpha_-,\beta_-,\beta_-)=-2E,
\end{align}
where
\begin{align}
  \label{eq:201}
  E\coloneqq\ppp{T_+^{tt}}{\alpha_+}\ppp{T_+^{rr}}{\alpha_+}-\left(\ppp{T_+^{tr}}{\alpha_+}\right)^2.
\end{align}
To proceed we need expressions for the second derivatives of the components of $T$. Using
\begin{align}
  \label{eq:202}
  \pp{\eta}{\alpha}=\frac{1}{2}\frac{d\eta}{d\tilde{\rho}},
\end{align}
together with \eqref{eq:177}, the first of \eqref{eq:181} and the first of \eqref{eq:176}, it follows from \eqref{eq:182} by a straightforward computation
\begin{align}
  \label{eq:203}
  E=-\left(\frac{G\psi_t^2}{4\eta^2}\right)^2\eta^2(1-v^2)^2\mu^2,
\end{align}
with $\mu$ given by \eqref{eq:186}. Therefore,
\begin{align}
  \label{eq:204}
  \frac{\partial^4J}{\partial\alpha_+^4}(\alpha_-,\alpha_-,\beta_-,\beta_-)=\left(\frac{(G\psi_t^2)^2}{8\eta^2}(1-v^2)^2\mu^2\right)(\alpha_-,\beta_-).
\end{align}
This concludes the proof of the proposition.
\end{proof}

We will use the following proposition.
\begin{proposition}
  Any smooth function $f(x,y)$ can be written as
  \begin{align}
    \label{eq:205}
    f(x,y)=f(x,0)+f(0,y)-f(0,0)+xy\overline{g}(x,y),
  \end{align}
where
\begin{align}
  \label{eq:206}
  g\coloneqq\frac{\partial^2f}{\partial x\partial y}
\end{align}
and $\overline{g}(x,y)$ is the mean value of $g$ in the rectangle $R(x,y)\coloneqq\{(x',y')\in\mathbb{R}^2:0\leq x'\leq x,0\leq y'\leq y\}$.\label{proposition_smooth}
\end{proposition}
\begin{proof}
  Integrating \eqref{eq:206} on the rectangle $R(x,y)$ yields the result.
\end{proof}
We now consider $J(\alpha_+,\alpha_-,\beta_+,\beta_-)$ as a function of $\jump{\alpha}$, $\jump{\beta}$ with given $\alpha_-$, $\beta_-$. We denote this function again by $J$. Using propositions \ref{proposition_expansion}, \ref{proposition_smooth} we can write
\begin{align}
  \label{eq:207}
  J\left(\jump{\alpha},\jump{\beta}\right)=\left(G\psi_t^2(1-v^2)\right)^2\left\{\jump{\alpha}\jump{\beta} M\left(\jump{\alpha},\jump{\beta}\right)+\frac{\mu^2}{192\eta^2}\left(\jump{\alpha}^4L\left(\jump{\alpha}\right)+\jump{\beta}^4N\left(\jump{\beta}\right)\right)\right\},
\end{align}
where the coefficients are evaluated at $(\alpha_-,\beta_-)$. Here $M$, $L$ and $N$ are smooth functions of their arguments and $M(0,0)=L(0)=N(0)=1$.
\begin{proposition}
  Let $f(x,y)$ be a smooth function on $\mathbb{R}^2$ of the form
  \begin{align}
    \label{eq:208}
    f(x,y)=xym(x,y)+x^4l(x)+y^4n(y),
  \end{align}
with $m(0,0)=1$, where $m$, $l$, $n$ are smooth functions. For small enough $x$, the equation
\begin{align}
  \label{eq:209}
  f(x,y)=0
\end{align}
has a unique solution for $y$, given by
\begin{align}
  \label{eq:210}
  y=x^3g(x),
\end{align}
where $g(x)$ is a smooth function and $g(0)=-l(0)$.
\end{proposition}
\begin{proof}
  Setting $y=x^3z$, \eqref{eq:209} becomes
  \begin{align}
    \label{eq:211}
    h(x,z)=0,
  \end{align}
where
\begin{align}
  \label{eq:212}
  h(x,z)\coloneqq zm(x,x^3z)+l(x)+x^8z^4n(x^3z).
\end{align}
Since $m(0,0)=1$, the pair $(x_0,z_0)\coloneqq(0,-l(0))$ satisfies \eqref{eq:211}. Now, since
\begin{align}
  \label{eq:213}
  \pp{h}{z}(x_0,z_0)=m(0,0)=1,
\end{align}
we can apply the implicit function theorem to deduce that there exists a smooth function $g(x)$ such that for small enough $x-x_0$ we have $z=g(x)$ with $g(0)=z_0$. It follows that for small enough $x$, $f(x,y)=0$ has a solution
\begin{align}
  \label{eq:214}
  y=x^3z=x^3g(x),\qquad \textrm{with}\qquad g(0)=-l(0).
\end{align}
\end{proof}
Applying this proposition to $J\left(\jump{\alpha},\jump{\beta}\right)=0$ and taking into account \eqref{eq:207} it follows that there is a smooth function $G\left(\jump{\alpha}\right)$ such that
\begin{align}
  \label{eq:1514}
  \jump{\beta}=\jump{\alpha}^3G\left(\jump{\alpha}\right),\qquad \textrm{with}\qquad G(0)=-\frac{\mu^2}{192\eta^2}.
\end{align}
We recall that above we considered $J(\alpha_+,\alpha_-,\beta_+,\beta_-)$ as a function of $\jump{\alpha}$, $\jump{\beta}$ with given $\alpha_-$, $\beta_-$. In the following we will make use of \eqref{eq:1514} in the form (with a different function $G$)
\begin{align}
  \label{eq:215}
  \jump{\beta}=\jump{\alpha}^3G(\alpha_+,\alpha_-,\beta_-).
\end{align}

\subsection{Boundary of the Maximal Development}

Let initial data be given on a spacelike hypersurface which coincides with the initial data of a constant state outside a bounded domain. According to \cite{ch2007} the boundary of the domain of the maximal solution consists of a regular part $\underline{C}$ and a singular part $\partial_-\mathcal{B}\cup\mathcal{B}$. Each component of $\partial_-\mathcal{B}$ is a smooth, space-like (w.r.t.~the acoustical metric), 2-dimensional submanifold, while the corresponding component of $\mathcal{B}$ is a smooth embedded 3-dimensional submanifold ruled by curves of vanishing arc length (w.r.t.~the acoustical metric), having past end points on the component of $\partial_-\mathcal{B}$. The corresponding component of $\underline{C}$ is the incoming null (w.r.t.~the acoustical metric) hypersurface associated to the component of $\partial_-\mathcal{B}$. It is ruled by incoming null geodesics of the acoustical metric with past end points on the component of $\partial_-\mathcal{B}$. The result of \cite{ch2007} holds for a general equation of state.

In the following we will restrict ourselves to the barotropic case. We also assume the initial data to be spherically symmetric. Therefore, also the solution is spherically symmetric and it suffices to study the problem in the $t$-$r$-plane, where $t$, $r$ are part of the standard spherical coordinates $(t,r,\vartheta,\varphi)$.

In the $t$-$r$-plane the boundary of the maximal development corresponds to a curve consisting of a regular part $\underline{C}$ and a singular part $\partial_-\mathcal{B}\cup\mathcal{B}$. Each component of $\mathcal{B}$ corresponds to a smooth curve of vanishing arc length with respect to the induced acoustical metric, having as its past end point the point corresponding to $\partial_-\mathcal{B}$. The corresponding component of $\underline{C}$ corresponds in the $t$-$r$-plane to an incoming null geodesic with respect to the induced acoustical metric with past end point being the point corresponding to $\partial_-\mathcal{B}$. We denote this point by $O$. See figure \ref{position} on the right.

In the following we will use $(t,w)$ as the acoustical coordinates (in contrast to \cite{ch2007}, where $(t,u)$ are playing the corresponding roles). We recall that the level sets of $w$ are the outgoing characteristic hypersurfaces with respect to the acoustical metric. The solution in the maximal development is a smooth solution with respect to the acoustical coordinates. In terms of these coordinates the solution also extends smoothly to the boundary. We recall the function $\mu$ which plays a central role in \cite{ch2007}, given by
\begin{align}
  \label{eq:216}
  \frac{1}{\mu}=-(g^{-1})^{\mu\nu}\partial_\mu t\partial_\nu w.
\end{align}
(See (2.13) of \cite{ch2007}). $\mu$ vanishes on the singular part of the boundary. On the other hand, $\mu$ is positive on the regular part $\underline{C}$ and the solution extends smoothly to this part also in the $(t,r)$ coordinates.

We now show that
\begin{align}
  \label{eq:217}
  \mu=-\eta\pp r w.
\end{align}
We use the vector field $T$, given in acoustical coordinates by (cf. (2.31) of \cite{ch2007})
\begin{align}
  \label{eq:218}
  T\coloneqq\pp{}{w}.
\end{align}
We have
\begin{align}
  \label{eq:219}
  Tr=\pp{r}{w},\qquad Tt=\pp{t}{w}=0.
\end{align}
Therefore
\begin{align}
  \label{eq:220}
  T=\pp{r}{w}\pp{}{r}.
\end{align}
Now we use the function $\kappa$ as defined by (2.24) of \cite{ch2007}
\begin{align}
  \label{eq:221}
  \kappa\coloneqq g(T,T)=g_{rr}(T^r)^2>0.
\end{align}
So
\begin{align}
  \label{eq:222}
  \pp{r}{w}=-\frac{\kappa}{\sqrt{g_{rr}}}<0.
\end{align}
The minus sign appears due to the initial condition
\begin{align}
  \label{eq:223}
  r(0,w)=-w+k,
\end{align}
where $k$ is a positive constant (see page 39 of \cite{ch2007}). We now recall the function $\alpha$ (see (2.41) of \cite{ch2007})
\begin{align}
  \label{eq:224}
  \frac{1}{\alpha^2}=-(g^{-1})^{\mu\nu}\partial_\mu t\partial_\nu t
\end{align}
and the relation (see (2.48) of \cite{ch2007})
\begin{align}
  \label{eq:225}
  \mu=\alpha\kappa.
\end{align}
Since (see \eqref{eq:116})
\begin{align}
  \label{eq:226}
  \alpha^2=-\frac{1}{(g^{-1})^{tt}}=\frac{1}{1+F(\partial_t\phi)^2},
\end{align}
and (see \eqref{eq:119}, \eqref{eq:120})
\begin{align}
  \label{eq:227}
  g_{rr}=1+F\eta^2(\partial_r\phi)^2,
\end{align}
we obtain
\begin{align}
  \label{eq:228}
  \alpha^2 g_{rr}=\eta^2.
\end{align}
Together with \eqref{eq:225} we arrive at
\begin{align}
  \label{eq:229}
  \frac{\kappa^2}{g_{rr}}=\frac{\mu^2}{\alpha^2g_{rr}}=\frac{\mu^2}{\eta^2},
\end{align}
which, in conjunction with \eqref{eq:222}, implies \eqref{eq:217}.
\begin{remark}
  The acoustical metric $h_{\mu\nu}$ as introduced in \cite{ch2007} coincides with $g_{\mu\nu}$, but quantities such as $\phi$, $\beta$, $F$ used in \cite{ch2007} do not coincide with the quantities denoted in the same way which were introduced in the present work. Making a distinction by putting a tilde on the quantities from \cite{ch2007} we have for example
  \begin{align}
    \label{eq:230}
    \tilde{\beta}_\mu=m\beta_\mu,
  \end{align}
where $\tilde{\beta}$ is introduced in (1.44) of \cite{ch2007}, while $\beta$ is the one form defined in \eqref{eq:28}. For $m$ see \eqref{eq:95}. Therefore, despite the fact that the wave equations of the present work and \cite{ch2007} have the same form, the physical meaning of the wave function is different. Nevertheless, functions such as $\alpha$, $\kappa$, $\mu$ and relations thereof such as \eqref{eq:225} are related only to the Lorentzian geometry given by the acoustical metric and can therefore be used in the present context as well.
\end{remark}
In the following we restrict ourselves to one component of $\partial_-\mathcal{B}\cup\mathcal{B}$ and the corresponding component of $\underline{C}$ with past end point $\partial_-\mathcal{B}$ which we denote by $O$. Now, the function $\mu$ vanishes on $O\cup\mathcal{B}$. From \eqref{eq:217} together with $\eta>0$, it follows that $\partial r/\partial w$ vanishes on $O\cup\mathcal{B}$. In particular
\begin{align}
  \label{eq:231}
  \left(\pp{r}{w}\right)_0=0,
\end{align}
where the index $0$ denotes evaluation at the cusp point $O$. Let the singular part of the boundary of the maximal solution be given by $t=t_\ast(w)$ and let us set $w_0=0$, $t_0=0$, i.e.~the cusp point $O$ is the origin of the acoustical coordinate system. In spherical symmetry, the results at the end of Chapter 15 of \cite{ch2007} translate into
\begin{align}
  \label{eq:232}
  t_\ast(w)=t_0+\frac{1}{2}aw^2+\mathcal{O}(w^3),
\end{align}
\begin{align}
  \label{eq:233}
  a=-\left(\frac{\partial^2\mu/\partial w^2}{\partial \mu/\partial t}\right)_0>0.
\end{align}
Using \eqref{eq:217} and \eqref{eq:231} we obtain
\begin{align}
  \label{eq:234}
  \left(\pp{\mu}{t}\right)_0=-\eta_0\kappa,
\end{align}
where we defined
\begin{align}
  \label{eq:235}
  \kappa\coloneqq\left(\frac{\partial^2r}{\partial w\partial t}\right)_0.
\end{align}

From Chapter 15 of \cite{ch2007} we have
\begin{align}
  \label{eq:236}
  \left(\pp{\mu}{t}\right)_0<0.
\end{align}
It follows that $\kappa>0$. Since $\mu(t_\ast(w),w)=0$, we obtain
\begin{align}
  \label{eq:237}
  \left(\frac{d\mu}{dw}\right)_0=\left(\pp{\mu}{t}\right)_0\left(\frac{dt_\ast}{dw}\right)_0+\left(\pp{\mu}{w}\right)_0=0,
\end{align}
which, using \eqref{eq:232}, implies
\begin{align}
  \label{eq:238}
  \left(\pp{\mu}{w}\right)_0=0.
\end{align}
Taking the partial derivative of \eqref{eq:217} with respect to $w$ and evaluating at the cusp point yields
\begin{align}
  \label{eq:239}
  \left(\frac{\partial^2r}{\partial w^2}\right)_0=0,
\end{align}
where we used \eqref{eq:231}. Taking the second partial derivative of \eqref{eq:217} with respect to $w$ and evaluating at the cusp point we obtain
\begin{align}
  \label{eq:240}
  \left(\frac{\partial^2\mu}{\partial w^2}\right)_0=\frac{\eta_0\lambda}{\kappa},
\end{align}
where we defined
\begin{align}
  \label{eq:241}
  \lambda\coloneqq -\kappa\left(\frac{\partial^3r}{\partial w^3}\right)_0.
\end{align}

From Chapter 15 of \cite{ch2007} we have
\begin{align}
  \label{eq:242}
  \left(\frac{\partial^2\mu}{\partial w^2}\right)_0>0,
\end{align}
which implies that $\lambda>0$. From \eqref{eq:232}, \eqref{eq:233}, \eqref{eq:234}, \eqref{eq:240} we deduce that $\mathcal{B}$, i.e.~the singular part of the boundary of the maximal development, is given in a neighborhood of the cusp point by
\begin{align}
  \label{eq:243}
  t_\ast(w)=t_0+\frac{\lambda}{2\kappa^2}w^2+\mathcal{O}(w^3).
\end{align}
In the following we will also make use of the definition
\begin{align}
  \label{eq:244}
  \xi\coloneqq\kappa\left(\frac{\partial^4r}{\partial w^4}\right)_0.
\end{align}
We summarize the behavior of the radial coordinate at the cusp point.
\begin{align}
  \label{eq:245}
  \left(\pp{r}{w}\right)_0=\left(\frac{\partial^2r}{\partial w^2}\right)_0=0,\qquad \left(\frac{\partial^3r}{\partial w^3}\right)_0<0,\qquad \left(\frac{\partial^2r}{\partial w\partial t}\right)_0>0.
\end{align}
We made the definitions
\begin{align}
  \label{eq:246}
  \kappa\coloneqq\left(\frac{\partial^2r}{\partial w\partial t}\right)_0,\qquad \lambda\coloneqq -\kappa\left(\frac{\partial^3r}{\partial w^3}\right)_0.
\end{align}
The boundary of the domain of the maximal solution close to a cusp point is shown in Figure \ref{position}.

\begin{figure}[h!]
\begin{center}
\subfloat{
\begin{tikzpicture}
\shade[top color=gray!30, bottom color=gray!30]
 (-1,0.5) parabola bend (0,0) (0,0) -- (0,0) parabola bend (0,0) (2.5,2.5) |- (-1,-1);
\draw [line width=0.8pt](-1,0.5) parabola bend (0,0) (0,0) parabola bend (0,0) (2.5,2.5);
\draw [dashed] (-2.5,0) -- (2.5,0);
\draw [<-] (-1.5,-1) -- (2.5,-1);
\draw [->] (2.5,-1) -- (2.5,3);
\node at (-1.8,-1) {$w$};
\node at (2.5,3.3) {$t$};
\node at (-1.9,0.2) {$t=t_0$};
\draw [dashed] (0,-1) -- (0,2);
\node at (0,2.2) {$w=0$};
\node at (1.8,1.8) {$\mathcal{B}$};
\draw (0,0) circle (0.4mm);
\fill (0,0) circle (0.4mm);
\node at (0.2,-0.2) {$O$};
\node at (-0.65,0.6) {$\underline{C}$};
\end{tikzpicture}
}
\qquad\quad\quad
\subfloat{
\begin{tikzpicture}
\node (o) at (0,0) {};
\node (a) at (1,0.8) {};
\node (b) at (1.3,3) {};
\node (c) at (1.7,1.7) {};
\node (d) at (0.7,1.2) {};
\node (e) at (-2,1.8) {};
\node (f) at (-0.3,2.2) {};
\node (g) at (-0.6,0.3) {};
\node (h) at (-1.5,-0.5) {};
\node (x) at (-1.5,-0.5) {};
\path [fill=gray!30] (o.center) to [bend right=25] (b.center) to (2.2,3) to (2.2,-1) to (-2,-1) to (-2,1.8) to [bend left=12] (o.center);
\draw [line width=0.8pt] (o.center) to [bend right=25] (b.center);
\draw [line width=0.8pt] (o.center) to [bend right=12] (e.center);
\node at (0.2,-0.2) {$O$};
\draw (o.center) circle (0.4mm);
\fill (o.center) circle (0.4mm);
\node at (1.3,1.5) {$\mathcal{B}$};
\draw [->] (-2,-1) -- (2.7,-1);
\draw [->] (-2,-1) -- (-2,3);
\node at (3,-1) {$r$};
\node at (-2,3.3) {$t$};
\node at (-1,1.55) {$\underline{C}$};
\end{tikzpicture}
}
\end{center}
\caption[Maximal development, comparison]{Left: Part of the maximal development in acoustical coordinates. Right: Part of the maximal development as a subset of spacetime. $\underline{C}$ denotes the incoming characteristic originating at the cusp point $O$ while $\mathcal{B}$ denotes the singular part of the boundary of the maximal development. Both figures show the maximal development just in a neighborhood of a cusp point.}
\label{position}
\end{figure}
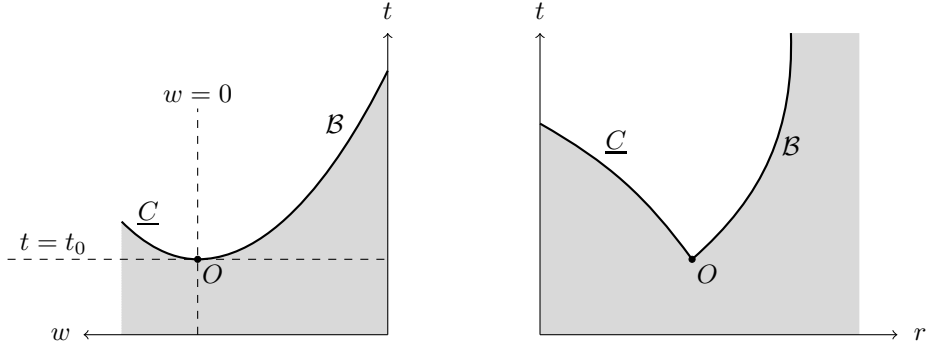

%\begin{remark}
%  It would be interesting to do the formation problem in the spherically symmetric case.
%\end{remark}

\subsubsection{Behavior of $\alpha$ and $\beta$ at the Cusp Point}

In \cite{ch2007} the null vector fields $L$, $\underline{L}$ are used. In the $t$-$r$-plane they are given in terms of acoustical coordinates by (see page 933 of \cite{ch2007})
\begin{align}
  \label{eq:247}
  L=\pp{}{t},\qquad \underline{L}=\frac{\mu}{\alpha^2}\pp{}{t}+2\pp{}{w}.
\end{align}
Therefore,
\begin{align}
  \label{eq:248}
  L_+=L=\pp{}{t},\qquad L_-=\frac{\alpha^2}{\mu}\underline{L}=\pp{}{t}+\frac{2\alpha^2}{\mu}\pp{}{w}.
\end{align}
Now, since the solution is smooth with respect to the acoustical coordinates $(t,w)$ and the Riemann invariants are given smooth functions of $\psi_\mu$ we have
\begin{align}
  \label{eq:249}
  \cp{L_+\alpha}<\infty,\qquad \cp{L_+\beta}<\infty.
\end{align}
\begin{remark}
  We note that in \eqref{eq:247}, \eqref{eq:248} (and also in \eqref{eq:253} and in the second line of \eqref{eq:256} below) $\alpha$ refers to the quantity given by \eqref{eq:224} as in \cite{ch2007}. However in \eqref{eq:249} and everywhere else in the present work $\alpha$ denotes the Riemann invariant defined by the first of \eqref{eq:144}. Also in \eqref{eq:251}, \eqref{eq:252}, \eqref{eq:256} below $\kappa$ denotes, as in \cite{ch2007}, the inverse spatial density of the outgoing characteristic hypersurfaces, defined in \eqref{eq:221} while everywhere else in the present work it denotes the quantity defined in \eqref{eq:235}.
\end{remark}
Now we look at the partial derivative of $\alpha$ and $\beta$ with respect to $w$. From the second of \eqref{eq:130}
\begin{align}
  \label{eq:250}
  \pp{\beta}{w}=T\beta=\pp{\beta}{\psi_\mu}T\psi_\mu=\lambda L_+^\mu T\psi_\mu.
\end{align}
Let us now use the vector field $\hat{T}$, collinear and in the same sense as $T$ and of unit magnitude with respect to the acoustical metric (see (2.57) of \cite{ch2007})
\begin{align}
  \label{eq:251}
  \hat{T}=\kappa^{-1}T.
\end{align}
Using now $X^\mu Y\psi_\mu=Y^\mu X\psi_\mu$ (recall that $\psi_\mu=\partial_\mu\phi$), we deduce from \eqref{eq:250}
\begin{align}
  \label{eq:252}
  \pp{\beta}{w}=\kappa f,
\end{align}
where $f$ is a smooth function of $(t,w)$. Therefore, in conjunction with \eqref{eq:225},
\begin{align}
  \label{eq:253}
  L_-\beta=\pp{\beta}{t}+2\alpha f\pp{\beta}{w},
\end{align}
which implies
\begin{align}
  \label{eq:254}
  \cp{L_-\beta}<\infty.
\end{align}
Now,
\begin{align}
  \label{eq:255}
  \pp{\alpha}{w}=T\alpha=\pp{\alpha}{\psi_\mu}T\psi_\mu=\xi L_-^\mu T\psi_\mu,
\end{align}
where we used the first of \eqref{eq:130}. Using again the vector field $\hat{T}$ we get
\begin{align}
  \label{eq:256}
  \pp{\alpha}{w}&=\xi \kappa \hat{T}^\mu L_-\psi_\mu\notag\\
&=\xi \hat{T}^\mu\left(\kappa\pp{\psi_\mu}{t}+2\alpha\pp{\psi_\mu}{w}\right),
\end{align}
where we used \eqref{eq:225}. Substituting this in $L_-\alpha$ we see that $\cp{L_-\alpha}$ blows up.

We note that from \eqref{eq:252} together with \eqref{eq:238} and the relation \eqref{eq:225} we have
\begin{align}
  \label{eq:257}
  \cp{\pp{\beta}{w}}=0,\qquad \cp{\ppp{\beta}{w}}=0,
\end{align}
while from \eqref{eq:256} we have
\begin{align}
  \label{eq:258}
  \cp{\pp{\alpha}{w}}<\infty.
\end{align}

\subsubsection{Incoming Characteristic Originating at the Cusp Point}
Let in acoustical coordinates $\underline{C}$ be given by $w=\underline{w}(t)$. Setting $\underline{r}(t)=r(t,\underline{w}(t))$, we obtain
\begin{align}
  \label{eq:259}
  \frac{d}{dt}\underline{r}(t)=c_-(t,\underline{w}(t)).
\end{align}
Since
\begin{align}
  \label{eq:260}
  c_-=\frac{d\underline{r}}{dt}=\pp{r}{t}+\pp{r}{w}\frac{d\underline{w}}{dt},
\end{align}
we have
\begin{align}
  \label{eq:261}
  \frac{d\underline{w}}{dt}=\frac{c_--c_+}{\pp{r}{w}}.
\end{align}
Therefore, the inverse function $\underline{t}(w)$ satisfies
\begin{align}
  \label{eq:262}
  \frac{d\underline{t}}{dw}=\frac{-\pp{r}{w}}{c_+-c_-}.
\end{align}
Using \eqref{eq:245} we deduce
\begin{align}
  \label{eq:263}
  \left(\frac{d\underline{t}}{dw}\right)_0=0.
\end{align}

Taking the derivative of \eqref{eq:262} with respect to $w$ we obtain from \eqref{eq:245}, in view of \eqref{eq:263}, that
\begin{align}
  \label{eq:264}
  \left(\frac{d^2\underline{t}}{dw^2}\right)_0=0.
\end{align}
Taking a second derivative of \eqref{eq:262} and using \eqref{eq:245}, \eqref{eq:246} yields
\begin{align}
  \label{eq:265}
  \left(\frac{d^3\underline{t}}{dw^3}\right)_0=\frac{\lambda}{\kappa(c_{+0}-c_{-0})}.
\end{align}
Taking the third derivative of \eqref{eq:262} and evaluating the result at the cusp point yields, in conjunction with the above results
\begin{align}
  \label{eq:266}
  \left(\frac{d^4\underline{t}}{dw^4}\right)_0=-\frac{\lambda(4\kappa-3l)}{\kappa(c_{+0}-c_{-0})^2}-\frac{\xi}{\kappa(c_{+0}-c_{-0})},
\end{align}
where we used the definition \eqref{eq:244} and we defined
\begin{align}
  \label{eq:267}
  l:=\left(\frac{dc_-}{dw}\right)_0.
\end{align}
We conclude from \eqref{eq:263}, \eqref{eq:264}, \eqref{eq:265} and \eqref{eq:266} that
\begin{align}
  \label{eq:268}
  \underline{t}(w)=t_0+\frac{\lambda}{\kappa(c_{+0}-c_{-0})}\frac{w^3}{6}-\left(\frac{\lambda(4\kappa-3l)}{\kappa(c_{+0}-c_{-0})^2}+\frac{\xi}{\kappa(c_{+0}-c_{-0})}\right)\frac{w^4}{24}+\mathcal{O}\left(w^5\right).
\end{align}

The function $\alpha$ along $\underline{C}$ is given by
\begin{align}
  \label{eq:2020}
  \underline{\alpha}(w)=\alpha(\underline{t}(w),w).
\end{align}
Taking into account \eqref{eq:263}, \eqref{eq:264} we obtain
\begin{align}
  \label{eq:2021}
  \cp{\frac{d\underline{\alpha}}{dw}}&=\cp{\pp{\alpha}{w}},\\
  \cp{\frac{d^2\underline{\alpha}}{dw^2}}&=\cp{\ppp{\alpha}{w}}.
\end{align}
Defining
\begin{align}
  \label{eq:2022}
  \dot{\alpha}_0:=\cp{\pp{\alpha}{w}},\qquad \ddot{\alpha}_0:=\cp{\ppp{\alpha}{w}},
\end{align}
we have
\begin{align}
  \label{eq:2028}
  \underline{\alpha}(w)=\alpha_0+\dot{\alpha}_0w+\frac{1}{2}\ddot{\alpha}_0w^2+\Landau(w^3).
\end{align}

For the function $\beta$ along $\underline{C}$ given by
\begin{align}
  \label{eq:2024}
  \underline{\beta}(w)=\beta(\underline{t}(w),w),
\end{align}
we find, taking into account \eqref{eq:257},
\begin{align}
  \label{eq:2025}
  \underline{\beta}(w)=\beta_0+\Landau(w^3).
\end{align}
Now, since
\begin{align}
  \label{eq:2026}
  c_+(\underline{\alpha}(w),\underline{\beta}(w))=\pp{r}{t}(\underline{t}(w),w),
\end{align}
applying $d/dw$ to this, evaluating at $w=0$ and using \eqref{eq:235}, \eqref{eq:257}, \eqref{eq:263} and \eqref{eq:2021} we obtain
\begin{align}
  \label{eq:2027}
  \cp{\pp{c_+}{\alpha}}\dot{\alpha}_0=\kappa.
\end{align}

\subsection{Shock Development Problem}

The notion of maximal development of the initial data is reasonable from the mathematical point of view and also the correct notion from the physical point of view up to $\underline{C}\cup \partial_-\mathcal{B}$. However, it is not the correct notion from the physical point of view up to $\mathcal{B}$. Let us consider a given component of $\mathcal{B}$ which we again denote by $\mathcal{B}$. Its past end point we denote by $O$ (this corresponds to $\partial_-\mathcal{B}$). We also consider the corresponding component of $\underline{C}$, i.e.~the incoming null curve originating at $O$ which we again denote by $\underline{C}$ (see figure \ref{position} on the right).

The shock development problem is the following:

Find a timelike curve $\mathcal{K}$ in the $t$-$r$ plane, lying in the past of $\mathcal{B}$ and originating at $O$, together with a solution of the equations of motion in the domain in Minkowski spacetime bounded in the past by $\mathcal{K}$ and $\underline{C}$, such that the data induced by this solution on $\underline{C}$ coincides with the data induced by the prior maximal solution, while across $\mathcal{K}$ the new solution displays jumps relative to the prior maximal solution, jumps which satisfy the jump conditions. The past of $\mathcal{K}$, where the prior maximal solution holds, is called the state ahead, and the future of $\mathcal{K}$, where the new solution holds, is called the state behind (see \ref{sec:det_ent_cond}). $\mathcal{K}$ is to be space-like relative to the acoustical metric induced by the maximal solution and time-like relative to the new solution which holds in the future of $\mathcal{K}$. The requirement in the last sentence is the determinism condition.

Let $T_\varepsilon$ be the subset bounded by $\underline{C}$, $\mathcal{K}$ and the outgoing characteristic originating at the point on $\underline{C}$ with acoustical coordinate $w=\varepsilon>0$. In $T_\varepsilon$ we use characteristic coordinates. We first shift the origin of the $(t,w)$ coordinate plane so that the cusp point $O$ has coordinates $(0,0)$. We then assign to a point in $T_\varepsilon$ the coordinates $(u,v)$ if it lies on the outgoing characteristic which intersects $\underline{C}$ at the point $w=u$ and on the incoming characteristic which intersects $\mathcal{K}$ at the point where the outgoing characteristic through the point $w=v$ on $\underline{C}$ intersects $\mathcal{K}$. It follows that (see figure \ref{situation})
\begin{align}
  \label{eq:269}
  T_\varepsilon=\left\{(u,v)\in \mathbb{R}^2:0\leq v\leq u\leq \varepsilon\right\}.
\end{align}

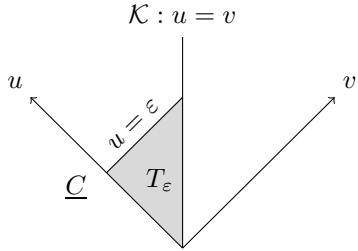
\begin{figure}[h!]
\begin{center}
\begin{tikzpicture}
\filldraw [gray!30] (0,0) -- (0,2) -- (-1,1);
\draw [->](0,0) -- (-2,2);
\draw [->](0,0) -- (2,2);
\node at (-2.2,2.2) {$u$};
\node at (2.2,2.2) {$v$};
\draw (0,0) -- (0,2.8);
\node at (0,3.1) {$\mathcal{K}:u=v$};
\draw (-1,1) -- (0,2) node[midway,sloped,above] {$u=\varepsilon$};
\node at (-1.4,0.8) {$\underline{C}$};
\node at (-0.3,0.9) {$T_\varepsilon$};
\end{tikzpicture}
\end{center}
\caption{The domain $T_\varepsilon$.}
\label{situation}
\end{figure}

\begin{remark}
  We note that to set up the characteristic coordinates in this way we have to a priori assume that the solution is smooth in these coordinates. This is shown to be true below.
\end{remark}
In the following we will denote $\alpha$, $\beta$ and $r$ corresponding to the solution in the maximal development by $\alpha^\ast$, $\beta^\ast$ and $r^\ast$ to distinguish them from $\alpha$, $\beta$, $r$ which we use in referring to the solution in $T_\varepsilon$. The quantities corresponding to the prior maximal solution are expressed in $(t,w)$ coordinates. The solution in $T_\varepsilon$ has to satisfy the characteristic system (see \eqref{eq:160}, \eqref{eq:159})
\begin{alignat}{3}
  \label{eq:270}
  \pp{\alpha}{v}&=\pp{t}{v}\tilde{A}(\alpha,\beta,r),\qquad &\pp{\beta}{u}&=\pp{t}{u}\tilde{B}(\alpha,\beta,r),\\
  \label{eq:271}
  \pp{r}{v}&=\pp{t}{v}c_+(\alpha,\beta),& \pp{r}{u}&=\pp{t}{u}c_-(\alpha,\beta),
\end{alignat}
together with initial data (for $\underline{t}$ see \eqref{eq:268})
\begin{align}
  \label{eq:158}
  t(u,0)&=h(u)\coloneqq \underline{t}(u),\\
  \alpha(u,0)&=\alpha^\ast(h(u),u)=:\alpha_i(u)
\end{align}
and
\begin{align}
  \label{eq:1347}
  r(0,0)&=r^\ast(0,0)=:r_0,\\
  \label{eq:1518}
  \beta(0,0)&=\beta^\ast(0,0)=:\beta_0.
\end{align}
The system consisting of the second of \eqref{eq:270} and the second of \eqref{eq:271} together with \eqref{eq:1347}, \eqref{eq:1518} constitutes, at $v=0$, a system of ordinary differential equations for $\beta$ and $r$. Hence the above conditions on $\beta$ and $r$ at $O$ imply that the data for $\beta$ and $r$ along $\underline{C}$ coincide with the data induced by the prior maximal solution.

Let
\begin{align}
  \label{eq:272}
  f(v)\coloneqq t(v,v),\qquad g(v)\coloneqq r(v,v)-r_0.
\end{align}
Condition \eqref{eq:168} is
\begin{align}
  \label{eq:273}
  J(\alpha_-(v),\alpha_+(v),\beta_-(v),\beta_+(v))=0,
\end{align}
where
\begin{align}
  \label{eq:274}
  \alpha_+(v)=\alpha(v,v),\qquad \beta_+(v)=\beta(v,v),
\end{align}
the right hand sides given by the solution in $T_\varepsilon$ and
\begin{align}
  \label{eq:275}
  \alpha_-(v)=\alpha^\ast(f(v),z(v)),\qquad \beta_-(v)=\beta^\ast(f(v),z(v)),
\end{align}
the right hand sides given by the solution in the maximal development, where $z(v)$ is the solution of the identification equation
\begin{align}
  \label{eq:276}
  g(v)+r_0=r^\ast(f(v),z(v)),
\end{align}
identifying the radial coordinate of points on $\mathcal{K}$ coming from the solution in the maximal development and from the solution in $T_\varepsilon$. Condition \eqref{eq:167} is
\begin{align}
  \label{eq:277}
  V(v)=\frac{\jump{T^{tr}(v)}}{\jump{T^{tt}(v)}}.
\end{align}
$f(v)$ and $g(v)$ have to satisfy
\begin{align}
  \label{eq:278}
  \frac{df}{dv}(v)V(v)=\frac{dg}{dv}(v).
\end{align}
We restate the free boundary problem as follows.\vspace{3mm}\\
\textit{For small enough $\varepsilon$ find in $T_\varepsilon$ a solution of \eqref{eq:270}, \eqref{eq:271} which attains along $\underline{C}$ the given data and along $\mathcal{K}$ satisfies \eqref{eq:273}, \eqref{eq:278}, where $V(v)$ is given by \eqref{eq:277} and $z(v)$ is given by \eqref{eq:276}}.\vspace{3mm}

We solve the problem using an iteration whose strategy is the following. We start with approximate solutions $z_m(v)$, $\beta_{+,m}(v)$, $V_m(v)$. Then we solve the characteristic system \eqref{eq:270}, \eqref{eq:271} with $(\alpha_{m+1},\beta_{m+1},t_{m+1},r_{m+1})$ in the role of $(\alpha,\beta,t,r)$ with initial data $t_{m+1}(u,0)=h(u)$, $\alpha_{m+1}(u,0)=\alpha_i(u)$ (on $\underline{C}$), boundary data $\beta_{m+1}(v,v)=\beta_{+,m}(v)$ (on $\mathcal{K}$) and $r(0,0)=r_0$ together with the requirement that
\begin{align}
  \label{eq:279}
  \frac{df_{m+1}}{dv}(v)V_m(v)=\frac{dg_{m+1}}{dv}(v),
\end{align}
where
\begin{align}
  \label{eq:280}
  f_{m+1}(v)\coloneqq t_{m+1}(v,v),\qquad g_{m+1}(v)\coloneqq r_{m+1}(v,v)-r_0.
\end{align}
We then substitute $f_{m+1}(v)$, $g_{m+1}(v)$ for $f(v)$, $g(v)$, respectively, in the identification equation \eqref{eq:276} and solve for $z$ in terms of $v$. The solution we define to be $z_{m+1}(v)$. Using now $z_{m+1}(v)$, $f_{m+1}(v)$ we obtain through \eqref{eq:275} $\alpha_{-,m+1}(v)$, $\beta_{-,m+1}(v)$. We then use these together with $\alpha_{+,m+1}(v)$ to solve \eqref{eq:273} for $\beta_{+}(v)$ which we define to be $\beta_{+,m+1}(v)$. Note that $\beta_{m+1}(v,v)=\beta_{+,m}(v)$ but $\alpha_{m+1}(v,v)=\alpha_{+,m+1}(v)$. We then define $V_{m+1}(v)$ by \eqref{eq:277} where the jumps on the right hand side correspond to $\alpha_{\pm,m+1}(v)$, $\beta_{\pm,m+1}(v)$. We summarize the strategy as follows
\begin{align}
  \label{eq:281}
  z_m, \beta_{+,m}, V_{m}\stackrel{1\,\,}{\rightarrow} \alpha_{m+1},\beta_{m+1},t_{m+1},r_{m+1}\stackrel{2\,\,}{\rightarrow}f_{m+1},g_{m+1},\alpha_{+,m+1}\stackrel{3\,\,}{\rightarrow} z_{m+1}\stackrel{4\,\,}{\rightarrow} \beta_{+,m+1},V_{m+1}.
\end{align}

In the following we shall call the triplet $(z_m,\beta_{+,m},V_m)$, which are functions on the boundary $\mathcal{K}$, boundary functions corresponding to the $m$'th iterate.

We make the following crucial observation. Let
\begin{align}
  \label{eq:282}
  F(v,z)\coloneqq g(v)+r_0-r^\ast(f(v),z).
\end{align}
Since (see \eqref{eq:231})
\begin{align}
  \label{eq:283}
  \cp{\pp{F}{z}}=-\cp{\pp{r^\ast}{w}}=0
\end{align}
it is not possible to directly solve the identification equation \eqref{eq:276} for $z$ in our iteration scheme. We will use leading order expansions of $g(v)$, $f(v)$, $z(v)$ in the identification equation to arrive, through a cancellation, at a reduced identification equation which can then be solved for the remainder function of $z(v)$.

%%% Local Variables: 
%%% mode: latex
%%% TeX-master: "./master"
%%% End: 

\section{Solution of the Fixed Boundary Problem}
 The characteristic system being nonlinear we use an iteration. In the first subsection we set up the iteration scheme and establish the inductive step. In the second subsection we show convergence.

\subsection{Setup of Iteration Scheme and Inductive Step}\label{fbp}
The goal is to find a solution of the system of equations
\begin{alignat}{3}
  \label{eq:284}
  \frac{\p \alpha}{\p v}&=\pp{t}{v}\tilde{A}(\alpha,\beta,r),& \qquad & \frac{\p \beta}{\p u}=\pp{t}{u}\tilde{B}(\alpha,\beta,r),\\
  \label{eq:285}
  \frac{\p r}{\p v}&=\frac{\p t}{\p v}c_+(\alpha,\beta),& &\frac{\p r}{\p u}=\frac{\p t}{\p u}c_-(\alpha,\beta),
\end{alignat}
together with initial data $\alpha(u,0)=\alpha_i(u)$ (on $\underline{C}$), boundary data $\beta(v,v)=\beta_+(v)$ (on $\mathcal{K}$), initial data $t(u,0)=h(u)$ (on $\underline{C}$) and $r(0,0)=r_0$, together with the requirement that for a given function $V(v)$ the equation 
\begin{align}
  \label{eq:286}
  \frac{df}{dv}(v)V(v)=\frac{dg}{dv}(v),
\end{align}
is to be satisfied, where
\begin{align}
  \label{eq:287}
  f(v)\coloneqq t(v,v),\qquad g(v)\coloneqq r(v,v)-r_0,
\end{align}
and the requirement that $t$ is a time function, i.e.
\begin{align}
  \label{eq:288}
  \pp{t}{u},\pp{t}{v}>0\quad \textrm{for}\quad u,v> 0.
\end{align}
Since we set $t_0=0$ we obtain from \eqref{eq:287} $f(0)=g(0)=0$. For the initial data of $t$ along $\underline{C}$ we assume (see \eqref{eq:268})
\begin{align}
  \label{eq:289}
   t(u,0)=h(u)=u^3\hat{h}(u),\qquad \hat{h}\in C^1[0,\varepsilon],\qquad \hat{h}(0)=\frac{\lambda}{6\kappa(c_{+0}-c_{-0})}.
\end{align}
For the initial data of $\alpha$ along $\underline{C}$ we assume (see \eqref{eq:2028})
\begin{align}
  \label{eq:290}
  \alpha(u,0)=\alpha_i(u)=\alpha_0+\dot{\alpha}_0u+u^2\hat{\alpha}_i(u),\qquad \hat{\alpha}_i\in C^1[0,\varepsilon],\qquad \hat{\alpha}_i(0)=\frac{1}{2}\ddot{\alpha}_0.
\end{align}
Furthermore, we assume $\beta_+(v)\in C^1[0,\varepsilon]$, $V(v)\in C^0[0,\varepsilon]$ and
\begin{align}
  \label{eq:291}
  \frac{d\beta_+}{dv}(v)= \mathcal{O}(v),\qquad V(v)=c_{+0}+\frac{\kappa}{2}(1+y(v))v+\mathcal{O}(v^2),
\end{align}
where $y\in C^1[0,\varepsilon]$ is a given function with
\begin{align}
  \label{eq:292}
  y(0)=-1.
\end{align}
We define
\begin{align}
  \label{eq:293}
  Y\coloneqq \sup_{[0,\varepsilon]}\left|\frac{dy}{dv}\right|.
\end{align}

The solution is to be found in a domain $T_\varepsilon$ for small enough $\varepsilon$, where
\begin{align}
  \label{eq:294}
  T_\varepsilon\coloneqq \left\{(u,v)\in \mathbb{R}^2:0\leq v \leq u\leq \varepsilon\right\}.
\end{align}
Taking the derivative of the first of \eqref{eq:285} with respect to $u$ and of the second of \eqref{eq:285} with respect to $v$ and subtracting yields
\begin{align}
  \label{eq:295}
  (c_+-c_-)\frac{\p^2 t}{\p u\p v}+\frac{\p c_+}{\p u}\frac{\p t}{\p v}-\frac{\p c_-}{\p v}\frac{\p t}{\p u}=0.
\end{align}
Defining
\begin{align}
  \label{eq:296}
  \mu\coloneqq \frac{1}{c_+-c_-}\frac{\p c_+}{\p u},\qquad \nu\coloneqq \frac{1}{c_+-c_-}\frac{\p c_-}{\p v},
\end{align}
equation \eqref{eq:295} becomes
\begin{align}
  \label{eq:297}
  \frac{\p^2 t}{\p u\p v}+\mu\frac{\p t}{\p v}-\nu\frac{\p t}{\p u}=0.
\end{align}
Using \eqref{eq:285} equation \eqref{eq:286} becomes
\begin{align}
  \label{eq:298}
  a(v)=\frac{1}{\gamma(v)}b(v),
\end{align}
where we use the definitions
\begin{align}
  \label{eq:299}
   a(v)\coloneqq \pp{t}{v}(v,v),\qquad b(v)\coloneqq \pp{t}{u}(v,v),\qquad  \gamma(v)\coloneqq \frac{\bar{c}_+(v)-V(v)}{V(v)-\bar{c}_-(v)},
\end{align}
where
\begin{align}
  \label{eq:300}
  \bar{c}_\pm(v)\coloneqq c_\pm(\alpha_+(v),\beta_+(v)).
\end{align}

Let us recall that in acoustical coordinates $(t,w)$, the boundary of the singular part of the maximal development $\mathcal{B}$ is given by $t_\ast(w)$ (see \eqref{eq:243}). We have
\begin{align}
  \label{eq:301}
  \cp{\frac{dt_\ast}{dw}}=0.
\end{align}
Since we are looking for a solution in which the shock $\mathcal{K}$ lies in the past of $\mathcal{B}$ we should have
\begin{align}
  \label{eq:302}
  \cp{\frac{df}{dv}}=0,
\end{align}
where $f(v)=t(v,v)$ describes the shock curve $\mathcal{K}$. By this assumption together with~\eqref{eq:289} we get
\begin{align}
  \label{eq:303}
  \pp{t}{v}(0,0)=0.
\end{align}
From the first of \eqref{eq:291} together with the second of \eqref{eq:284} in conjunction with \eqref{eq:289} we obtain
\begin{align}
  \label{eq:304}
  \pp{\beta}{v}(0,0)=0.
\end{align}

Taking the derivative of the second of \eqref{eq:284} with respect to $v$ we find
\begin{align}
  \label{eq:305}
  \pppp{\beta}{u}{v}=\pppp{t}{u}{v}\tilde{B}+\pp{t}{u}\left(\pp{\tilde{B}}{\alpha}\pp{t}{v}\tilde{A}+\pp{\tilde{B}}{\beta}\pp{\beta}{v}+\pp{\tilde{B}}{r}c_+\pp{t}{v}\right),
\end{align}
where we also used the first of \eqref{eq:284} together with the first of \eqref{eq:285}. Using the first of \eqref{eq:284} we obtain
\begin{align}
  \label{eq:306}
  \pp{c_-}{v}&=\pp{c_-}{\alpha}\pp{t}{v}\tilde{A}+\pp{c_-}{\beta}\pp{\beta}{v}.
\end{align}
Using this in \eqref{eq:295} we get
\begin{align}
  \label{eq:307}
  \pppp{t}{u}{v}=-\frac{1}{c_+-c_-}\left\{\pp{c_+}{u}\pp{t}{v}-\pp{t}{u}\left(\pp{c_-}{\alpha}\pp{t}{v}\tilde{A}+\pp{c_-}{\beta}\pp{\beta}{v}\right)\right\}.
\end{align}
Using this for the first term in \eqref{eq:305} we get
\begin{align}
  \label{eq:308}
  \pppp{\beta}{u}{v}&=-\frac{1}{c_+-c_-}\left\{\pp{c_+}{u}\pp{t}{v}-\pp{t}{u}\left(\pp{c_-}{\alpha}\pp{t}{v}\tilde{A}+\pp{c_-}{\beta}\pp{\beta}{v}\right)\right\}\tilde{B}\notag\\
&\qquad\hspace{30mm}+\pp{t}{u}\left(\pp{\tilde{B}}{\alpha}\pp{t}{v}\tilde{A}+\pp{\tilde{B}}{\beta}\pp{\beta}{v}+\pp{\tilde{B}}{r}c_+\pp{t}{v}\right).
\end{align}
Along $\underline{C}$ \eqref{eq:307}, \eqref{eq:308} build a system of the form
\begin{align}
  \label{eq:309}
  \frac{d}{du}
\left(\begin{array}{c}
      \partial\beta/\partial v\\
      \partial t/\partial v
    \end{array}\right)
=
\left(\begin{array}{cc}
  a_{11} & a_{12}\\
  a_{21} & a_{22}
\end{array}\right)
\left(\begin{array}{c}
  \partial\beta/\partial v\\
  \partial t/\partial v
\end{array}\right)\quad \textrm{for} \quad v=0.
\end{align}
Together with the initial conditions given by \eqref{eq:303}, \eqref{eq:304}, we arrive at
\begin{align}
  \label{eq:310}
  \pp{t}{v}(u,0)=0, \qquad \pp{\beta}{v}(u,0)=0.
\end{align}
Hence, we expect
\begin{align}
  \label{eq:311}
  \pp{t}{v}(u,v)=\Landau(v),\qquad \pp{\beta}{v}(u,v)=\Landau(v).
\end{align}
Therefore, we base our iteration scheme with this expectation in mind.

We construct a solution of the fixed boundary problem as the limit of a sequence of functions $((\alpha_n,\beta_n,t_n,r_n);n=0,1,2,\ldots)$. Given $(\alpha_n,\beta_n)$ we find $(\alpha_{n+1},\beta_{n+1})$ in the following way. We set
\begin{align}
  \label{eq:312}
  \mu_n\coloneqq \frac{1}{c_{+n}-c_{-n}}\pp{c_{+n}}{u},\qquad\nu_n\coloneqq \frac{1}{c_{+n}-c_{-n}}\pp{c_{-n}}{v},
\end{align}
where
\begin{align}
  \label{eq:313}
  c_{\pm n}\coloneqq c_\pm(\alpha_n,\beta_n).
\end{align}
Let $t_{n}$ be the solution of the linear equation
\begin{align}
  \label{eq:314}
  \frac{\p^2t_{n}}{\p u\p v}+\mu_{n}\frac{\p t_{n}}{\p v}-\nu_{n}\frac{\p t_{n}}{\p u}=0,
\end{align}
together with the initial data on $\underline{C}$ (cf.~\eqref{eq:289}) and the boundary condition on $\mathcal{K}$
\begin{align}
  \label{eq:315}
  a_{n}(v)=\frac{1}{\gamma_{n}(v)}b_{n}(v).
\end{align}
We then find $r_{n}$ by integrating \eqref{eq:285},~i.e.
\begin{align}
  \label{eq:316}
  r_{n}(u,v)&=r_n(u,0)+\int_0^v\left(c_{+,n}\pp{t}{v}\right)(u,v')dv'\notag\\
&=r_0+\int_0^u\left(c_{-,n}\frac{\p t_{n}}{\p u}\right)(u',0)du'+\int_0^v\left(c_{+,n}\frac{\p t_{n}}{\p v}\right)(u,v')dv'.
\end{align}
We then define $\alpha_{n+1}$ and $\beta_{n+1}$ by
\begin{align}
  \label{eq:317}
  \alpha_{n+1}(u,v)\coloneqq \alpha_i(u)+\int_0^vA_n(u,v')dv',\qquad \beta_{n+1}(u,v)\coloneqq \beta_+(v)+\int_v^uB_n(u',v)du',
\end{align}
where
\begin{align}
  \label{eq:318}
  A_n\coloneqq \pp{t_n}{v}\tilde{A}(\alpha_n,\beta_n,r_n),\qquad B_n\coloneqq \pp{t_n}{u}\tilde{B}(\alpha_n,\beta_n,r_n).
\end{align}
We have thus found $(\alpha_{n+1} ,\beta_{n+1})$.

To initiate the sequence we set
\begin{align}
  \label{eq:319}
  \alpha_0(u,v)=\alpha_i(u),\qquad \beta_0(u,v)=\beta_+(v).
\end{align}
$t_0(u,v)$ is given by the solution of \eqref{eq:314} (with $0$ in the role of $n$) with $\mu_0$, $\nu_0$ given by \eqref{eq:312} (with $0$ in the role of $n$). $r_0(u,v)$ is then given by \eqref{eq:316}.

The way we set things up we see that to each pair $(\alpha_n,\beta_n)$ there corresponds a unique pair $(t_n,r_n)$ given by \eqref{eq:314}, \eqref{eq:316}. It therefore suffices to show that the iteration mapping maps the respective spaces to itself (by induction) and the convergence only for the sequence $((\alpha_n,\beta_n);n=0,1,2,\ldots)$. Let us denote by $(\alpha,\beta)$ the limit of $((\alpha_n,\beta_n);n=0,1,2,\ldots)$. The convergence of $(\alpha_n,\beta_n)$ to $(\alpha,\beta)$ will imply the convergence of $(t_n,r_n)$ to $(t,r)$, where $t$ is the solution of \eqref{eq:314} with the coefficients $\mu$, $\nu$ given by $c_\pm(\alpha,\beta)$ and $r$ is given by \eqref{eq:316} such that when $t$, $r$ are substituted into the right hand sides of \eqref{eq:332} below, the left hand sides of \eqref{eq:332} are $\alpha$, $\beta$ respectively. Therefore this will imply the existence to a solution of the fixed boundary problem.

The first of \eqref{eq:291} is equivalent to
\begin{align}
  \label{eq:320}
  \left|\frac{d\beta_+}{dv}(v)\right|\leq Cv.
\end{align}
We now define
\begin{align}
  \label{eq:321}
  b_0\coloneqq |\beta_0|+\frac{C\varepsilon^2}{2},
\end{align}
where the constant $C$ is the constant from \eqref{eq:320}. Therefore,
\begin{align}
  \label{eq:322}
  \sup_{v\in[0,\varepsilon]}|\beta_{+}(v)|\leq b_0.
\end{align}
Let now
\begin{align}
  \label{eq:323}
  a_0\coloneqq \sup_{v\in[0,\varepsilon]}|\alpha_i(v)|.
\end{align}
Let $\delta>0$ and let $R_{\delta}$ be the rectangle given by (see figure \ref{domain_R})
\begin{align}
  \label{eq:324}
  R_{\delta}\coloneqq \{(\alpha,\beta)\in\mathbb{R}^2:|\alpha|\leq\delta+a_0,|\beta|\leq\delta+b_0\}.
\end{align}

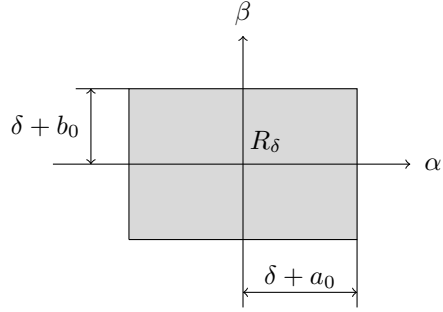
\begin{figure}[h!]
\begin{center}
\begin{tikzpicture}
\filldraw [gray!30] (-1.5,-1) -- (1.5,-1) -- (1.5,1) -- (-1.5,1);
\draw (-1.5,-1) -- (1.5,-1) -- (1.5,1) -- (-1.5,1) -- (-1.5,-1);
\draw [->](-2.5,0) -- (2.2,0);
\draw [->](0,-1.9) -- (0,1.7);
\draw [<->](0,-1.7) -- (1.5,-1.7);
\draw (1.5,-1) -- (1.5,-1.9);
\node at (0.75,-1.5) {$\delta+a_0$};
\draw [<->](-2,0) -- (-2,1);
\draw (-1.5,1) -- (-2.2,1);
\node at (-2.6,0.5) {$\delta+b_0$};
\node at (2.5,0) {$\alpha$};
\node at (0,2) {$\beta$};
\node at (0.3,0.3) {$R_\delta$};
\end{tikzpicture}
\end{center}
\caption{The rectangle $R_\delta$.}
\label{domain_R}
\end{figure}

We define
\begin{align}
  \label{eq:325}
  N\coloneqq \max\left\{\sup_{\Omega_\delta}| \tilde{A}(\alpha,\beta,r)|,\sup_{\Omega_\delta}| \tilde{B}(\alpha,\beta,r)|\right\},
\end{align}
where we defined
\begin{align}
  \label{eq:326}
  {\Omega_\delta}\coloneqq R_\delta\times\left[\tfrac{1}{2}r_0,\tfrac{3}{2}r_0\right].
\end{align}
Let
\begin{align}
  \label{eq:327}
  C_0\coloneqq \frac{2\lambda}{3\kappa^2}.
\end{align}
We now choose a constant $N_0$ such that
\begin{align}
  \label{eq:328}
  N_0>C_0N.
\end{align}
\begin{comment}
Let
\begin{align}
    \label{eq:329}
    \tilde{C}':&=\max\left\{\max_{i\in \{\alpha,\beta,r\}}\sup_{\Omega_\delta}|\partial_i\tilde{A}|,\max_{i\in \{\alpha,\beta,r\}}\sup_{\Omega_\delta}|\partial_i\tilde{B}|\right\},\\
  \label{eq:330}
  \tilde{\tilde{C}}':&=\max\left\{\max_{i\in \{\alpha,\beta\}}\sup_{R_\delta}|\partial_i^nc_+|,\max_{i\in \{\alpha,\beta\}}\sup_{R_\delta}|\partial_i^nc_-|\right\},\qquad n\leq 2.
\end{align}
\end{comment}

Defining
\begin{align}
  \label{eq:331}
  \alpha_n'(u,v)\coloneqq \alpha_n(u,v)-\alpha_i(u),\qquad \beta_n'(u,v)\coloneqq \beta_n(u,v)-\beta_+(v),
\end{align}
we have
\begin{align}
  \label{eq:332}
  \alpha_{n+1}'(u,v)=\int_0^vA_n(u,v')dv',\qquad \beta_{n+1}'(u,v)=\int_v^uB_n(u',v)du'.
\end{align}
Let $X$ be the closed subspace of $C^1(T_\varepsilon,\mathbb{R}^2)$ consisting of those functions $F=(F_1,F_2)$ which satisfy
\begin{enumerate}
\item \begin{align}
  \label{eq:333}
  F_1(u,0)=0,\qquad F_2(v,v)=0,
\end{align}
\item\begin{align}
  \label{eq:334}
  \|F\|_X\coloneqq \max\left\{\sup_{T_\varepsilon}\left|\frac{1}{u}\pp{F_1}{u}\right|,\sup_{T_\varepsilon}\left|\frac{1}{v}\pp{F_1}{v}\right|,\sup_{T_\varepsilon}\left|\frac{1}{u}\pp{F_2}{u}\right|,\sup_{T_\varepsilon}\left|\frac{1}{v}\pp{F_2}{v}\right|\right\}\leq N_0.
\end{align}
\end{enumerate}
As a preliminary result concerning the linear equation \eqref{eq:314} we have the following proposition.
\begin{proposition}\label{prop_inner_iteration}
Let
\begin{align}
  \label{eq:335}
  \mu(u,v)&=\frac{\kappa}{c_{+0}-c_{-0}}(1+\tau(u,v)),\quad\textrm{where}\quad \tau(u,v)=\mathcal{O}(u),\\
  \nu(u,v)&=\mathcal{O}(v),\label{eq:336}\\
  \frac{1}{\gamma(v)}&=\frac{c_{+0}-c_{-0}}{\kappa v}(1+\rho(v)),\quad\textrm{where}\quad \rho(v)=\rho_0(v)+\mathcal{O}(v),\quad\textrm{with}\quad \rho_0(v)=\frac{\frac{1}{2}(y(v)+1)}{1-\frac{1}{2}(y(v)+1)},\label{eq:337}\\
h(u)&=\frac{\lambda}{\kappa(c_{+0}-c_{-0})}\frac{u^3}{6}+\mathcal{O}(u^4).\label{eq:338}
\end{align}
Then, provided we choose $\varepsilon$ small enough, depending on $Y$, the solution of the equation
\begin{align}
  \label{eq:339}
  \frac{\partial^2 t}{\partial u\partial v}+\mu\pp{t}{v}-\nu\pp{t}{u}=0,
\end{align}
with initial condition
\begin{align}
  \label{eq:340}
  t(u,0)=h(u)
\end{align}
and boundary condition
\begin{align}
  \label{eq:341}
  a(v)=\frac{1}{\gamma(v)}b(v),\qquad\textrm{where}\qquad a(v)\coloneqq \pp{t}{v}(v,v),\qquad b(v)\coloneqq \pp{t}{u}(v,v),
\end{align}
is in $C^1(T_\varepsilon)$ and satisfies
\begin{align}
  \label{eq:342}
  \left|\pp{t}{v}(u,v)-\frac{\lambda}{3\kappa^2}v\right|\leq C(Y)uv,\qquad \left|\pp{t}{u}(u,v)-\frac{\lambda(3u^2-v^2)}{6\kappa(c_{+0}-c_{-0})}\right|\leq C(Y)u^3
\end{align}
for $(u,v)\in T_\varepsilon$, where $C(Y)$ are non-negative, non-decreasing, continuous functions of $Y$. Furthermore, let $f(v)\coloneqq t(v,v)$, then
\begin{align}
  \label{eq:343}
  \frac{df}{dv}(v)-\frac{\lambda}{3\kappa^2}v=\frac{\lambda}{18\kappa^2}\left\{2v(y(v)+1)+\frac{1}{v^2}\int_0^vv'^3\frac{dy}{dv}(v')dv'\right\}+\mathcal{O}(v^2),
\end{align}
for $v\in[0,\varepsilon]$.
\end{proposition}
%\begin{comment}
\begin{proof}
Integrating \eqref{eq:339} with respect to $v$ from $v=0$ yields
\begin{align}
  \label{eq:344}
  \frac{\p t}{\p u}(u,v)=e^{-K(u,v)}\left\{h'(u)-\int_0^ve^{K(u,v')}\left(\mu\frac{\p t}{\p v}\right)(u,v')dv'\right\},
\end{align}
while integrating \eqref{eq:339} with respect to $u$ from $u=v$ yields
\begin{align}
  \label{eq:345}
   \frac{\p t}{\p v}(u,v)=e^{-L(u,v)}\left\{a(v)+\int_v^ue^{L(u',v)}\left(\nu\frac{\p t}{\p u}\right)(u',v)du'\right\},
\end{align}
where we used the definitions
\begin{align}
  \label{eq:346}
  K(u,v)\coloneqq \int_0^v(-\nu)(u,v')dv',\qquad L(u,v)\coloneqq \int_v^u\mu(u',v)du'.
\end{align}

From the first of \eqref{eq:346} together with \eqref{eq:336} we obtain
\begin{align}
  \label{eq:347}
  K(u,v)=\mathcal{O}(v^2).
\end{align}
Evaluating \eqref{eq:344} at $u=v$ yields
\begin{align}
  \label{eq:348}
  b(v)=e^{-K(v,v)}\left\{h'(v)-\int_0^ve^{K(v,v')}\left(\mu\pp{t}{v}\right)(v,v')dv'\right\}.
\end{align}
Defining (see \eqref{eq:335} for the definition of $\tau(u,v)$)
\begin{align}
  \label{eq:349}
  I(v)\coloneqq \int_0^v\left\{e^{K(v,v')}-1+e^{K(v,v')}\tau(v,v')\right\}\pp{t}{v}(v,v')dv',
\end{align}
we have
\begin{align}
  \label{eq:350}
  b(v)=e^{-K(v,v)}\left\{h'(v)-\frac{\kappa}{c_{+0}-c_{-0}}\left(f(v)-h(v)+I(v)\right)\right\},
\end{align}
where we used
\begin{align}
  \label{eq:351}
  \int_0^v\pp{t}{v}(v,v')dv'=t(v,v)-t(v,0)=f(v)-h(v).
\end{align}
From \eqref{eq:298}, \eqref{eq:299} we have
\begin{align}
  \label{eq:352}
  \frac{df}{dv}(v)=\left(\frac{1}{\gamma(v)}+1\right)b(v),
\end{align}
which implies, using \eqref{eq:337}, \eqref{eq:350},
\begin{align}
  \label{eq:353}
  v\frac{df}{dv}(v)&=-e^{-K(v,v)}\left(1+\rho(v)+\frac{\kappa v}{c_{+0}-c_{-0}}\right)f(v)\notag\\
&\qquad+e^{-K(v,v)}\left(1+\rho(v)+\frac{\kappa v}{c_{+0}-c_{-0}}\right)\left(\frac{c_{+0}-c_{-0}}{\kappa}h'(v)+h(v)-I(v)\right).
\end{align}
Using the definitions
\begin{align}
  \label{eq:354}
  A(v)&\coloneqq e^{-K(v,v)}\left(\frac{\rho(v)}{v}+\frac{\kappa}{c_{+0}-c_{-0}}\right)-\frac{1}{v}\left(1-e^{-K(v,v)}\right),\\
  \label{eq:355}
  B(v)&\coloneqq \frac{e^{-K(v,v)}}{v^2}\left(1+\rho(v)+\frac{\kappa v}{c_{+0}-c_{-0}}\right)\left(\frac{c_{+0}-c_{-0}}{\kappa}h'(v)+h(v)-I(v)\right),
\end{align}
equation \eqref{eq:353} becomes
\begin{align}
  \label{eq:356}
  \frac{d(vf(v))}{dv}+A(v)vf(v)=v^2B(v).
\end{align}

Since (cf.~\eqref{eq:289})
\begin{align}
  \label{eq:357}
  h'(v)=\frac{\lambda}{2\kappa(c_{+0}-c_{-0})}v^2+\mathcal{O}(v^3),
\end{align}
we can write
\begin{align}
  \label{eq:358}
  B(v)=\frac{\lambda}{2\kappa^2}+\hat{B}(v),
\end{align}
where
\begin{align}
  \label{eq:359}
  \hat{B}(v)&\coloneqq \frac{\lambda}{2\kappa^2}\left\{e^{-K(v,v)}\left(\rho(v)+\frac{\kappa v}{c_{+0}-c_{-0}}\right)-\left(1-e^{-K(v,v)}\right)\right\}\notag\\
  &\qquad +\frac{1}{v^2}e^{-K(v,v)}\left(1+\rho(v)+\frac{\kappa v}{c_{+0}-c_{-0}}\right)\left(\frac{c_{+0}-c_{-0}}{\kappa}h'(v)-\frac{\lambda}{2\kappa^2}v^2+h(v)-I(v)\right).
\end{align}

Integrating \eqref{eq:356} from $v=0$ yields
\begin{align}
  \label{eq:360}
  vf(v)=\int_0^ve^{-\int_{v'}^vA(v'')dv''}v'^2B(v')dv'.
\end{align}
Substituting this back into \eqref{eq:356} gives
\begin{align}
  \label{eq:361}
  \frac{df}{dv}(v)=vB(v)-\frac{1}{v^2}(1+vA(v))\int_0^ve^{-\int_{v'}^vA(v'')dv''}v'^2B(v')dv'.
\end{align}
Using now \eqref{eq:358} we find
\begin{align}
  \label{eq:362}
  \frac{df}{dv}(v)=\frac{\lambda}{2\kappa^2}M(v)+N(v),
\end{align}
where $N(v)$ is linear in $\hat{B}(v)$, while $M(v)$ is independent of $\hat{B}(v)$, i.e.
\begin{align}
  \label{eq:363}
  M(v)=M_{0}(v)+M_{1}(v)+M_{2}(v),\qquad N(v)=N_{0}(v)+N_{1}(v),
\end{align}
with
\begin{align}
  \label{eq:364}
  M_{0}(v)&\coloneqq v-\frac{1}{v^2}\int_0^vv'^2dv'=\frac{2}{3}v,\\
  M_{1}(v)&\coloneqq \frac{1}{v^2}\int_0^v\left(1-e^{-\int_{v'}^vA(v'')dv''}\right)v'^2dv',\label{eq:365}\\
  M_{2}(v)&\coloneqq -\frac{A(v)}{v}\int_0^ve^{-\int_{v'}^vA(v'')dv''}v'^2dv',\label{eq:366}
\end{align}
and
\begin{align}
  \label{eq:367}
  N_{i}(v)\coloneqq v\hat{B}_{i}(v)-\frac{1}{v^2}(1+vA(v))\int_0^ve^{-\int_{v'}^vA(v'')dv''}v'^2\hat{B}_{i}(v')dv',\qquad \textrm{$i=0,1$},
\end{align}
where we split $\hat{B}(v)$ into a part depending on $I(v)$ and a part independent of $I(v)$, i.e.
\begin{align}
  \label{eq:368}
  \hat{B}_{0}(v)&\coloneqq \frac{\lambda}{2\kappa^2}\left\{e^{-K(v,v)}\left(\rho(v)+\frac{\kappa v}{c_{+0}-c_{-0}}\right)-\left(1-e^{-K(v,v)}\right)\right\}\notag\\
  &\qquad +\frac{1}{v^2}e^{-K(v,v)}\left(1+\rho(v)+\frac{\kappa v}{c_{+0}-c_{-0}}\right)\left(\frac{c_{+0}-c_{-0}}{\kappa}h'(v)-\frac{\lambda}{2\kappa^2}v^2+h(v)\right),\\
  \hat{B}_{1}(v)&\coloneqq -\frac{1}{v^2}e^{-K(v,v)}\left(1+\rho(v)+\frac{\kappa v}{c_{+0}-c_{-0}}\right)I(v).\label{eq:369}
\end{align}
In view of \eqref{eq:337} we have
\begin{align}
  \label{eq:370}
  \rho(v)+\frac{\kappa v}{c_{+0}-c_{-0}}=\rho_0(v)+\mathcal{O}(v).
\end{align}

In the arguments to follow $q>0$ will denote a number which we can make as small as we wish by choosing $\varepsilon$ suitably small. From $y(0)=-1$ we have
\begin{align}
  \label{eq:371}
  |y(v)+1|\leq vY\leq q.
\end{align}
Now, since we can make $1-\tfrac{1}{2}vY$ as close to $1$ as we wish by choosing $\varepsilon$ suitably small depending on $Y$ (in the following we will not state this dependence explicitly anymore) and since (cf.~\eqref{eq:371})
\begin{align}
  \label{eq:372}
  1-\tfrac{1}{2}|y(v)+1|\geq 1-\tfrac{1}{2}vY,
\end{align}
we obtain
\begin{align}
  \label{eq:373}
  \frac{1}{1-\tfrac{1}{2}|y(v)+1|}\leq p,
\end{align}
for $p>1$ but as close to $1$ as we wish by choosing $\varepsilon$ suitably small. Hence
\begin{align}
  \label{eq:374}
  |\rho_0(v)|\leq \frac{p}{2}|y(v)+1|\leq \frac{p}{2}v Y\leq \frac{p}{2}\varepsilon Y\leq q,
\end{align}
where $q>0$ as small as we wish by restricting $\varepsilon$ (and therefore $\varepsilon Y$) suitably (cf.~\eqref{eq:371}). Therefore,
\begin{align}
  \label{eq:375}
  |\rho(v)|\leq q+Cv,
\end{align}

We now look at $\hat{B}_{0}$. From \eqref{eq:289} it follows that the asymptotic form of the second bracket on the second line of \eqref{eq:368} is $\mathcal{O}(v^3)$. Taking into account \eqref{eq:347}, \eqref{eq:370}, as well as \eqref{eq:375}, we obtain
\begin{align}
  \label{eq:376}
  \hat{B}_{0}(v)=\frac{\lambda}{2\kappa^2}\rho_0(v)+\mathcal{O}(v).
\end{align}

\eqref{eq:354} together with \eqref{eq:374} yields
\begin{align}
  \label{eq:377}
  v|A(v)|=|\rho_0(v)|+\mathcal{O}(v)\leq q'.
\end{align}
From \eqref{eq:354} we obtain
\begin{align}
  \label{eq:378}
  \int_{v'}^v|A(v'')|dv''\leq \int_{v'}^v\frac{|\rho_0(v'')|}{v''}dv''+\int_{v'}^v\mathcal{O}(1)dv''\leq q'',
\end{align}
where for the first integral we use (cf.~\eqref{eq:374})
\begin{align}
  \label{eq:379}
  |\rho_0(v)|\leq \frac{p}{2}v Y,
\end{align}
and we again choose $\varepsilon$ sufficiently small. $q'$, $q''$ are like $q$, positive and as small as we wish by restricting $\varepsilon$ suitably. From the two bounds in \eqref{eq:377} and \eqref{eq:378} it follows that the contribution of the $\mathcal{O}(v)$ term in $\hat{B}_{0}$ (cf.~\eqref{eq:376}) to $N_{0}$ (cf.~\eqref{eq:367} with $i=0$) has the asymptotic form $\mathcal{O}(v^2)$.

We now look at the contribution of the first term in \eqref{eq:376} to $N_{0}$ (cf.~\eqref{eq:367}). This contribution is
\begin{align}
  \label{eq:380}
  &\frac{\lambda}{2\kappa^2}\left\{v\rho_0(v)-\frac{1}{v^2}\int_0^vv'^2\rho_0(v')dv'\right\}\notag\\
&\hspace{10mm}+\frac{\lambda}{2\kappa^2}\left\{\frac{1}{v^2}\int_0^v\left(1-e^{-\int_{v'}^vA(v'')dv''}\right)v'^2\rho_0(v')dv'-\frac{A(v)}{v}\int_0^ve^{-\int_{v'}^vA(v'')dv''}v'^2\rho_0(v')dv'\right\}.
\end{align}
Since the function $(1-e^{x})/x$ is bounded for $x\in [-1,1]$, it follows from \eqref{eq:378} that the first term in the second curly bracket in \eqref{eq:380} is bounded by
\begin{align}
  \label{eq:381}
  \frac{C}{v^2}\int_0^v\left(\int_{v'}^v|A(v'')|dv''\right)v'^2|\rho_0(v')|dv'.
\end{align}
Now, since
\begin{align}
  \label{eq:382}
  \int_{v'}^v|A(v'')|dv''=\int_{v'}^v\frac{|\rho_0(v'')|}{v''}+\mathcal{O}(1)dv''\leq C(Y+1)v,
\end{align}
where we used \eqref{eq:379}, we obtain
\begin{align}
  \label{eq:383}
  \frac{C}{v^2}\int_0^v\left(\int_{v'}^v|A(v'')|dv''\right)v'^2|\rho_0(v')|dv'\leq \frac{C}{v}\int_0^v(Y+1)v'^2|\rho_0(v')|dv'\leq C(Y+1)Yv^3,
\end{align}
where we again used \eqref{eq:379}. Choosing $\varepsilon$ sufficiently small such that $Y^2\varepsilon\leq 1$, it follows that $C(Y+1)Yv^3\leq Cv^2$, i.e.~the first term in the second curly bracket of \eqref{eq:380} is bounded in absolute value by $Cv^2$.

Now we look at the second term in the second curly bracket in \eqref{eq:380}. We will use
\begin{align}
  \label{eq:384}
  v^2|\rho_0(v)|\leq Cv^3Y,\qquad |A(v)|=\frac{|\rho_0(v)|}{v}+\mathcal{O}(1)\leq C(Y+1).
\end{align}
Those are consequences of \eqref{eq:377} and \eqref{eq:379}. Using these and \eqref{eq:378}, we have
\begin{align}
  \label{eq:385}
  \left|\frac{A(v)}{v}\int_0^ve^{-\int_{v'}^vA(v'')dv''}v'^2\rho_0(v')dv'\right|\leq \frac{C(Y+1)Y}{v}\int_0^vv'^3dv'\leq Cv^2,
\end{align}
where in the last step we again use the assumption that $Y^2\varepsilon\leq 1$. We conclude that the second term in the second curly bracket of \eqref{eq:380} is bounded in absolute value by $Cv^2$. Therefore, the second curly bracket in \eqref{eq:380} is bounded in absolute value by $Cv^2$.

We rewrite the first curly bracket in \eqref{eq:380} as
\begin{align}
  \label{eq:386}
  &v\tfrac{1}{2}(y(v)+1)-\frac{1}{v^2}\int_0^v\frac{v'^2}{2}(y(v')+1)dv'\notag\\
&\hspace{10mm}+v\Big\{\rho_0(v)-\tfrac{1}{2}(y(v)+1)\Big\}-\frac{1}{v^2}\int_0^vv'^2\Big\{\rho_0(v')-\tfrac{1}{2}(y(v')+1)\Big\}dv'.
\end{align}
For the curly brackets we use the estimate
\begin{align}
  \label{eq:387}
  \left|\rho_0(v)-\tfrac{1}{2}(y(v)+1)\right|=\left|\frac{\tfrac{1}{4}(y(v)+1)^2}{1-\tfrac{1}{2}(y(v)+1)}\right|\leq Cv^2Y^2,
\end{align}
where we used \eqref{eq:371}. We deduce that the second line in \eqref{eq:386} is bounded in absolute value by $Cv^2$ where we again make use of the assumption that $Y^2\varepsilon\leq 1$. We conclude that
\begin{align}
  \label{eq:388}
  N_{0}(v)=\frac{\lambda}{4\kappa^2}\left\{v(y(v)+1)-\frac{1}{v^2}\int_0^vv'^2(y(v')+1)dv'\right\}+\mathcal{O}(v^2).
\end{align}
Integrating by parts yields
\begin{align}
  \label{eq:389}
  N_{0}(v)=\frac{\lambda}{12\kappa^2}\left\{2v(y(v)+1)+\frac{1}{v^2}\int_0^vv'^3\frac{dy}{dv}(v')dv'\right\}+\mathcal{O}(v^2).
\end{align}

We now turn to $N_{1}$. For this we have to estimate $\hat{B}_{1}$. In view of \eqref{eq:369} we have
\begin{align}
  \label{eq:390}
  |\hat{B}_{1}(v)|\leq \frac{C}{v^2}|I(v)|.
\end{align}
Now we look at $I(v)$. From \eqref{eq:347}, \eqref{eq:335} we have
\begin{align}
  \label{eq:391}
  |e^{K(v,v')}-1|\leq Cv'^2,\qquad |\tau(v,v')|\leq Cv.
\end{align}
These imply
\begin{align}
  \label{eq:392}
  |I(v)|\leq Cv\int_0^v\left|\pp{t}{v}(v,v')\right|dv',
\end{align}
which yields
\begin{align}
  \label{eq:393}
  |\hat{B}_{1}(v)|\leq\frac{C}{v}\int_0^v\left|\pp{t}{v}(v,v')\right|dv'.
\end{align}
Now,
\begin{align}
  \label{eq:394}
  \frac{1}{v^2}\int_0^vv'^2|\hat{B}_{1}(v')|dv'&\leq \frac{1}{v}\int_0^vv'|\hat{B}_{1}(v')|dv'\notag\\
&\leq \frac{C}{v}\int_0^v\left\{\int_0^{v'}\left|\pp{t}{v}(v',v'')\right|dv''\right\}dv'\notag\\
&\leq \frac{C}{v}\int_0^v\left\{\int_{v''}^v\left|\pp{t}{v}(v',v'')\right|dv'\right\}dv''\notag\\
&\leq \frac{C}{v}\int_0^v(v-v'')T(v,v'')dv''\notag\\
&\leq C\int_0^vT(v,v'')dv'',
\end{align}
where to go from the second to the third line we changed the order of integration and in the fourth line we used the definition
\begin{align}
  \label{eq:395}
  T(u,v)\coloneqq \sup_{u'\in[v,u]}\left|\pp{t}{v}(u',v)\right|.
\end{align}
Using \eqref{eq:393} and \eqref{eq:394} in \eqref{eq:367} with $i=1$ we obtain
\begin{align}
  \label{eq:396}
  |N_{1}(v)|\leq C\int_0^vT(v,v')dv'.
\end{align}
We postpone the estimation of $T$ until the estimates for $M$ are completed.

We now look at $M_{1}$ (cf.~\eqref{eq:365}). We rewrite $M_{1}$ as
\begin{align}
  \label{eq:397}
  M_{1}(v)&=\frac{1}{v^2}\int_0^v\left\{\int_{v'}^vA(v'')dv''\right\}v'^2dv'\notag\\
&\hspace{10mm}+\frac{1}{v^2}\int_0^v\left(1-e^{-\int_{v'}^vA(v'')dv''}-\int_{v'}^vA(v'')dv''\right)v'^2dv'.
\end{align}
Since
\begin{align}
  \label{eq:398}
  A(v)=\frac{\rho_0(v)}{v}+\Landau(1),
\end{align}
we can rewrite the term on the first line as
\begin{align}
  \label{eq:399}
  \frac{1}{3v^2}\int_0^vA(v')v'^3dv'=\frac{1}{3v^2}\int_0^v\left(\frac{\rho_0(v')}{v'}+\mathcal{O}(1)\right)v'^3dv'=\frac{1}{3v^2}\int_0^v\rho_0(v')v'^2dv'+\mathcal{O}(v^2).
\end{align}
We now rewrite
\begin{align}
  \label{eq:400}
  \frac{1}{3v^2}\int_0^v\rho_0(v')v'^2dv'&=\frac{1}{3v^2}\left\{\int_0^v\tfrac{1}{2}(y(v')+1)v'^2dv'+\int_0^v\left(\rho_0(v')-\tfrac{1}{2}(y(v')+1)\right)v'^2dv'\right\}\notag\\
&=\frac{1}{18v^2}\left\{v^3(y(v)+1)-\int_0^vv'^3\frac{dy}{dv}(v')dv'\right\}+\mathcal{O}(v^2),
\end{align}
where for the first term we used integration by parts while for the second term we used \eqref{eq:387} together with the assumption that $\varepsilon Y^2\leq 1$. This implies that the term in the first line of \eqref{eq:397} is equal to
\begin{align}
  \label{eq:401}
  \frac{1}{18v^2}\left\{v^3(y(v)+1)-\int_0^vv'^3\frac{dy}{dv}(v')dv'\right\}+\mathcal{O}(v^2).
\end{align}
Using the fact that the function $(1-e^{-x}-x)/x^2$ is bounded for $x\in[-1,1]$, implies that the term on the second line in \eqref{eq:397} is bounded by
\begin{align}
  \label{eq:402}
  \frac{C}{v^2}\int_0^v\left(\int_{v'}^v|A(v'')|dv''\right)^2v'^2dv'.
\end{align}
Now, since
\begin{align}
  \label{eq:403}
  \int_{v'}^v|A(v'')|dv''\leq\int_0^v|A(v')|dv'\leq Cv(Y+1),
\end{align}
where we used \eqref{eq:384}, we obtain
\begin{align}
  \label{eq:404}
  \frac{C}{v^2}\int_0^v\left(\int_{v'}^v|A(v'')|dv''\right)^2v'^2dv'\leq Cv^3(Y+1)^2\leq Cv^2,
\end{align}
where we used the assumption $\varepsilon Y^2\leq 1$. This together with \eqref{eq:401} implies
\begin{align}
  \label{eq:405}
  M_{1}(v)=\frac{1}{18v^2}\left\{v^3(y(v)+1)-\int_0^vv'^3\frac{dy}{dv}(v')dv'\right\}+\mathcal{O}(v^2).
\end{align}

We now look at $M_{2}$ (cf.~\eqref{eq:366}). We rewrite $M_{2}$ as
\begin{align}
  \label{eq:406}
  M_{2}(v)=-\frac{A(v)}{3}v^2+\frac{A(v)}{v}\int_0^v\left(1-e^{-\int_{v'}^vA(v'')dv''}\right)v'^2dv'.
\end{align}
For the first term we have
\begin{align}
  \label{eq:407}
  -\frac{A(v)}{3}v^2&=-\frac{\rho_0(v)}{3}v+\mathcal{O}(v^2)\notag\\
&=-\frac{v}{6}(y(v)+1)-\frac{v}{3}\Big(\tfrac{1}{2}(y(v)+1)-\rho_0(v)\Big)+\mathcal{O}(v^2)\notag\\
&=-\frac{v}{6}(y(v)+1)+\mathcal{O}(v^2),
\end{align}
where we used \eqref{eq:387} together with the assumption $\varepsilon Y^2\leq 1$. Taking into account that the function $(1-e^{-x})/x$ is bounded for $x\in[-1,1]$ and integrating by parts we find that the second term in \eqref{eq:406} can be estimated in absolute value by
\begin{align}
  \label{eq:408}
  C|A(v)|\left|\int_0^vA(v')dv'\right|v^2.
\end{align}
Using the second of \eqref{eq:384} we get
\begin{align}
  \label{eq:409}
  C|A(v)|\left|\int_0^vA(v')dv'\right|v^2\leq C(Y+1)^2v^3\leq Cv^2,
\end{align}
where we again used the assumption $\varepsilon Y^2\leq 1$. From \eqref{eq:407}, \eqref{eq:409} we obtain
\begin{align}
  \label{eq:410}
  M_{2}(v)=-\frac{v}{6}(y(v)+1)+\mathcal{O}(v^2).
\end{align}

We shall now turn to the estimate of $T(u,v)$ and prove that $T(u,v)\leq Cv$. To accomplish this we will use the ode \eqref{eq:362} but we will use less delicate estimates for $M_{i}$ and $N_{i}$ than derived above. We derive these rough estimates from the delicate estimates \eqref{eq:389}, \eqref{eq:405}, \eqref{eq:410} by using the two estimates (cf.~\eqref{eq:371})
\begin{align}
  \label{eq:411}
  |y(v)+1|\leq C,\qquad \left|v\frac{dy}{dv}(v)\right|\leq C.
\end{align}
Using the resulting (rough) estimates together with \eqref{eq:364}, \eqref{eq:396} in \eqref{eq:362} we arrive at
\begin{align}
  \label{eq:412}
  \left|\frac{df}{dv}(v)\right|\leq C\left\{v+\int_0^vT(v,v')dv'\right\}.
\end{align}
Using the asymptotic form of $\mu(u,v)$ as given by \eqref{eq:335} it follows from the second of\eqref{eq:346} that $|L(u,v)|\leq C\varepsilon$. We deduce from \eqref{eq:345} together with the asymptotic form of $\nu$ given by \eqref{eq:336}, very roughly,
\begin{align}
  \label{eq:413}
  \left|\pp{t}{v}(u,v)\right|\leq C\left\{|a(v)|+v\int_v^u\left|\pp{t}{u}(u',v)\right|du'\right\}.
\end{align}
From \eqref{eq:344} we deduce, very roughly,
\begin{align}
  \label{eq:414}
  \left|\pp{t}{u}(u,v)\right|\leq C\left\{u^2+\int_0^v\left|\pp{t}{v}(u,v')\right|dv'\right\}.
\end{align}
Substituting in the integral in \eqref{eq:413} the bound \eqref{eq:414} yields
\begin{align}
  \label{eq:415}
  \int_v^u\left|\pp{t}{u}(u',v)\right|du'\leq C(u-v)\left\{u^2+\int_0^vT(u,v')dv'\right\}.
\end{align}

Now,
\begin{align}
  \label{eq:416}
  \frac{df}{dv}(v)=a(v)+b(v)=(1+\gamma(v))a(v),
\end{align}
which implies
\begin{align}
  \label{eq:417}
  |a(v)|=\left|\frac{1}{1+\gamma(v)}\,\frac{df}{dv}(v)\right|\leq \left|\frac{df}{dv}(v)\right|\leq C\left\{v+\int_0^vT(v,v')dv'\right\},
\end{align}
where for the first inequality we used
\begin{align}
  \label{eq:418}
 0\leq\gamma(v)\leq Cv,
\end{align}
which follows from \eqref{eq:337}, and for the second inequality in \eqref{eq:417} we used \eqref{eq:412}. Substituting \eqref{eq:415} and \eqref{eq:417} into \eqref{eq:413} yields
\begin{align}
  \label{eq:419}
  \left|\pp{t}{v}(u,v)\right|\leq C\left\{v+\int_0^vT(u,v')dv'\right\}.
\end{align}
Since $T(u,v)$ is non-decreasing in its first argument (cf.~\eqref{eq:395}) we have
\begin{align}
  \label{eq:420}
  \left|\pp{t}{v}(u',v)\right|\leq C\left\{v+\int_0^vT(u',v')dv'\right\}\quad \textrm{for}\quad u'\in[v,u].
\end{align}
Taking the supremum over all $u'\in[v,u]$ we find
\begin{align}
  \label{eq:421}
  T(u,v)\leq C\left\{v+\int_0^vT(u,v')dv'\right\}.
\end{align}
Defining
\begin{align}
  \label{eq:422}
  \Sigma_u(v)\coloneqq \int_0^vT(u,v')dv',
\end{align}
equation \eqref{eq:421} becomes
\begin{align}
  \label{eq:423}
  \frac{d\Sigma_u}{dv}(v)\leq C(v+\Sigma_u(v)).
\end{align}
Since $\Sigma_u(0)=0$, we get
\begin{align}
  \label{eq:424}
  \Sigma_u(v)\leq \int_0^ve^{C(v-v')}Cv'dv'\leq C'v^2.
\end{align}
Therefore,
\begin{align}
  \label{eq:425}
  T(u,v)\leq Cv.
\end{align}
Plugging this estimate into the estimate for $N_{1}$, i.e.~into \eqref{eq:396}, we obtain
\begin{align}
  \label{eq:426}
  |N_1(v)|\leq Cv^2.
\end{align}
Using now the delicate estimates for $M_{0}$, $N_{0}$, $M_{1}$, $M_{2}$ as given by \eqref{eq:364}, \eqref{eq:389}, \eqref{eq:405}, \eqref{eq:410} in the ode \eqref{eq:362}, we arrive at
\begin{align}
  \label{eq:427}
  \frac{df}{dv}(v)-\frac{\lambda}{3\kappa^2}v=\frac{\lambda}{18\kappa^2}\left\{2v(y(v)+1)+\frac{1}{v^2}\int_0^vv'^3\frac{dy}{dv}(v')dv'\right\}+\mathcal{O}(v^2).
\end{align}
This is \eqref{eq:343}.

Using now
\begin{align}
  \label{eq:428}
  |y(v)+1|\leq Yv,\qquad \left|\frac{dy}{dv}\right|\leq Y,
\end{align}
in \eqref{eq:427} we obtain
\begin{align}
  \label{eq:429}
  \left|\frac{df}{dv}(v)-\frac{\lambda}{3\kappa^2}v\right|\leq C(Y)v^2.
\end{align}
Since
\begin{align}
  \label{eq:430}
  a(v)=\frac{df}{dv}(v)-b(v)=\frac{df}{dv}(v)-\frac{\gamma(v)}{1+\gamma(v)}\frac{df}{dv}(v),
\end{align}
we obtain from \eqref{eq:418},
\begin{align}
  \label{eq:431}
  a(v)=\frac{df}{dv}(v)+\mathcal{O}(v^2).
\end{align}
Together with \eqref{eq:429} we conclude
\begin{align}
  \label{eq:432}
  \left|a(v)-\frac{\lambda}{3\kappa^2}v\right|\leq C(Y)v^2.
\end{align}

Recalling the definition \eqref{eq:395} and using \eqref{eq:425} we get
\begin{align}
  \label{eq:433}
  \int_0^v\left|\pp{t}{v}(u,v)\right|dv'\leq\int_0^vT(u,v')dv'\leq Cv^2.
\end{align}
Substituting this into \eqref{eq:414} yields
\begin{align}
  \label{eq:434}
  \left|\pp{t}{u}(u,v)\right|\leq Cu^2.
\end{align}

We now revisit \eqref{eq:345}, in particular the term
\begin{align}
  \label{eq:435}
  \int_v^ue^{-(L(u,v)-L(u',v))}\left(\nu\pp{t}{u}\right)(u',v)du'.
\end{align}
From \eqref{eq:335} we derive
\begin{align}
  \label{eq:436}
  |L(u,v)-L(u',v)|=\left|\int_{u'}^u\mu(u'',v)du''\right|\leq Cu.
\end{align}
Using now \eqref{eq:336} and \eqref{eq:434} we have
\begin{align}
  \label{eq:437}
  \left|\int_v^ue^{-(L(u,v)-L(u',v))}\left(\nu\pp{t}{u}\right)(u',v)du'\right|\leq Cu^3v.
\end{align}

We now look at the first term in \eqref{eq:345} and rewrite it as
\begin{align}
  \label{eq:438}
  e^{-L(u,v)}a(v)=e^{-L(u,v)}\frac{\lambda}{3\kappa^2}v+e^{-L(u,v)}\left(a(v)-\frac{\lambda}{3\kappa^2}v\right).
\end{align}
The second term can be estimated using \eqref{eq:432}, while for the first term we use \eqref{eq:335} to estimate
\begin{align}
  \label{eq:439}
  \left|L(u,v)-\frac{\kappa(u-v)}{c_{+0}-c_{-0}}\right|\leq C\int_v^uu'du'\leq Cu^2.
\end{align}
Now,
\begin{align}
  \label{eq:440}
  e^{-\frac{\kappa(u-v)}{c_{+0}-c_{-0}}}=1-\frac{\kappa(u-v)}{c_{+0}-c_{-0}}+\mathcal{O}(u^2),
\end{align}
which, together with \eqref{eq:439}, implies
\begin{align}
  \label{eq:441}
  e^{-L(u,v)}=1-\frac{\kappa(u-v)}{c_{+0}-c_{-0}}+\mathcal{O}(u^2).
\end{align}
Using the estimates \eqref{eq:432}, \eqref{eq:441} in \eqref{eq:438} and the resulting estimate together with \eqref{eq:437} in \eqref{eq:345} yields
\begin{align}
  \label{eq:442}
  \left|\pp{t}{v}(u,v)-\frac{\lambda}{3\kappa^2}v\right|\leq C(Y)uv.
\end{align}
This is the first of \eqref{eq:342}.

We now revisit \eqref{eq:344}, in particular the term
\begin{align}
  \label{eq:443}
  \int_0^ve^{K(u,v')}\left(\mu\pp{t}{v}\right)(u,v')dv'.
\end{align}
We rewrite it as
\begin{align}
  \label{eq:444}
  \int_0^ve^{K(u,v')}\left(\mu\pp{t}{v}\right)(u,v')dv'&=\frac{\lambda}{3\kappa^2}\int_0^ve^{K(u,v')}\mu(u,v')v'dv'\notag\\
&\hspace{10mm}+\int_0^ve^{K(u,v')}\mu(u,v')\left(\pp{t}{v}(u,v')-\frac{\lambda}{3\kappa^2}v'\right)dv'.
\end{align}
Using \eqref{eq:335}, the first term on the right of \eqref{eq:444} possesses the following asymptotic form
\begin{align}
  \label{eq:445}
  \frac{\lambda}{3\kappa^2}\int_0^ve^{K(u,v')}\mu(u,v')v'dv'=\frac{\lambda}{6\kappa(c_{+0}-c_{-0})}v^2+\mathcal{O}(uv^2).
\end{align}
By \eqref{eq:335} and \eqref{eq:442}, the second integral in \eqref{eq:444} can be bounded in absolute value by $C(Y)uv^2$. Together with \eqref{eq:289} we obtain from \eqref{eq:344},
\begin{align}
  \label{eq:446}
  \left|\pp{t}{u}(u,v)-\frac{\lambda(3u^2-v^2)}{6\kappa(c_{+0}-c_{-0})}\right|\leq C(Y)u^3.
\end{align}
This is the second of \eqref{eq:342}. This completes the proof of the proposition.
\end{proof}

\begin{lemma}\label{lemma_fxdbp_ind}
  For $\varepsilon$ sufficiently small depending, on $N_0$, $Y$, the sequence $((\alpha'_n,\beta'_n);n=0,1,2,\ldots)$ is contained in $X$.
\end{lemma}

\begin{proof}
  From \eqref{eq:319}, \eqref{eq:331} we have
  \begin{align}
    \label{eq:447}
    \alpha_0'(u,v)=0,\qquad \beta_0'(u,v)=0.
  \end{align}
Therefore
\begin{align}
  \label{eq:448}
  (\alpha_0',\beta_0')\in X.
\end{align}

We now show the inductive step. In the following the generic constants will depend on $N_0$ but we shall not specify this dependence. It suffices for us that this constants are non-negative non-decreasing continuous functions of $N_0$. The inductive hypothesis is
\begin{align}
  \label{eq:449}
  \|(\alpha_n',\beta_n')\|_X\leq N_0.
\end{align}
We start with
\begin{align}
  \label{eq:450}
  |\alpha_n'|(u,v)=\left|\alpha_n'(u,0)+\int_0^v\pp{\alpha_n'}{v}(u,v')dv'\right|\leq N_0\varepsilon^2,
\end{align}
where we used $\alpha_n'(u,0)=0$.
\begin{align}
  \label{eq:451}
  |\beta_n'|(u,v)=\left|\beta_n'(v,v)+\int_v^u\pp{\beta_n'}{u}(u',v)du'\right|\leq N_0\varepsilon^2,
\end{align}
where we used $\beta'_n(v,v)=0$. Choosing $\varepsilon$ small enough we deduce
\begin{align}
  \label{eq:452}
  (\alpha'_n+\alpha_i,\beta'_n+\beta_+)=(\alpha_n,\beta_n)\in R_\delta.
\end{align}
Therefore, the functions $c_{\pm,n}$ and derivatives thereof are bounded, i.e.
\begin{align}
  \label{eq:453}
  \sup_{T_\varepsilon}|c_{\pm,n}|,\sup_{T_\varepsilon}\left|\left(\pp{c_\pm}{\alpha}\right)_n\right|,\sup_{T_\varepsilon}\left|\left(\pp{c_\pm}{\beta}\right)_n\right|\leq C.
\end{align}
(also for higher order derivatives). From \eqref{eq:331} we have
\begin{alignat}{4}
  \label{eq:454}
  \pp{\alpha'_n}{u}&=\pp{\alpha_n}{u}-\frac{d\alpha_i}{du},& \qquad\quad \pp{\alpha_n'}{v}&=\pp{\alpha_n}{v},\\
  \label{eq:455}
  \pp{\beta_n'}{u}&=\pp{\beta_n}{u},& \pp{\beta_n'}{v}&=\pp{\beta_n}{v}-\frac{d\beta_+}{dv}.
\end{alignat}

We now derive bounds on $\mu_n$, $\nu_n$. From the first of \eqref{eq:296} we have
\begin{align}
  \label{eq:456}
  \mu_n=\frac{1}{c_{+n}-c_{-n}}\pp{c_{+n}}{u}=\frac{1}{c_{+n}-c_{-n}}\left\{\left(\pp{c_+}{\alpha}\right)_n\pp{\alpha_n}{u}+\left(\pp{c_+}{\beta}\right)_n\pp{\beta_n}{u}\right\}.
\end{align}
For the second term we have
\begin{align}
  \label{eq:457}
  \frac{1}{c_{+n}-c_{-n}}\left(\pp{c_+}{\beta}\right)_n\pp{\beta_n}{u}=\mathcal{O}(u),
\end{align}
where we used the first of \eqref{eq:455} together with the inductive hypothesis \eqref{eq:449}. Defining the functions
\begin{align}
  \label{eq:458}
  f_1\coloneqq \frac{1}{c_{+0}-c_{-0}}-\frac{1}{c_{+n}-c_{-n}},\qquad f_2\coloneqq \left(\pp{c_+}{\alpha}\right)_0-\left(\pp{c_+}{\alpha}\right)_n,\qquad f_3\coloneqq \left(\pp{\alpha}{u}\right)_0-\left(\pp{\alpha}{u}\right)_n,
\end{align}
the first term in \eqref{eq:456} becomes
\begin{align}
  \label{eq:459}
  \frac{1}{c_{+n}-c_{-n}}\left(\pp{c_+}{\alpha}\right)_n\pp{\alpha_n}{u}=\left(\frac{1}{c_{+0}-c_{-0}}-f_1\right)\left(\left(\pp{c_+}{\alpha}\right)_0-f_2\right)\left(\left(\pp{\alpha}{u}\right)_0-f_3\right).
\end{align}
For $f_{1,2}$ we note that $f_{1,2}\in C^1[0,\varepsilon]$ with $f_{1,2}(0,0)=0$, which implies
\begin{align}
  \label{eq:460}
  f_{1,2}(u,v)=\mathcal{O}(u).
\end{align}
For $f_3$ we have
\begin{align}
  \label{eq:461}
  f_3(u,v)=\dot{\alpha}_0-\left(\pp{\alpha_n}{u}\right)(u,v)=\dot{\alpha}_0-\pp{\alpha_n'}{u}(u,v)-\frac{d\alpha_i}{du}(u),
\end{align}
where we used the first of \eqref{eq:454}. The inductive hypothesis together with \eqref{eq:290} yields
\begin{align}
  \label{eq:462}
  f_3(u,v)=\mathcal{O}(u).
\end{align}
From \eqref{eq:460}, \eqref{eq:462} together with
\begin{align}
  \label{eq:463}
  \frac{1}{c_{+0}-c_{-0}}\left(\pp{c_+}{\alpha}\right)_0\left(\pp{\alpha}{u}\right)_0=\frac{\kappa}{c_{+0}-c_{-0}},
\end{align}
where we recalled $(\partial c_+/\partial\alpha)_0\dot{\alpha}_0=\kappa$ (see \eqref{eq:2027}), we obtain
\begin{align}
  \label{eq:464}
  \frac{1}{c_{+n}-c_{-n}}\left(\pp{c_+}{\alpha}\right)_n\pp{\alpha_n}{u}= \frac{\kappa}{c_{+0}-c_{-0}}+\mathcal{O}(u).
\end{align}
Together with \eqref{eq:457} we conclude
\begin{align}
  \label{eq:465}
  \mu_n(u,v)=\frac{\kappa}{c_{+0}-c_{-0}}+\mathcal{O}(u).
\end{align}

We now look at
\begin{align}
  \label{eq:466}
  \nu_n=\frac{1}{c_{+n}-c_{-n}}\left\{\left(\pp{c_-}{\alpha}\right)_n\pp{\alpha_n}{v}+\left(\pp{c_-}{\beta}\right)_n\pp{\beta_n}{v}\right\}.
\end{align}
For the first term we use the second of \eqref{eq:454} which together with the inductive hypothesis \eqref{eq:449} implies
\begin{align}
  \label{eq:467}
  \frac{1}{c_{+n}-c_{-n}}\left(\pp{c_-}{\alpha}\right)_n\pp{\alpha_n}{v}=\mathcal{O}(v).
\end{align}
For the second term in \eqref{eq:466} we use the second of \eqref{eq:455} together with the inductive hypothesis \eqref{eq:449} and the property of the boundary data given by the first of \eqref{eq:291}. We get
\begin{align}
  \label{eq:468}
  \frac{1}{c_{+n}-c_{-n}}\left(\pp{c_-}{\beta}\right)_n\pp{\beta_n}{v}=\mathcal{O}(v).
\end{align}
Therefore
\begin{align}
  \label{eq:469}
  \nu_n(u,v)=\mathcal{O}(v).
\end{align}

We now look at $\gamma_n(v)$. Since
\begin{align}
  \label{eq:470}
  \bar{c}_{+n}(v)=c_+(\alpha_{+n}(v),\beta_+(v)),
\end{align}
we have
\begin{align}
  \label{eq:471}
  \frac{d\bar{c}_{+n}}{dv}(v)=\pp{c_+}{\alpha}(\alpha_{+n}(v),\beta_+(v))\frac{d\alpha_{+n}}{dv}(v)+\pp{c_+}{\beta}(\alpha_{+n}(v),\beta_+(v))\frac{d\beta_+}{dv}(v).
\end{align}
Now,
\begin{align}
  \label{eq:472}
  \alpha_{+n}(v)=\alpha'_n(v,v)+\alpha_i(v).
\end{align}
From $\alpha'(u,0)=0$ together with the inductive hypothesis \eqref{eq:449} we obtain
\begin{align}
  \label{eq:473}
  |\alpha'_n(v,v)|\leq \frac{1}{2}N_0v^2.
\end{align}
Hence
\begin{align}
  \label{eq:474}
  \alpha_{+n}(v)=\alpha_0+\dot{\alpha}_0v+\Landau(v^2).
\end{align}
From \eqref{eq:320} we have
\begin{align}
  \label{eq:475}
  \beta_+(v)=\beta_0+\Landau(v^2).
\end{align}
From
\begin{align}
  \label{eq:476}
  \frac{d\alpha_{+n}}{dv}(v)=\left(\pp{\alpha'_n}{u}+\pp{\alpha'_n}{v}\right)(v,v)+\frac{d\alpha_i}{dv}(v),
\end{align}
together with the inductive hypothesis we have
\begin{align}
  \label{eq:477}
  \frac{d\alpha_{+n}}{dv}(v)=\dot{\alpha}_0+\Landau(v).
\end{align}

Expanding now $(\partial c_+/\partial \alpha)(\alpha,\beta)$, $(\partial c_+/\partial \beta)(\alpha,\beta)$ and making use of \eqref{eq:474}, \eqref{eq:475} we obtain
\begin{align}
  \label{eq:478}
  \pp{c_+}{\alpha}(\alpha_{+n}(v),\beta_+(v))&=\cp{\pp{c_+}{\alpha}}+\cp{\ppp{c_+}{\alpha}}\dot{\alpha}_0v+\Landau(v^2),\\
  \pp{c_+}{\beta}(\alpha_{+n}(v),\beta_+(v))&=\cp{\pp{c_+}{\beta}}+\cp{\pppp{c_+}{\beta}{\alpha}}\dot{\alpha}_0v+\Landau(v^2).
\end{align}
Using these together with \eqref{eq:320}, \eqref{eq:477} in \eqref{eq:471} and recalling $\cp{\partial c_+/\partial\alpha}\dot{\alpha}_0=\kappa$ we find
\begin{align}
  \label{eq:479}
  \frac{d\bar{c}_{+n}}{dv}(v)=\kappa+\Landau(v).
\end{align}
Hence
\begin{align}
  \label{eq:480}
  \bar{c}_{+n}(v)=c_{+0}+\kappa v+\Landau(v^2).
\end{align}
Therefore, using the second of \eqref{eq:291},
\begin{align}
  \label{eq:481}
  \bar{c}_{+n}(v)-V(v)=\kappa v-\frac{\kappa}{2}(1+y(v))v+\mathcal{O}(v^2).
\end{align}
From the fact that $\bar{c}_{-n}(v)=c_-(\alpha_{+n}(v),\beta_+(v))$ is in $C^1[0,\varepsilon]$ we obtain, in conjunction with the second of \eqref{eq:291},
\begin{align}
  \label{eq:482}
  V(v)-\bar{c}_{-n}(v)=c_{+0}-c_{-0}+\mathcal{O}(v).
\end{align}
Now, using \eqref{eq:481} and \eqref{eq:482} in the third of \eqref{eq:299} we find
\begin{align}
  \label{eq:483}
    \frac{1}{\gamma_n(v)}=\frac{c_{+0}-c_{-0}}{\kappa v}(1+\rho_n(v)),
\end{align}
where
\begin{align}
  \label{eq:484}
    \rho_n(v)=\rho_0(v)+\mathcal{O}(v),\qquad \textrm{with}\qquad \rho_0(v)=\frac{\frac{1}{2}(y(v)+1)}{1-\frac{1}{2}(y(v)+1)}.
\end{align}
\begin{remark}
There is no index missing on $\beta_+$ in \eqref{eq:470}. This is because $\beta_+$ originates from the boundary data which is not subjected to the iteration while $\alpha_{+n}$ is part of the solution of the fixed boundary problem, therefore subjected to the iteration.
\end{remark}

In view of \eqref{eq:289}, \eqref{eq:465}, \eqref{eq:469}, \eqref{eq:483}, \eqref{eq:484} we are in the position to apply  proposition  \ref{prop_inner_iteration} with $(\mu_n,\tau_n,\nu_n,\gamma_n,\rho_n,t_n)$ in the role of $(\mu,\tau,\nu,\gamma,\rho,t)$. From proposition \ref{prop_inner_iteration} we have (recalling the constant $C_0$ given in \eqref{eq:327})
\begin{align}
  \label{eq:485}
  \left|\pp{t_n}{u}\right|\leq C_0u^2,\qquad \left|\pp{t_n}{v}\right|\leq C_0v,
\end{align}
provided that $\varepsilon$ is sufficiently small. Using \eqref{eq:485} in \eqref{eq:316} together with \eqref{eq:453} we get
\begin{align}
  \label{eq:486}
   \tfrac{1}{2}r_0\leq r_n\leq\tfrac{3}{2}r_0,
\end{align}
provided that $\varepsilon$ is sufficiently small. It follows from \eqref{eq:452}, \eqref{eq:486} and the definitions \eqref{eq:324}, \eqref{eq:325}
\begin{align}
  \label{eq:487}
  \sup_{T_\varepsilon}|\tilde{A}(\alpha_n,\beta_n,r_n)|,\sup_{T_\varepsilon}|\tilde{B}(\alpha_n,\beta_n,r_n)|\leq N.
\end{align}
Taking the derivative of the first of \eqref{eq:332} with respect to $v$ and of the second of \eqref{eq:332} with respect to $u$ we obtain from \eqref{eq:485} together with \eqref{eq:487} that
\begin{align}
  \label{eq:488}
  \left|\pp{\alpha_{n+1}'}{v}\right|=\left|\pp{t_n}{v}\tilde{A}_n\right|\leq C_0vN,\qquad \left|\pp{\beta_{n+1}'}{u}\right|=\left|\pp{t_n}{u}\tilde{B}_n\right|\leq C_0uN.
\end{align}

Now we establish bounds for $\partial\alpha_{n+1}'/\partial u$, $\partial\beta_{n+1}'/\partial v$. \eqref{eq:452}, \eqref{eq:486} imply
\begin{align}
  \label{eq:489}
  \sup_{T_\varepsilon}\left|\left(\pp{\tilde{A}}{k}\right)_n\right|,\sup_{T_\varepsilon}\left|\left(\pp{\tilde{B}}{k}\right)_n\right|\leq C,\qquad k\in\{\alpha,\beta,r\}.
\end{align}
From the first of \eqref{eq:332},
\begin{align}
  \label{eq:490}
  \pp{\alpha_{n+1}'}{u}(u,v)&=\int_0^v\pp{A_n}{u}(u,v')dv'\notag\\
&=\int_0^v\bigg\{\pp{t_n}{v}\left(\pp{\tilde{A}}{\alpha}\right)_n\pp{\alpha_n}{u}+\pp{t_n}{v}\left(\pp{\tilde{A}}{\beta}\right)_n\pp{\beta_n}{u}\notag\\
&\hspace{40mm}+\pp{t_n}{v}\left(\pp{\tilde{A}}{r}\right)_nc_{-n}\pp{t_n}{u}+\tilde{A}_n\frac{\partial^2t_n}{\partial u\partial v}\bigg\}(u,v')dv',
\end{align}
where for the third term we used the second of \eqref{eq:285}. From \eqref{eq:314}, \eqref{eq:335}, \eqref{eq:336}, \eqref{eq:485} we obtain
\begin{align}
  \label{eq:491}
  \left|\frac{\partial^2t_n}{\partial u\partial v}\right|\leq Cv.
\end{align}

Now we bound \eqref{eq:490}. For the first three terms we use the second of \eqref{eq:485}, \eqref{eq:489} and the inductive hypothesis \eqref{eq:449} in conjunction with \eqref{eq:331} (and the fact that we have bounds on the derivative of the initial data $\alpha_i$). For the fourth term in \eqref{eq:490} we use \eqref{eq:491}. We conclude
\begin{align}
  \label{eq:492}
  \left|\pp{\alpha_{n+1}'}{u}\right|\leq Cu^2.
\end{align}
From the second of \eqref{eq:331},
\begin{align}
  \label{eq:493}
   \pp{\beta_{n+1}'}{v}(u,v)&=\int_v^u\pp{B_n}{v}(u',v)du'-B_n(v,v)\notag\\
  &=\int_v^u\bigg\{\pp{t_n}{u}\left(\pp{\tilde{B}}{\alpha}\right)_n\pp{\alpha_n}{v}+\pp{t_n}{u}\left(\pp{\tilde{B}}{\beta}\right)_n\pp{\beta_n}{v}\notag\\
&\hspace{30mm}+\pp{t_n}{u}\left(\pp{\tilde{B}}{r}\right)_nc_{+n}\pp{t_n}{v}+\tilde{B}_n\frac{\partial^2t_n}{\partial u\partial v}\bigg\}(u',v)du'-B_n(v,v),
\end{align}
We can bound the integral in the same way as we did for \eqref{eq:490} (and using the first of \eqref{eq:291}). For the term $B_n(v,v)$ we use
\begin{align}
  \label{eq:494}
  |B_n(v,v)|=\left|\pp{t_n}{u}(v,v)\tilde{B}_n(v,v)\right|\leq C_0vN,
\end{align}
where we used the second of \eqref{eq:488}. Therefore,
\begin{align}
  \label{eq:495}
  \left|\pp{\beta_{n+1}'}{u}\right|\leq Cuv+C_0vN.
\end{align}
From \eqref{eq:488}, \eqref{eq:492}, \eqref{eq:495} we have
\begin{align}
  \label{eq:496}
  \frac{1}{u}\left|\pp{\alpha_{n+1}'}{u}\right|,\frac{1}{v}\left|\pp{\alpha_{n+1}'}{v}\right|,\frac{1}{u}\left|\pp{\beta_{n+1}'}{u}\right|,\frac{1}{v}\left|\pp{\beta_{n+1}'}{v}\right|\leq Cu+C_0N.
\end{align}
Replacing $u$ by $u'$, taking the supremum over $u'\in [v,u]$ and choosing $\varepsilon$ sufficiently small we find
\begin{align}
  \label{eq:497}
  \|(\alpha_{n+1}',\beta_{n+1}')\|_X\leq N_0.
\end{align}
This completes the inductive step and therefore the proof of the lemma.
\end{proof}

%\end{comment}

\subsection{Convergence}
\begin{lemma}\label{lemma_fxdbp_conv}
For $\varepsilon$ sufficiently small depending on $N_0$, $Y$, the sequence $((\alpha_n',\beta_n');n=0,1,2,\ldots)$ converges in $X$.
\end{lemma}
\begin{proof}
We use the notation
\begin{align}
  \label{eq:498}
  \Delta_n\psi\coloneqq \psi_n-\psi_{n-1}.
\end{align}
From the first of \eqref{eq:332} we have
\begin{align}
  \label{eq:499}
  \pp{\Delta_{n+1}\alpha'}{v}=\Delta_nA=\tilde{A}_n\pp{\Delta_nt}{v}+\pp{t_{n-1}}{v}\Delta_n\tilde{A}.
\end{align}
Therefore
\begin{align}
  \label{eq:500}
  \left|\pp{\Delta_{n+1}\alpha'}{v}\right|\leq C\left\{\left|\pp{\Delta_nt}{v}\right|+v\Big(|\Delta_n\alpha|+|\Delta_n\beta|+|\Delta_nr|\Big)\right\}.
\end{align}
From the first of \eqref{eq:331} we have
\begin{align}
  \label{eq:501}
  \Delta_n\alpha=\Delta_n\alpha'=\int_0^v\pp{\Delta_n\alpha'}{v}(u,v')dv',
\end{align}
which implies
\begin{align}
  \label{eq:502}
  |\Delta_n\alpha|\leq v^2\sup_{T_\varepsilon}\left|\frac{1}{v}\pp{\Delta_n\alpha'}{v}\right|.
\end{align}
Similarly we get
\begin{align}
  \label{eq:503}
  |\Delta_n\beta|\leq u^2\sup_{T_\varepsilon}\left|\frac{1}{u}\pp{\Delta_n\beta'}{u}\right|.
\end{align}

For an estimate of $|\Delta_nr(u,v)|$ we use \eqref{eq:316}, which gives
\begin{align}
  \label{eq:504}
   \Delta_nr= \int_0^u\left\{c_{-,n}\pp{\Delta_nt}{u}+\pp{t_{n-1}}{u}\Delta_nc_-\right\}(u',0)du'+\int_0^v\left\{c_{+,n}\pp{\Delta_nt}{v}+\pp{t_{n-1}}{v}\Delta_nc_+\right\}(u,v')dv'.
\end{align}
We first look at the terms $\Delta_nc_\pm$.
\begin{align}
  \label{eq:505}
  |\Delta_nc_\pm|\leq C\left\{|\Delta_n\alpha|+|\Delta_n\beta|\right\}\leq Cu^2\left\{\sup_{T_\varepsilon}\left|\frac{1}{v}\pp{\Delta_n\alpha'}{v}\right|+\sup_{T_\varepsilon}\left|\frac{1}{u}\pp{\Delta_n\beta'}{u}\right|\right\},
\end{align}
where for the second inequality we use \eqref{eq:502}, \eqref{eq:503}. To find an estimate for the differences of partial derivatives of $t$ we subtract \eqref{eq:297} with $t_{n-1}$ in the role of $t$ from \eqref{eq:297} with $t_n$ in the role of $t$ (analogously for the roles of $\mu$ and $\nu$). We arrive at
\begin{align}
  \label{eq:506}
  \frac{\partial^2\Delta_nt}{\partial u\partial v}+\mu_n\pp{\Delta_nt}{v}-\nu_n\pp{\Delta_nt}{u}=\Xi_n,
\end{align}
where
\begin{align}
  \label{eq:507}
  \Xi_n\coloneqq \pp{t_{n-1}}{u}\Delta_n\nu-\pp{t_{n-1}}{v}\Delta_n\mu.
\end{align}
The initial conditions for \eqref{eq:506} are $\Delta_nt(u,0)=0$. Furthermore we have the boundary condition
\begin{align}
  \label{eq:508}
  \pp{\Delta_nt}{v}=\frac{1}{\gamma_n}\pp{\Delta_nt}{u}+\pp{t_{n-1}}{u}\Delta_n\left(\frac{1}{\gamma}\right),\quad\textrm{for}\quad u=v.
\end{align}
Integrating \eqref{eq:506} with respect to $v$ from $v=0$ yields
\begin{align}
  \label{eq:509}
  \pp{\Delta_nt}{u}(u,v)=e^{-K_n(u,v)}\int_0^ve^{K_n(u,v')}\left(\Xi_n-\mu_n\pp{\Delta_nt}{v}\right)(u,v')dv',
\end{align}
where we used the first of \eqref{eq:346}. Evaluating \eqref{eq:509} at $u=v$ gives
\begin{align}
  \label{eq:510}
  \Delta_nb(v)=e^{-K_n(v,v)}\left\{P_n(v)-\frac{\kappa}{c_{+0}-c_{-0}}\Big(\Delta_nf(v)+\Delta_nI(v)\Big)\right\},
\end{align}
where we used the definitions
\begin{align}
  \label{eq:511}
  \Delta_nI(v)&\coloneqq \int_0^v\left\{e^{K_n(v,v')}-1+e^{K_n(v,v')}\tau_n(v,v')\right\}\pp{\Delta_nt}{v}(v,v')dv',\\
  \label{eq:512}
  P_n(v)&\coloneqq \int_0^ve^{K_n(v,v')}\Xi_n(v,v')dv'.
\end{align}
For the definition of $\tau_n(v)$ see \eqref{eq:335}.

We now estimate the differences $\Delta_n\mu$, $\Delta_n\nu$. From the first equality in \eqref{eq:456} we obtain
\begin{align}
  \label{eq:513}
  \Delta_n\mu=-\frac{1}{(c_{+,n}-c_{-,n})(c_{+,n-1}-c_{-,n-1})}\pp{c_{+,n}}{u}\Delta_n(c_+-c_-)+\frac{1}{c_{+,n-1}-c_{-,n-1}}\pp{\Delta_nc_+}{u}.
\end{align}
The first term can be bounded using the estimates \eqref{eq:505}. For the second term we use
\begin{align}
  \label{eq:514}
  \pp{c_{+,n}}{u}=\left(\pp{c_+}{\alpha}\right)_n\pp{\alpha_n}{u}+\left(\pp{c_+}{\beta}\right)_n\pp{\beta_n}{u},
\end{align}
which implies
\begin{align}
  \label{eq:515}
  \pp{\Delta_nc_+}{u}=\left(\pp{c_+}{\alpha}\right)_n\pp{\Delta_n\alpha}{u}+\pp{\alpha_{n-1}}{u}\Delta_n\left(\pp{c_+}{\alpha}\right)+\left(\pp{c_+}{\beta}\right)_n\pp{\Delta_n\beta}{u}+\pp{\beta_{n-1}}{u}\Delta_n\left(\pp{c_+}{\beta}\right).
\end{align}
For the first and third difference we use
\begin{align}
  \label{eq:516}
  \pp{\Delta_n\alpha}{u}=\pp{\Delta_n\alpha'}{u},\qquad \pp{\Delta_n\beta}{u}=\pp{\Delta_n\beta'}{u}.
\end{align}
For the difference of derivatives of $c_{n,\pm}$ we use
\begin{align}
  \label{eq:517}
  \left|\Delta_n\left(\pp{c_+}{\alpha}\right)\right|\leq C\left\{|\Delta_n\alpha|+|\Delta_n\beta|\right\}\leq Cu^2\left\{\sup_{T_\varepsilon}\left|\frac{1}{v}\pp{\Delta_n\alpha'}{v}\right|+\sup_{T_\varepsilon}\left|\frac{1}{u}\pp{\Delta_n\beta'}{u}\right|\right\},
\end{align}
where we used \eqref{eq:502}, \eqref{eq:503}. The same result holds with $\beta$ in the role of $\alpha$. Using \eqref{eq:516}, \eqref{eq:517} in \eqref{eq:515} and the resulting estimate together with \eqref{eq:514} in \eqref{eq:513}, we obtain
\begin{align}
  \label{eq:518}
  |\Delta_n\mu|\leq C\left\{\left|\pp{\Delta_n\alpha'}{u}\right|+\left|\pp{\Delta_n\beta'}{u}\right|+u^2\left(\sup_{T_\varepsilon}\left|\frac{1}{v}\pp{\Delta_n\alpha'}{v}\right|+\sup_{T_\varepsilon}\left|\frac{1}{u}\pp{\Delta_n\beta'}{u}\right|\right)\right\}.
\end{align}

From \eqref{eq:466} we get
\begin{align}
  \label{eq:519}
  \Delta_n\nu=-\frac{1}{(c_{+,n}-c_{-,n})(c_{+,n-1}-c_{-,n-1})}\pp{c_{-,n}}{v}\Delta_n\left(c_+-c_-\right)+\frac{1}{c_{+,n-1}-c_{-,n-1}}\pp{\Delta_nc_-}{v}.
\end{align}
The first term can be bounded using the estimate \eqref{eq:505} and taking into account that
\begin{align}
  \label{eq:520}
  \left|\pp{c_{-,n}}{v}\right|=\left|\left(\pp{c_-}{\alpha}\right)_n\pp{\alpha_n}{v}+\left(\pp{c_-}{\beta}\right)_n\pp{\beta_n}{v}\right|\leq Cv,
\end{align}
where we used the inductive hypothesis \eqref{eq:449}. For the second term in \eqref{eq:519} we use the expression for $\partial \Delta_nc_-/\partial v$ analogous to \eqref{eq:515} and take into account the expressions analogous to \eqref{eq:516}. We arrive at
\begin{align}
  \label{eq:521}
  |\Delta_n\nu|\leq C\left\{\left|\pp{\Delta_n\alpha'}{v}\right|+\left|\pp{\Delta_n\beta'}{v}\right|+u^2v\left(\sup_{T_\varepsilon}\left|\frac{1}{v}\pp{\Delta_n\alpha'}{v}\right|+\sup_{T_\varepsilon}\left|\frac{1}{u}\pp{\Delta_n\beta'}{u}\right|\right)\right\}.
\end{align}

Defining (see \eqref{eq:334})
\begin{align}
  \label{eq:522}
  \Lambda\coloneqq \|(\Delta_n\alpha',\Delta_n\beta')\|_X,
\end{align}
we deduce from \eqref{eq:518}, \eqref{eq:521},
\begin{align}
  \label{eq:523}
  |\Delta_n\mu|\leq Cu\Lambda,\qquad |\Delta_n\nu|\leq Cv\Lambda.
\end{align}

Recalling \eqref{eq:485} we deduce from \eqref{eq:523} that
\begin{align}
  \label{eq:524}
  |\Xi_n|\leq Cuv\Lambda.
\end{align}
Therefore,
\begin{align}
  \label{eq:525}
  |P_n(v)|\leq Cv^3\Lambda.
\end{align}
Since
\begin{align}
  \label{eq:526}
  \frac{d\Delta_nf}{dv}=\Delta_na+\Delta_nb,
\end{align}
the boundary condition \eqref{eq:508} gives
\begin{align}
  \label{eq:527}
  \frac{d\Delta_nf}{dv}=\left(\frac{1}{\gamma_n}+1\right)\Delta_nb+k_n,
\end{align}
where we use the definition
\begin{align}
  \label{eq:528}
  k_n(v)\coloneqq \pp{t_{n-1}}{u}(v,v)\Delta_n\left(\frac{1}{\gamma(v)}\right).
\end{align}
Using \eqref{eq:484} we obtain
\begin{align}
  \label{eq:529}
  v\frac{d\Delta_nf}{dv}=\left(\frac{c_{+0}-c_{-0}}{\kappa}+v+\frac{c_{+0}-c_{-0}}{\kappa}\rho_n(v)\right)\Delta_nb+vk_n.
\end{align}
Substituting \eqref{eq:510} we get
\begin{align}
  \label{eq:530}
  \frac{d(v\Delta_nf)}{dv}+A_n(v\Delta_nf)=v^2\Delta_nB,
\end{align}
where $A_n(v)$ is given by \eqref{eq:354} and
\begin{align}
  \label{eq:531}
  \Delta_nB(v)\coloneqq \frac{e^{-K_n(v,v)}}{v^2}\left(1+\rho_n(v)+\frac{\kappa v}{c_{+0}-c_{-0}}\right)\left(\frac{c_{+0}-c_{-0}}{\kappa}P_n(v)-\Delta_nI(v)\right)+\frac{k_n(v)}{v}.
\end{align}
Integrating \eqref{eq:530} from $v=0$ to $v$ yields
\begin{align}
  \label{eq:532}
  v\Delta_nf(v)=\int_0^ve^{-\int_{v'}^vA_n(v'')dv''}v'^2\Delta_nB(v')dv'.
\end{align}
Substituting this back into \eqref{eq:530} gives
\begin{align}
  \label{eq:533}
  \frac{d\Delta_nf(v)}{dv}=v\Delta_nB(v)-\frac{1}{v^2}(1+vA_n(v))\int_0^ve^{-\int_{v'}^vA_n(v'')dv''}v'^2\Delta_nB(v')dv'.
\end{align}

We decompose $\Delta_nB$ according to
\begin{align}
  \label{eq:534}
  \Delta_nB=\overset{0}{\Delta}_nB+\overset{1}{\Delta}_nB,
\end{align}
where $\overset{1}{\Delta}_nB$ contains the terms of $\Delta_nB$ which are linear in $\Delta_nI$. The right hand side of \eqref{eq:533} being linear in $\Delta_nB$, we decompose analogous to the decomposition \eqref{eq:534}, i.e.
\begin{align}
  \label{eq:535}
  \frac{d\Delta_nf(v)}{dv}=\overset{0}{R}_n+\overset{1}{R}_n,
\end{align}
$\overset{1}{R}_n$ being linear in $\Delta_nI$. Recalling the third of \eqref{eq:299} together with the second of \eqref{eq:291} we deduce
\begin{align}
  \label{eq:536}
  \left|\Delta_n\left(\frac{1}{\gamma}\right)(v)\right|\leq C\frac{|\Delta_n\alpha(v,v)|}{v^2}\leq C\sup_{T_\varepsilon}\left|\frac{1}{v}\pp{\Delta_n\alpha'}{v}\right|\leq C\Lambda,
\end{align}
where we used \eqref{eq:502} (notice that in the third of \eqref{eq:299} we have $\gamma_n$ in the role of $\gamma$ and $\bar{c}_{\pm,n}(v)=c_\pm(\alpha_{+n}(v),\beta_+(v))$ in the role of $\bar{c}_\pm(v)$ but $\beta_+(v)$ as well as $V(v)$ are given by the boundary data). From the first of \eqref{eq:485} we then get
\begin{align}
  \label{eq:537}
  |k_n(v)|\leq Cv^2\Lambda.
\end{align}
Using \eqref{eq:525} and \eqref{eq:537} we find
\begin{align}
  \label{eq:538}
  |\overset{0}{\Delta}_nB(v)|\leq Cv\Lambda.
\end{align}
From \eqref{eq:377} we have
\begin{align}
  \label{eq:539}
  |vA_n(v)|\leq C.
\end{align}
Therefore, through \eqref{eq:533},
\begin{align}
  \label{eq:540}
  |\overset{0}{R}_n(v)|\leq Cv^2\Lambda.
\end{align}
We now estimate $\Delta_nI$. From \eqref{eq:347} we get
\begin{align}
  \label{eq:541}
  |e^{K(v,v')}-1|\leq Cv^2,
\end{align}
while from \eqref{eq:335} we have
\begin{align}
  \label{eq:542}
  |\tau_n(v,v)|\leq Cv.
\end{align}
Using \eqref{eq:541} and \eqref{eq:542} in \eqref{eq:511} we obtain
\begin{align}
  \label{eq:543}
  |\Delta_nI(v)|\leq Cv\int_0^v\left|\pp{\Delta_nt}{v}(v,v')\right|dv'.
\end{align}
Therefore,
\begin{align}
  \label{eq:544}
  |\overset{1}{\Delta}_nB(v)|\leq \frac{C}{v^2}|\Delta_nI(v)|\leq \frac{C}{v}\int_0^v\left|\pp{\Delta_nt}{v}(v,v')\right|dv',
\end{align}
which implies
\begin{align}
  \label{eq:545}
  \frac{1}{v^2}\int_0^vv'^2|\overset{1}{\Delta}_nB(v')|dv'&\leq\frac{1}{v}\int_0^vv'|\overset{1}{\Delta}_nB(v')|dv'\notag\\
&\leq\frac{C}{v}\int_0^v\left(\int_0^{v'}\left|\pp{\Delta_nt}{v}(v',v'')\right|dv''\right)dv'\notag\\
&\leq C\int_0^v\Delta_nT(v,v'')dv'',
\end{align}
where
\begin{align}
  \label{eq:546}
  \Delta_nT(u,v)\coloneqq \sup_{u'\in[v,u]}\left|\pp{\Delta_nt}{v}(u',v)\right|.
\end{align}
Since, from \eqref{eq:544}
\begin{align}
  \label{eq:547}
  v|\overset{1}{\Delta}_nB(v)|\leq C\int_0^v\Delta_nT(v,v'')dv'',
\end{align}
we conclude, together with \eqref{eq:545}, that the part of the right hand side of \eqref{eq:533} which is linear in $\Delta_nI$ (which we denote by $\overset{1}{R}_n$) satisfies the estimate
\begin{align}
  \label{eq:548}
  |\overset{1}{R}_n(v)|\leq C\int_0^v\Delta_nT(v,v'')dv''.
\end{align}
From \eqref{eq:540}, \eqref{eq:548} we conclude
\begin{align}
  \label{eq:549}
  \left|\frac{d\Delta_nf(v)}{dv}\right|\leq Cv^2\Lambda+C\int_0^v\Delta_nT(v,v')dv'.
\end{align}

Integrating \eqref{eq:506} with respect to $u$ from $u=v$ yields
\begin{align}
  \label{eq:550}
  \pp{\Delta_nt}{v}(u,v)=e^{-L_n(u,v)}\left\{\Delta_na(v)+\int_v^ue^{L_n(u',v)}\left(\Xi_n(u',v)+\left(\nu_n\pp{\Delta_nt}{u}\right)(u',v)\right)du'\right\}.
\end{align}
From \eqref{eq:335}, \eqref{eq:336} we have
\begin{align}
  \label{eq:551}
  |\nu_n(u,v)|\leq Cv,\qquad |\mu_n(u,v)|\leq C.
\end{align}
The second of these implies, through the second of \eqref{eq:346}, that
\begin{align}
  \label{eq:552}
  |L_n(u,v)|\leq Cu.
\end{align}
Using the first of \eqref{eq:551} as well as \eqref{eq:552} and \eqref{eq:524} we obtain
\begin{align}
  \label{eq:553}
  \left|\pp{\Delta_nt}{v}(u,v)\right|\leq C\left\{|\Delta_na(v)|+u^2v\Lambda+v\int_v^u\left|\pp{\Delta_nt}{u}(u',v)\right|du'\right\}.
\end{align}
Now, from \eqref{eq:509} in conjunction with \eqref{eq:524} and the second of \eqref{eq:551} we get
\begin{align}
  \label{eq:554}
  \left|\pp{\Delta_nt}{u}(u,v)\right|\leq C\left\{uv^2\Lambda+\int_0^v\left|\pp{\Delta_nt}{v}(u,v')\right|dv'\right\}.
\end{align}
Using this for the integral of the right hand side of \eqref{eq:553} we estimate
\begin{align}
  \label{eq:555}
  \int_v^u\left|\pp{\Delta_nt}{u}(u',v)\right|du'&\leq Cuv^2(u-v)\Lambda+\int_v^u\left(\int_0^v\left|\pp{\Delta_nt}{v}(u',v')\right|dv'\right)du'\notag\\
&\leq C\left\{uv^2(u-v)\Lambda+(u-v)\int_0^v\Delta_nT(u,v')dv'\right\}.
\end{align}
Substituting this into \eqref{eq:553} we obtain
\begin{align}
  \label{eq:556}
  \left|\pp{\Delta_nt}{v}(u,v)\right|\leq C\left\{|\Delta_na(v)|+u^2v\Lambda+v(u-v)\int_0^v\Delta_nT(u,v')dv'\right\}.
\end{align}
From \eqref{eq:526}, \eqref{eq:527},
\begin{align}
  \label{eq:557}
  \Delta_na(v)=\frac{1}{1+\gamma_n(v)}\left(\frac{d\Delta_nf(v)}{dv}+\gamma_n(v)k_n(v)\right).
\end{align}
Using \eqref{eq:537}, \eqref{eq:549} we obtain
\begin{align}
  \label{eq:558}
  |\Delta_na(v)|\leq C\left\{v^2\Lambda+\int_0^v\Delta_nT(v,v')dv'\right\}.
\end{align}
Using this estimate in \eqref{eq:556} we find
\begin{align}
  \label{eq:559}
  \left|\pp{\Delta_nt}{v}(u,v)\right|\leq C\left\{uv\Lambda+\int_0^v\Delta_nT(u,v')dv'\right\}.
\end{align}
Taking the supremum over $u$ in $[v,u]$ we deduce
\begin{align}
  \label{eq:560}
  \Delta_nT(u,v)\leq C\left\{uv\Lambda+\int_0^v\Delta_nT(u,v')dv'\right\}.
\end{align}
Setting
\begin{align}
  \label{eq:561}
  \Delta_n\Sigma_u(v)\coloneqq \int_0^v\Delta_nT(u,v')dv',
\end{align}
\eqref{eq:560} becomes
\begin{align}
  \label{eq:562}
  \frac{d}{dv}\Delta_n\Sigma_u(v)\leq Cuv\Lambda+C\Delta_n\Sigma_u(v).
\end{align}
Integrating yields
\begin{align}
  \label{eq:563}
  \Delta_n\Sigma_u(v)\leq Cuv^2\Lambda.
\end{align}
Using this in \eqref{eq:560}, we find
\begin{align}
  \label{eq:564}
  |\Delta_nT(u,v)|\leq Cuv\Lambda.
\end{align}

The above imply, through \eqref{eq:559},
\begin{align}
  \label{eq:565}
  \left|\pp{\Delta_nt}{v}(u,v)\right|\leq C uv\Lambda.
\end{align}
Using this in \eqref{eq:554} we conclude
\begin{align}
  \label{eq:566}
  \left|\pp{\Delta_nt}{u}(u,v)\right|\leq C uv^2\Lambda.
\end{align}
Using these estimates together with \eqref{eq:485} and \eqref{eq:505} in \eqref{eq:504} we deduce
\begin{align}
  \label{eq:567}
  |\Delta_nr(u,v)|\leq Cu^3\Lambda.
\end{align}
It then follows from \eqref{eq:500} together with \eqref{eq:502}, \eqref{eq:503}, \eqref{eq:565}, \eqref{eq:567} that
\begin{align}
  \label{eq:568}
  \left|\pp{\Delta_{n+1}\alpha'}{v}\right|\leq C uv\Lambda.
\end{align}

From the second of \eqref{eq:332} we deduce
\begin{align}
  \label{eq:569}
  \pp{\Delta_{n+1}\beta'}{u}=\Delta_nB=\tilde{B}_n\pp{\Delta_nt}{u}+\pp{t_{n-1}}{u}\Delta_n\tilde{B}.
\end{align}
Therefore (cf.~\eqref{eq:500}),
\begin{align}
  \label{eq:570}
  \left|\pp{\Delta_{n+1}\beta'}{u}\right|\leq C\left\{\left|\pp{\Delta_nt}{u}\right|+u\Big(|\Delta_n\alpha|+|\Delta_n\beta|+|\Delta_nr|\Big)\right\}.
\end{align}
Using now \eqref{eq:502}, \eqref{eq:503}, \eqref{eq:566} \eqref{eq:567} we arrive at
\begin{align}
  \label{eq:571}
  \left|\pp{\Delta_{n+1}\beta'}{u}(u,v)\right|\leq C u^2\Lambda.
\end{align}

Now, from the first of \eqref{eq:332} we deduce
\begin{align}
  \label{eq:572}
  \pp{\Delta_{n+1}\alpha'}{u}(u,v)&=\int_0^v\bigg\{\left(\pp{A}{\alpha}\right)_n\pp{\Delta_n\alpha}{u}+\pp{\alpha_{n-1}}{u}\Delta_n\left(\pp{A}{\alpha}\right)+\left(\pp{A}{\beta}\right)_n\pp{\Delta_n\beta}{u}\notag\\
&\hspace{20mm}+\pp{\beta_{n-1}}{u}\Delta_n\left(\pp{A}{\beta}\right)+\left(\pp{A}{r}\right)_nc_{-,n}\pp{\Delta_nt}{u}\notag\\
&\hspace{20mm}+\pp{t_{n-1}}{u}\left(\pp{A}{r}\right)_n\Delta_nc_-+\pp{t_{n-1}}{u}c_{-,n-1}\Delta_n\left(\pp{A}{r}\right)\notag\\
&\hspace{60mm}+\tilde{A}_n\frac{\partial^2\Delta_nt}{\partial u\partial v}+\frac{\partial^2t_{n-1}}{\partial u\partial v}\Delta_n\tilde{A}\bigg\}(u,v')dv'.
\end{align}
For the second term on the right of \eqref{eq:572} we use
\begin{align}
  \label{eq:573}
  \left|\Delta_n\left(\pp{A}{\alpha}\right)\right|\leq C\left\{|\Delta_n\alpha|+|\Delta_n\beta|+|\Delta_nr|+\left|\pp{\Delta_nt}{v}\right|\right\}\leq Cu^2\Lambda,
\end{align}
where for the last inequality we used \eqref{eq:502}, \eqref{eq:503}, \eqref{eq:565}, \eqref{eq:567}. The fourth term on the right of \eqref{eq:572} we treat analogous to the second. For the fifth term we use \eqref{eq:566} and for the sixth we use \eqref{eq:505}. The seventh and the last term in \eqref{eq:572} can be bounded in the same way as the second and the fourth. This leaves us with the eighth term. From \eqref{eq:506} in conjunction with \eqref{eq:524}, \eqref{eq:565}, \eqref{eq:566} we get
\begin{align}
  \label{eq:574}
  \left|\frac{\partial^2\Delta_nt}{\partial u\partial v}(u,v)\right|\leq Cuv\Lambda.
\end{align}
Putting things together we deduce from \eqref{eq:572}
\begin{align}
  \label{eq:575}
  \left|\pp{\Delta_{n+1}\alpha'}{u}(u,v)\right|\leq Cu^2\Lambda.
\end{align}
Using the second of \eqref{eq:332} we see that for the difference
\begin{align}
  \label{eq:576}
  \pp{\Delta_{n+1}\beta'}{v}
\end{align}
we get an analogous equation as we got in \eqref{eq:572} with the exception of the additional term
\begin{align}
  \label{eq:577}
  -\Delta_nB(v,v).
\end{align}
For the terms analogous to the ones showing up in \eqref{eq:572} we use the analogous estimates while for the term \eqref{eq:577} we use the first equality in \eqref{eq:569} with $u=v$ and the estimate \eqref{eq:571}. Therefore,
\begin{align}
  \label{eq:578}
  \left|\pp{\Delta_{n+1}\beta'}{v}(u,v)\right|\leq C uv\Lambda.
\end{align}

Using \eqref{eq:568}, \eqref{eq:571}, \eqref{eq:575}, \eqref{eq:578}, it follows
\begin{align}
  \label{eq:579}
  \|(\Delta_{n+1}\alpha',\Delta_{n+1}\beta')\|_X\leq Cu\Lambda= Cu \|(\Delta_{n}\alpha',\Delta_{n}\beta')\|_X.
\end{align}
It follows that for $\varepsilon$ small enough we have convergence of the sequence in the space $X$. This concludes the proof of the lemma.
\end{proof}

The two lemmas above show that the sequence $(\alpha_n',\beta_n')$ converges to $(\alpha',\beta')\in X$ uniformly in $T_\varepsilon$. Therefore we also have uniform convergence of $(\alpha_n,\beta_n)$ to $(\alpha,\beta)\in C^1(T_\varepsilon)$ (see \eqref{eq:331}).  Now, \eqref{eq:565}, \eqref{eq:566} show the convergence of the derivatives of $t_n$. Therefore, the pair of integral equations \eqref{eq:344}, \eqref{eq:345} are satisfied in the limit. We denote by $t$ the limit of $(t_n)$. It then follows that the mixed derivative $\partial^2t/\partial u\partial v$ satisfies \eqref{eq:339}. In view of the Hodograph system \eqref{eq:285} the partial derivatives of $r_n$ converge and the limit satisfies the Hodograph system. Let us denote by $r$ the limit of $(r_n)$. We have thus found a solution of the fixed boundary problem. Since every member of the sequence $(t_n)$ satisfies the expressions for the asymptotic form as given in the statement of proposition \ref{prop_inner_iteration} these expressions (i.e.~\eqref{eq:342}, \eqref{eq:343}) also hold for the limit $t$. We have therefore proven the following proposition.

\begin{proposition}\label{prop_fxdbp}
  Let $h(u)$ and $\alpha_i(u)$ be given by
  \begin{alignat}{5}
    \label{eq:580}
    h(u)&=u^3\hat{h}(u),& \qquad\hat{h}&\in C^1[0,\varepsilon],&\qquad \hat{h}(0)&=\frac{\lambda}{6\kappa(c_{+0}-c_{-0})},\\
    \alpha_i(u)&=\alpha_0+\dot{\alpha}_0v+v^2\hat{\alpha}_i(v),&\qquad \hat{\alpha}_i&\in C^1[0,\varepsilon],& \qquad\hat{\alpha}_i(0)&=\frac{1}{2}\ddot{\alpha}_0.
  \end{alignat}
Furthermore, let $\beta_{+}(v)\in C^1[0,\varepsilon]$, $V(v)\in C^0[0,\varepsilon]$ satisfy
\begin{align}
  \label{eq:581}
  \frac{d\beta_{+}}{dv}(v)= \mathcal{O}(v),\qquad V(v)=c_{+0}+\frac{\kappa}{2}(1+y(v))v+\mathcal{O}(v^2),
\end{align}
where $y$ is a given function such that
\begin{align}
  \label{eq:582}
  y\in C^1[0,\varepsilon],\qquad y(0)=-1.
\end{align}
Let
\begin{align}
  \label{eq:583}
  Y\coloneqq \sup_{[0,\varepsilon]}\left|\frac{dy}{dv}\right|.
\end{align}
Let $r_0>0$ and let $N_0$ be given as on page \pageref{eq:328}. Then there exists a solution $(\alpha, \beta,t,r)\in C^1(T_\varepsilon)$ of the characteristic system \eqref{eq:284}, \eqref{eq:285} such that $\alpha(u,0)=\alpha_i(u)$, $\beta(v,v)=\beta_{+}(v)$, $t(u,0)=h(u)$, $r(0,0)=r_0$ and
\begin{align}
  \label{eq:584}
  \pp{t}{v}(v,v)=\frac{1}{\gamma(v)}\pp{t}{u}(v,v),\qquad\textrm{with}\qquad \gamma(v)=\frac{c_+(\alpha(v,v), \beta_{+}(v))-V(v)}{V(v)-c_-(\alpha(v,v),\beta_+(v))},
\end{align}
\begin{align}
  \label{eq:585}
  \|(\alpha-\alpha_i,\beta-\beta_{+})\|_X\leq N_0,
\end{align}
\begin{align}
  \label{eq:586}
  \frac{df}{dv}(v)-\frac{\lambda}{3\kappa^2}v=\frac{\lambda}{18\kappa^2}\left\{2v(y(v)+1)+\frac{1}{v^2}\int_0^vv'^3\frac{dy}{dv}(v')dv'\right\}+\mathcal{O}(v^2),
\end{align}
\begin{align}
  \label{eq:587}
  \left|\pp{t}{v}(u,v)-\frac{\lambda}{3\kappa^2}v\right|\leq C(Y)uv,\qquad \left|\pp{t}{u}(u,v)-\frac{\lambda(3u^2-v^2)}{6\kappa(c_{+0}-c_{-0})}\right|\leq C(Y)u^3,
\end{align}
provided $\varepsilon$ is sufficiently small depending on $N_0$, $Y$.

\end{proposition}

%%% Local Variables: 
%%% mode: latex
%%% TeX-master: "./master"
%%% End: 

\section{Construction}
As outlined in the end of section \ref{chapter_setting} we solve the free boundary problem using an iteration. In the first subsection we specify the form of the boundary functions to be iterated and the corresponding function spaces. Then we establish the inductive step of the iteration. In the subsecond section we show convergence.

\subsection{Inductive Step}\label{part_one_inductive_step}
We recall briefly the strategy of the iteration. We start with the boundary functions corresponding to the $m$'th iterate $(z_m,\beta_{+,m},V_m)$. We then solve the corresponding fixed boundary problem using the result from the previous chapter. The solution of the fixed boundary problem provides us with the functions $\alpha_{+,m+1}$, $f_{m+1}$, $g_{m+1}$. Using $f_{m+1}$, $g_{m+1}$ we solve the identification equation, the solution of which we denote by $z_{m+1}$. Using then $\alpha_{+,m+1}$, $z_{m+1}$ in the jump conditions we obtain $\beta_{+,m+1}$, $V_{m+1}$. We have thus obtained the boundary functions corresponding to the $(m+1)$'th iterate $(z_{m+1},\beta_{+,m+1},V_{m+1})$. This concludes the iteration.

The input for the construction problem are the following assumptions for the boundary functions $z_m$, $\beta_{+,m}$, $V_m$:
\begin{align}
  \label{eq:588}
  z_m(v)&=vy_m(v)\hspace{9mm} \qquad\qquad\qquad\quad\textrm{with} \hspace{10.5mm} y_m(0)=-1,\\
  \beta_{+,m}(v)&=\beta_0+v^2\hat{\beta}_{+,m}(v)\qquad\qquad\qquad \textrm{with} \qquad \hat{\beta}_{+,m}(0)=\frac{\lambda}{6\kappa^2}\left(\frac{\partial\beta^\ast}{\partial t}\right)_0,\label{eq:589}\\
  V_m(v)&=c_{+0}+\frac{\kappa}{2}(1+y_m(v))v+v^2\hat{V}_m(v) \label{eq:590},
\end{align}
with $y_m,\hat{\beta}_{+,m}\in C^1[0,\varepsilon]$, $\hat{V}_{m}\in C^0[0,\varepsilon]$. $\lambda$, $\kappa$ and $(\partial \beta^\ast/\partial t)_0$ are given by the solution in the maximal development. We recall (see \eqref{eq:235}, \eqref{eq:241})
\begin{align}
  \label{eq:591}
  \kappa=\cp{\pppp{r^\ast}{w}{t}},\qquad \lambda=-\kappa\cp{\frac{\partial^3r^\ast}{\partial w^3}}.
\end{align}
We choose closed balls in the function spaces as follows
\begin{align}
  \label{eq:592}
  B_Y&=\left\{f\in C^1[0,\varepsilon]:f(0)=-1,\left|\frac{df}{dv}\right|\leq Y\right\},\\
    \label{eq:593}
    B_{\delta_1}&=\left\{f\in C^1[0,\varepsilon]:f(0)=\frac{\lambda}{6\kappa^2}\left(\pp{\beta^\ast}{t}\right)_0,\left|\frac{df}{dv}\right|\leq \delta_1\right\},\\
    \label{eq:594}
B_{\delta_2}&=\left\{f\in C^0[0,\varepsilon]:|f|\leq \delta_2\right\},
\end{align}
where $Y$, $\delta_1$, $\delta_2$ are to be chosen appropriately.

We initiate the sequence by
\begin{align}
  \label{eq:595}
  y_0\coloneqq -1,\qquad \hat{\beta}_{+,0}\coloneqq \frac{\lambda}{6\kappa^2}\left(\frac{\partial\beta^\ast}{\partial t}\right)_0,\qquad \hat{V}_0\coloneqq 0.
\end{align}
%thus
%\begin{align}
%  \label{eq:596}
%  z_0(v)=-v,\qquad \beta_{+,0}=\beta_0+\frac{\lambda}{6\kappa^2}\left(\frac{\partial\beta}{\partial t}\right)_0v^2,\qquad V_0(v)=c_{+0}.
%\end{align}

\begin{proposition}\label{prop_ind_step}
Choosing the constants $Y$, $\delta_1$, $\delta_2$ appropriately, the sequence $((y_m,\hat{\beta}_{+,m},\hat{V}_m);m=0,1,2,\ldots)$ is contained in $B_Y\times B_{\delta_1}\times B_{\delta_2}$, provided we choose $\varepsilon$ suitably small.
\end{proposition}

\begin{proof}
We see that
\begin{align}
  \label{eq:597}
  (y_0,\hat{\beta}_{+,0},\hat{V}_0)\in B_Y\times B_{\delta_1}\times B_{\delta_2}.
\end{align}

The inductive hypothesis is
\begin{align}
  \label{eq:598}
  y_m\in B_Y,\qquad \hat{\beta}_{+,m}\in B_{\delta_1},\qquad \hat{V}_m\in B_{\delta_2}.
\end{align}
Therefore,
\begin{align}
  \label{eq:599}
  \sup_{v\in[0,\varepsilon]}\left|\frac{dy_m}{dv}\right|\leq Y.
\end{align}

In the arguments to follow $q>0$ will denote a number which we can make as small as we wish by choosing $\varepsilon$ suitably small. From $y_m(0)=-1$ we obtain
\begin{align}
  \label{eq:600}
  |y_m(v)+1|\leq vY\leq q.
\end{align}
In the following we use the notation $g(v)=\mathcal{O}_d(v^n)$ to denote
\begin{align}
  \label{eq:601}
  |g(v)|\leq C(d)v^n,
\end{align}
where the constant $C$ is a non-negative, non-decreasing, continuous function of $d$.

From \eqref{eq:589} and \eqref{eq:590} we get
\begin{align}
  \label{eq:602}
  \frac{d\beta_{+,m}}{dv}(v)= \mathcal{O}(v),\qquad V_m(v)=c_{+0}+\frac{\kappa}{2}(1+y_m(v))v+\mathcal{O}_{\delta_2}(v^2),
\end{align}
provided that $\varepsilon$ is sufficiently small. This can be seen as follows. The statements \eqref{eq:602} are equivalent, respectively, to
\begin{align}
  \label{eq:603}
  \left|\frac{d\beta_{+,m}}{dv}(v)\right|\leq Cv,\qquad \left|V_m(v)-c_{+0}-\frac{\kappa}{2}(1+y_m(v))v\right|\leq C(\delta_2)v^2.
\end{align}
From the inductive hypothesis \eqref{eq:598} we have
\begin{align}
  \label{eq:604}
  \bigg|\hat{\beta}_{+,m}(v)-\frac{\lambda}{6\kappa^2}\left(\pp{\beta^\ast}{t}\right)_0\bigg|\leq v\delta_1,
\end{align}
which implies
\begin{align}
  \label{eq:605}
  |\hat{\beta}_{+,m}(v)|\leq C
\end{align}
for a fixed numerical constant $C$ if we choose $\varepsilon$ sufficiently small.  Using this in (see \eqref{eq:589})
\begin{align}
  \label{eq:606}
  \frac{1}{v}\frac{d\beta_{+,m}}{dv}(v)=2\hat{\beta}_{+,m}(v)+v\frac{d\hat{\beta}_{+,m}}{dv}(v)
\end{align}
we obtain
\begin{align}
  \label{eq:607}
  \bigg|\frac{1}{v}\frac{d\beta_{+,m}}{dv}(v)\bigg|\leq C+v\delta_1\leq C
\end{align}
if we choose $\varepsilon$ sufficiently small. This is equivalent to the first of \eqref{eq:603}. The second of \eqref{eq:603} follows directly from the inductive hypothesis.

Recalling the definition of $N_0$ in the fixed boundary problem (see \eqref{eq:320},\ldots,\eqref{eq:328}), we note that since the constant in the first of \eqref{eq:603} is a fixed numerical constant, also $N_0$ is a fixed numerical constant.

We now apply proposition \ref{prop_fxdbp} with $(y_m,\beta_{+,m},V_m)$ in the role of $(y,\beta_+,V)$. The resulting solution we denote by $(\alpha_{m+1},\beta_{m+1},t_{m+1},r_{m+1})$. We also denote
\begin{align}
  \label{eq:1540}
  f_{m+1}(v)&\coloneqq t_{m+1}(v,v),\\
  g_{m+1}(v)&\coloneqq r_{m+1}(v,v)-r_0,\\
  \alpha_{+,m+1}(v)&\coloneqq \alpha_{m+1}(v,v).
\end{align}

From the solution of the fixed boundary problem we have
\begin{align}
  \label{eq:608}
  \frac{df_{m+1}}{dv}(v)-\frac{\lambda}{3\kappa^2}v=\frac{\lambda}{18\kappa^2}\left\{2v(y_{m}(v)+1)+\frac{1}{v^2}\int_0^vv'^3\frac{dy_{m}}{dv}(v')dv'\right\}+\mathcal{O}_{\delta_2}(v^2).
\end{align}
Defining the function $\hat{f}_{m+1}$ by
\begin{align}
  \label{eq:609}
  f_{m+1}(v)=v^2\hat{f}_{m+1}(v),
\end{align}
we deduce
\begin{align}
  \label{eq:610}
  \frac{d\hat{f}_{m+1}}{dv}(v)&=\frac{1}{v^2}\left(\frac{df_{m+1}}{dv}(v)-\frac{2}{v}f_{m+1}(v)\right)\notag\\
&=\frac{\lambda}{18\kappa^2v^2}\left\{2v(y_m(v)+1)+\frac{1}{v^2}\int_0^vv'^3\frac{dy_m}{dv}(v')dv'-\frac{2}{v}(A_m+B_m)\right\}+\mathcal{O}_{\delta_2}(1),
\end{align}
where
\begin{align}
  \label{eq:611}
  A_m\coloneqq \int_0^v2v'(y_m(v')+1)dv',\qquad B_m\coloneqq \int_0^v\frac{1}{v'^2}\left(\int_0^{v'}v''^3\frac{dy_m}{dv}(v'')dv''\right)dv'
\end{align}
and we used \eqref{eq:608}. Integrating by parts we obtain
\begin{align}
  \label{eq:612}
  A_m=v^2(y_m(v)+1)-\int_0^vv'^2\frac{dy_m}{dv}(v')dv',\qquad B_m=-\frac{1}{v}\int_0^vv'^3\frac{dy_m}{dv}(v')dv'+\int_0^vv'^2\frac{dy_m}{dv}(v')dv'.
\end{align}
This implies
\begin{align}
  \label{eq:613}
  \frac{d\hat{f}_{m+1}}{dv}(v)=\frac{\lambda}{6\kappa^2v^4}\int_0^vv'^3\frac{dy_m}{dv}(v')dv'+\mathcal{O}_{\delta_2}(1).
\end{align}
Hence
\begin{align}
  \label{eq:614}
  \bigg|\frac{d\hat{f}_{m+1}}{dv}(v)\bigg|\leq \frac{\lambda Y}{24\kappa^2}+C(\delta_2).
\end{align}

Defining the function $\delta_m(v)$ by
\begin{align}
  \label{eq:615}
  \delta_m(v)\coloneqq g_m(v)-c_{+0}f_m(v),
\end{align}
we obtain
\begin{align}
  \label{eq:616}
  \frac{d\delta_{m+1}}{dv}(v)=(V_m(v)-c_{+0})\frac{df_{m+1}}{dv}(v).
\end{align}
\begin{remark}
The appearance of the indices $m$ and $m+1$ originates from the basic strategy (see \eqref{eq:281}) in which the solution of the fixed boundary problem carrying the index $m+1$ satisfies
  \begin{align}
    \label{eq:617}
    \frac{df_{m+1}}{dv}(v)V_m(v)=\frac{dg_{m+1}}{dv}(v).
  \end{align}
This confusion did not appear up to now since we dropped the index $m$ from the outer iteration during the solution process of the fixed boundary problem. An analogous situation appears when evaluating the function $\beta_{m+1}$ on the boundary $u=v$. There we have
\begin{align}
  \label{eq:618}
  \beta_{m+1}(v,v)=\beta_{+,m}(v),
\end{align}
since $\beta_{+,m}$ is the boundary value for the fixed boundary problem whose solution carries the index $m+1$. The function $\beta_{+,m+1}$ is determined later on by making use of the jump conditions. The appearance of the indices $m$ and $m+1$ in \eqref{eq:608} are explained in the same way.
\end{remark}

We define the function $\phi(v)$ by
\begin{align}
  \label{eq:619}
  \phi_{m+1}(v)\coloneqq \frac{df_{m+1}}{dv}(v)-\frac{\lambda}{3\kappa^2}v.
\end{align}
We split up the function $\delta_{m+1}(v)$ according to
\begin{align}
  \label{eq:620}
  \delta_{m+1}(v)=\delta_0(v)+\delta_1(v),
\end{align}
where the functions $\delta_0(v)$ and $\delta_1(v)$ are given by $\delta_0(0)=0$, $\delta_1(0)=0$ and
\begin{align}
  \label{eq:621}
  \frac{d\delta_0}{dv}(v)&=\frac{\lambda}{6\kappa}(1+y_m(v))v^2,\notag\\
 \frac{d\delta_1}{dv}(v)&=\left(V_m(v)-c_{+0}-\frac{\kappa}{2}(1+y_m(v))v\right)\frac{\lambda}{3\kappa^2}v+(V_m(v)-c_{+0})\phi_{m+1}(v).
\end{align}
(We make use of $V_m(v)=c_{+0}+\frac{\kappa}{2}(1+y_m(v))v+\mathcal{O}_{\delta_2}(v^2)$). Defining the functions $\hat{\delta}_i(v),i=0,1$ by $\delta_i(v)=v^3\hat{\delta}_i(v)$ we get
\begin{align}
  \label{eq:622}
  \frac{d\hat{\delta}_0}{dv}(v)&=\frac{1}{v^3}\frac{d\delta_0}{dv}(v)-\frac{3}{v^4}\delta_0(v)\notag\\
&=\frac{\lambda}{6\kappa}\left\{\frac{1+y_m(v)}{v}-\frac{3}{v^4}\int_0^v(1+y_m(v'))v'^2dv'\right\}\notag\\
&=\frac{\lambda}{6\kappa v^4}\int_0^vv'^3\frac{dy_m}{dv}(v')dv',
\end{align}
where we integrated by parts.

From $V_m(v)-c_{+0}=\frac{\kappa}{2}(1+y_m(v))v+\mathcal{O}_{\delta_2}(v^2)$ we have (see \eqref{eq:600})
\begin{align}
  \label{eq:623}
  |V_m(v)-c_{+0}|\leq \frac{\kappa}{2}v^2Y+C(\delta_2)v^2.
\end{align}
Together with (see the right hand side of \eqref{eq:608})
\begin{align}
  \label{eq:624}
  |\phi_{m+1}(v)|\leq CYv^2+C(\delta_2)v^2,
\end{align}
we obtain
\begin{align}
  \label{eq:625}
  |V_m(v)-c_{+0}||\phi_{m+1}(v)|\leq CY^2v^4+C(\delta_2)Yv^4.
\end{align}
Therefore, for $\varepsilon$ sufficiently small,
\begin{align}
  \label{eq:626}
  \left|\frac{d\delta_1}{dv}(v)\right|\leq C(\delta_2)v^3,\qquad |\delta_1(v)|\leq C(\delta_2)v^4.
\end{align}
Using this in
\begin{align}
  \label{eq:627}
  \frac{d\hat{\delta}_1}{dv}(v)=\frac{1}{v^3}\frac{d\delta_1}{dv}(v)-\frac{3}{v^4}\delta_1(v)
\end{align}
we deduce
\begin{align}
  \label{eq:628}
  \bigg|\frac{d\hat{\delta}_1}{dv}(v)\bigg|\leq C(\delta_2).
\end{align}

Using \eqref{eq:622}, \eqref{eq:628} we arrive at
\begin{align}
  \label{eq:629}
  \frac{d\hat{\delta}_{m+1}}{dv}(v)=\frac{\lambda}{6\kappa v^4}\int_0^vv'^3\frac{dy_m}{dv}(v')dv'+\mathcal{O}_{\delta_2}(1).
\end{align}
Hence
\begin{align}
  \label{eq:630}
  \bigg|\frac{d\hat{\delta}_{m+1}}{dv}(v)\bigg|\leq \frac{\lambda Y}{24\kappa}+C(\delta_2).
\end{align}

In view of \eqref{eq:608}, \eqref{eq:616}, \eqref{eq:623} we have
\begin{align}
  \label{eq:631}
  \left|\frac{d\delta_{m+1}}{dv}(v)\right|\leq C(Y,\delta_2)v^3.
\end{align}
Therefore, in conjunction with $\delta_{m+1}(0)=g_{m+1}(0)-c_{+0}f_{m+1}(0)=0$, we have
\begin{align}
  \label{eq:632}
  |\delta_{m+1}(v)|\leq C(Y,\delta_2)v^4,
\end{align}
which implies
\begin{align}
  \label{eq:633}
  \hat{\delta}_{m+1}(0)=0.
\end{align}

Now we look at the identification equation, i.e.~at
\begin{align}
  \label{eq:634}
  g_{m+1}(v)+r_0=r^\ast(f_{m+1}(v),vy_{m+1}).
\end{align}
Here the function on the right hand side is the solution $r^\ast(t,w)$ given in the maximal development (recall that we set $t_0=w_0=0$), while the left hand side is given by the solution of the fixed boundary problem (The identification equation is an equation for $y_{m+1}$ as a function of $v$ given the functions $g_{m+1}(v)$, $f_{m+1}(v)$). We have
\begin{align}
  \label{eq:635}
  g_{m+1}(v)=\delta_{m+1}(v)+c_{+0}f_{m+1}(v)=v^3\hat{\delta}_{m+1}(v)+c_{+0}v^2\hat{f}_{m+1}(v).
\end{align}
In the following discussion of the identification equation we will omit the index $m+1$. We define the function
\begin{align}
  \label{eq:636}
  h(t,w)&\coloneqq r^\ast(t,w)-r_0-\left(\pp{r^\ast}{t}\right)_0t-\left(\frac{\partial^2r^\ast}{\partial t^2}\right)_0\frac{t^2}{2}-\left(\frac{\partial^2r^\ast}{\partial t\partial w}\right)_0tw\notag\\
&\qquad\qquad-\left(\frac{\partial^4r^\ast}{\partial w^4}\right)_0\frac{w^4}{24}-\left(\frac{\partial^3r^\ast}{\partial t\partial w^2}\right)_0\frac{tw^2}{2}-\left(\frac{\partial^3r^\ast}{\partial w^3}\right)_0\frac{w^3}{6}.
\end{align}
Thus
\begin{align}
  \label{eq:637}
  r^\ast(t,w)&=r_0+\left(\pp{r^\ast}{t}\right)_0t+\left(\frac{\partial^2r^\ast}{\partial t^2}\right)_0\frac{t^2}{2}+\left(\frac{\partial^2r^\ast}{\partial t\partial w}\right)_0tw\notag\\
&\quad\qquad +\left(\frac{\partial^4r^\ast}{\partial w^4}\right)_0\frac{w^4}{24}+\left(\frac{\partial^3r^\ast}{\partial t\partial w^2}\right)_0\frac{tw^2}{2}+\left(\frac{\partial^3r^\ast}{\partial w^3}\right)_0\frac{w^3}{6}+h(t,w).
\end{align}
\begin{remark}
  The function $h(t,w)$ is introduced in order to represent the terms $\mathcal{O}(v^k)$ for $k\geq 5$ in the expansion of $r^\ast(t,w)$, when $v^2\hat{f}(v)$ for $t$ and $vy$ for $w$ are being substituted. We use
  \begin{align}
    \label{eq:638}
        \left(\pp{r^\ast}{w}\right)_0=\left(\frac{\partial^2r^\ast}{\partial w^2}\right)_0=0.
  \end{align}
\end{remark}
Let now
\begin{align}
  \label{eq:639}
  F(v,y)\coloneqq g(v)+r_0-r^\ast(f(v),vy).
\end{align}
The identification equation becomes
\begin{align}
  \label{eq:640}
  F(v,y)=0.
\end{align}
Using now
\begin{align}
    \label{eq:641}
\left(\pp{r^\ast}{t}\right)_0=c_{+0},\qquad\left(\frac{\partial^2r^\ast}{\partial w\partial t}\right)_0=\kappa,\qquad \left(\frac{\partial^3r^\ast}{\partial w^3}\right)_0=-\frac{\lambda}{\kappa},
\end{align}
and expressing $(t,w)$ in terms of $v$ and $y$ according to $t=f(v)=v^2\hat{f}(v)$, $w=vy$ and making use of \eqref{eq:635} the function $F(v,y)$ becomes
\begin{align}
  \label{eq:642}
  F(v,y)&=\frac{\lambda}{6\kappa}v^3y^3-\kappa v^3\hat{f}(v)y+v^3\hat{\delta}(v)\notag\\
&\qquad-\left(\frac{\partial^2r^\ast}{\partial t^2}\right)_0v^4\left(\hat{f}(v)\right)^2-\left(\frac{\partial^4r^\ast}{\partial w^4}\right)_0\frac{y^4v^4}{24}-\left(\frac{\partial^3r^\ast}{\partial t\partial w^2}\right)_0\frac{y^2v^4\hat{f}(v)}{2}-h(v^2\hat{f}(v),vy).
\end{align}
We note that
\begin{align}
  \label{eq:643}
 h(v^2\hat{f}(v),vy)=v^5H(\hat{f}(v),y),
\end{align}
where $H$ is a smooth function of its two arguments.

Defining the function $\hat{F}$ by the relation
\begin{align}
  \label{eq:644}
  F(v,y)=v^3\hat{F}(v,y),
\end{align}
\eqref{eq:642} is equivalent to
\begin{align}
  \label{eq:645}
  \hat{F}(v,y)=\frac{\lambda}{6\kappa}y^3-\kappa\hat{f}(v)y+\hat{\delta}(v)+vR(v,y),
\end{align}
where the remainder $R$ is given by
\begin{align}
  \label{eq:646}
  R(v,y)\coloneqq -\left(\frac{\partial^2r^\ast}{\partial t^2}\right)_0\left(\hat{f}(v)\right)^2-\left(\frac{\partial^4r^\ast}{\partial w^4}\right)_0\frac{y^4}{24}-\left(\frac{\partial^3r^\ast}{\partial t\partial w^2}\right)_0\frac{y^2\hat{f}(v)}{2}-vH(\hat{f}(v),y).
\end{align}

The identification equation is now equivalent to
\begin{align}
  \label{eq:647}
  \hat{F}(v,y)=0.
\end{align}
At $v=0$ this becomes (we recall that $\hat{\delta}(0)=0$ and note that $\hat{f}(0)=\lambda/6\kappa^2$ (see \eqref{eq:608}, \eqref{eq:609}))
\begin{align}
  \label{eq:648}
  \frac{\lambda}{6\kappa}y(y^2-1)=0.
\end{align}
The only physical solution is $y=-1$. We set $v_0\coloneqq 0$, $y_0\coloneqq -1$. We now see that
\begin{align}
  \label{eq:649}
  \pp{\hat{F}}{y}(v_0,y_0)=\frac{\lambda}{3\kappa}> 0.
\end{align}
Therefore we are able to solve the identification equation for $y$ as a function of $v$ for $v$ small enough. We have
\begin{align}
  \label{eq:650}
  |y|\leq C.
\end{align}

Differentiating \eqref{eq:647} implicitly yields
\begin{align}
  \label{eq:651}
  \frac{dy}{dv}(v)=-\frac{\displaystyle{\pp{\hat{F}}{v}(v,y(v))}}{\displaystyle{\pp{\hat{F}}{y}(v,y(v))}}.
\end{align}
We have
\begin{align}
  \label{eq:652}
  \pp{\hat{F}}{y}(v,y)&=\frac{\lambda}{2\kappa}y^2-\kappa\hat{f}(v)+v\pp{R}{y}(v,y),\\
  \label{eq:653}
  \pp{\hat{F}}{v}(v,y)&=-\kappa y\frac{d\hat{f}}{dv}(v)+\frac{d\hat{\delta}}{dv}(v)+R(v,y)+v\pp{R}{v}(v,y).
\end{align}
We first derive bounds for the remainder $R$ and its derivatives of first order.
%We start with a bound for $\hat{H}$. We have
%\begin{align}
%  \label{eq:654}
%  \hat{H}(v,y)=\frac{1}{v^5}\sum_{2\alpha_1+\alpha_2\geq 5}\frac{(\partial^\alpha r)_0}{|\alpha|!}\left(f(v)\right)^{\alpha_1}\left(vy\right)^{\alpha_2}=\sum_{2\alpha_1+\alpha_2\geq 5}\frac{(\partial^\alpha r)_0}{|\alpha|!}\left(\hat{f}(v)\right)^{\alpha_1}y^{\alpha_2}v^{2\alpha_1+\alpha_2-5}.
%\end{align}
From \eqref{eq:614} we have, for $\varepsilon$ small enough,
\begin{align}
  \label{eq:655}
  |\hat{f}(v)|\leq C,\qquad \bigg|v\frac{d\hat{f}}{dv}(v)\bigg|\leq C.
\end{align}
%it follows from \eqref{eq:654} that
%\begin{align}
%  \label{eq:656}
%  |\hat{H}(v,y)|\leq C.
%\end{align}
%Taking the derivative of \eqref{eq:654} with respect to $y$ we get
%\begin{align}
%  \label{eq:657}
%  \pp{\hat{H}}{y}(v,y)=\sum_{2\alpha_1+\alpha_2\geq 5}\frac{(\partial^\alpha r)_0}{|\alpha|!}\left(\hat{f}(v)\right)^{\alpha_1}\alpha_2 y^{\alpha_2-1}v^{2\alpha_1+\alpha_2-5},
%\end{align}
%while taking the derivative of \eqref{eq:651} yields
%\begin{align}
%  \label{eq:658}
%  v\pp{\hat{H}}{v}(v,y)&=\sum_{2\alpha_1+\alpha_2\geq 5}\frac{(\partial^\alpha r)_0}{|\alpha|!}\bigg\{\alpha_1\left(\hat{f}(v)\right)^{\alpha_1-1}y^{\alpha_2}v^{2\alpha_1+\alpha_2-4}\frac{d\hat{f}}{dv}(v)\notag\\
%&\hspace{45mm}+\left(\hat{f}(v)\right)^{\alpha_1}y^{\alpha_2}(2\alpha_1+\alpha_2-5)v^{2\alpha_1+\alpha_2-5}\bigg\}
%\end{align}
%Using the same reasoning used to arrive at \eqref{eq:656} we deduce
%\begin{align}
%  \label{eq:659}
%  \bigg|v\pp{\hat{H}}{v}(v,y)\bigg|\leq C.
%\end{align}
%Using again \eqref{eq:655} and \eqref{eq:650} we see that the first three terms in \eqref{eq:646} together with their derivatives with respect to $y$ are bounded in absolute value by a constant.
%To bound the derivatives with respect of $v$ in \eqref{eq:646} we make use of the fact that for $\varepsilon$ small enough we have (see \eqref{eq:614})
%\begin{align}
%  \label{eq:660}
%  
%\end{align}
Therefore,
\begin{align}
  \label{eq:661}
  \left|v\pp{R}{v}(v,y)\right|\leq C,\qquad \left|v\pp{R}{y}(v,y)\right|\leq C,\qquad |R(v,y)|\leq C.
\end{align}

We now examine \eqref{eq:651}. Using \eqref{eq:652} we find for the denominator
\begin{align}
  \label{eq:662}
  \bigg|\pp{\hat{F}}{y}(v,y)-\pp{\hat{F}}{y}(v_0,y_0)\bigg|\leq \frac{\lambda}{2\kappa}|y^2-1|+\kappa |\hat{f}(v)-\hat{f}(0)|+v\left|\pp{R}{v}(v,y)\right|.
\end{align}
Now, for small enough $\varepsilon$, we have
\begin{align}
  \label{eq:663}
  |y^2-1|&=|(y+1)(y-1)|\leq C|y+1|\leq CYv\leq Cv^{\frac{1}{2}},\\
  \label{eq:664}
  |\hat{f}(v)-\hat{f}(0)|&\leq \sup_{v'\in[0,v]}\bigg|\frac{d\hat{f}}{dv}(v')\bigg|v\leq CYv+C(\delta_2)v\leq Cv^{\frac{1}{2}},
\end{align}
which, together with the second of \eqref{eq:661} implies
\begin{align}
  \label{eq:665}
  \bigg|\pp{\hat{F}}{y}(v,y)-\pp{\hat{F}}{y}(v_0,y_0)\bigg|\leq Cv^{\frac{1}{2}}.
\end{align}

Now we look at the numerator of \eqref{eq:651}. Making use of \eqref{eq:653} we get
\begin{align}
  \label{eq:666}
  \bigg|\pp{\hat{F}}{v}(v,y)-\kappa\frac{d\hat{f}}{dv}(v)-\frac{d\hat{\delta}}{dv}(v)\bigg|\leq \kappa \bigg|\frac{d\hat{f}}{dv}(v)\bigg||y+1|+|R(v,y)|+v\left|\pp{R}{v}(v,y)\right|.
\end{align}
From \eqref{eq:614} together with \eqref{eq:600} we have, for small enough $\varepsilon$,
\begin{align}
  \label{eq:667}
  \bigg|\frac{d\hat{f}}{dv}(v)\bigg||y+1|\leq C.
\end{align}
Together with the first and the third of \eqref{eq:661} we obtain
\begin{align}
  \label{eq:668}
  \pp{\hat{F}}{v}(v,y)=\kappa \frac{d\hat{f}}{dv}(v)+\frac{d\hat{\delta}}{dv}(v)+\mathcal{O}(1),
\end{align}
which implies, through \eqref{eq:613}, \eqref{eq:629},
\begin{align}
  \label{eq:669}
  \pp{\hat{F}}{v}(v,y)=\frac{\lambda}{3\kappa v^4}\int_0^vv'^3\frac{dy}{dv}(v')dv'+\mathcal{O}_{\delta_2}(1).
\end{align}
%\begin{align}
%  \label{eq:670}
%  \bigg|\pp{\hat{F}}{v}(v,y)-\pp{\hat{F}}{v}(v_0,y_0)\bigg|&\leq \kappa \bigg|\frac{d\hat{f}}{dv}(v)\bigg|+\bigg|\frac{d\hat{\delta}}{dv}(v)\bigg|+C(\delta_2)\notag\\
%&\leq \frac{\lambda Y}{12\kappa}+C(\delta_2),
%\end{align}
%where we used \eqref{eq:614} and \eqref{eq:630}. 

Substituting \eqref{eq:669} for the numerator in \eqref{eq:651}, using the estimate \eqref{eq:665} together with \eqref{eq:649} for the denominator in \eqref{eq:651} and putting back the indices $m$ and $m+1$ we arrive at
\begin{align}
  \label{eq:671}
  \frac{dy_{m+1}}{dv}(v)=\frac{1}{v^4(1-\varepsilon_m(v))}\int_0^vv'^3\frac{dy_m}{dv}(v')dv'+\mathcal{O}_{\delta_2}(1),
\end{align}
where
\begin{align}
  \label{eq:672}
  |\varepsilon_m(v)|\leq Cv^{\frac{1}{2}}.
\end{align}
Taking the absolute value and then taking the supremum over $v\in[0,\varepsilon]$ yields
\begin{align}
  \label{eq:674}
  \sup_{v\in[0,\varepsilon]}\bigg|\frac{dy_{m+1}}{dv}(v)\bigg|\leq \frac{\frac{1}{4}Y}{1-C\varepsilon^{\frac{1}{2}}}+C'(\delta_2).
\end{align}
Choosing then $\varepsilon$ suitably small such that
\begin{align}
  \label{eq:675}
  C\varepsilon^{\frac{1}{2}}\leq \frac{1}{2},
\end{align}
where $C$ is the constant appearing in the denominator of \eqref{eq:674}, we obtain
\begin{align}
  \label{eq:676}
  \sup_{v\in[0,\varepsilon]}\bigg|\frac{dy_{m+1}}{dv}(v)\bigg|\leq \frac{1}{2}Y+C'(\delta_2).
\end{align}
Therefore, choosing now $Y=2C'(\delta_2)$,
\begin{align}
  \label{eq:677}
  \sup_{v\in[0,\varepsilon]}\bigg|\frac{dy_{m+1}}{dv}(v)\bigg|\leq Y.
\end{align}
\begin{remark}
  $Y$ depends on $\delta_2$.
\end{remark}

In the following we will establish closure for the iterations of the functions $\hat{\beta}_m$, $\hat{V}_m$. The balls for the respective iterations have been chosen according to \eqref{eq:593}, \eqref{eq:594}. Since there are no more indices $m$ to appear, only indices $m+1$, we will omit in the following the index $m+1$. We have
\begin{align}
  \label{eq:678}
  Y=\mathcal{O}_{\delta_2}(1).
\end{align}
From \eqref{eq:614} it follows (using $\hat{f}(0)=\lambda/6\kappa^2$)
\begin{align}
  \label{eq:679}
  \frac{d\hat{f}}{dv}(v)=\mathcal{O}_{\delta_2}(1),\qquad \hat{f}(v)=\frac{\lambda}{6\kappa^2}+\mathcal{O}_{\delta_2}(v).
\end{align}
This implies, through \eqref{eq:609} and the first line of \eqref{eq:610},
\begin{align}
  \label{eq:680}
  \frac{df}{dv}(v)&=\frac{\lambda}{3\kappa^2}v+\mathcal{O}_{\delta_2}(v^2),\\
  \label{eq:681}
  f(v)&=\frac{\lambda}{6\kappa^2}v^2+\Landau_{\delta_2}(v^3).
\end{align}
From $z_{m+1}(v)=vy_{m+1}(v)$ we obtain
\begin{align}
  \label{eq:682}
  \frac{dz}{dv}(v)&=v\frac{dy}{dv}(v)+y(v)\notag\\
&=-1+\mathcal{O}_{\delta_2}(v).
\end{align}
Therefore,
\begin{align}
  \label{eq:683}
  z(v)=-v+\Landau_{\delta_2}(v^2).
\end{align}

Now we look at the asymptotic form of $\beta_-(v)$. We have
\begin{align}
  \label{eq:684}
  \beta_-(v)=\beta^\ast(f(v),z(v)),
\end{align}
where the function on the right is $\beta^\ast(t,w)$ from the state ahead (i.e.~given by the solution in the maximal development) and we recall that $t_0=w_0=0$. We have
\begin{align}
  \label{eq:685}
  \frac{d\beta_-}{dv}(v)=\pp{\beta^\ast}{t}(f(v),z(v))\frac{df}{dv}(v)+\pp{\beta^\ast}{w}(f(v),z(v))\frac{dz}{dv}(v).
\end{align}
Expanding $(\partial\beta^\ast/\partial t)(t,w)$ to first order and substituting $t=f(v)$, $w=z(v)$ we obtain
\begin{align}
  \label{eq:686}
  \pp{\beta^\ast}{t}(f(v),z(v))=\cp{\pp{\beta^\ast}{t}}-\cp{\pppp{\beta^\ast}{t}{w}}v+\Landau_{\delta_2}(v^2),
\end{align}
while expanding $(\partial\beta^\ast/\partial w)(t,w)$ to second order, substituting $t=f(v)$, $w=z(v)$ and using (see \eqref{eq:257})
\begin{align}
  \label{eq:687}
  \left(\pp{\beta^\ast}{w}\right)_0=0,\qquad\left(\frac{\partial^2\beta^\ast}{\partial w^2}\right)_0=0,
\end{align}
yields
\begin{align}
  \label{eq:688}
  \pp{\beta^\ast}{w}(f(v),z(v))=\left\{\cp{\pppp{\beta^\ast}{t}{w}}\frac{\lambda}{6\kappa^2}+\frac{1}{2}\cp{\frac{\partial^3\beta^\ast}{\partial w^3}}\right\}v^2+\Landau_{\delta_2}(v^3).
\end{align}
%\begin{align}
%  \label{eq:689}
%  \pp{\beta}{t}(f(v),z(v))=\left(\pp{\beta}{t}\right)_0+\mathcal{O}_{\delta_2}(v),\qquad \pp{\beta}{w}(f(v),z(v))=\mathcal{O}_{\delta_2}(v^2).
%\end{align}
Therefore,
\begin{align}
  \label{eq:690}
  \frac{d\beta_-}{dv}(v)=\cp{\pp{\beta^\ast}{t}}\frac{\lambda}{3\kappa^2}v+\Landau_{\delta_2}(v^2).
\end{align}
Hence,
\begin{align}
  \label{eq:691}
  \beta_-(v)=\beta_0+\left(\pp{\beta^\ast}{t}\right)_0\frac{\lambda}{6\kappa^2}v^2+\mathcal{O}_{\delta_2}(v^3).
\end{align}

Now we find the asymptotic form of $\alpha_-(v)$. From
\begin{align}
  \label{eq:692}
  \alpha_-(v)=\alpha^\ast(f(v), z(v)),
\end{align}
we have
\begin{align}
  \frac{d\alpha_-}{dv}(v)=\pp{\alpha^\ast}{t}(f(v),z(v))\frac{df}{dv}(v)+\pp{\alpha^\ast}{w}(f(v),z(v))\frac{dz}{dv}(v).
\end{align}
Expanding now $(\partial\alpha^\ast/\partial t)(t,w)$ and $(\partial \alpha^\ast/\partial w)(t,w)$ to first order and substituting $t=f(v)$, $w=z(v)$ we obtain
\begin{align}
  \label{eq:695}
  \pp{\alpha^\ast}{t}(f(v),z(v))&=\cp{\pp{\alpha^\ast}{t}}-\cp{\pppp{\alpha^\ast}{t}{w}}v+\Landau_{\delta_2}(v^2),\\
  \pp{\alpha^\ast}{w}(f(v),z(v))&=\dot{\alpha}_0-\cp{\ppp{\alpha^\ast}{w}}v+\Landau_{\delta_2}(v^2).
\end{align}
Therefore,
\begin{align}
  \label{eq:696}
  \frac{d\alpha_-}{dv}(v)=\left\{\cp{\pp{\alpha^\ast}{t}}\frac{\lambda}{3\kappa^2}+\cp{\ppp{\alpha^\ast}{w}}\right\}v+\dot{\alpha}_0\frac{d}{dv}(vy(v))+\mathcal{O}_{\delta_2}(v^2).
\end{align}
Thus,
\begin{align}
  \label{eq:697}
  \alpha_-(v)&=\alpha_0+\dot{\alpha}_0vy(v)+\left\{\cp{\pp{\alpha^\ast}{t}}\frac{\lambda}{3\kappa^2}+\cp{\ppp{\alpha^\ast}{w}}\right\}\frac{v^2}{2}+\Landau_{\delta_2}(v^3)\notag\\
  &=\alpha_0+\dot{\alpha}_0vy(v)+\mathcal{O}(v^2),
\end{align}
where the last equality holds provided we choose $\varepsilon$ suitably small.

We now deal with the asymptotic form of $\alpha_+(v)$. From
\begin{align}
  \label{eq:698}
  \pp{\alpha}{v}=\pp{t}{v}\tilde{A}(\alpha,\beta,r)
\end{align}
we obtain
\begin{align}
  \label{eq:699}
  \alpha_+(v)=\alpha(v,v)=\alpha_i(v)+\int_0^v\left(\pp{t}{v}\tilde{A}(\alpha,\beta,r)\right)(v,v')dv',
\end{align}
where we made use of $\alpha(v,0)=\alpha_i(v)$. This implies
\begin{align}
  \label{eq:700}
  \frac{d\alpha_+}{dv}(v)&=\frac{d\alpha_i}{dv}(v)+\left(\pp{t}{v}\tilde{A}(\alpha,\beta,r)\right)(v,v)\notag\\
&\qquad+\int_0^v\Bigg\{\left(\frac{\partial^2t}{\partial u\partial v}\tilde{A}(\alpha,\beta,r)\right)\notag\\
&\hspace{20mm}+\pp{t}{v}\left(\pp{\tilde{A}}{\alpha}(\alpha,\beta,r)\pp{\alpha}{u}+\pp{\tilde{A}}{\beta}(\alpha,\beta,r)\pp{\beta}{u}+\pp{\tilde{A}}{r}(\alpha,\beta,r)c_-(\alpha,\beta)\pp{t}{u}\right)\Bigg\}(v,v')dv'.
\end{align}

The solution of the fixed boundary problem satisfies (see proposition \ref{prop_fxdbp})
\begin{align}
  \label{eq:701}
  \left|\pp{t}{v}(u,v)-\frac{\lambda}{3\kappa^2}v\right|\leq Cuv,\qquad \left|\pp{t}{u}(u,v)-\frac{\lambda(3u^2-v^2)}{6\kappa(c_{+0}-c_{-0})}\right|\leq Cu^3.
\end{align}
Here the constants depend on $Y$ and $\delta_2$. Therefore (cf.~\eqref{eq:678}),
\begin{align}
  \label{eq:702}
  \pp{t}{v}(u,v)=\frac{\lambda}{3\kappa^2}v+\mathcal{O}_{\delta_2}(uv),\qquad \pp{t}{u}(u,v)=\frac{\lambda(3u^2-v^2)}{6\kappa(c_{+0}-c_{-0})}+\Landau_{\delta_2}(u^3).
\end{align}
The first implies
\begin{align}
  \label{eq:703}
  \left(\pp{t}{v}\tilde{A}(\alpha,\beta,r)\right)(v,v)=\frac{\lambda}{3\kappa^2}\tilde{A}_0v+\mathcal{O}_{\delta_2}(v^2).
\end{align}
From
\begin{align}
  \label{eq:704}
  \frac{\partial^2t}{\partial u\partial v}=-\mu\pp{t}{v}+\nu\pp{t}{u},
\end{align}
we deduce, together with \eqref{eq:702} and (cf.~\eqref{eq:465}, \eqref{eq:469})
\begin{align}
  \label{eq:705}
  \mu(u,v)=\frac{\kappa}{c_{+0}-c_{-0}}+\Landau_{\delta_2}(u),\qquad \nu(u,v)=\mathcal{O}_{\delta_2}(v),
\end{align}
that
\begin{align}
  \label{eq:706}
  \frac{\partial^2t}{\partial u\partial v}(v,v')=-\frac{\lambda}{3\kappa(c_{+0}-c_{-0})}v+\Landau_{\delta_2}(v^2)=\mathcal{O}(v),
\end{align}
where the second equality holds provided we choose $\varepsilon$ small enough. Therefore
\begin{align}
  \label{eq:707}
  \int_0^v\left(\frac{\partial^2t}{\partial u\partial v}\tilde{A}(\alpha,\beta,r)\right)(v,v')dv'=\mathcal{O}(v^2).
\end{align}
Using the first of \eqref{eq:702} we find
\begin{align}
  \label{eq:708}
  \int_0^v\left(\pp{t}{v}(\ldots)\right)(v,v')dv'=\mathcal{O}(v^2),
\end{align}
where we denote by $(\ldots)$ the bracket in the last line of \eqref{eq:700}. We conclude from \eqref{eq:703}, \eqref{eq:707}, \eqref{eq:708} that
\begin{align}
  \label{eq:709}
  \frac{d\alpha_+}{dv}(v)=\frac{d\alpha_i}{dv}(v)+\frac{\lambda\tilde{A}_0}{3\kappa^2}v+\mathcal{O}_{\delta_2}(v^2),
\end{align}
which implies
\begin{align}
  \label{eq:710}
  \alpha_+(v)&=\alpha_i(v)+\frac{\lambda\tilde{A}_0}{6\kappa^2}v^2+\Landau_{\delta_2}(v^3)\notag\\
&=\alpha_i(v)+\mathcal{O}(v^2),
\end{align}
where the second equality holds provided we choose $\varepsilon$ small enough.

We now turn to the jumps $\jump{\alpha(v)}$, $\jump{\beta(v)}$. The first line of \eqref{eq:710} together with the first line of \eqref{eq:697} yields (note that $\tilde{A}_0=\cp{\partial\alpha^\ast/\partial t}$)
\begin{align}
  \label{eq:711}
  \jump{\alpha(v)}&=\dot{\alpha}_0(1-y(v))v-\cp{\ppp{\alpha^\ast}{w}}\frac{v^2}{2}+\Landau_{\delta_2}(v^3)\notag\\
&=\dot{\alpha}_0(1-y(v))v+\mathcal{O}(v^2),
\end{align}
where the second equality holds provided we choose $\varepsilon$ sufficiently small. Using now \eqref{eq:215}, i.e.
\begin{align}
  \label{eq:712}
  \jump{\beta(v)}=\jump{\alpha(v)}^3G(\alpha_+(v),\alpha_-(v),\beta_-(v))
\end{align}
we get from \eqref{eq:711}
\begin{align}
  \label{eq:713}
  \jump{\beta(v)}=8G_0\dot{\alpha}_0^3v^3+\Landau_{\delta_2}(v^4),
\end{align}
where
\begin{align}
  \label{eq:714}
  G_0\coloneqq G(\alpha_{+}(0),\alpha_-(0),\beta_-(0))=G(\alpha_0,\alpha_0,\beta_0).
\end{align}
Therefore, in view of \eqref{eq:691}, we get
\begin{align}
  \label{eq:715}
  \beta_+(v)=\beta_0+\left(\pp{\beta^\ast}{t}\right)_0\frac{\lambda}{6\kappa^2}v^2+\mathcal{O}_{\delta_2}(v^3).
\end{align}

From $\beta_+(v)=\beta_0+v^2\hat{\beta}_+(v)$ we have
\begin{align}
  \label{eq:716}
  \frac{d\hat{\beta}_+}{dv}(v)=-\frac{2}{v^3}(\beta_+(v)-\beta_0)+\frac{1}{v^2}\frac{d\beta_+}{dv}(v).
\end{align}
Taking the derivative of \eqref{eq:712} we obtain
\begin{align}
  \label{eq:717}
  \frac{d}{dv}\jump{\beta(v)}=\left(3\jump{\alpha(v)}^2\frac{d}{dv}\jump{\alpha(v)}\right)G(\alpha_+(v),\alpha_-(v),\beta_-(v))+\jump{\alpha(v)}^3\frac{dG}{dv}(\alpha_+(v),\alpha_-(v),\beta_-(v)).
\end{align}
%Now (see \eqref{eq:682}),
%\begin{align}
%  \label{eq:718}
%  \frac{dz}{dv}(v)=\frac{d}{dv}(vy(v))=-1+\Landau_{\delta_2}(v).
%\end{align}
From \eqref{eq:682}, \eqref{eq:696}, \eqref{eq:709} we have
\begin{align}
  \label{eq:719}
  \frac{d}{dv}\jump{\alpha(v)}=2\dot{\alpha}_0+\Landau_{\delta_2}(v).
\end{align}
Using this together with \eqref{eq:711} in \eqref{eq:717} we obtain
\begin{align}
  \label{eq:720}
  \frac{d}{dv}\jump{\beta(v)}=24G_0\dot{\alpha}_0^3v^2+\Landau_{\delta_2}(v^3).
\end{align}
Using now \eqref{eq:690} we find
\begin{align}
  \label{eq:721}
  \frac{d\beta_+}{dv}(v)=\cp{\pp{\beta^\ast}{t}}\frac{\lambda}{3\kappa^2}v+\Landau_{\delta_2}(v^2).
\end{align}
Therefore, substituting \eqref{eq:715}, \eqref{eq:721} into \eqref{eq:716}, we conclude (putting back the index $m+1$)
\begin{align}
  \label{eq:722}
  \frac{d\hat{\beta}_{+,m+1}}{dv}(v)=\mathcal{O}_{\delta_2}(1).
\end{align}

We now find an expression for the asymptotic form of $V_{m+1}(v)$. We again omit the index $m+1$ for now. We have (see \eqref{eq:167})
\begin{align}
  \label{eq:723}
  V=\frac{\jump{T^{tr}}}{\jump{T^{tt}}}.
\end{align}
We rewrite the numerator as
\begin{align}
  \label{eq:724}
  \jump{T^{tr}}=T^{tr}(\alpha_+,\beta_+)-T^{tr}(\alpha_-,\beta_-)&=T^{tr}(\alpha_+,\beta_+)-T^{tr}(\alpha_+,\beta_-)\notag\\
&\qquad+T^{tr}(\alpha_+,\beta_-)-T^{tr}(\alpha_-,\beta_-).
\end{align}
We expand
\begin{align}
  \label{eq:725}
  T^{tr}(\alpha_+,\beta_+)-T^{tr}(\alpha_+,\beta_-)=\pp{T^{tr}}{\beta}(\alpha_+,\beta_-)\jump{\beta}+\mathcal{O}\left(\jump{\beta}^2\right)
\end{align}
and
\begin{align}
  \label{eq:726}
  T^{tr}(\alpha_+,\beta_-)-T^{tr}(\alpha_-,\beta_-)=\pp{T^{tr}}{\alpha}(\alpha_+,\beta_-)\jump{\alpha}-\frac{1}{2}\frac{\partial^2T^{tr}}{\partial \alpha^2}(\alpha_+,\beta_-)\jump{\alpha}^2+\mathcal{O}\left(\jump{\alpha}^3\right).
\end{align}
Similar expressions hold for the denominator of \eqref{eq:723}. Using $\jump{\beta}=\jump{\alpha}^3G(\alpha_+,\alpha_-,\beta_-)$  (see \eqref{eq:215}) it follows
\begin{align}
  \label{eq:727}
  V&=\frac{\displaystyle\pp{T^{tr}}{\alpha}(\alpha_+,\beta_-)-\frac{1}{2}\frac{\partial^2T^{tr}}{\partial\alpha^2}(\alpha_+,\beta_-)\jump{\alpha}+\mathcal{O}\left(\jump{\alpha}^2\right)}{\displaystyle\pp{T^{tt}}{\alpha}(\alpha_+,\beta_-)-\frac{1}{2}\frac{\partial^2T^{tt}}{\partial\alpha^2}(\alpha_+,\beta_-)\jump{\alpha}+\mathcal{O}\left(\jump{\alpha}^2\right)}\notag\\
&=\frac{\displaystyle\pp{T^{tr}}{\alpha}(\alpha_+,\beta_-)-\frac{1}{2}\frac{\partial^2T^{tr}}{\partial\alpha^2}(\alpha_+,\beta_-)\jump{\alpha}}{\displaystyle\pp{T^{tt}}{\alpha}(\alpha_+,\beta_-)-\frac{1}{2}\frac{\partial^2T^{tt}}{\partial\alpha^2}(\alpha_+,\beta_-)\jump{\alpha}}+\mathcal{O}\left(\jump{\alpha}^2\right),
\end{align}
where for the second equality we used $\partial T^{tt}/\partial\alpha\neq 0$, which follows from the first of \eqref{eq:182}. We now use the first of \eqref{eq:184}, i.e.
\begin{align}
  \label{eq:728}
  \pp{T^{tr}}{\alpha}=c_+\pp{T^{tt}}{\alpha},
\end{align}
which implies
\begin{align}
  \label{eq:729}
  V=c_+(\alpha_+,\beta_-)-\frac{1}{2}\pp{c_+}{\alpha}(\alpha_+,\beta_-)\jump{\alpha}+\mathcal{O}\left(\jump{\alpha}^2\right).
\end{align}

We look at $c_+(\alpha_+,\beta_-)$. We have
\begin{align}
  \label{eq:730}
  \frac{d}{dv}c_+(\alpha_+(v),\beta_-(v))=\pp{c_+}{\alpha}(\alpha_+(v),\beta_-(v))\frac{d\alpha_+}{dv}(v)+\pp{c_+}{\beta}(\alpha_+(v),\beta_-(v))\frac{d\beta_-}{dv}(v).
\end{align}
Using $\cp{d\beta_-/dv}=0$ we obtain
\begin{align}
  \label{eq:731}
  \pp{c_+}{\alpha}(\alpha_+(v),\beta_-(v))&=\cp{\pp{c_+}{\alpha}}+\left\{\cp{\ppp{c_+}{\alpha}}\cp{\frac{d\alpha_+}{dv}}+\cp{\pppp{c_+}{\alpha}{\beta}}\cp{\frac{d\beta_-}{dv}}\right\}v+\mathcal{O}(v^2)\notag\\
&=\frac{\kappa}{\dot{\alpha}_0}+\cp{\ppp{c_+}{\alpha}}\dot{\alpha}_0v+\mathcal{O}(v^2),
\end{align}
where we also used $\cp{\partial c_+/\partial \alpha}=\kappa/\dot{\alpha}_0$. 
We also have
\begin{align}
  \label{eq:736}
  \pp{c_+}{\beta}(\alpha_+(v),\beta_-(v))=\cp{\pp{c_+}{\beta}}+\mathcal{O}(v).
\end{align}

Using now \eqref{eq:731} and \eqref{eq:736} together with the asymptotic forms of $d\alpha_+/dv$, $d\beta_-/dv$ given by \eqref{eq:709}, \eqref{eq:721} respectively, we obtain
\begin{align}
  \label{eq:737}
  \frac{d}{dv}c_+(\alpha_+(v),\beta_-(v))&=\frac{\kappa}{\dot{\alpha}_0}\frac{d\alpha_i}{dv}(v)\notag\\
&\qquad+\left\{\frac{\lambda}{3\kappa^2}\left(\frac{\kappa\tilde{A}_0}{\dot{\alpha}_0}+\cp{\pp{c_+}{\beta}}\cp{\pp{\beta^\ast}{t}}\right)+\dot{\alpha}_0\cp{\ppp{c_+}{\alpha}}\frac{d\alpha_i}{dv}(v)\right\}v+\Landau_{\delta_2}(v^2).
\end{align}
Therefore,
\begin{align}
  \label{eq:738}
  c_+(\alpha_+(v),\beta_-(v))&=c_{+0}+\frac{\kappa}{\dot{\alpha}_0}(\alpha_i(v)-\alpha_0)+\dot{\alpha}_0\cp{\ppp{c_+}{\alpha}}(\alpha_i(v)-\overline{\alpha_i}(v))v\notag\\
&\qquad\qquad\qquad+\frac{\lambda}{6\kappa^2}\left\{\frac{\kappa\tilde{A}_0}{\dot{\alpha}_0}+\cp{\pp{c_+}{\beta}}\cp{\pp{\beta^\ast}{t}}\right\}v^2+\Landau_{\delta_2}(v^3),
\end{align}
where we introduced
\begin{align}
  \label{eq:739}
  \overline{\alpha_i}(v)\coloneqq \frac{1}{v}\int_0^v\alpha_i(v')dv'.
\end{align}
Therefore, provided we choose $\varepsilon$ small enough,
\begin{align}
  \label{eq:740}
  c_+(\alpha_+(v),\beta_-(v))=c_{+0}+\kappa v+\Landau(v^2).
\end{align}

Using \eqref{eq:711}, \eqref{eq:731}, \eqref{eq:740} in \eqref{eq:729} we arrive at
\begin{align}
  \label{eq:741}
  V(v)=c_{+0}+\frac{\kappa}{2}(1+y(v))v+\mathcal{O}(v^2).
\end{align}
Using now $\hat{V}(v)=(1/v^2)(V(v)-c_{+0}-(\kappa/2)(1+y(v))v)$ we deduce, putting back the index $m+1$,
\begin{align}
  \label{eq:742}
  \hat{V}_{m+1}(v)=\mathcal{O}(1),
\end{align}
i.e.
\begin{align}
  \label{eq:743}
  \sup_{v\in[0,\varepsilon]}|\hat{V}_{m+1}(v)|\leq C
\end{align}
for a fixed numerical constant $C$. Choosing now the size of the ball for the iteration of the function $\hat{V}_m$, i.e.~$\delta_2$, to be equal to the numerical constant $C$ appearing in \eqref{eq:743} we see that $\hat{V}_{m+1}\in B_{\delta_2}$. From \eqref{eq:722} we then have
\begin{align}
  \label{eq:744}
  \sup_{v\in[0,\varepsilon]}\bigg|\frac{d\hat{\beta}_{+,m+1}}{dv}(v)\bigg|\leq C
\end{align}
for a fixed numerical constant $C$. Choosing now the size of the ball for the iteration of the function $\hat{\beta}_{+,m}$, i.e.~$\delta_1$, to be equal to the numerical constant appearing in \eqref{eq:744} we see that $\hat{\beta}_{+,m+1}\in B_{\delta_1}$. From \eqref{eq:678} we get
\begin{align}
  \label{eq:745}
  Y\leq C,
\end{align}
for a fixed numerical constant $C$. In view of \eqref{eq:677} this shows that $y_{m+1}\in B_Y$. This concludes the proof of the proposition.

\end{proof}

\subsection{Convergence}

We define
\begin{align}
  \label{eq:746}
  \nd{f}\coloneqq \sup_{[0,v]}\left|\frac{df}{dv}\right|.
\end{align}
For differences between successive members of the iteration we will use the notation $\Delta_m f\coloneqq f_m-f_{m-1}$. For any sequence of functions $h_m\in C^1$ with identical values $h_m(0)=:h_0$ for any $m$ we have
\begin{align}
  \label{eq:747}
  \Delta_mh(v)=\int_0^v\frac{d}{dv}\Delta_mh(v')dv',
\end{align}
which implies
\begin{align}
  \label{eq:748}
  |\Delta_mh|\leq v\nd{\Delta_mh}.
\end{align}

\subsubsection{Estimates for the Solution of the Fixed Boundary Problem}

\begin{lemma}\label{lemma_est_fxdbp}
  Let $(t_{m+1},r_{m+1},\alpha_{m+1},\beta_{m+1})$ be the solution of the fixed boundary problem as given by proposition \ref{prop_fxdbp}, corresponding to the boundary functions $(z_m,\beta_{+,m},V_m)$. Then the following estimates hold
  \begin{align}
  \label{eq:749}
  \bigg|\frac{d\Delta_{m+1}\hat{f}}{dv}(v)\bigg|&\leq \frac{\lambda}{24\kappa^2}\nd{\Delta_my}+C\left\{\sup_{[0,v]}|\Delta_m\hat{V}|+v\nd{\Delta_m\hat{\beta}_+}+v\nd{\Delta_my}\right\},\\
    \label{eq:750}
  \left|\frac{d\Delta_{m+1}\hat{\delta}}{dv}(v)\right|&\leq\frac{\lambda}{24\kappa}\nd{\Delta_my}+C\left\{\sup_{[0,v]}|\Delta_m\hat{V}|+v\nd{\Delta_m\hat{\beta}_+}+v\nd{\Delta_my}\right\},\\
    \label{eq:751}
  \nd{\Delta_{m+1}\alpha_+}&\leq Cv^2\left\{\sup_{[0,v]}|\Delta_{m}\hat{V}|+v\nd{\Delta_{m}\hat{\beta}_+}+\nd{\Delta_{m}y}\right\}.
  \end{align}
\end{lemma}

\begin{proof}
The difference $\Delta_{m+1}t$ satisfies
\begin{align}
  \label{eq:752}
  \pppp{\Delta_{m+1}t}{u}{v}+\mu_{m+1}\pp{\Delta_{m+1}t}{v}-\nu_{m+1}\pp{\Delta_{m+1}t}{u}=\Xi_{m+1},
\end{align}
where
\begin{align}
  \label{eq:753}
  \Xi_{m+1}=\pp{t_m}{u}\Delta_{m+1}\nu-\pp{t_m}{v}\Delta_{m+1}\mu.
\end{align}
In addition we have the boundary condition
\begin{align}
  \label{eq:754}
  \pp{\Delta_{m+1}t}{v}=\frac{1}{\gamma_{m+1}}\pp{\Delta_{m+1}t}{u}+\pp{t_m}{u}\Delta_{m+1}\left(\frac{1}{\gamma}\right)\quad\textrm{for}\quad u=v,
\end{align}
where
\begin{align}
  \label{eq:755}
  \gamma_m=\frac{\bar{c}_{+,m}(v)-V_{m-1}(v)}{V_{m-1}(v)-\bar{c}_{-,m(v)}},\qquad \bar{c}_{\pm,m}(v)=c_\pm(\alpha_{+,m}(v),\beta_{+,m-1}(v)),
\end{align}
and the initial condition
\begin{align}
  \label{eq:756}
  \Delta_{m+1}t(u,0)=0.
\end{align}

Before we study equation \eqref{eq:752} we estimate the difference $\Delta_{m+1}(1/\gamma)$ and the term $\Xi_{m+1}$. We have
\begin{align}
  \label{eq:757}
  \Delta_{m+1}\left(\frac{1}{\gamma}\right)=\frac{\omega_{m+1}}{(\bar{c}_{+,m+1}-V_{m})(\bar{c}_{+,m}-V_{m-1})},
\end{align}
where
\begin{align}
  \label{eq:758}
  \omega_{m+1}=(V_{m}-\bar{c}_{-,m+1})(\bar{c}_{+,m}-V_{m-1})-(V_{m-1}-\bar{c}_{-,m})(\bar{c}_{+,m+1}-V_{m}).
\end{align}
\begin{remark}
The mixture of indices in \eqref{eq:755}, \eqref{eq:758} arises because the index increases with the solution of the fixed boundary problem as described in the basic strategy. $\alpha_+$ and $c_\pm$ carry the index of the solution of the fixed boundary problem while $\beta_+$ and $V_m$ carry the index of the data which goes into the fixed boundary problem.
\end{remark}
We consider first the denominator of \eqref{eq:757}. From
\begin{align}
  \label{eq:759}
  \frac{d\bar{c}_{+,m}}{dv}=\pp{c_+}{\alpha}(\alpha_{+,m},\beta_{+,m-1})\frac{d\alpha_{+,m}}{dv}+\pp{c_+}{\beta}(\alpha_{+,m},\beta_{+,m-1})\frac{d\beta_{+,m-1}}{dv}
\end{align}
together with $\cp{\partial c_+/\partial\alpha}=\kappa/\dot{\alpha}_0$ and the asymptotic forms of $d\alpha_{+,m}/dv$, $d\beta_{+,m}/dv$ given by \eqref{eq:709}, \eqref{eq:721} respectively, we find
\begin{align}
  \label{eq:760}
  \bar{c}_{+,m}(v)=c_{+0}+\kappa v+\Landau(v^2).
\end{align}
Using this together with
\begin{align}
  \label{eq:761}
  V_m(v)=c_{+0}+\Landau(v^2),
\end{align}
(recall that $V_m(v)=c_{+0}+\frac{\kappa}{2}(1+y_m(v))v+v^2\hat{V}(v)$, $y_m(v)=-1+\Landau(v)$) we get
\begin{align}
  \label{eq:762}
  (\bar{c}_{+,m+1}-V_{m})(\bar{c}_{+,m}-V_{m-1})=\kappa^2v^2+\mathcal{O}(v^3).
\end{align}

We turn to the numerator of \eqref{eq:757}. We rewrite it as
\begin{align}
  \label{eq:763}
  \omega_{m+1}=(\Delta_mV-\Delta_{m+1}\bar{c}_-)(\bar{c}_{+,m}-V_{m-1})-(V_{m-1}-\bar{c}_{-,m})(\Delta_{m+1}\bar{c}_+-\Delta_mV).
\end{align}
We have
\begin{align}
  \label{eq:764}
  \Delta_mV(\bar{c}_{+,m}-V_{m-1})=\mathcal{O}(v^2|\Delta_my|)+\mathcal{O}(v^3|\Delta_m\hat{V}|).
\end{align}

Now we study the difference $\Delta_{m+1}\bar{c}_\pm$. Since $\bar{c}_{\pm,m}(v)=c_\pm(\alpha_{+,m}(v),\beta_{+,m-1}(v))$ we need to estimate the difference $\Delta_{m+1}\alpha_+$. From (see the first of \eqref{eq:284})
\begin{align}
  \label{eq:765}
  \Delta_{m+1}\alpha(u,v)=\int_0^v\Delta_{m+1}A(u,v')dv',
\end{align}
we obtain
\begin{align}
  \label{eq:766}
  |\Delta_{m+1}\alpha(u,v)|&\leq Cv\sup_{T_u}|\Delta_{m+1}A(u,v)|\notag\\
&\leq Cv\left\{\sup_{T_u}\left|\pp{\Delta_{m+1}t}{v}\right|+v\left(\sup_{T_u}|\Delta_{m+1}\alpha|+\sup_{T_u}|\Delta_{m+1}\beta|+\sup_{T_u}|\Delta_{m+1}r|\right)\right\}.
\end{align}
From this we deduce that for $\varepsilon$ small enough
\begin{align}
  \label{eq:767}
  \sup_{T_u}|\Delta_{m+1}\alpha|\leq Cu\left\{\sup_{T_u}\left|\pp{\Delta_{m+1}t}{v}\right|+u\left(\sup_{T_u}|\Delta_{m+1}\beta|+\sup_{T_u}|\Delta_{m+1}r|\right)\right\}.
\end{align}

For the difference $|\Delta_{m+1}r|$ we integrate \eqref{eq:285} (cf.~also \eqref{eq:316}) to get
\begin{align}
  \label{eq:768}
  \Delta_{m+1}r(u,v)&=\int_0^u\left(c_{-,m+1}\pp{\Delta_{m+1}t}{u}+\pp{t_m}{u}\Delta_{m+1}c_-\right)(u',0)du'\notag\\
&\qquad+\int_0^v\left(c_{+,m+1}\pp{\Delta_{m+1}t}{v}+\pp{t_m}{v}\Delta_{m+1}c_+\right)(u,v')dv',
\end{align}
which implies
\begin{align}
  \label{eq:769}
  |\Delta_{m+1}r(u,v)|\leq Cu\left\{\sup_{T_u}\left|\pp{\Delta_{m+1}t}{u}\right|+\sup_{T_u}\left|\pp{\Delta_{m+1}t}{v}\right|+u\left(\sup_{T_u}|\Delta_{m+1}\alpha|+\sup_{T_u}|\Delta_{m+1}\beta|\right)\right\}.
\end{align}
Using this in \eqref{eq:767} and choosing $\varepsilon$ sufficiently small we find
\begin{align}
  \label{eq:770}
  \sup_{T_u}|\Delta_{m+1}\alpha|\leq Cu\left\{u^2\sup_{T_u}\left|\pp{\Delta_{m+1}t}{u}\right|+\sup_{T_u}\left|\pp{\Delta_{m+1}t}{v}\right|+u\sup_{T_u}|\Delta_{m+1}\beta|\right\}.
\end{align}

To estimate the difference $\Delta_{m+1}\beta$ we use (see the second of \eqref{eq:284})
\begin{align}
  \label{eq:771}
  \Delta_{m+1}\beta(u,v)=\Delta_{m}\beta_+(v)+\int_v^u\Delta_{m+1}B(u',v)du',
\end{align}
which implies
\begin{align}
  \label{eq:772}
  |\Delta_{m+1}\beta(u,v)|&\leq |\Delta_m\beta_+(v)|+Cu\bigg\{\sup_{T_u}\left|\pp{\Delta_{m+1}t}{u}\right|\notag\\
&\hspace{35mm}+u^2\left(\sup_{T_u}|\Delta_{m+1}\alpha|+\sup_{T_u}|\Delta_{m+1}\beta|+\sup_{T_u}|\Delta_{m+1}r|\right)\bigg\}.
\end{align}
Using \eqref{eq:769} and choosing $\varepsilon$ small enough, this in turn implies
\begin{align}
  \label{eq:773}
  \sup_{T_u}|\Delta_{m+1}\beta|\leq\sup_{[0,u]}|\Delta_m\beta_+|+Cu\left\{\sup_{T_u}\left|\pp{\Delta_{m+1}t}{u}\right|+u^2\left(u\sup_{T_u}\left|\pp{\Delta_{m+1}t}{v}\right|+\sup_{T_u}|\Delta_{m+1}\alpha|\right)\right\}.
\end{align}
Using this in \eqref{eq:770} we obtain
\begin{align}
  \label{eq:774}
  \sup_{T_u}|\Delta_{m+1}\alpha|\leq Cu\left\{u^2\sup_{T_u}\left|\pp{\Delta_{m+1}t}{u}\right|+\sup_{T_u}\left|\pp{\Delta_{m+1}t}{v}\right|+u\sup_{[0,u]}|\Delta_m\beta_+|\right\},
\end{align}
provided that we choose $\varepsilon$ suitably small. For future reference we use this in \eqref{eq:773} which implies
\begin{align}
  \label{eq:775}
  \sup_{T_u}|\Delta_{m+1}\beta|\leq C\sup_{[0,u]}|\Delta_m\beta_+|+Cu\left\{\sup_{T_u}\left|\pp{\Delta_{m+1}t}{u}\right|+u^3\sup_{T_u}\left|\pp{\Delta_{m+1}t}{v}\right|\right\}.
\end{align}
From \eqref{eq:774} we have
\begin{align}
  \label{eq:776}
  \sup_{[0,v]}|\Delta_{m+1}\alpha_+|\leq\sup_{T_v}|\Delta_{m+1}\alpha|\leq Cv\left\{v^2\sup_{T_v}\left|\pp{\Delta_{m+1}t}{u}\right|+\sup_{T_v}\left|\pp{\Delta_{m+1}t}{v}\right|+v\sup_{[0,v]}|\Delta_{m}\beta_+|\right\}.
\end{align}

It now follows for the difference $\Delta_{m+1}\bar{c}_\pm$
\begin{align}
  \label{eq:777}
  |\Delta_{m+1}\bar{c}_\pm|\leq C\left\{v^3\sup_{T_v}\left|\pp{\Delta_{m+1}t}{u}\right|+v\sup_{T_v}\left|\pp{\Delta_{m+1}t}{v}\right|+\sup_{[0,v]}|\Delta_{m}\beta_+|\right\}.
\end{align}
Using this we get
\begin{align}
  \label{eq:778}
  |\Delta_{m+1}\bar{c}_-(\bar{c}_{+,m}-V_{m-1})|&\leq Cv \left\{v^3\sup_{T_v}\left|\pp{\Delta_{m+1}t}{u}\right|+v\sup_{T_v}\left|\pp{\Delta_{m+1}t}{v}\right|+\sup_{[0,v]}|\Delta_{m}\beta_+|\right\},\\
  \label{eq:779}
  |(V_{m-1}-\bar{c}_{-,m})\Delta_{m+1}\bar{c}_+|&\leq C\left\{v^3\sup_{T_v}\left|\pp{\Delta_{m+1}t}{u}\right|+v\sup_{T_v}\left|\pp{\Delta_{m+1}t}{v}\right|+\sup_{[0,v]}|\Delta_{m}\beta_+|\right\},
\end{align}
where for \eqref{eq:778} we used $\bar{c}_{+,m}-V_{m-1}=\mathcal{O}(v)$ and for \eqref{eq:779} we used $V_{m-1}-\bar{c}_{-,m}=c_{+0}-c_{-0}+\mathcal{O}(v)$.

The remaining part of the numerator of \eqref{eq:757} is
\begin{align}
  \label{eq:780}
  (V_{m-1}-\bar{c}_{-m})(\Delta_{m}V)=\frac{(c_{+0}-c_{-0})\kappa v}{2}\Delta_my+\mathcal{O}(v^2|\Delta_my|)+\mathcal{O}(v^2|\Delta_m\hat{V}|).
\end{align}
Using now \eqref{eq:764}, \eqref{eq:778}, \eqref{eq:779}, \eqref{eq:780} together with \eqref{eq:762} in \eqref{eq:757} we arrive at
\begin{align}
  \label{eq:781}
  \Delta_{m+1}\left(\frac{1}{\gamma}\right)&=\frac{c_{+0}-c_{-0}}{2\kappa v}\Delta_my+\mathcal{O}(|\Delta_my|)+\mathcal{O}(|\Delta_m\hat{V}|)+\mathcal{O}\bigg(\frac{1}{v^2}\sup_{[0,v]}|\Delta_m\beta_+|\bigg)\notag\\
&\qquad+\mathcal{O}\left(\frac{1}{v}\sup_{T_v}\left|\pp{\Delta_{m+1}t}{v}\right|\right)+\mathcal{O}\left(v\sup_{T_v}\left|\pp{\Delta_{m+1}t}{u}\right|\right).
\end{align}

We turn to $\Xi_{m+1}$. For an estimate of this we need estimates for the differences $\Delta_{m+1}\mu$, $\Delta_{m+1}\nu$. From the first of \eqref{eq:296} we have
\begin{align}
  \label{eq:782}
  \Delta_{m+1}\mu=-\frac{1}{(c_{+,m+1}-c_{-,m+1})(c_{+,m}-c_{-,m})}\pp{c_{+,m+1}}{u}\Delta_{m+1}(c_+-c_-)+\frac{1}{c_{+,m}-c_{-,m}}\pp{\Delta_{m+1}c_+}{u},
\end{align}
which implies
\begin{align}
  \label{eq:783}
  |\Delta_{m+1}\mu|\leq C\left\{|\Delta_{m+1}c_+|+|\Delta_{m+1}c_-|+\left|\pp{\Delta_{m+1}c_+}{u}\right|\right\}.
\end{align}
For the last term we use
\begin{align}
  \label{eq:784}
  \pp{\Delta_{m+1}c_+}{u}&=\left(\pp{c_+}{\alpha}\right)_{m+1}\pp{\Delta_{m+1}\alpha}{u}+\pp{\alpha_m}{u}\Delta_{m+1}\left(\pp{c_+}{\alpha}\right)\notag\\
&\qquad+\left(\pp{c_+}{\beta}\right)_{m+1}\pp{\Delta_{m+1}\beta}{u}+\pp{\beta_m}{u}\Delta_{m+1}\left(\pp{c_+}{\beta}\right).
\end{align}

For the partial derivatives of $c_+$ with respect to $\alpha$, $\beta$ we have
\begin{align}
  \label{eq:785}
  \left|\Delta_{m+1}\left(\pp{c_+}{\alpha}\right)\right|,\left|\Delta_{m+1}\left(\pp{c_+}{\beta}\right)\right|\leq C\left\{|\Delta_{m+1}\alpha|+|\Delta_{m+1}\beta|\right\}.
\end{align}
For the partial derivative of $\beta$ with respect to $u$ we use the second of \eqref{eq:284}. I.e.~we have
\begin{align}
  \label{eq:786}
  \pp{\Delta_{m+1}\beta}{u}=\Delta_{m+1}B,
\end{align}
which implies
\begin{align}
  \label{eq:787}
  \left|\pp{\Delta_{m+1}\beta}{u}\right|\leq C \left\{\left|\pp{\Delta_{m+1}t}{u}\right|+u^2\Big(|\Delta_{m+1}\alpha|+|\Delta_{m+1}\beta|+|\Delta_{m+1}r|\Big)\right\}.
\end{align}

Using \eqref{eq:785}, \eqref{eq:787} in \eqref{eq:784} and the resulting estimate in \eqref{eq:783} we obtain
\begin{align}
  \label{eq:788}
  |\Delta_{m+1}\mu|\leq C\left\{|\Delta_{m+1}\alpha|+|\Delta_{m+1}\beta|+u^2|\Delta_{m+1}r|+\left|\pp{\Delta_{m+1}t}{u}\right|+\left|\pp{\Delta_{m+1}\alpha}{u}\right|\right\}.
\end{align}
For the term involving the partial derivative of $\alpha$ with respect to $u$ we use
\begin{align}
  \label{eq:789}
  \alpha_m(u,v)=\alpha_i(u)+\int_0^vA_m(u,v')dv',
\end{align}
which implies
\begin{align}
  \label{eq:790}
  \pp{\Delta_{m+1}\alpha}{u}(u,v)&=\int_0^v\bigg\{\left(\pp{A}{\alpha}\right)_{m+1}\pp{\Delta_{m+1}\alpha}{u}+\pp{\alpha_{m}}{u}\Delta_{m+1}\left(\pp{A}{\alpha}\right)\notag\\
&\hspace{10mm}+\left(\pp{A}{\beta}\right)_{m+1}\pp{\Delta_{m+1}\beta}{u}+\pp{\beta_{m}}{u}\Delta_{m+1}\left(\pp{A}{\beta}\right)+\left(\pp{A}{r}\right)_{m+1}c_{-,m+1}\pp{\Delta_{m+1}t}{u}\notag\\
&\hspace{10mm}+\pp{t_{m}}{u}\left(\left(\pp{A}{r}\right)_{m+1}\Delta_{m+1}c_-+c_{-,m}\Delta_{m+1}\left(\pp{A}{r}\right)\right)\notag\\
&\hspace{50mm}+\tilde{A}_{m+1}\frac{\partial^2\Delta_{m+1}t}{\partial u\partial v}+\frac{\partial^2t_{m}}{\partial u\partial v}\Delta_{m+1}\tilde{A}\bigg\}(u,v')dv'.
\end{align}
We rewrite this as
\begin{align}
  \label{eq:791}
  \pp{\Delta_{m+1}\alpha}{u}(u,v)=\int_0^v\sum_{i=1}^9I_i(u,v')dv',
\end{align}
where the $I_i$ denote the nine terms appearing in \eqref{eq:790} (the $(\ldots)$ bracket counting as two terms). The first term we will absorb on the left hand side. For $I_2$ we use
\begin{align}
  \label{eq:792}
  \left|\Delta_{m+1}\left(\pp{A}{\alpha}\right)\right|&\leq \bigg|\pp{t_{m+1}}{v}\Delta_{m+1}\bigg(\pp{\tilde{A}}{\alpha}\bigg)\bigg|+\bigg|\bigg(\pp{\tilde{A}}{\alpha}\bigg)_m\pp{\Delta_{m+1}t}{v}\bigg|\notag\\
%&\leq C\bigg\{v\bigg|\Delta_{m+1}\bigg(\pp{\tilde{A}}{\alpha}\bigg)\bigg|+\bigg|\pp{\Delta_{m+1}t}{v}\bigg|\bigg\}\notag\\
&\leq C\left\{v\Big(|\Delta_{m+1}\alpha|+|\Delta_{m+1}\beta|+|\Delta_{m+1}r|\Big)+\bigg|\pp{\Delta_{m+1}t}{v}\bigg|\right\},
\end{align}
which implies
\begin{align}
  \label{eq:793}
  \left|\int_0^vI_2(u,v')dv'\right|\leq Cv\left\{v\Big(\sup_{T_u}|\Delta_{m+1}\alpha|+\sup_{T_u}|\Delta_{m+1}\beta|+\sup_{T_u}|\Delta_{m+1}r|\Big)+\sup_{T_u}\bigg|\pp{\Delta_{m+1}t}{v}\bigg|\right\}.
\end{align}
For $I_3$ we use
\begin{align}
  \label{eq:794}
  \bigg|\left(\pp{A}{\beta}\right)_{m+1}\bigg|=\bigg|\pp{t_{m+1}}{v}\bigg(\pp{\tilde{A}}{\beta}\bigg)_{m+1}\bigg|\leq Cv.
\end{align}
Using now \eqref{eq:787} this implies
\begin{align}
  \label{eq:795}
  \left|\int_0^vI_3(u,v')dv'\right|\leq Cv^2\left\{\sup_{T_u}\left|\pp{\Delta_{m+1}t}{u}\right|+u^2\Big(\sup_{T_u}|\Delta_{m+1}\alpha|+\sup_{T_u}|\Delta_{m+1}\beta|+\sup_{T_u}|\Delta_{m+1}r|\Big)\right\}.
\end{align}
$I_4$ can be treated in the same way as $I_2$. For $I_5$ we use \eqref{eq:794} with $r$ in the role of $\beta$, which implies
\begin{align}
  \label{eq:796}
  \left|\int_0^vI_5(u,v')dv'\right|\leq Cv^2\sup_{T_u}\left|\pp{\Delta_{m+1}t}{u}\right|.
\end{align}
For $I_6$ we use
\begin{align}
  \label{eq:797}
  \left|\pp{t_m}{u}\left(\pp{A}{r}\right)_{m+1}\right|=\left|\pp{t_m}{u}\pp{t_{m+1}}{v}\bigg(\pp{\tilde{A}}{r}\bigg)_{m+1}\right|\leq Cu^2v,
\end{align}
which implies
\begin{align}
  \label{eq:798}
  \left|\int_0^vI_6(u,v')dv'\right|\leq Cu^2v^2\left\{\sup_{T_u}|\Delta_{m+1}\alpha|+\sup_{T_u}|\Delta_{m+1}\beta|\right\}.
\end{align}
For $I_7$ we use
\begin{align}
  \label{eq:799}
  \left|\Delta_{m+1}\left(\pp{A}{r}\right)\right|\leq C\left\{\left|\pp{\Delta_{m+1}t}{v}\right|+v\Big(|\Delta_{m+1}\alpha|+|\Delta_{m+1}\beta|+|\Delta_{m+1}r|\Big)\right\},
\end{align}
which implies
\begin{align}
  \label{eq:800}
  \left|\int_0^vI_7(u,v')dv'\right|\leq Cu^2v\left\{\sup_{T_u}\left|\pp{\Delta_{m+1}t}{v}\right|+v\Big(\sup_{T_u}|\Delta_{m+1}\alpha|+\sup_{T_u}|\Delta_{m+1}\beta|+\sup_{T_u}|\Delta_{m+1}r|\Big)\right\}.
\end{align}
For $I_8$ we use
\begin{align}
  \label{eq:801}
  \pppp{\Delta_{m+1}t}{u}{v}&=\Xi_{m+1}-\mu_{m+1}\pp{\Delta_{m+1}t}{v}+\nu_{m+1}\pp{\Delta_{m+1}t}{u}\notag\\
&=\pp{t_m}{u}\Delta_{m+1}\nu-\pp{t_m}{v}\Delta_{m+1}\mu-\mu_{m+1}\pp{\Delta_{m+1}t}{v}+\nu_{m+1}\pp{\Delta_{m+1}t}{u}.
\end{align}
From this it follows that
\begin{align}
  \label{eq:802}
  \left|\pppp{\Delta_{m+1}t}{u}{v}\right|\leq C\left\{v|\Delta_{m+1}\mu|+u^2|\Delta_{m+1}\nu|+\left|\pp{\Delta_{m+1}t}{v}\right|+v\left|\pp{\Delta_{m+1}t}{u}\right|\right\},
\end{align}
which implies
\begin{align}
  \label{eq:803}
  \left|\int_0^vI_8(u,v')dv'\right|\leq Cv\left\{v\sup_{T_u}|\Delta_{m+1}\mu|+u^2\sup_{T_u}|\Delta_{m+1}\nu|+\sup_{T_u}\left|\pp{\Delta_{m+1}t}{v}\right|+v\sup_{T_u}\left|\pp{\Delta_{m+1}t}{u}\right|\right\}.
\end{align}
For $I_9$ we use
\begin{align}
  \label{eq:804}
  \left|\pppp{t_m}{u}{v}\right|\leq \left|\mu_m\pp{t_m}{v}\right|+\left|\nu_m\pp{t_m}{u}\right|\leq Cu,
\end{align}
which implies
\begin{align}
  \label{eq:805}
  \left|\int_0^vI_9(u,v')dv'\right|\leq Cuv\left\{\sup_{T_u}|\Delta_{m+1}\alpha|+\sup_{T_u}|\Delta_{m+1}\beta|+\sup_{T_u}|\Delta_{m+1}r|\right\}.
\end{align}
Using the estimates for the integrals of $I_2\ldots I_9$ in \eqref{eq:791}, taking the supremum of the resulting estimate in $T_u$ and absorbing the term involving $I_1$ on the left hand side it follows that for small enough $\varepsilon$ we have
\begin{align}
  \label{eq:806}
  \sup_{T_u}\left|\pp{\Delta_{m+1}\alpha}{u}\right|&\leq C\Bigg\{u^2\left(\sup_{T_u}|\Delta_{m+1}\alpha|+\sup_{T_u}|\Delta_{m+1}\beta|+\sup_{T_u}|\Delta_{m+1}r|\right)\notag\\
&\qquad +u\sup_{T_u}\left|\pp{\Delta_{m+1}t}{v}\right|+u^2\sup_{T_u}\left|\pp{\Delta_{m+1}t}{u}\right|+u^2\sup_{T_u}|\Delta_{m+1}\mu|+u^3\sup_{T_u}|\Delta_{m+1}\nu|\Bigg\}.
\end{align}
Using this estimate in \eqref{eq:788} it follows that for $\varepsilon$ small enough we have
\begin{align}
  \label{eq:807}
  \sup_{T_u}|\Delta_{m+1}\mu|&\leq C\bigg\{\sup_{T_u}|\Delta_{m+1}\alpha|+\sup_{T_u}|\Delta_{m+1}\beta|+u^2\sup_{T_u}|\Delta_{m+1}r|\notag\\
&\quad\qquad +u\sup_{T_u}\left|\pp{\Delta_{m+1}t}{v}\right|+\sup_{T_u}\left|\pp{\Delta_{m+1}t}{u}\right|+u^3\sup_{T_u}|\Delta_{m+1}\nu|\bigg\}.
\end{align}

To get an estimate for $\Delta_{m+1}\nu$ we apply a similar procedure. The estimate \eqref{eq:783} is replaced by
\begin{align}
  \label{eq:808}
  |\Delta_{m+1}\nu|\leq C\left\{|\Delta_{m+1}c_+|+|\Delta_{m+1}c_-|+\left|\pp{\Delta_{m+1}c_-}{v}\right|\right\}.
\end{align}
Using now
\begin{align}
  \label{eq:809}
  \pp{\Delta_{m+1}\alpha}{v}=\Delta_{m+1}A,
\end{align}
it follows that in the role of the estimate \eqref{eq:788} we have
\begin{align}
  \label{eq:810}
  |\Delta_{m+1}\nu|\leq C\left\{|\Delta_{m+1}\alpha|+|\Delta_{m+1}\beta|+u|\Delta_{m+1}r|+\left|\pp{\Delta_{m+1}t}{v}\right|+\left|\pp{\Delta_{m+1}\beta}{v}\right|\right\}.
\end{align}

Now for the estimate of $\partial\Delta_{m+1}\beta/\partial v$ we use
\begin{align}
  \label{eq:811}
  \pp{\beta_{m+1}}{v}=\frac{d\beta_{+,m}}{dv}-B_{m+1}+\int_v^u\pp{B_{m+1}}{v}(u',v)du'.
\end{align}
This follows from the second of \eqref{eq:284} integrated with respect to $u$ from $u=v$ up to $u$. Using this we get for the difference $\partial\Delta_{m+1}\beta/\partial v$ terms which are analogous to the ones in \eqref{eq:790} ($A$ replaced by $B$, $\partial/\partial u$ replaced by $\partial /\partial v$) but in addition to the integral we have the terms
\begin{align}
  \label{eq:812}
  \frac{d\Delta_{m}\beta_+}{dv}-\Delta_{m+1}B.
\end{align}
%\begin{remark}
%  Again there is no mistake with the indices in \eqref{eq:811}, \eqref{eq:812}. $\beta_+$ comes from the boundary condition posessing index $m$ while $\beta$ is part %of the solution of the fixed boundary problem posessing index $m+1$.
%\end{remark}
For the second difference in \eqref{eq:812} we use \eqref{eq:786}, \eqref{eq:787}. It follows
\begin{align}
  \label{eq:813}
  \sup_{T_u}\left|\pp{\Delta_{m+1}\beta}{v}\right|&\leq \sup_{[0,u]}\left|\frac{d\Delta_m\beta_+}{dv}\right|+C\Bigg\{u^2\left(\sup_{T_u}|\Delta_{m+1}\alpha|+\sup_{T_u}|\Delta_{m+1}\beta|+\sup_{T_u}|\Delta_{m+1}r|\right)\notag\\
&\quad\qquad\qquad +u\sup_{T_u}\left|\pp{\Delta_{m+1}t}{v}\right|+\sup_{T_u}\left|\pp{\Delta_{m+1}t}{u}\right|+u^2\sup_{T_u}|\Delta_{m+1}\mu|+u^3\sup_{T_u}|\Delta_{m+1}\nu|\Bigg\}.
\end{align}
Therefore, for $\varepsilon$ small enough, we have
\begin{align}
  \label{eq:814}
  \sup_{T_u}|\Delta_{m+1}\nu|&\leq C\Bigg\{\sup_{T_u}|\Delta_{m+1}\alpha|+\sup_{T_u}|\Delta_{m+1}\beta|+u\sup_{T_u}|\Delta_{m+1}r|+\sup_{[0,u]}\left|\frac{d\Delta_{m}\beta_+}{dv}\right|\notag\\
&\qquad\hspace{25mm} +\sup_{T_u}\left|\pp{\Delta_{m+1}t}{v}\right|+\sup_{T_u}\left|\pp{\Delta_{m+1}t}{u}\right|+u^2\sup_{T_u}|\Delta_{m+1}\mu|\Bigg\}.
\end{align}

From \eqref{eq:807}, \eqref{eq:814} together with \eqref{eq:769} we obtain
\begin{align}
  \label{eq:815}
  \sup_{T_u}|\Delta_{m+1}\mu|&\leq C\Bigg\{\sup_{T_u}|\Delta_{m+1}\alpha|+\sup_{T_u}|\Delta_{m+1}\beta|+u^3\sup_{[0,u]}\left|\frac{d\Delta_{m}\beta_+}{dv}\right|\notag\\
&\hspace{40mm} +u\sup_{T_u}\left|\pp{\Delta_{m+1}t}{v}\right|+\sup_{T_u}\left|\pp{\Delta_{m+1}t}{u}\right|\Bigg\},\\
  \label{eq:816}
  \sup_{T_u}|\Delta_{m+1}\nu|&\leq C\Bigg\{\sup_{T_u}|\Delta_{m+1}\alpha|+\sup_{T_u}|\Delta_{m+1}\beta|+\sup_{[0,u]}\left|\frac{d\Delta_{m}\beta_+}{dv}\right|\notag\\
&\hspace{40mm}+\sup_{T_u}\left|\pp{\Delta_{m+1}t}{v}\right|+\sup_{T_u}\left|\pp{\Delta_{m+1}t}{u}\right|\Bigg\}.
\end{align}
Substituting now \eqref{eq:774}, \eqref{eq:775} we get
\begin{align}
  \label{eq:817}
  \sup_{T_u}|\Delta_{m+1}\mu|&\leq C\left\{u\sup_{[0,u]}\left|\frac{d\Delta_{m}\beta_+}{dv}\right|+u\sup_{T_u}\left|\pp{\Delta_{m+1}t}{v}\right|+\sup_{T_u}\left|\pp{\Delta_{m+1}t}{u}\right|\right\},\\
  \label{eq:818}
  \sup_{T_u}|\Delta_{m+1}\nu|&\leq C\left\{\sup_{[0,u]}\left|\frac{d\Delta_{m}\beta_+}{dv}\right|+\sup_{T_u}\left|\pp{\Delta_{m+1}t}{v}\right|+\sup_{T_u}\left|\pp{\Delta_{m+1}t}{u}\right|\right\},
\end{align}
where we also used $\sup_{[0,u]}|\Delta_{m+1}\beta_+|\leq u\sup_{[0,u]}|d\Delta_{m+1}\beta/dv|$. Using the estimates \eqref{eq:817}, \eqref{eq:818} in \eqref{eq:753} we arrive at
\begin{align}
  \label{eq:819}
  \sup_{T_u}|\Xi_{m+1}|\leq CJ_1(u),
\end{align}
where
\begin{align}
  \label{eq:820}
  J_1(u)\coloneqq u^2\sup_{[0,u]}\left|\frac{d\Delta_{m}\beta_+}{dv}\right|+u^2\sup_{T_u}\left|\pp{\Delta_{m+1}t}{v}\right|+u\sup_{T_u}\left|\pp{\Delta_{m+1}t}{u}\right|.
\end{align}

We now look at equation \eqref{eq:752}. We are going to deal with this equation in a similar way as we dealt with the one appearing in the convergence proof of the fixed boundary problem. Integrating with respect to $v$ yields (for $K_{m+1}$ recall the first of \eqref{eq:346})
\begin{align}
  \label{eq:821}
  \pp{\Delta_{m+1}t}{u}(u,v)=e^{-K_{m+1}(u,v)}\int_0^ve^{K_{m+1}(u,v')}\left(\Xi_{m+1}-\mu_{m+1}\pp{\Delta_{m+1}t}{v}\right)(u,v')dv'.
\end{align}
We define
\begin{align}
  \label{eq:822}
  \Delta_{m+1}I(v)&\coloneqq \int_0^v\left\{e^{K_{m+1}(v,v')}-1+e^{K_{m+1}(v,v')}\tau_{m+1}(v,v')\right\}\pp{\Delta_{m+1}t}{v}(v,v')dv',\\
  \label{eq:823}
  P_{m+1}(v)&\coloneqq \int_0^ve^{K_{m+1}(v,v')}\Xi_{m+1}(v,v')dv'.
\end{align}
We recall that $\tau_{m+1}(u,v)$ is given by
\begin{align}
  \label{eq:824}
  \mu_{m+1}(u,v)=\frac{\kappa}{c_{+0}-c_{-0}}(1+\tau_{m+1}(u,v))
\end{align}
and that $\tau_{m+1}(u,v)=\mathcal{O}(u)$. We obtain from \eqref{eq:821} (recall that $b(v)=(\partial t/\partial u)(v,v)$)
\begin{align}
  \label{eq:825}
  \Delta_{m+1}b(v)=e^{-K_{m+1}(v,v)}\left\{P_{m+1}(v)-\frac{\kappa}{c_{+0}-c_{-0}}\left(\Delta_{m+1}f(v)+\Delta_{m+1}I(v)\right)\right\}.
\end{align}

Defining
\begin{align}
  \label{eq:826}
  \Lambda_{m+1}(v)\coloneqq \pp{t_m}{u}(v,v)\Delta_{m+1}\left(\frac{1}{\gamma(v)}\right),
\end{align}
we obtain from \eqref{eq:754} (recall $a(v)=(\partial t/\partial v)(v,v)$)
\begin{align}
  \label{eq:827}
  \Delta_{m+1}a(v)=\frac{1}{\gamma_{m+1}(v)}\Delta_{m+1}b(v)+\Lambda_{m+1}(v),
\end{align}
which implies
\begin{align}
  \label{eq:828}
  \frac{d\Delta_{m+1}f}{dv}(v)=\left(\frac{1}{\gamma_{m+1}(v)}+1\right)\Delta_{m+1}b(v)+\Lambda_{m+1}(v).
\end{align}

Using the estimate \eqref{eq:819} in \eqref{eq:823} we find
\begin{align}
  \label{eq:829}
  |P_{m+1}(v)|\leq Cv^2\left\{v\nd{\Delta_{m}\beta_+}+v\sup_{T_v}\left|\pp{\Delta_{m+1}t}{v}\right|+\sup_{T_v}\left|\pp{\Delta_{m+1}t}{u}\right|\right\}.
\end{align}
Using \eqref{eq:781} together with the second of \eqref{eq:342} we obtain
\begin{align}
  \label{eq:830}
  \Lambda_{m+1}(v)&=\frac{\lambda}{6\kappa^2}v\Delta_my(v)+\mathcal{O}(v^2|\Delta_my(v)|)+\mathcal{O}(v^2|\Delta_m\hat{V}(v)|)+\mathcal{O}(\sup_{[0,v]}|\Delta_m\beta_+(v)|)\notag\\
&\qquad+\mathcal{O}\left(v\sup_{T_v}\left|\pp{\Delta_{m+1}t}{v}\right|\right)+\mathcal{O}\left(v^3\sup_{T_v}\left|\pp{\Delta_{m+1}t}{u}\right|\right).
\end{align}
We substitute \eqref{eq:825} into \eqref{eq:828} and arrive at
\begin{align}
  \label{eq:831}
  \frac{d(v\Delta_{m+1}f(v))}{dv}+A_{m+1}(v)(v\Delta_{m+1}f(v))=v^2\Delta_{m+1}B(v),
\end{align}
where (for $\rho$ see \eqref{eq:337})
\begin{align}
  \label{eq:832}
  A_{m+1}(v)&=e^{-K_{m+1}(v,v)}\left(\frac{\rho_{m+1}(v)}{v}+\frac{\kappa}{c_{+0}-c_{-0}}\right)-\frac{1}{v}\left(1-e^{-K_{m+1}(v,v)}\right),\\
  \label{eq:833}
  \Delta_{m+1}B(v)&=\frac{e^{-K_{m+1}(v,v)}}{v^2}\left(1+\rho_{m+1}(v)+\frac{\kappa v}{c_{+0}-c_{-0}}\right)\left(\frac{c_{+0}-c_{-0}}{\kappa}P_{m+1}(v)-\Delta_{m+1}I(v)\right)\notag\\
&\qquad+\frac{\Lambda_{m+1}(v)}{v}.
\end{align}
Integrating \eqref{eq:831} from $v=0$ gives
\begin{align}
  \label{eq:834}
  v\Delta_{m+1}f(v)=\int_0^ve^{-\int_{v'}^vA_{m+1}(v'')dv''}v'^2\Delta_{m+1}B(v')dv'.
\end{align}
Substituting back into \eqref{eq:831} yields
\begin{align}
  \label{eq:835}
  \frac{d\Delta_{m+1}f}{dv}(v)=v\Delta_{m+1}B(v)-\frac{1}{v^2}(1+vA_{m+1}(v))\int_0^ve^{-\int_{v'}^vA_{m+1}(v'')dv''}v'^2\Delta_{m+1}B(v')dv'.
\end{align}

Now we decompose $\Delta_{m+1}B$ into
\begin{align}
  \label{eq:836}
  \Delta_{m+1}B=\overset{0}{\Delta}_{m+1}B+\overset{1}{\Delta}_{m+1}B,
\end{align}
where $\overset{1}{\Delta}_{m+1}B$ contains the terms of $\Delta_{m+1}B$ which are linear in $\Delta_{m+1}I$. Analogously we decompose
\begin{align}
  \label{eq:837}
  \frac{d\Delta_{m+1}f}{dv}=\overset{0}{R}_{m+1}+\overset{1}{R}_{m+1}.
\end{align}

We estimate $\overset{0}{\Delta}_{m+1}B$ using \eqref{eq:829},\eqref{eq:830}. We obtain
\begin{align}
  \label{eq:838}
  |\overset{0}{\Delta}_{m+1}B(v)|\leq C\Bigg\{|\Delta_{m}y(v)|+v|\Delta_{m}\hat{V}(v)|+\nd{\Delta_m\beta_+}+\sup_{T_v}\left|\pp{\Delta_{m+1}t}{v}\right|+\sup_{T_v}\left|\pp{\Delta_{m+1}t}{u}\right|\Bigg\}.
\end{align}
This implies
\begin{align}
  \label{eq:839}
  |\overset{0}{R}_{m+1}(v)|\leq CvJ_0(v),
\end{align}
where
\begin{align}
  \label{eq:840}
  J_0(v)\coloneqq |\Delta_{m}y(v)|+v|\Delta_{m}\hat{V}(v)|+\nd{\Delta_m\beta_+}+\sup_{T_v}\left|\pp{\Delta_{m+1}t}{v}\right|+\sup_{T_v}\left|\pp{\Delta_{m+1}t}{u}\right|.
\end{align}

We now estimate $\overset{1}{\Delta}_{m+1}B$. For this we need an estimate for $\Delta_{m+1}I$. From $K_{m+1}(u,v)=\mathcal{O}(v^2)$ we have
\begin{align}
  \label{eq:841}
  e^{K_{m+1}(v,v')}-1=\mathcal{O}(v^2).
\end{align}
Since we also have $\tau_{m+1}(v)=\mathcal{O}(v)$ we obtain from \eqref{eq:822}
\begin{align}
  \label{eq:842}
  |\Delta_{m+1}I(v)|\leq Cv\int_0^v\left|\pp{\Delta_{m+1}t}{v}(v,v')\right|dv',
\end{align}
which implies
\begin{align}
  \label{eq:843}
  |\overset{1}{\Delta}_{m+1}B(v)|\leq\frac{C}{v^2}|\Delta_{m+1}I(v)|\leq \frac{C}{v}\int_0^v\left|\pp{\Delta_{m+1}t}{v}(v,v')\right|dv'.
\end{align}
Therefore,
\begin{align}
  \label{eq:844}
  \frac{1}{v^2}\int_0^vv'^2|\overset{1}{\Delta}_{m+1}B(v')|dv'&\leq \frac{1}{v}\int_0^vv'|\overset{1}{\Delta}_{m+1}B(v')|dv'\notag\\
&\leq\frac{C}{v}\int_0^v\left(\int_0^{v'}\left|\pp{\Delta_{m+1}t}{v}(v',v'')\right|dv''\right)dv'\notag\\
&\leq C\int_0^v\Delta_{m+1}T(v,v'')dv'',
\end{align}
where we use the definition
\begin{align}
  \label{eq:845}
  \Delta_{m+1}T(v,v'')\coloneqq \sup_{v'\in[v'',v]}\left|\pp{\Delta_{m+1}t}{v}(v',v'')\right|.
\end{align}
Using this definition also in \eqref{eq:843} we obtain
\begin{align}
  \label{eq:846}
  v|\overset{1}{\Delta}_{m+1}B(v)|\leq C\int_0^v\Delta_{m+1}T(v,v')dv'.
\end{align}
From the estimates \eqref{eq:844}, \eqref{eq:846} we find
\begin{align}
  \label{eq:847}
  |\overset{1}{R}_{m+1}(v)|\leq C\int_0^v\Delta_{m+1}T(v,v')dv''.
\end{align}
This together with \eqref{eq:839} gives
\begin{align}
  \label{eq:848}
  \left|\frac{d\Delta_{m+1}f}{dv}(v)\right|\leq C\left\{vJ_0(v)+\int_0^v\Delta_{m+1}T(v,v')dv'\right\}.
\end{align}

Integrating \eqref{eq:752} with respect to $u$ from $u=v$ yields
\begin{align}
  \label{eq:849}
  \pp{\Delta_{m+1}t}{v}(u,v)&=e^{-L_{m+1}(u,v)}\bigg\{\Delta_{m+1}a(v)\notag\\
&\hspace{27mm}+\int_v^ue^{L_{m+1}(u',v)}\left(\Xi_{m+1}(u',v)+\left(\nu_{m+1}\pp{\Delta_{m+1}t}{u}\right)(u',v)\right)du'\bigg\}.
\end{align}
Using now $\nu_{m+1}(u,v)=\mathcal{O}(v)$ (cf.~\eqref{eq:469}) together with \eqref{eq:819} we obtain
\begin{align}
  \label{eq:850}
  \left|\pp{\Delta_{m+1}t}{v}(u,v)\right|\leq C\left\{|\Delta_{m+1}a(v)|+uJ_1(u)+v\int_v^u\left|\pp{\Delta_{m+1}t}{u}(u',v)\right|du'\right\}.
\end{align}

Now, from \eqref{eq:821} we get
\begin{align}
  \label{eq:851}
  \left|\pp{\Delta_{m+1}t}{u}(u,v)\right|\leq C\left\{vJ_1(u)+\int_0^v\left|\pp{\Delta_{m+1}t}{v}(u,v')\right|dv'\right\}.
\end{align}
From this we obtain
\begin{align}
  \label{eq:852}
  \int_v^u\left|\pp{\Delta_{m+1}t}{u}(u',v)\right|du'\leq C\left\{v(u-v)J_1(u)+(u-v)\int_0^v\Delta_{m+1}T(u,v')dv'\right\}.
\end{align}
Using this in \eqref{eq:850} we find
\begin{align}
  \label{eq:853}
  \left|\pp{\Delta_{m+1}t}{v}(u,v)\right|\leq C\left\{|\Delta_{m+1}a(v)|+uJ_1(u)+v(u-v)\int_0^v\Delta_{m+1}T(u,v')dv'\right\}.
\end{align}

From \eqref{eq:827} we have
\begin{align}
  \label{eq:854}
  \Delta_{m+1}a(v)=\frac{1}{1+\gamma_{m+1}(v)}\left(\frac{d\Delta_{m+1}f}{dv}(v)+\gamma_{m+1}(v)\Lambda_{m+1}(v)\right).
\end{align}
Using the estimates for $\Lambda_{m+1}$ and $d\Delta_{m+1}f/dv$ given by \eqref{eq:830}, \eqref{eq:848} respectively, we obtain
\begin{align}
  \label{eq:855}
  |\Delta_{m+1}a(v)|\leq C\left\{vJ_0(v)+\int_0^v\Delta_{m+1}T(v,v')dv'\right\}.
\end{align}
Using this in \eqref{eq:853} we get
\begin{align}
  \label{eq:856}
  \left|\pp{\Delta_{m+1}t}{v}(u,v)\right|&\leq C\left\{vJ_0(v)+uJ_1(u)+\int_0^v\Delta_{m+1}T(v,v')dv'\right\}\notag\\
&\leq C\left\{uF_m(u)+u\sup_{T_u}\left|\pp{\Delta_{m+1}t}{u}\right|+u\sup_{T_u}\left|\pp{\Delta_{m+1}t}{v}\right|+\int_0^v\Delta_{m+1}T(u,v')dv'\right\},
\end{align}
where we use the definition
\begin{align}
  \label{eq:857}
  F_m(u)\coloneqq \sup_{[0,u]}|\Delta_{m}y|+u\sup_{[0,u]}|\Delta_{m}\hat{V}|+\sup_{[0,u]}\left|\frac{d\Delta_{m}\beta_+}{dv}\right|.
\end{align}
%\begin{leftbar}
%  $F(u)$ corresponds to $\{\ldots\}_3$ in the notes from 17.1.14.
%\end{leftbar}
We also define
\begin{align}
  \label{eq:858}
  \tilde{F}_m(u)\coloneqq F_m(u)+\sup_{T_u}\left|\pp{\Delta_{m+1}t}{u}\right|+\sup_{T_u}\left|\pp{\Delta_{m+1}t}{v}\right|.
\end{align}

From \eqref{eq:856} with $u'$ in the role of $u$ and taking the supremum over $u'\in[v,u]$ we deduce
\begin{align}
  \label{eq:859}
  \Delta_{m+1}T(u,v)\leq C\left\{u\tilde{F}_m(u)+\int_0^v\Delta_{m+1}T(u,v')dv'\right\}.
\end{align}
Defining now
\begin{align}
  \label{eq:860}
  \Delta_{m+1}\Sigma_u(v)\coloneqq \int_0^v\Delta_{m+1}T(u,v')dv',
\end{align}
\eqref{eq:859} becomes
\begin{align}
  \label{eq:861}
  \frac{d}{dv}\Delta_{m+1}\Sigma_u(v)\leq Cu\tilde{F}_m(u)+C\Delta_{m+1}\Sigma_u(v),
\end{align}
which implies
\begin{align}
  \label{eq:862}
  \Delta_{m+1}\Sigma_u(v)\leq Cuv\tilde{F}_m(u).
\end{align}
Therefore,
\begin{align}
  \label{eq:863}
  |\Delta_{m+1}T(u,v)|\leq Cu\tilde{F}_m(u),
\end{align}
i.e.
\begin{align}
  \label{eq:864}
  \left|\pp{\Delta_{m+1}t}{v}(u,v)\right|\leq Cu\tilde{F}_m(u),
\end{align}
which in turn implies, through \eqref{eq:851},
\begin{align}
  \label{eq:865}
  \left|\pp{\Delta_{m+1}t}{u}(u,v)\right|\leq Cu\tilde{F}_m(u).
\end{align}
It follows that for $\varepsilon$ small enough
\begin{align}
  \label{eq:866}
  \left|\pp{\Delta_{m+1}t}{v}(u,v)\right|&\leq CuF_m(u),\\
  \label{eq:867}
  \left|\pp{\Delta_{m+1}t}{u}(u,v)\right|&\leq CuF_m(u).
\end{align}
We call these the rough estimates. In the following we will use these to get more precise estimates.
%Also
%\begin{align}
%  \label{eq:868}
%  \left|\frac{d\Delta_{m+1}f}{dv}(v)\right|\leq CvE(v).
%\end{align}

Using \eqref{eq:866} in \eqref{eq:822} we get
\begin{align}
  \label{eq:869}
  |\Delta_{m+1}I(v)|\leq Cv^3F_m(v),
\end{align}
which implies, through the first inequality of \eqref{eq:843},
\begin{align}
  \label{eq:870}
  |\overset{1}{\Delta}_{m+1}B(v)|\leq CvF_m(v).
\end{align}

Now, from \eqref{eq:838}, together with the rough estimates, we obtain
\begin{align}
  \label{eq:871}
  |\overset{0}{\Delta}_{m+1}B(v)|\leq CF_m(v).
\end{align}
Therefore,
\begin{align}
  \label{eq:872}
  |\Delta_{m+1}B(v)|\leq CF_m(v),
\end{align}
which, in conjunction with \eqref{eq:834}, implies
\begin{align}
  \label{eq:873}
  |\Delta_{m+1}f(v)|\leq Cv^2F_m(v).
\end{align}
Using the rough estimates in \eqref{eq:829} yields
\begin{align}
  \label{eq:874}
  |P_{m+1}(v)|\leq Cv^3F_m(v).
\end{align}
Now we use \eqref{eq:869}, \eqref{eq:873}, \eqref{eq:874} in \eqref{eq:825} and arrive at
\begin{align}
  \label{eq:875}
  \left|\Delta_{m+1}b(v)+\frac{\kappa\Delta_{m+1}f(v)}{c_{+0}-c_{-0}}\right|\leq Cv^3F_m(v).
\end{align}

Using the rough estimates in \eqref{eq:830} we get
\begin{align}
  \label{eq:876}
  \Lambda_{m+1}(v)=\frac{\lambda}{6\kappa^2}v\Delta_my(v)+\mathcal{O}(v^2|\Delta_m\hat{V}(v)|)+\mathcal{O}\left(v\nd{\Delta_m\beta_+}\right)+\mathcal{O}(v^2F_m(v)).
\end{align}
Using now \eqref{eq:874}, \eqref{eq:876} in \eqref{eq:833} we get
\begin{align}
  \label{eq:877}
  \overset{0}{\Delta}_{m+1}B=\frac{\lambda}{6\kappa^2}\Delta_my(v)+\mathcal{O}(v\sup_{[0,v]}|\Delta_m\hat{V}|)+\mathcal{O}(\nd{\Delta_m\beta_+})+\mathcal{O}(vF_m(v)).
\end{align}
Together with \eqref{eq:870} we find
\begin{align}
  \label{eq:878}
  \Delta_{m+1}B(v)=\frac{\lambda}{6\kappa^2}\Delta_my(v)+\mathcal{O}(v\sup_{[0,v]}|\Delta_m\hat{V}|)+\mathcal{O}(\nd{\Delta_m\beta_+})+\mathcal{O}(vF_m(v)).
\end{align}

Now we derive the precise estimate for $A_{m+1}(v)$. In view of \eqref{eq:832} we need first an estimate for $\rho_{m+1}(v)$. From \eqref{eq:337} we have
\begin{align}
  \label{eq:879}
  \rho_{m+1}(v)=\rho_{0,m+1}(v)+\mathcal{O}(v).
\end{align}
From \eqref{eq:374} we have
\begin{align}
  \label{eq:880}
  \rho_{0,m+1}(v)=\mathcal{O}(v).
\end{align}
Therefore,
\begin{align}
  \label{eq:881}
  A_{m+1}(v)=\mathcal{O}(1),
\end{align}
which, in conjunction with \eqref{eq:873}, implies
\begin{align}
  \label{eq:882}
  A_{m+1}(v)(v\Delta_{m+1}f(v))=\mathcal{O}(v^3F_m(v)).
\end{align}

Using \eqref{eq:878}, \eqref{eq:882} in \eqref{eq:831} we arrive at
\begin{align}
  \label{eq:883}
  \frac{d(v\Delta_{m+1}f(v))}{dv}=\frac{\lambda}{6\kappa^2}v^2\Delta_my(v)+\mathcal{O}(v^3\sup_{[0,v]}|\Delta_m\hat{V}|)+\mathcal{O}(v^2\nd{\Delta_m\beta_+})+\mathcal{O}(v^3F_m(v)).
\end{align}
Let us make the following generic definition. By $E_m(v)$ we mean a function of the kind
\begin{align}
  \label{eq:884}
  E_m(v)\coloneqq \mathcal{O}(v\nd{\Delta_my})+\mathcal{O}(\sup_{[0,v]}|\Delta_m\hat{V}|)+\mathcal{O}(v\nd{\Delta_m\hat{\beta}_+}).
\end{align}
%\begin{leftbar}
%  $E_m(v)$ corresponds to $E_m(v)$ in the notes from 17.1.14 (see page 37 of these notes).
%\end{leftbar}
Using now
\begin{align}
  \label{eq:885}
  \sup_{[0,v]}\left|\frac{d\Delta_{m}\beta_+}{dv}\right|\leq Cv^2\sup_{[0,v]}\bigg|\frac{d\Delta_m\hat{\beta}_+}{dv}\bigg|,
\end{align}
\eqref{eq:883} becomes
\begin{align}
  \label{eq:886}
   \frac{d(v\Delta_{m+1}f(v))}{dv}=\frac{\lambda}{6\kappa^2}v^2\Delta_my(v)+v^3E_m(v).
\end{align}
Therefore,
\begin{align}
  \label{eq:887}
  \Delta_{m+1}f(v)=\frac{1}{v}\frac{\lambda}{6\kappa^2}\int_0^vv'^2\Delta_my(v')dv'+v^3E_m(v).
\end{align}
Integrating by parts we obtain
\begin{align}
  \label{eq:888}
  \frac{6\kappa^2}{\lambda}\Delta_{m+1}f(v)=\frac{v^3}{3}\Delta_my(v)-\frac{1}{3v}\int_0^vv'^3\frac{d\Delta_my}{dv}(v')dv'+v^3E_m(v).
\end{align}

Now we express the above in terms of $\Delta_{m+1}\hat{f}(v)$ given by $\Delta_{m+1}f(v)=v^2\Delta_{m+1}\hat{f}(v)$. Using $|\Delta_my(v)|\leq v\nd{\Delta_my}$, we obtain from \eqref{eq:887},
\begin{align}
  \label{eq:889}
  |\Delta_{m+1}\hat{f}(v)|\leq \frac{\lambda}{24\kappa^2}v\nd{\Delta_my}+vE_m(v).
\end{align}
Since
\begin{align}
  \label{eq:890}
  \frac{d\Delta_{m+1}\hat{f}}{dv}=\frac{1}{v^3}\left(\frac{d(v\Delta_{m+1}f(v))}{dv}-3\Delta_{m+1}f(v)\right),
\end{align}
from \eqref{eq:886}, \eqref{eq:888} we deduce
\begin{align}
  \label{eq:891}
  \bigg|\frac{d\Delta_{m+1}\hat{f}}{dv}(v)\bigg|\leq \frac{\lambda}{24\kappa^2}\nd{\Delta_my}+E_m(v).
\end{align}

We turn to $\delta_{m+1}(v)=g_{m+1}(v)-c_{+0}f_{m+1}(v)$. From (see \eqref{eq:616})
\begin{align}
  \label{eq:892}
  \frac{d\delta_{m+1}}{dv}=(V_{m}-c_{+0})\frac{df_{m+1}}{dv}
\end{align}
we have
\begin{align}
  \label{eq:893}
  \frac{d\Delta_{m+1}\delta}{dv}=\frac{df_{m+1}}{dv}\Delta_{m}V+(V_{m-1}-c_{+0})\frac{d\Delta_{m+1}f}{dv}.
\end{align}

For the first term in \eqref{eq:893} we use (see \eqref{eq:608} and $V_m(v)-c_{+0}=\frac{\kappa}{2}(1+y_m(v))v+\mathcal{O}(v^2)$)
\begin{align}
  \label{eq:894}
  \frac{df_{m+1}}{dv}(v)=\frac{\lambda}{3\kappa^2}v+\mathcal{O}(v^2),\qquad \Delta_mV(v)=\frac{\kappa}{2}v\Delta_my(v)+\mathcal{O}(v^2E_m(v)),
\end{align}
and get
\begin{align}
  \label{eq:895}
  \frac{df_{m+1}}{dv}(v)\Delta_mV(v)=\frac{\lambda}{6\kappa}v^2\Delta_my(v)+\mathcal{O}(v^3E_m(v)).
\end{align}

For the second term in \eqref{eq:893} we use
\begin{align}
  \label{eq:896}
  \frac{df_{m+1}}{dv}(v)=2v\hat{f}_{m+1}(v)+v^2\frac{d\hat{f}_{m+1}}{dv}(v),
\end{align}
which implies, in conjunction with \eqref{eq:889}, \eqref{eq:891},
\begin{align}
  \label{eq:897}
  \frac{d\Delta_{m+1}f}{dv}(v)=\mathcal{O}(vE_m(v)).
\end{align}
Therefore,
\begin{align}
  \label{eq:898}
  (V_{m-1}(v)-c_{+0})\frac{d\Delta_{m+1}f}{dv}(v)=\mathcal{O}(v^3E_m(v)).
\end{align}
Putting things together we arrive at
\begin{align}
  \label{eq:899}
  \frac{d\Delta_{m+1}\delta}{dv}(v)=\frac{\lambda}{6\kappa}v^2\Delta_my(v)+\mathcal{O}(v^3E_m(v)).
\end{align}
Integrating (and also integrating by parts) we obtain
\begin{align}
  \label{eq:900}
  \Delta_{m+1}\delta(v)=\frac{\lambda}{18\kappa}\left(v^3\Delta_my(v)-\int_0^vv'^3\frac{d\Delta_my}{dv}(v')dv'\right)+\mathcal{O}(v^4E_m(v)).
\end{align}
Now we use $\hat{\delta}_{m+1}$ given by $\delta_{m+1}(v)=v^3\hat{\delta}_{m+1}(v)$. We obtain from \eqref{eq:899}, \eqref{eq:900}
\begin{align}
  \label{eq:901}
  \frac{d\Delta_{m+1}\hat{\delta}}{dv}=\frac{1}{v^4}\frac{\lambda}{6\kappa}\int_0^vv'^3\frac{d\Delta_{m}y}{dv}(v')dv'+\mathcal{O}(E_m(v)),
\end{align}
which implies
\begin{align}
  \label{eq:902}
  \left|\frac{d\Delta_{m+1}\hat{\delta}}{dv}(v)\right|\leq\frac{\lambda}{24\kappa}\nd{\Delta_my}+E_m(v).
\end{align}

We have
\begin{align}
  \label{eq:903}
  \frac{d\Delta_{m+1}\alpha_+}{dv}(v)=\left(\pp{\Delta_{m+1}\alpha}{v}+\pp{\Delta_{m+1}\alpha}{u}\right)(v,v).
\end{align}
For the first term we use as a starting point \eqref{eq:809}, i.e.
\begin{align}
  \label{eq:904}
  \pp{\Delta_{m+1}\alpha}{v}=\Delta_{m+1}A,
\end{align}
which implies
\begin{align}
  \label{eq:905}
  \left|\pp{\Delta_{m+1}\alpha}{v}\right|\leq C\left\{\left|\pp{\Delta_{m+1}t}{v}\right|+v\Big(|\Delta_{m+1}\alpha|+|\Delta_{m+1}\beta|+|\Delta_{m+1}r|\Big)\right\}.
\end{align}
To deal with the difference $\Delta_{m+1}r$ we use \eqref{eq:769}. For the differences $\Delta_{m+1}\alpha$, $\Delta_{m+1}\beta$ we use \eqref{eq:774}, \eqref{eq:775} respectively. Using also the estimates \eqref{eq:866}, \eqref{eq:867} together with the definition \eqref{eq:857}, we find
\begin{align}
  \label{eq:906}
  \sup_{T_u}\left|\pp{\Delta_{m+1}\alpha}{v}\right|\leq CuF_m(u),
\end{align}
For the second term in \eqref{eq:903} we use \eqref{eq:806}. For the differences in the first line of \eqref{eq:806} we use the same estimates which we used for the first term in \eqref{eq:903}. For the differences in the second line of \eqref{eq:806} we use again the estimates \eqref{eq:866}, \eqref{eq:867}. We obtain
\begin{align}
  \label{eq:907}
  \sup_{T_u}\left|\pp{\Delta_{m+1}\alpha}{u}\right|\leq C\left\{u^2F_m(u)+u^2\sup_{T_u}|\Delta_{m+1}\mu|+u^3\sup_{[0,u]}|\Delta_{m+1}\nu|\right\}.
\end{align}
For the differences of $\mu$ and $\nu$ we use \eqref{eq:817}, \eqref{eq:818} together with \eqref{eq:866}, \eqref{eq:867}. We obtain
\begin{align}
  \label{eq:908}
  \sup_{T_u}\left|\pp{\Delta_{m+1}\alpha}{u}\right|\leq Cu^2F_m(u).
\end{align}
Therefore,
\begin{align}
  \label{eq:909}
  \nd{\Delta_{m+1}\alpha_+}\leq CvF_m(v),
\end{align}
i.e.~(see \eqref{eq:857} and take into account that $\beta_{+,m}(v)=\beta_0+v^2\hat{\beta}_{+.m}(v)$)
\begin{align}
  \label{eq:910}
  \nd{\Delta_{m+1}\alpha_+}\leq Cv^2\left\{\nd{\Delta_{m}y}+\sup_{[0,v]}|\Delta_{m}\hat{V}|+v\nd{\Delta_{m}\hat{\beta}_+}\right\}.
\end{align}

In view of \eqref{eq:891}, \eqref{eq:902} and \eqref{eq:910}, the proof of the lemma is complete.
\end{proof}

\subsubsection{Estimates for the Identification Equation}
\begin{lemma}\label{lemma_est_ie}
Let $f_m(v)=v^2\hat{f}_m(v)$, $y_m(v)$, $g_m(v)$ satisfy the identification equation
\begin{align}
  \label{eq:911}
  g_m(v)+r_0=r^\ast(f_m(v),vy_m(v)).
\end{align}
Then the following relation holds
\begin{align}
  \label{eq:912}
  \frac{d\Delta_{m+1}y}{dv}(v)= \frac{3\kappa}{\lambda}\left\{\kappa\frac{d\Delta_{m+1}\hat{f}}{dv}(v)+\frac{d\Delta_{m+1}\hat{\delta}}{dv}(v)\right\}+\Landau(v\nd{\Delta_{m+1}\hat{f}})+\Landau(v\nd{\Delta_{m+1}y}),
\end{align}
where
\begin{align}
  \label{eq:913}
  v^3\hat{\delta}_m(v)=\delta_m(v)=g_m(v)-c_{+0}f_m(v).
\end{align}
\end{lemma}

\begin{proof}
We look at (see \eqref{eq:651})
\begin{align}
  \label{eq:914}
  \frac{dy_{m+1}}{dv}(v)=-\frac{\displaystyle\pp{\hat{F}_{m+1}}{v}(v,y_{m+1}(v))}{\displaystyle\pp{\hat{F}_{m+1}}{y}(v,y_{m+1}(v))},
\end{align}
where (we omit the argument of $y$)
\begin{align}
  \label{eq:916}
  \pp{\hat{F}_{m+1}}{y}(v,y_{m+1})&=\frac{\lambda}{2\kappa}y_{m+1}^2-\kappa \hat{f}_{m+1}(v)+v\pp{R}{y}(v,y_{m+1}),\\
  \label{eq:915}
  \pp{\hat{F}_{m+1}}{v}(v,y_{m+1})&=-\kappa y_{m+1}\frac{d\hat{f}_{m+1}}{dv}(v)+\frac{d\hat{\delta}_{m+1}}{dv}(v)+R_{m+1}(v,y_{m+1})+v\pp{R_{m+1}}{v}(v,y_{m+1}).
\end{align}
The difference we have to estimate is given by
\begin{align}
  \label{eq:917}
  \frac{d\Delta_{m+1}y}{dv}(v)=-\frac{\displaystyle \Delta_{m+1}\left(\pp{\hat{F}}{v}(v,y)\right)}{\displaystyle\pp{\hat{F}_{m+1}}{y}(v,y_{m+1})}-\pp{\hat{F}_m}{v}(v,y_m)\Delta_{m+1}\left(\frac{1}{\displaystyle\pp{\hat{F}}{y}(v,y)}\right).
\end{align}

Now
\begin{align}
  \label{eq:918}
  \Delta_{m+1}\left(\pp{\hat{F}}{v}(v,y)\right)&=-\kappa y_{m+1}\frac{d\Delta_{m+1}\hat{f}}{dv}(v)-\kappa \frac{d\hat{f}_m}{dv}(v)\Delta_{m+1}y+\frac{d\Delta_{m+1}\hat{\delta}}{dv}(v)\notag\\
&\qquad+\Delta_{m+1}R(v,y)+v\Delta_{m+1}\left(\pp{R}{v}(v,y)\right).
\end{align}
For the second term we use (cf.~\eqref{eq:614})
\begin{align}
  \label{eq:919}
  \bigg|\kappa \frac{d\hat{f}_m}{dv}(v)\Delta_{m+1}y\bigg|\leq Cv\nd{\Delta_{m+1}y}.
\end{align}
For the fourth term in \eqref{eq:918} we recall (see \eqref{eq:646})
\begin{align}
  \label{eq:920}
  R(v,y)=-\cp{\ppp{r^\ast}{t}}\left(\hat{f}(v)\right)^2-\cp{\frac{\partial^4r^\ast}{\partial w^4}}\frac{y^4}{24}-\cp{\frac{\partial^3r^\ast}{\partial t\partial w^2}}\frac{y^2\hat{f}(v)}{2}-vH(\hat{f}(v),y),
\end{align}
where $H$ is a smooth function of its two arguments. We estimate the difference of the first term by
\begin{align}
  \label{eq:921}
  \Delta_{m+1}\left(\hat{f}(v)\right)^2=(\hat{f}_{m+1}(v)+\hat{f}_m(v))\Delta_{m+1}\hat{f}(v)=\mathcal{O}(|\Delta_{m+1}\hat{f}(v)|)&=\mathcal{O}(v\nd{\Delta_{m+1}\hat{f}}).
\end{align}
The difference of the second term in \eqref{eq:920} we estimate using (recall $y(v)=\mathcal{O}(1)$)
\begin{align}
  \label{eq:922}
  \Delta_{m+1}y^4=\mathcal{O}(v\nd{\Delta_{m+1}y}),
\end{align}
while for the difference of the third term in \eqref{eq:920} we use
\begin{align}
  \label{eq:923}
  \Delta_{m+1}(y^2\hat{f}(v))&=\hat{f}_{m+1}(v)\Delta_{m+1}y^2+y_m^2\Delta_{m+1}\hat{f}(v)\notag\\
&=\mathcal{O}(v\nd{\Delta_{m+1}y})+\mathcal{O}(v\nd{\Delta_{m+1}\hat{f}}).
\end{align}
For the difference of the fourth term in \eqref{eq:920} we use
\begin{align}
  \label{eq:924}
  |v\Delta_{m+1}H(\hat{f}(v),y)|\leq Cv^2(\nd{\Delta_{m+1}\hat{f}}+\nd{\Delta_{m+1}y}).
\end{align}
Therefore,
\begin{align}
  \label{eq:925}
  |\Delta_{m+1}R(v,y)|\leq Cv(\nd{\Delta_{m+1}\hat{f}}+\nd{\Delta_{m+1}y})
\end{align}

Now we look at the fifth term in \eqref{eq:918}. We have
\begin{align}
  \label{eq:926}
  \pp{R}{v}(v,y)=-2\cp{\ppp{r^\ast}{t}}\hat{f}(v)\frac{d\hat{f}}{dv}(v)-\cp{\frac{\partial^3r^\ast}{\partial t\partial w^2}}\frac{y^2}{2}\frac{d\hat{f}}{dv}(v)-v(\partial_1H)(\hat{f}(v),y)\frac{d\hat{f}}{dv}(v)-H(\hat{f}(v),y),
\end{align}
where by $\partial_1H$ we denote the partial derivative of $H$ with respect to its first argument. For the difference of the first term we use
\begin{align}
  \label{eq:927}
  \Delta_{m+1}\left(\hat{f}(v)\frac{d\hat{f}}{dv}(v)\right)&=\frac{d\hat{f}_{m+1}}{dv}(v)\Delta_{m+1}\hat{f}(v)+\hat{f}_m(v)\frac{d\Delta_{m+1}\hat{f}}{dv}(v)\notag\\
  &=\mathcal{O}(\nd{\Delta_{m+1}\hat{f}}),
\end{align}
while for the difference of the second term we use
\begin{align}
  \label{eq:928}
  \Delta_{m+1}\left(y^2\frac{d\hat{f}}{dv}(v))\right)&=\frac{d\hat{f}_{m+1}}{dv}(v)\Delta_{m+1}(y^2)+y_m^2\frac{d\Delta_{m+1}\hat{f}}{dv}(v)\notag\\
&=\Landau(v\nd{\Delta_{m+1}y})+\Landau({\nd{\Delta_{m+1}\hat{f}}}).
\end{align}
For the differences of the last two terms in \eqref{eq:926} we use
\begin{align}
  \label{eq:929}
  |\Delta_{m+1}\partial_1H(\hat{f}(v),y)|,|\Delta_{m+1}H(\hat{f}(v),y)|\leq Cv(\nd{\Delta_{m+1}\hat{f}}+\nd{\Delta_{m+1}y}).
\end{align}
Therefore,
\begin{align}
  \label{eq:930}
  \left|\Delta_{m+1}\left\{v(\partial_1H)(\hat{f}(v),y)\frac{d\hat{f}}{dv}(v)+H(\hat{f}(v),y)\right\}\right|\leq Cv(\nd{\Delta_{m+1}\hat{f}}+\nd{\Delta_{m+1}y}).
\end{align}
It follows that
\begin{align}
  \label{eq:931}
  \left|v\Delta_{m+1}\left(\pp{R}{v}(v,y)\right)\right|\leq Cv(\nd{\Delta_{m+1}\hat{f}}+v\nd{\Delta_{m+1}y}).
\end{align}
We conclude that
\begin{align}
  \label{eq:932}
  \Delta_{m+1}\left(\pp{\hat{F}}{v}(v,y)\right)=-\kappa y_{m+1}\frac{d\Delta_{m+1}\hat{f}}{dv}(v)+\frac{d\Delta_{m+1}\hat{\delta}}{dv}(v)+\Landau(v\nd{\Delta_{m+1}\hat{f}})+\Landau(v\nd{\Delta_{m+1}y}).
\end{align}

Now we look at the second term in \eqref{eq:917}. We use
\begin{align}
  \label{eq:933}
  \Delta_{m+1}\left(\frac{1}{f}\right)=-\frac{\Delta_{m+1}f}{f_{m+1}f_m}.
\end{align}
We need to estimate the difference $\Delta_{m+1}\left((\partial\hat{F}/\partial y)(v,y)\right)$. We have
\begin{align}
  \label{eq:934}
  \Delta_{m+1}\left(\pp{\hat{F}}{y}(v,y)\right)=\frac{\lambda}{2\kappa}\Delta_{m+1}(y^2)-\kappa\Delta_{m+1}\hat{f}(v)+v\Delta_{m+1}\left(\pp{R}{y}(v,y)\right).
\end{align}
For the first term we have
\begin{align}
  \label{eq:935}
  \Delta_{m+1}(y^2)=\mathcal{O}(v\nd{\Delta_{m+1}y}).
\end{align}
For the third term in \eqref{eq:934} we use
\begin{align}
  \label{eq:936}
  \pp{R}{y}(v,y)=-\cp{\frac{\partial^4r^\ast}{\partial w^4}}\frac{y^3}{6}-\cp{\frac{\partial^3r^\ast}{\partial t\partial^2w}}y\hat{f}(v)-v\pp{H}{y}(\hat{f}(v),y).
\end{align}
For difference of the first term we use
\begin{align}
  \label{eq:937}
  \Delta_{m+1}y^3=\mathcal{O}(v\nd{\Delta_{m+1}y}),
\end{align}
while for the difference of the second term we use
\begin{align}
  \label{eq:938}
  \Delta_{m+1}(y\hat{f}(v))&=\mathcal{O}(v\nd{\Delta_{m+1}y})+\mathcal{O}(v\nd{\Delta_{m+1}\hat{f}}).
\end{align}
For the difference of the third term in \eqref{eq:936} we have
\begin{align}
  \label{eq:939}
  \left|\Delta_{m+1}\left(v\pp{H}{y}(\hat{f}(v),y)\right)\right|\leq Cv^2(\nd{\Delta_{m+1}\hat{f}}+\nd{\Delta_{m+1}y}).
\end{align}
Therefore, we obtain for the third term in \eqref{eq:934}
\begin{align}
  \label{eq:940}
  \left|v\Delta_{m+1}\left(\pp{R}{y}(v,y)\right)\right|\leq Cv^2(\nd{\Delta_{m+1}\hat{f}}+\nd{\Delta_{m+1}y}).
\end{align}
We conclude that
\begin{align}
  \label{eq:941}
  \left|\Delta_{m+1}\left(\pp{\hat{F}}{y}(v,y)\right)\right|\leq Cv(\nd{\Delta_{m+1}\hat{f}}+\nd{\Delta_{m+1}y}).
\end{align}

To deal with \eqref{eq:917} we also need (see \eqref{eq:649}, \eqref{eq:652})
\begin{align}
  \label{eq:942}
  \pp{\hat{F}}{y}(v,y_{m+1}(v))=\frac{\lambda}{3\kappa}+\mathcal{O}(v).
\end{align}
Using now the asymptotic forms as given by \eqref{eq:932}, \eqref{eq:941} and \eqref{eq:942} in \eqref{eq:917} we arrive at \eqref{eq:912}.
\end{proof}

\subsubsection{Estimates for the Jump Conditions}
\begin{lemma}\label{lemma_est_jc}
Let $\alpha_{+,m}(v)$, $\beta_{+,m}(v)$, $f_m(v)$, $z_m(v)$, $V_m(v)$ satisfy
\begin{align}
  \label{eq:943}
  \jump{\beta_m(v)}=\jump{\alpha_m(v)}^3G(\alpha_{+,m}(v),\alpha_{-,m}(v),\beta_{-,m}(v)),\qquad V_m(v)=\frac{\jump{T_m^{tr}(v)}}{\jump{T_m^{tt}(v)}},
\end{align}
where
\begin{align}
  \label{eq:944}
  \alpha_{-,m}(v)=\alpha^\ast(f_m(v),z_m(v)),\qquad \beta_{-,m}(v)=\beta^\ast(f_m(v),z_m(v)),
\end{align}
where the right hand sides are given by the state ahead, i.e.~by the solution in the maximal development, and
\begin{align}
  \label{eq:945}
  \jump{T_m^{\mu\nu}(v)}=T^{\mu\nu}(\alpha_{+,m}(v),\beta_{+,m}(v))-T^{\mu\nu}(\alpha_{-,m}(v),\beta_{-,m}(v)).
\end{align}
Then, the following estimates hold:
\begin{align}
  \label{eq:946}
\left|\frac{d\Delta_{m+1}\beta_+}{dv}(v)\right|&\leq C\left\{v^2\nd{\Delta_{m+1}z}+\nd{\Delta_{m+1}f}+v^2\nd{\Delta_{m+1}\alpha_+}\right\},\\
 \label{eq:947}
  \left|\Delta_{m+1}V(v)-\frac{\kappa}{2}\Delta_{m+1}z(v)\right|&\leq C\left\{v^2\nd{\Delta_{m+1}z}+v\nd{\Delta_{m+1}f}+v\nd{\Delta_{m+1}\alpha_+}\right\}. 
\end{align}
\end{lemma}

\begin{proof}
We start with
\begin{align}
  \label{eq:948}
  \jump{\beta}=\jump{\alpha}^3G(\alpha_+,\alpha_-,\beta_-).
\end{align}
We have
\begin{align}
  \label{eq:949}
  \frac{d\Delta_{m+1}\jump{\beta}}{dv}=3\Delta_{m+1}\left(\jump{\alpha}^2\frac{d\jump{\alpha}}{dv}G\right)+\Delta_{m+1}\left(\jump{\alpha}^3\frac{dG}{dv}\right).
\end{align}
We start with an estimate for the difference of
\begin{align}
  \label{eq:950}
  \beta_{-,m}(v)=\beta^\ast(f_m(v),z_m(v)),
\end{align}
where $\beta^\ast(t,w)$ is given by the state ahead, i.e.~by the solution in the maximal development. We have
\begin{align}
  \label{eq:951}
  \frac{d\Delta_{m+1}\beta_-}{dv}&=\left(\pp{\beta^\ast}{t}\right)_{m+1}\frac{d\Delta_{m+1}f}{dv}+\frac{df_m}{dv}\Delta_{m+1}\left(\pp{\beta^\ast}{t}\right)\notag\\
&\qquad+\left(\pp{\beta^\ast}{w}\right)_{m+1}\frac{d\Delta_{m+1}z}{dv}+\frac{dz_m}{dv}\Delta_{m+1}\left(\pp{\beta^\ast}{w}\right).
\end{align}
Here the partial derivatives of $\beta$ possess the same arguments as on the right hand side of \eqref{eq:950}. Taking into account the second of \eqref{eq:687} we obtain
\begin{align}
  \label{eq:952}
  \left|\Delta_{m+1}\left(\pp{\beta^\ast}{t}\right)\right|\leq C\left\{|\Delta_{m+1}f|+|\Delta_{m+1}z|\right\},\qquad \left|\Delta_{m+1}\left(\pp{\beta^\ast}{w}\right)\right|\leq C\left\{|\Delta_{m+1}f|+v|\Delta_{m+1}z|\right\}.
\end{align}
In view of \eqref{eq:680}, \eqref{eq:682}, \eqref{eq:686}, \eqref{eq:688} we arrive at
\begin{align}
  \label{eq:953}
  \left|\frac{d\Delta_{m+1}\beta_-}{dv}\right|\leq C\left\{\nd{\Delta_{m+1}f}+v^2\nd{\Delta_{m+1}z}\right\}.
\end{align}
Similarly we find
\begin{align}
  \label{eq:954}
  \left|\frac{d\Delta_{m+1}\alpha_-}{dv}\right|\leq C\left\{\nd{\Delta_{m+1}f}+\nd{\Delta_{m+1}z}\right\}.
\end{align}
(The factor $v^2$ does not appear since the conditions \eqref{eq:687} do not hold for the partial derivatives of $\alpha$ with respect to $w$).

We split the first term on the right hand side of \eqref{eq:949} into $I_1+I_2$, where
\begin{align}
  \label{eq:955}
  I_1\coloneqq 3\jump{\alpha_{m+1}}^2G_{m+1}\frac{d\Delta_{m+1}\jump{\alpha}}{dv},\qquad I_2\coloneqq 3\frac{d\jump{\alpha_m}}{dv}\Delta_{m+1}\left(\jump{\alpha}^2G\right).
\end{align}
For $I_1$ we use
\begin{align}
  \label{eq:956}
  \frac{d\Delta_{m+1}\jump{\alpha}}{dv}=\frac{d\Delta_{m+1}\alpha_+}{dv}-\frac{d\Delta_{m+1}\alpha_-}{dv}.
\end{align}
The second term is estimated by \eqref{eq:954}. Taking into account $\jump{\alpha(v)}=\mathcal{O}(v)$ (see \eqref{eq:711}) we get
\begin{align}
  \label{eq:957}
  |I_1|\leq Cv^2\left\{\nd{\Delta_{m+1}\alpha_+}+\nd{\Delta_{m+1}f}+\nd{\Delta_{m+1}z}\right\}.
\end{align}
For $I_2$ we use
\begin{align}
  \label{eq:958}
  |\Delta_{m+1}G|\leq C\left\{\nd{\Delta_{m+1}\alpha_+}+\nd{\Delta_{m+1}\alpha_-}+\nd{\Delta_{m+1}\beta_-}\right\}.
\end{align}
Together with \eqref{eq:953}, \eqref{eq:954} we obtain
\begin{align}
  \label{eq:959}
  |I_2|\leq Cv^2\left\{\nd{\Delta_{m+1}\alpha_+}+\nd{\Delta_{m+1}f}+\nd{\Delta_{m+1}z}\right\}.
\end{align}

The second term in \eqref{eq:949} we split into $I_3+I_4$, where
\begin{align}
  \label{eq:960}
  I_3\coloneqq \left(\frac{dG}{dv}\right)_{m+1}\Delta_{m+1}\jump{\alpha}^3,\qquad I_4\coloneqq \jump{\alpha_m}^3\Delta_{m+1}\left(\frac{dG}{dv}\right).
\end{align}
Reasoning in a similar way as to arrive at \eqref{eq:957}, \eqref{eq:959} we find
\begin{align}
  \label{eq:961}
  |I_3|&\leq Cv^3\left\{\nd{\Delta_{m+1}\alpha_+}+\nd{\Delta_{m+1}f}+\nd{\Delta_{m+1}z}\right\},\\
  \label{eq:962}
  |I_4|&\leq Cv^4\left\{\nd{\Delta_{m+1}\alpha_+}+\nd{\Delta_{m+1}f}+\nd{\Delta_{m+1}z}\right\}.
\end{align}

Using now \eqref{eq:957}, \eqref{eq:959}, \eqref{eq:961}, \eqref{eq:962} in \eqref{eq:949}, %which is
%\begin{align}
%  \label{eq:963}
%  \frac{d\Delta_{m+1}\beta_+}{dv}=\frac{d\Delta_{m+1}\beta_-}{dv}+I_a+I_b+I_c+I_d,
%\end{align}
we obtain
\begin{align}
  \label{eq:964}
  \left|\frac{d\Delta_{m+1}\jump{\beta}}{dv}\right|&\leq \sum_{i=1}^4|I_i|\notag\\
&\leq Cv^2\left\{\nd{\Delta_{m+1}\alpha_+}+\nd{\Delta_{m+1}f}+\nd{\Delta_{m+1}z}\right\},
\end{align}
which implies
\begin{align}
  \label{eq:965}
  |\Delta_{m+1}\jump{\beta}|\leq Cv^3\left\{\nd{\Delta_{m+1}\alpha_+}+\nd{\Delta_{m+1}f}+\nd{\Delta_{m+1}z}\right\}.
\end{align}
From \eqref{eq:964} together with \eqref{eq:953} we conclude
\begin{align}
  \label{eq:966}
  \left|\frac{d\Delta_{m+1}\beta_+}{dv}\right|\leq C\left\{\nd{\Delta_{m+1}f}+v^2\nd{\Delta_{m+1}z}+v^2\nd{\Delta_{m+1}\alpha_+}\right\}.
\end{align}

We turn to the jump condition
\begin{align}
  \label{eq:967}
  V=\frac{\jump{T^{tr}}}{\jump{T^{tt}}}.
\end{align}
We have
\begin{align}
  \label{eq:968}
  \Delta_{m+1}V=\frac{A_{m+1}B_m-A_mB_{m+1}}{B_{m+1}B_m},
\end{align}
where
\begin{align}
  \label{eq:969}
  A_m\coloneqq T^{tr}(\alpha_{+,m},\beta_{+,m})-T^{tr}(\alpha_{-,m},\beta_{-,m}),\qquad B_m\coloneqq T^{tt}(\alpha_{+,m},\beta_{+,m})-T^{tt}(\alpha_{-,m},\beta_{-,m}).
\end{align}

Let us denote the numerator of \eqref{eq:968} by $N$. We have
\begin{align}
  \label{eq:970}
  -N= A_mB_{m+1}-A_{m+1}B_m=A_m\Delta_{m+1}B-B_m\Delta_{m+1}A.
\end{align}
Now we rewrite
\begin{align}
  \label{eq:971}
  A_m=\overset{\,\,1}{A}_{m}+\overset{\,\,2}{A}_{m},
\end{align}
where
\begin{align}
  \label{eq:972}
  \overset{\,\,1}{A}_{m}=T^{tr}(\alpha_{+,m},\beta_{+,m})-T^{tr}(\alpha_{+,m},\beta_{-,m}),\qquad \overset{\,\,2}{A}_{m}=T^{tr}(\alpha_{+,m},\beta_{-,m})-T^{tr}(\alpha_{-,m},\beta_{-,m}).
\end{align}

Now, for a smooth function $F(x,y)$ we have
\begin{align}
  \label{eq:973}
  F(x,y_2)-F(x,y_1)&=\left(\pp{F}{y}\right)(x,y_1)(y_2-y_1)+\mathcal{O}((y_2-y_1)^2).
\end{align}
Suppose now $g_2(v)-g_1(v)=\mathcal{O}(v)$. It follows
\begin{align}
  \label{eq:974}
  F(f(v),g_2(v))-F(f(v),g_1(v))&=\cp{\pp{F}{y}}(g_2(v)-g_1(v))+\mathcal{O}(v(g_2(v)-g_1(v)))\notag\\
&=\cp{\pp{F}{y}}(g_2(v)-g_1(v))+\mathcal{O}(v^2).
\end{align}
Using this we obtain
\begin{align}
  \label{eq:975}
  \overset{\,\,1}{A}_{m}=\cp{\pp{T^{tr}}{\beta}}\jump{\beta_m}+\mathcal{O}(v^2),\qquad \overset{\,\,2}{A}_{m}=\cp{\pp{T^{tr}}{\alpha}}\jump{\alpha_m}+\mathcal{O}(v^2).
\end{align}
Now we recall (for the first see \eqref{eq:711})
\begin{align}
  \label{eq:976}
  \jump{\alpha_m(v)}=\dot{\alpha}_0(1-y_m(v))v+\mathcal{O}(v^2)=\mathcal{O}(v),\qquad \jump{\beta_m(v)}=\jump{\alpha_m(v)}^3G(\alpha_{+,m}(v),\alpha_{-,m}(v),\beta_{-,m}(v)).
\end{align}
Therefore,
\begin{align}
  \label{eq:977}
  A_m=\cp{\pp{T^{tr}}{\alpha}}\dot{\alpha}_0(1-y_m(v))v+\mathcal{O}(v^2).
\end{align}
Similarly we find
\begin{align}
  \label{eq:978}
  B_m=\cp{\pp{T^{tt}}{\alpha}}\dot{\alpha}_0(1-y_m(v))v+\mathcal{O}(v^2).
\end{align}

Now we look at the difference $\Delta_{m+1}B$. We have
\begin{align}
  \label{eq:979}
  \Delta_{m+1}B=\,\overset{1}{\Delta}_{m+1}B+\overset{2}{\Delta}_{m+1}B,
\end{align}
where
\begin{align}
  \label{eq:980}
  \overset{1}{\Delta}_{m+1}B&=T^{tt}(\alpha_{+,m+1},\beta_{+,m+1})-T^{tt}(\alpha_{+,m},\beta_{+,m}),\\
  \label{eq:981}
  \overset{2}{\Delta}_{m+1}B&=T^{tt}(\alpha_{-,m},\beta_{-,m})-T^{tt}(\alpha_{-,m+1},\beta_{-,m+1}).
\end{align}

We rewrite
\begin{align}
  \label{eq:982}
  \overset{1}{\Delta}_{m+1}B=\overset{1a}{\Delta}_{m+1}B+\overset{1b}{\Delta}_{m+1}B,
\end{align}
where
\begin{align}
  \label{eq:983}
  \overset{1a}{\Delta}_{m+1}B&=T^{tt}(\alpha_{+,m+1},\beta_{+,m+1})-T^{tt}(\alpha_{+,m},\beta_{+,m+1}),\\
  \label{eq:984}
  \overset{1b}{\Delta}_{m+1}B&=T^{tt}(\alpha_{+,m},\beta_{+,m+1})-T^{tt}(\alpha_{+,m},\beta_{+,m}).
\end{align}
Now, appealing to the first line of \eqref{eq:974}, we have
\begin{align}
  \label{eq:985}
  \overset{1a}{\Delta}_{m+1}B&=\cp{\pp{T^{tt}}{\alpha}}\Delta_{m+1}\alpha_++\mathcal{O}(v|\Delta_{m+1}\alpha_+|),\\
  \label{eq:986}
  \overset{1b}{\Delta}_{m+1}B&=\cp{\pp{T^{tt}}{\beta}}\Delta_{m+1}\beta_++\mathcal{O}(v|\Delta_{m+1}\beta_+|).
\end{align}
Therefore,
\begin{align}
  \label{eq:987}
  \overset{1}{\Delta}_{m+1}B=\cp{\pp{T^{tt}}{\alpha}}\Delta_{m+1}\alpha_++\cp{\pp{T^{tt}}{\beta}}\Delta_{m+1}\beta_++\mathcal{O}(v|\Delta_{m+1}\alpha_+|)+\mathcal{O}{(v|\Delta_{m+1}\beta_+|)}.
\end{align}

We rewrite
\begin{align}
  \label{eq:988}
  \overset{2}{\Delta}_{m+1}B=\overset{2a}{\Delta}_{m+1}B+\overset{2b}{\Delta}_{m+1}B,
\end{align}
where
\begin{align}
  \label{eq:989}
  \overset{2a}{\Delta}_{m+1}B&=-\Big(T^{tt}(\alpha_{-,m+1},\beta_{-,m+1})-T^{tt}(\alpha_{-,m},\beta_{-,m+1})\Big),\\
  \label{eq:990}
  \overset{2b}{\Delta}_{m+1}B&=-\Big(T^{tt}(\alpha_{-,m},\beta_{-,m+1})-T^{tt}(\alpha_{-,m},\beta_{-,m})\Big).
\end{align}
Again, by the first line of \eqref{eq:974}, we have
\begin{align}
  \label{eq:991}
  \overset{2b}{\Delta}_{m+1}B=-\cp{\pp{T^{tt}}{\beta}}\Delta_{m+1}\beta_-+\mathcal{O}(v|\Delta_{m+1}\beta_-|).
\end{align}

Defining
\begin{align}
  \label{eq:992}
  \Psi(t,w)\coloneqq T^{tt}(\alpha^\ast(t,w),\beta^\ast(f_{m+1},z_{m+1})),
\end{align}
where $\alpha^\ast(t,w)$, $\beta^\ast(t,w)$ are given by the state ahead, we have
\begin{align}
  \label{eq:993}
  \overset{2a}{\Delta}_{m+1}B=-\Big(\Psi(f_{m+1},z_{m+1})-\Psi(f_m,z_m)\Big).
\end{align}
We have
\begin{align}
  \label{eq:994}
  \Psi(t,w)&=\Psi(0,0)+\cp{\pp{\Psi}{t}}t+\cp{\pp{\Psi}{w}}w\notag\\
&\qquad +\cp{\ppp{\Psi}{t}}\frac{t^2}{2}+\cp{\pppp{\Psi}{t}{w}}tw+\cp{\ppp{\Psi}{w}}\frac{w^2}{2}+\mathcal{O}(t^3,t^2w,tw^2,w^3).
\end{align}
Using
\begin{align}
  \label{eq:995}
  \cp{\pp{\Psi}{t}}&=\cp{\pp{T^{tt}}{\alpha}}\cp{\pp{\alpha^\ast}{t}},\qquad \cp{\pp{\Psi}{w}}=\cp{\pp{T^{tt}}{\alpha}}\dot{\alpha}_0,\notag\\
\cp{\ppp{\Psi}{t}}&=\cp{\ppp{T^{tt}}{\alpha}}\cp{\pp{\alpha^\ast}{t}}^2+\cp{\pp{T^{tt}}{\alpha}}\cp{\ppp{\alpha^\ast}{t}},\notag\\
\cp{\ppp{\Psi}{w}}&=\cp{\ppp{T^{tt}}{\alpha}}\dot{\alpha}_0^2+\cp{\pp{T^{tt}}{\alpha}}\ddot{\alpha}_0,\notag\\
\cp{\pppp{\Psi}{t}{w}}&=\cp{\ppp{T^{tt}}{\alpha}}\cp{\pp{\alpha^\ast}{t}}\dot{\alpha}_0+\cp{\pp{T^{tt}}{\alpha}}\cp{\pppp{\alpha^\ast}{t}{w}},
\end{align}
where we also use $(\partial\alpha^\ast/\partial w)_0=\dot{\alpha}_0$, $(\partial^2\alpha^\ast/\partial w^2)_0=\ddot{\alpha}_0$ (see \eqref{eq:2022}), we obtain
\begin{align}
  \label{eq:996}
  -\overset{2a}{\Delta}_{m+1}B&=\cp{\pp{T^{tt}}{\alpha}}\cp{\pp{\alpha^\ast}{t}}\Delta_{m+1}f+\cp{\pp{T^{tt}}{\alpha}}\dot{\alpha}_0\Delta_{m+1}z\notag\\
&\qquad -\left\{\cp{\ppp{T^{tt}}{\alpha}}\dot{\alpha}_0^2+\cp{\pp{T^{tt}}{\alpha}}\ddot{\alpha}_0\right\}v\Delta_{m+1}z+\mathcal{O}(v|\Delta_{m+1}f|)+\mathcal{O}(v^2|\Delta_{m+1}z|),
\end{align}
where we also use $\tfrac{1}{2}(\Delta_{m+1}(z^2))=-v(\Delta_{m+1}z)+\mathcal{O}(v^2|\Delta_{m+1}z|)$, since $z(v)=-v+\mathcal{O}(v^2)$ and $f(v)=\mathcal{O}(v^2)$. This together with \eqref{eq:991} implies
\begin{align}
  \label{eq:997}
  \overset{2}{\Delta}_{m+1}B&=-\cp{\pp{T^{tt}}{\alpha}}\dot{\alpha}_0\Delta_{m+1}z+\left\{\cp{\ppp{T^{tt}}{\alpha}}\dot{\alpha}_0^2+\cp{\pp{T^{tt}}{\alpha}}\ddot{\alpha}_0\right\}v\Delta_{m+1}z+\mathcal{O}(v^2|\Delta_{m+1}z|)\notag\\
&\qquad-\cp{\pp{T^{tt}}{\alpha}}\cp{\pp{\alpha^\ast}{t}}\Delta_{m+1}f+\mathcal{O}(v|\Delta_{m+1}f|)-\cp{\pp{T^{tt}}{\beta}}\Delta_{m+1}\beta_-+\mathcal{O}(v|\Delta_{m+1}\beta_-|).
\end{align}

From \eqref{eq:987} and \eqref{eq:997} we obtain
\begin{align}
  \label{eq:998}
  \Delta_{m+1}B&=-\cp{\pp{T^{tt}}{\alpha}}\dot{\alpha}_0\Delta_{m+1}z+\left\{\cp{\ppp{T^{tt}}{\alpha}}\dot{\alpha}_0^2+\cp{\pp{T^{tt}}{\alpha}}\ddot{\alpha}_0\right\}v\Delta_{m+1}z\notag\\
&\qquad+\cp{\pp{T^{tt}}{\alpha}}\Delta_{m+1}\alpha_+-\cp{\pp{T^{tt}}{\alpha}}\cp{\pp{\alpha^\ast}{t}}\Delta_{m+1}f\notag\\
&\qquad +\cp{\pp{T^{tt}}{\beta}}\Delta_{m+1}\jump{\beta}+\mathcal{O}(v|\Delta_{m+1}\beta_+|)+\mathcal{O}(v|\Delta_{m+1}\beta_-|)\notag\\
&\qquad+\mathcal{O}(v^3\nd{\Delta_{m+1}z})+\mathcal{O}(v^2\nd{\Delta_{m+1}\alpha_+})+\mathcal{O}(v^2\nd{\Delta_{m+1}f}).
\end{align}
Making use of the estimate \eqref{eq:965} we arrive at
\begin{align}
  \label{eq:999}
  \Delta_{m+1}B&=-\cp{\pp{T^{tt}}{\alpha}}\dot{\alpha}_0\Delta_{m+1}z+\left\{\cp{\ppp{T^{tt}}{\alpha}}\dot{\alpha}_0^2+\cp{\pp{T^{tt}}{\alpha}}\ddot{\alpha}_0\right\}v\Delta_{m+1}z\notag\\
&\qquad+\cp{\pp{T^{tt}}{\alpha}}\Delta_{m+1}\alpha_+-\cp{\pp{T^{tt}}{\alpha}}\cp{\pp{\alpha^\ast}{t}}\Delta_{m+1}f\notag\\
&\qquad+\mathcal{O}(v^3\nd{\Delta_{m+1}z})+\mathcal{O}(v^2\nd{\Delta_{m+1}\alpha_+})+\mathcal{O}(v^2\nd{\Delta_{m+1}f}),
\end{align}
where we also used \eqref{eq:953}, \eqref{eq:966}. The analogous expression holds for $\Delta_{m+1}A$ but with $T^{tr}$ in the role of $T^{tt}$. Putting things together we arrive at the following expression for the numerator of \eqref{eq:968}
\begin{align}
  \label{eq:1000}
  N&=-v^2\dot{\alpha}_0^3(1-y_m(v))\left\{\cp{\ppp{T^{tt}}{\alpha}}\cp{\pp{T^{tr}}{\alpha}}-\cp{\ppp{T^{tr}}{\alpha}}\cp{\pp{T^{tt}}{\alpha}}\right\}\Delta_{m+1}z\notag\\
&\qquad +\mathcal{O}(v^4\nd{\Delta_{m+1}z})+\mathcal{O}(v^3\nd{\Delta_{m+1}\alpha_+})+\mathcal{O}(v^3\nd{\Delta_{m+1}f}).
\end{align}

Let us inspect the curly bracket in \eqref{eq:1000}. For this we use (see the first and second of \eqref{eq:182})
\begin{align}
  \label{eq:1001}
  \pp{T^{tt}}{\alpha}=\frac{G\psi_t^2}{2\eta}(1+v\eta)^2,\qquad \pp{T^{tr}}{\alpha}=\frac{G\psi_t^2}{2\eta}(v+\eta)(1+v\eta).
\end{align}
Using \eqref{eq:177}, \eqref{eq:202}, the first of \eqref{eq:176} and the first of \eqref{eq:181} we obtain
\begin{align}
  \label{eq:1002}
  \ppp{T^{tt}}{\alpha}&=\frac{G\psi_t^2}{4\eta^2}(1+v\eta)\left((1+3\eta^2+3v\eta+v\eta^3)-(1-v\eta)\frac{d\eta}{d\tilde{\rho}}\right),\\
  \label{eq:1003}
  \ppp{T^{tr}}{\alpha}&=\frac{G\psi_t^2}{4\eta^2}\left((v+2\eta+2v^2\eta+2\eta^3+6v\eta^2+2v^2\eta^3+v\eta^4)-v(1-\eta^2)\frac{d\eta}{d\tilde{\rho}}\right).
\end{align}
Therefore, we obtain the following expression for the curly bracket in \eqref{eq:1000}
\begin{align}
  \label{eq:1004}
  -\frac{G_0^2\psi_{t0}^4}{8\eta_0^2}(1-v_0^2)(1+v_0\eta_0)^2\left(1+\cp{\frac{d\eta}{d\tilde{\rho}}}-\eta_0^2\right).
\end{align}
Using the definition of $\mu$ (see \eqref{eq:186}) we see that \eqref{eq:1004} is equal to
\begin{align}
  \label{eq:1005}
  -\frac{G_0^2\psi_{t0}^4}{8\eta_0^2}(1-v_0^2)(1+v_0\eta_0)^2\mu_0.
\end{align}
Therefore,
\begin{align}
  \label{eq:1006}
  N&=v^2\dot{\alpha}_0^3\frac{G_0^2\psi_{t0}^4}{4\eta_0^2}(1-v_0^2)(1+v_0\eta_0)^2\mu_0\Delta_{m+1}z\notag\\
&\qquad +\mathcal{O}(v^4\nd{\Delta_{m+1}z})+\mathcal{O}(v^3\nd{\Delta_{m+1}\alpha_+})+\mathcal{O}(v^3\nd{\Delta_{m+1}f}).
\end{align}

We turn to the investigation of the denominator of \eqref{eq:968}. Let
\begin{align}
  \label{eq:1007}
  D\coloneqq B_{m+1}B_m.
\end{align}
Rewriting
\begin{align}
  \label{eq:1008}
  B_m=\overset{1}{B}_m+\overset{2}{B}_m,
\end{align}
where
\begin{align}
  \label{eq:1009}
  \overset{1}{B}_m&=T^{tt}(\alpha_{+,m},\beta_{+,m})-T^{tt}(\alpha_{+,m},\beta_{-,m}),\\
  \label{eq:1010}
  \overset{2}{B}_m&=T^{tt}(\alpha_{+,m},\beta_{-,m})-T^{tt}(\alpha_{-,m},\beta_{-,m})
\end{align}
and arguing as we did for the derivation of the asymptotic forms of $\overset{\,\,1}{A}_m$ and $\overset{\,\,2}{A}_m$, we arrive at
\begin{align}
  \label{eq:1011}
  B_m=\cp{\pp{T^{tt}}{\alpha}}\dot{\alpha}_0(1-y_m(v))v+\mathcal{O}(v^2).
\end{align}
Similarly with $m+1$ in the role of $m$. Therefore, we obtain
\begin{align}
  \label{eq:1012}
  D=4\cp{\pp{T^{tt}}{\alpha}}^2\dot{\alpha}_0^2v^2+\mathcal{O}(v^3),
\end{align}
where we used $y_m=-1+\mathcal{O}(v)$. Using the first of \eqref{eq:1001} we find
\begin{align}
  \label{eq:1013}
  D=v^2\dot{\alpha}_0^2\frac{G_0^2\psi_{t0}^4}{\eta_0^2}(1+v_0\eta_0)^4+\mathcal{O}(v^3).
\end{align}

Now, from $c_+=(v+\eta)/(1+v\eta)$ together with \eqref{eq:181}, \eqref{eq:186}, \eqref{eq:202} we have
\begin{align}
  \label{eq:1014}
  (1+v\eta)^2\pp{c_+}{\alpha}=\frac{1}{2}\mu(1-v^2).
\end{align}
Using $\kappa=\dot{\alpha}_0(\partial c_+/\partial \alpha)_0$ we obtain from \eqref{eq:1006}, \eqref{eq:1013}
\begin{align}
  \label{eq:1015}
  \Delta_{m+1}V(v)=\frac{\kappa}{2}\Delta_{m+1}z(v)+\mathcal{O}(v^2\nd{\Delta_{m+1}z})+\mathcal{O}(v\nd{\Delta_{m+1}\alpha_+})+\mathcal{O}(v\nd{\Delta_{m+1}f}).
\end{align}

\end{proof}

\subsubsection{Closing the Argument}

\begin{proposition}\label{prop_est_combination}
  For $\varepsilon$ small enough the sequence $(y_m,\hat{\beta}_{+,m},\hat{V}_m)$ converges in $B_Y\times B_{\delta_1}\times B_{\delta_2}$.
\end{proposition}

\begin{proof}

We first note
\begin{align}
  \label{eq:1016}
  \nd{\Delta_{m+1}\hat{\beta}_+}\leq \frac{C}{v^2}\nd{\Delta_{m+1}\beta_+}.
\end{align}
We now use the first estimate from lemma \ref{lemma_est_jc} together with the estimates for $|d\hat{f}/dv|$ and $\nd{\Delta_{m+1}\alpha_+}$ from lemma \ref{lemma_est_fxdbp}. We obtain
\begin{align}
  \label{eq:1017}
  \nd{\Delta_{m+1}\hat{\beta}_+}\leq C\left\{\nd{\Delta_{m}y}+\sup_{[0,v]}|\Delta_m\hat{V}|+v\nd{\Delta_{m}\hat{\beta}_+}\right\}.
\end{align}
For an estimate of $\nd{\Delta_{m+1}y}$ we look at the estimate of lemma \ref{lemma_est_ie}. Using also the first two estimates of lemma \ref{lemma_est_fxdbp} we obtain
\begin{align}
  \label{eq:1018}
  \nd{\Delta_{m+1}y}\leq \frac{1}{3}\nd{\Delta_my}+C\left\{v\nd{\Delta_m\hat{\beta}_+}+\sup_{[0,v]}|\Delta_m\hat{V}|\right\},
\end{align}
provided we choose $\varepsilon$ suitably small.
Now we note
\begin{align}
  \label{eq:1019}
  \Delta_{m+1}\hat{V}(v)=\frac{1}{v^2}\left(\Delta_{m+1}V(v)-v\frac{\kappa}{2}\Delta_{m+1}y(v)\right).
\end{align}
We now use the second estimate from lemma \ref{lemma_est_jc} together with \eqref{eq:1018}. We obtain
\begin{align}
  \label{eq:1020}
  \sup_{[0,v]}|\Delta_{m+1}\hat{V}|\leq C\left\{v\nd{\Delta_my}+v\sup_{[0,v]}|\Delta_{m}\hat{V}|+v^2\nd{\Delta_m\hat{\beta}_+}\right\}.
\end{align}

We rewrite the estimates \eqref{eq:1017}, \eqref{eq:1018}, \eqref{eq:1020} as
\begin{align}
  \label{eq:1021}
  \nd{\Delta_{m+1}y}&\leq \frac{1}{3}\nd{\Delta_my}+C_1\left\{v\nd{\Delta_m\hat{\beta}_+}+\sup_{[0,v]}|\Delta_m\hat{V}|\right\},\notag\\
\nd{\Delta_{m+1}\hat{\beta}_+}&\leq C_2\left\{\nd{\Delta_{m}y}+\sup_{[0,v]}|\Delta_m\hat{V}|+v\nd{\Delta_{m}\hat{\beta}_+}\right\},\notag\\
\sup_{[0,v]}|\Delta_{m+1}\hat{V}|&\leq C_3\left\{v\nd{\Delta_my}+v\sup_{[0,v]}|\Delta_{m}\hat{V}|+v^2\nd{\Delta_m\hat{\beta}_+}\right\}.
\end{align}
It follows that for $\varepsilon$ small enough the sequence $(y_m,\hat{\beta}_{+,m},\hat{V}_m)$ converges in $B_Y\times B_{\delta_1}\times B_{\delta_2}$.

\begin{comment}
One way to proceed is to (re)define the norms by
\begin{align}
  \label{eq:1022}
  \|\hat{V}\|\coloneqq 2\max\{C_1,C_2\}\sup_{[0,v]}|\hat{V}|,\qquad \|y\|\coloneqq 2C_2\nd{y}.
\end{align}
It follows that
\begin{align}
  \label{eq:1023}
  \nd{\Delta_{m+1}y}&\leq \frac{1}{3}\nd{\Delta_my}+C_1v\nd{\Delta_m\hat{\beta}_+}+\frac{1}{2}\sup_{[0,v]}|\Delta_m\hat{V}|,\notag\\
\nd{\Delta_{m+1}\hat{\beta}_+}&\leq \frac{1}{2}\nd{\Delta_{m}y}+\frac{1}{2}\sup_{[0,v]}|\Delta_m\hat{V}|+v\nd{\Delta_{m+1}\hat{\beta}_+},\notag\\
\sup_{[0,v]}|\Delta_{m+1}\hat{V}|&\leq C_3\left\{v\nd{\Delta_my}+v\sup_{[0,v]}|\Delta_{m}\hat{V}|+v^2\nd{\Delta_m\hat{\beta}_+}\right\}.
\end{align}
Therefore, we see that for $\varepsilon$ small enough we have convergence. Another way to proceed is that if $T$ is the iteration mapping there exists $n\in\NN$ such that for $\varepsilon$ small enough $T^n$ is a contraction. From this it follows by a corollary to the fixed point theorem that also $T$ is a contraction.
\begin{remark}
  In fact, $n$ has to satisfy the conditions
  \begin{align}
    \label{eq:1024}
    \left(\frac{1}{3}\right)^{n-1}C_1<1,\qquad \left(\frac{1}{3}\right)^{n}C_2<1,\qquad \left(\frac{1}{3}\right)^{n-1}C_1C_2<1.
  \end{align}
\end{remark}
\end{comment}

\end{proof}

The two propositions above show that the sequence $(y_m,\hat{\beta}_{+,m},\hat{V}_m)$ converges uniformly in $[0,\varepsilon]$ to $(y,\hat{\beta}_+,\hat{V})\in B_Y\times B_{\delta_1}\times B_{\delta_2}$.

We see that $(F_m)$ as given by \eqref{eq:857} converges to $0$ as $m\rightarrow \infty$ uniformly in $[0,\varepsilon]$. Therefore, in view of \eqref{eq:866}, \eqref{eq:867} also $\partial t_m/\partial v$, $\partial t_m/\partial u$ converge uniformly in $T_\varepsilon$. Let us denote the limit of $(t_m)$ by $t$. This, in view of \eqref{eq:817}, \eqref{eq:818}, implies the convergence of $(\mu_m,\nu_m)$ to $(\mu,\nu)$ uniformly in $T_\varepsilon$ and, in view of \eqref{eq:819}, also the convergence of $(\Xi_m)$ to $0$. Therefore, the pair of integral equations \eqref{eq:821}, \eqref{eq:849} are satisfied in the limit. It then follows that $t$ satisfies \eqref{eq:339}.

In view of \eqref{eq:774}, \eqref{eq:775} we have convergence of $(\alpha_m,\beta_m)$ to $(\alpha,\beta)$ uniformly in $T_\varepsilon$, which implies convergence of $(c_{\pm,m})$ to $c_\pm$ uniformly in $T_\varepsilon$. In view of the Hodograph system the partial derivatives of $(r_m)$ converge uniformly in $T_\varepsilon$ and the limit satisfies this system.

Now, the uniform convergence of $(\alpha_m,\beta_m)$, $(r_m)$ and the partial derivatives of $(t_m)$ imply the uniform convergence of $(A_m)$, $(B_m)$ to $A$, $B$. Therefore, the partial derivatives $\partial\alpha_m/\partial v$, $\partial\beta_m/\partial u$ converge to $\partial\alpha/\partial v$, $\partial\beta/\partial u$, uniformly in $T_\varepsilon$ and it holds $\partial\alpha/\partial v=A$, $\partial\beta/\partial u=B$, i.e.~also the other two equations of the characteristic system are satisfied in the limit. In view of \eqref{eq:806}, \eqref{eq:813} also the partial derivatives $(\partial\alpha_m/\partial u)$, $(\partial\beta_m/\partial v)$ converge to $\partial\alpha/\partial u$, $\partial\beta/\partial v$ uniformly in $T_\varepsilon$.

From \eqref{eq:749}, \eqref{eq:750} we see that $(\hat{f}_m)$, $(\hat{\delta}_m)$ converge to $\hat{f}$, $\hat{\delta}$ uniformly in $C^1[0,\varepsilon]$. Therefore, $z=vy$ satisfies the identification equation when $f\coloneqq v^2\hat{f}$, $g\coloneqq v^3\hat{\delta}+c_{+0}v^2\hat{f}$ are substituted. Also $V$, $\beta_+$ are given by the jump conditions when $\alpha_+$, $f$, $z$ are substituted. We have thus found a solution to the free boundary problem.

We have that $z(v)$ is given by $z(v)=vy(v)$, where $y\in C^1[0,\varepsilon]$, $y(0)=-1$ (see \eqref{eq:683}). $f(v)$ is given by $f(v)=v^2\hat{f}(v)$, with $\hat{f}\in C^1[0,\varepsilon]$, $\hat{f}(0)=\lambda/6\kappa^2$ (see \eqref{eq:681}). $\beta_+(v)$ is given by $\beta_+(v)=\beta_0+v^2\hat{\beta}_+(v)$ with $\hat{\beta}_+\in C^1[0,\varepsilon]$, $\hat{\beta}_+(0)=\cp{\partial\beta/\partial t}\lambda/6\kappa^2$ (see \eqref{eq:715}). $\alpha_+(v)=\alpha_i(v)+v^2\hat{\alpha}_+(v)$ with $\hat{\alpha}_+\in C^1[0,\varepsilon]$, $\hat{\alpha}_+(0)=\lambda\tilde{A}_0/6\kappa^2$ (see \eqref{eq:710}). From \eqref{eq:615} together with $\delta(v)=v^3\hat{\delta}(v)$, we have $g(v)=v^2\hat{g}(v)$ with $\hat{g}\in C^1[0,\varepsilon]$, $\hat{g}(0)=\lambda c_{+0}/6\kappa^2$.

We recall \eqref{eq:243} which is the singular boundary of the maximal development in acoustical coordinates $(t,w)$:
\begin{align}
  \label{eq:1025}
  t_\ast(w)=t_0+\frac{\lambda}{2\kappa^2}w^2+\Landau(w^3).
\end{align}
Therefore, using $w=z(v)=-v+\Landau(v^2)$, we have
\begin{align}
  \label{eq:1026}
  t_\ast(z(v))=t_0+\frac{\lambda}{2\kappa^2}v^2+\Landau(v^3).
\end{align}
Comparing this with (see \eqref{eq:681})
\begin{align}
  \label{eq:1027}
  f(v)+t_0=t_0+\frac{\lambda}{6\kappa^2}v^2+\Landau(v^3),
\end{align}
we see that for $\varepsilon$ sufficiently small the shock curve $\mathcal{K}$ lies in the past of the singular boundary of the maximal development $\mathcal{B}$.

From \eqref{eq:761} we have
\begin{align}
  \label{eq:1028}
  V(v)=c_{+0}+\Landau(v^2).
\end{align}
From \eqref{eq:760} we have that the characteristic speed of the outgoing characteristics along $\mathcal{K}$ in the state behind is
\begin{align}
  \label{eq:1029}
  \bar{c}_{+}(v)=c_{+0}+\kappa v+\Landau(v^2).
\end{align}
Now, let us denote by $\underline{c}_+$ the characteristic speed of the outgoing 
characteristics along $\mathcal{K}$ in the state ahead. We have
\begin{align}
  \label{eq:1030}
  \underline{c}_+(v)=c_+(\alpha^\ast(f(v),z(v)),\beta^\ast(f(v),z(v))).
\end{align}
Now,
\begin{align}
  \label{eq:1031}
  \frac{d\underline{c}_+}{dv}=\pp{c_+}{\alpha}\left(\pp{\alpha^\ast}{t}\frac{df}{dv}+\pp{\alpha^\ast}{w}\frac{dz}{dv}\right)+\pp{c_+}{\beta}\left(\pp{\beta^\ast}{t}\frac{df}{dv}+\pp{\beta^\ast}{w}\frac{dz}{dv}\right).
\end{align}
Therfore, using \eqref{eq:2027} and
\begin{align}
  \label{eq:1032}
  \frac{df}{dv}(v)=\Landau(v),\qquad \pp{\beta^\ast}{w}(f(v),z(v))=\Landau(v),
\end{align}
we find
\begin{align}
  \label{eq:1033}
  \frac{d\underline{c}_+}{dv}(v)=-\kappa+\Landau(v),
\end{align}
which implies
\begin{align}
  \label{eq:1034}
  \underline{c}_+(v)=c_{+0}-\kappa v+\Landau(v^2).
\end{align}
From \eqref{eq:1028}, \eqref{eq:1029} and \eqref{eq:1034} we obtain
\begin{align}
  \label{eq:1035}
  V(v)-\underline{c}_+(v)&=\kappa v+\Landau(v^2),\\
  \bar{c}_+(v)-V(v)&=\kappa v+\Landau(v^2).
\end{align}
These imply the determinism condition, provided that $\varepsilon$ is sufficiently small.

We have therefore proven the following existence theorem.

\begin{theorem}\label{existence_theorem}
Let initial data for $t$ and $\alpha$ be given along $\underline{C}$. Let $r_0>0$. Let the solution in the state ahead be given by $\alpha^\ast(t,w)$, $\beta^\ast(t,w)$, $r^\ast(t,w)$. Then, for $\varepsilon$ small enough, there exists a continuously differentiable solution $(t,r,\alpha,\beta)$ of the characteristic system in $T_\varepsilon$ such that
\begin{enumerate}
\item along $\underline{C}$ it attains the initial data and $r(0,0)=r_0$.
\item  $\alpha_+(v)\coloneqq \alpha(v,v)$, $\beta_+(v)\coloneqq \beta(v,v)$, $\alpha_-(v)\coloneqq \alpha^\ast(f(v),z(v))$, $\beta_-(v)\coloneqq \beta^\ast(f(v),z(v))$ satisfy the jump conditions
 \begin{align}
   \label{eq:1036}
  -\jump{T^{tt}(v)}V(v)+\jump{T^{tr}(v)}&=0,\\
  \label{eq:1037}
  -\jump{T^{rt}(v)}V(v)+\jump{T^{rr}(v)}&=0,
 \end{align}
where
\begin{align}
  \label{eq:1038}
  T^{\mu\nu}_+(v)=T^{\mu\nu}(\alpha_+(v),\beta_+(v)),\qquad T^{\mu\nu}_-(v)=T^{\mu\nu}(\alpha_-(v),\beta_-(v)),
\end{align}
$V(v)$ satisfies
\begin{align}
  \label{eq:1039}
  \frac{df}{dv}(v)V(v)=\frac{dg}{dv}(v)
\end{align}
and $z(v)$ satisfies the identification equation
\begin{align}
  \label{eq:1040}
  g(v)+r_0=r^\ast(f(v),z(v)),
\end{align}
where
\begin{align}
  \label{eq:1041}
  f(v)\coloneqq t(v,v),\qquad g(v)\coloneqq r(v,v)-r_0.
\end{align}
Furthermore, we have $\hat{V}\in C^0[0,\varepsilon]$, where $\hat{V}(v)$ is given by
\begin{align}
  \label{eq:11214}
  V(v)=c_{+0}+\frac{\kappa}{2}(1+y(v))v+v^2\hat{V}(v).
\end{align}
\item We have
\begin{align}
  \label{eq:1042}
  z(v)&=vy(v),& y&\in C^1[0,\varepsilon],& y(0)&=-1,\\
  f(v)&=v^2\hat{f}(v),& \hat{f}&\in C^1[0,\varepsilon],& \hat{f}(0)&=\frac{\lambda}{6\kappa^2},\\
  g(v)&=v^2\hat{g}(v),&\hat{g}&\in C^1[0,\varepsilon],& \hat{g}(0)&=\frac{\lambda c_{+0}}{6\kappa^2},\\
  \alpha_+(v)-\alpha_i(v)&=v^2\hat{\alpha}_+(v),&\hat{\alpha}_+&\in C^1[0,\varepsilon],&\hat{\alpha}_+(0)&=\frac{\lambda\tilde{A}_0}{6\kappa^2},\\
  \beta_+(v)-\beta_0&=v^2\hat{\beta}(v),&\hat{\beta}_+&\in C^1[0,\varepsilon],&\hat{\beta}_+(0)&=\frac{\lambda}{6\kappa^2}\cp{\pp{\beta}{t}}.
\end{align}
\item The curve $\mathcal{K}$ given in rectangular coordinates by
  \begin{align}
    \label{eq:1043}
  v\mapsto (f(v),g(v)+r_0)
  \end{align}
lies in the past of the singular boundary of the maximal development $\mathcal{B}$ and satisfies the determinism condition, i.e.~it is supersonic relative to the state ahead and subsonic relative to the state behind.
\end{enumerate}

\end{theorem}

%%% Local Variables: 
%%% mode: latex
%%% TeX-master: "./master"
%%% End: 

\section{Uniqueness}

\subsection{Asymptotic Form}
In the following we prove that any solution of the characteristic system satisfying the smoothness conditions from the existence theorem possesses the same leading order behavior as the solution given by the existence theorem.

\begin{proposition}\label{prop_ass_form}
Let $(t,r,\alpha,\beta)$ be a continuously differentiable solution of the free boundary problem and let $z(v)$ be the corresponding solution of the identification equation. Let $z(v)$ and
\begin{align}
  \label{eq:1044}
  f(v)\coloneqq t(v,v),\qquad g(v)\coloneqq r(v,v)-r_0,\qquad \alpha_+(v)\coloneqq \alpha(v,v),\qquad \beta_+(v)\coloneqq \beta(v,v)
\end{align}
be given by
\begin{align}
  \label{eq:1045}
  z(v)&=vy(v),\\\label{eq:1046}
  f(v)&=v^2\hat{f}(v),\\\label{eq:1047}
  g(v)&=v^2\hat{g}(v),\\\label{eq:1048}
  \alpha_+(v)-\alpha_i(v)&=v^2\hat{\alpha}_+(v),\\
  \beta_+(v)-\beta_0&=v^2\hat{\beta}_+(v),\label{eq:1049}
\end{align}
with
\begin{align}
  \label{eq:1050}
  y\in C^0[0,\varepsilon],\qquad\hat{f},\hat{g},\hat{\alpha}_+,\hat{\beta}_+\in C^1[0,\varepsilon].
\end{align}
Then it follows that
\begin{align}
  \label{eq:1051}
  y(0)=-1,\qquad \hat{f}(0)=\frac{\lambda}{6\kappa^2},\qquad \hat{g}(0)=\frac{\lambda c_{+0}}{6\kappa^2},\qquad \hat{\alpha}_+(0)=\frac{\lambda \tilde{A}_0}{6\kappa^2},\qquad \hat{\beta}_+(0)=\frac{\lambda}{6\kappa^2}\cp{\pp{\beta^\ast}{t}}
\end{align}
and
\begin{align}
  \label{eq:1052}
  \frac{d\hat{f}}{dv}(0)=\frac{1}{c_{+0}}\frac{d\hat{g}}{dv}(0).
\end{align}
\end{proposition}

\begin{proof}
We first note that the characteristic system together with the solution being continuously differentiable implies
\begin{align}
  \label{eq:1053}
  \pppp{t}{u}{v}=\frac{1}{c_+-c_-}\left(\pp{c_-}{v}\pp{t}{u}-\pp{c_+}{u}\pp{t}{v}\right)\in C^0(T_\varepsilon).
\end{align}
We recall the initial data for $t$:
\begin{align}
  \label{eq:1054}
  t(u,0)=h(u)=u^3\hat{h}(u),\qquad \hat{h}\in C^1[0,\varepsilon],\qquad \hat{h}(0)=\frac{\lambda}{6\kappa(c_{+0}-c_{-0})}.
\end{align}
Taking into account
\begin{align}
  \label{eq:1055}
  \pp{t}{u}(v,v)+\pp{t}{v}(v,v)&=\frac{df}{dv}(v)\notag\\
&=2v\hat{f}(v)+v^2\frac{d\hat{f}}{dv}(v),
\end{align}
we deduce
\begin{align}
  \label{eq:1056}
  \cp{\pp{t}{u}}=\cp{\pp{t}{v}}=\cp{\pppp{t}{u}{v}}=0.
\end{align}

Now,
\begin{align}
  \label{eq:1057}
  \frac{dg}{dv}(v)&=\pp{r}{u}(v,v)+\pp{r}{v}(v,v)\notag\\
&=\bar{c}_-(v)\pp{t}{u}(v,v)+\bar{c}_+(v)\pp{t}{v}(v,v)\notag\\
&=\bar{c}_+(v)\frac{df}{dv}(v)+\pp{t}{u}(v,v)\left(\bar{c}_-(v)-\bar{c}_+(v)\right)\notag\\
&=\bar{c}_+(v)\frac{df}{dv}(v)+\left\{\pp{t}{u}(v,0)+\int_0^v\pppp{t}{u}{v}(v,v')dv'\right\}\left(\bar{c}_-(v)-\bar{c}_+(v)\right).
\end{align}
Substituting
\begin{align}
  \label{eq:1058}
  \frac{dg}{dv}(v)&=2v\hat{g}(v)+v^2\frac{d\hat{g}}{dv}(v),\\
\pp{t}{u}(v,0)=\frac{dh}{dv}(v)&=3v^2\hat{h}(v)+v^3\frac{d\hat{h}}{dv}(v),\label{eq:1059}
\end{align}
together with \eqref{eq:1055} yields, after dividing by $v$ and taking the limit $v\rightarrow 0$,
\begin{align}
  \label{eq:1060}
  \hat{g}(0)=c_{+0}\hat{f}(0).
\end{align}

Since by \eqref{eq:1049}
\begin{align}
  \label{eq:1212}
  \frac{d\beta_+}{dv}(0)=0
\end{align}
and the second of \eqref{eq:160}, namely
\begin{align}
  \label{eq:1214}
  \pp{\beta}{u}=\pp{t}{u}\tilde{B}(\alpha,\beta,r),
\end{align}
gives, by the first of \eqref{eq:1056},
\begin{align}
  \label{eq:1216}
  \cp{\pp{\beta}{u}}=0,
\end{align}
it follows
\begin{align}
  \label{eq:1062}
  \cp{\pp{\beta}{v}}=0.
\end{align}

Now, $\partial t/\partial v$ and $\partial\beta/\partial v$ satisfy along $\underline{C}$ a system of the form (see page \pageref{eq:309})
\begin{align}
  \label{eq:1063}
    \frac{d}{du}
\left(\begin{array}{c}
      \partial\beta/\partial v\\
      \partial t/\partial v
    \end{array}\right)
=
\left(\begin{array}{cc}
  a_{11} & a_{12}\\
  a_{21} & a_{22}
\end{array}\right)
\left(\begin{array}{c}
  \partial\beta/\partial v\\
  \partial t/\partial v
\end{array}\right)\quad \textrm{for} \quad v=0,
\end{align}
which, together with \eqref{eq:1056}, \eqref{eq:1062} implies
\begin{align}
  \label{eq:1064}
  \pp{t}{v}(u,0)=0, \qquad \pp{\beta}{v}(u,0)=0.
\end{align}
Hence
\begin{align}
  \label{eq:1065}
  \pp{t}{v}(u,v)=\landau_v(1),\qquad \pp{\beta}{v}(u,v)=\landau_v(1),
\end{align}
the indices on the Landau symbols representing the variable with respect to which the limit is taken, the limit being uniform in the other variable.

Making use of the equations for $\alpha$ and $\beta$ from the characteristic system (see \eqref{eq:160}) we obtain
\begin{align}
  \label{eq:1066}
  \pp{c_+}{u}&=\pp{c_+}{\alpha}\pp{\alpha}{u}+\pp{c_+}{\beta}\pp{t}{u}\tilde{B},\\
  \pp{c_+}{v}&=\pp{c_+}{\alpha}\pp{t}{v}\tilde{A}+\pp{c_+}{\beta}\pp{\beta}{v}.
\end{align}
Therefore,
\begin{align}
  \label{eq:1067}
  \pp{c_+}{u}(u,v)&=\kappa+\landau_u(1),\\
  \pp{c_+}{v}(u,v)&=\landau_v(1).
\end{align}
where for the first we used $(\partial c_+/\partial\alpha)_0=\kappa /\dot{\alpha}_0$ (see \eqref{eq:2027}). Hence
\begin{align}
  \label{eq:1068}
  \bar{c}_+(v)=c_{+0}+\kappa v+{\scriptstyle\mathcal{O}}(v).
\end{align}

We turn to the integral in \eqref{eq:1057}. Making again use of the characteristic system we obtain
\begin{align}
  \label{eq:1069}
  \pp{c_-}{v}=\pp{c_-}{\alpha}\pp{t}{v}\tilde{A}+\pp{c_-}{\beta}\pp{\beta}{v}.
\end{align}
Hence
\begin{align}
  \label{eq:1070}
  \pp{c_-}{v}(u,v)=\landau_v(1),
\end{align}
which, in view of \eqref{eq:1053}, implies
\begin{align}
  \label{eq:1071}
  \pppp{t}{u}{v}(u,v)=\landau_v(1).
\end{align}

We have
\begin{align}
  \label{eq:1072}
  \pp{t}{u}(v,v')&=\pp{t}{u}(v,0)+\int_0^{v'}\pppp{t}{u}{v}(v,v'')dv''\notag\\
&=\landau_{v}(v),
\end{align}
where we made use of \eqref{eq:1059}. Together with (see \eqref{eq:1070})
\begin{align}
  \label{eq:1073}
  \left(\frac{1}{c_+-c_-}\pp{c_-}{v}\right)(v,v')=\landau_{v'}(1),
\end{align}
we find
\begin{align}
  \label{eq:1074}
  \int_0^v\left(\frac{1}{c_+-c_-}\pp{c_-}{v}\pp{t}{u}\right)(v,v')dv'={\scriptstyle\mathcal{O}}(v^2).
\end{align}

We have
\begin{align}
  \label{eq:1075}
  \pp{t}{v}(v,v')&=\pp{t}{v}(v',v')+\int_{v'}^{v}\pppp{t}{u}{v}(u',v')du'\notag\\
&=\frac{df}{dv}(v')-\pp{t}{u}(v',v')+\int_{v'}^{v}\pppp{t}{u}{v}(u',v')du'\notag\\
&=2v'\hat{f}(0)+v\landau_{v'}(1),
\end{align}
where for the second term in the second line we substituted \eqref{eq:1072}, setting $v=v'$. Together with (see \eqref{eq:1067})
\begin{align}
  \label{eq:1076}
  \left(\frac{1}{c_+-c_-}\pp{c_+}{u}\right)(v,v')=\frac{\kappa}{c_{+0}-c_{-0}}+{\scriptstyle\mathcal{O}}_{v}(1),
\end{align}
we find
\begin{align}
  \label{eq:1077}
  \int_0^v\left(\frac{1}{c_+-c_-}\pp{c_+}{u}\pp{t}{v}\right)(v,v')dv'=\frac{\kappa \hat{f}(0)}{c_{+0}-c_{-0}}v^2+{\scriptstyle\mathcal{O}}(v^2).
\end{align}

Now, rewriting the integrand in \eqref{eq:1057} using \eqref{eq:1053} and then substituting \eqref{eq:1074} and \eqref{eq:1077}, we obtain
\begin{align}
  \label{eq:1078}
  \int_0^v\pppp{t}{u}{v}(v,v')dv'=-\frac{\kappa \hat{f}(0)}{c_{+0}-c_{-0}}v^2+{\scriptstyle\mathcal{O}}(v^2).
\end{align}
Substituting now \eqref{eq:1055}, \eqref{eq:1058}, \eqref{eq:1059}, \eqref{eq:1068}, \eqref{eq:1078}  into \eqref{eq:1057}, noting that
\begin{alignat}{3}
  \label{eq:1079}
  \hat{f}(v)&=\hat{f}(0)+\frac{d\hat{f}}{dv}(0)v+{\scriptstyle\mathcal{O}}(v),&\qquad\hat{g}(v)&=\hat{g}(0)+\frac{d\hat{g}}{dv}(0)v+{\scriptstyle\mathcal{O}}(v),\\
  \frac{d\hat{f}}{dv}(v)&=\frac{d\hat{f}}{dv}(0)+\landau(v),&\qquad \frac{d\hat{g}}{dv}(v)&=\frac{d\hat{g}}{dv}(0)+\landau(v),
\end{alignat}
and that by \eqref{eq:1060} the terms linear in $v$ cancel, dividing by $v^2$ and taking the limit as $v\rightarrow 0$, we find
\begin{align}
  \label{eq:1080}
  \frac{d\hat{g}}{dv}(0)=c_{+0}\frac{d\hat{f}}{dv}(0)+\kappa \hat{f}(0)-\frac{\lambda}{6\kappa}.
\end{align}

We now consider the identification equation
\begin{align}
  \label{eq:1081}
  g(v)+r_0=r^\ast(f(v),z(v)).
\end{align}
In view of \eqref{eq:1060}, \eqref{eq:1080} the left hand side is
\begin{align}
  \label{eq:1082}
  r_0+c_{+0}\hat{f}(0)v^2+\left(c_{+0}\frac{d\hat{f}}{dv}(0)+\kappa \hat{f}(0)-\frac{\lambda}{6\kappa}\right)v^3+\landau(v^3).
\end{align}
For the right hand side of \eqref{eq:1081} we expand $r^\ast(t,w)$ up to third order. Substituting $t=f(v)=v^2\hat{f}(v)$, $w=z(v)=vy(v)$ we obtain
\begin{align}
  \label{eq:1083}
  r_0+c_{+0}\hat{f}(0)v^2+\left(c_{+0}\frac{d\hat{f}}{dv}(0)+y(0)\kappa\hat{f}(0)-\frac{\lambda y(0)^3}{6\kappa}\right)v^3+\landau(v^3).
\end{align}
Therefore, setting \eqref{eq:1082} equal to \eqref{eq:1083}, dividing by $v^3$ and taking the limit $v\rightarrow 0$ we obtain
\begin{align}
  \label{eq:1084}
  \kappa\hat{f}(0)-\frac{\lambda}{6\kappa}=y(0)\kappa\hat{f}(0)-\frac{\lambda y(0)^3}{6\kappa}.
\end{align}
Defining now
\begin{align}
  \label{eq:1085}
  m\coloneqq 3\kappa-\frac{\lambda}{2\hat{f}(0)\kappa},\qquad p\coloneqq -y(0),
\end{align}
(we recall that for a physical solution we need $p>0$), this becomes
\begin{align}
  \label{eq:1086}
  m(1+p^3)+3\kappa p(1-p^2)=0.
\end{align}

We now turn to the jump conditions. Using (cf.~\eqref{eq:1079})
\begin{align}
  \label{eq:1087}
  \frac{df}{dv}(v)&=2v\hat{f}(v)+v^2\frac{d\hat{f}}{dv}(v)\notag\\
&=2v\hat{f}(0)+3v^2\frac{d\hat{f}}{dv}(0)+\landau(v^2)
\end{align}
and the analogous expansion for $dg/dv(v)$ in
\begin{align}
  \label{eq:1088}
  V(v)=\frac{dr}{dv}(v)=\frac{\displaystyle{\frac{dg}{dv}(v)}}{\displaystyle{\frac{df}{dv}(v)}},
\end{align}
we obtain
\begin{align}
  \label{eq:1089}
  V(v)=c_{+0}+\frac{3}{2\hat{f}(0)}\left(\frac{d\hat{g}}{dv}(0)-c_{+0}\frac{d\hat{f}}{dv}(0)\right)v+\landau(v).
\end{align}
Hence, in view of \eqref{eq:1080} and the first of \eqref{eq:1085},
\begin{align}
  \label{eq:1090}
  V(v)=c_{+0}+\frac{m}{2}v+\landau(v).
\end{align}

We recall the jump conditions:
\begin{align}
  \label{eq:1091}
  \jump{T^{tt}(v)}V(v)&=\jump{T^{tr}(v)},\\
  \jump{T^{tr}(v)}V(v)&=\jump{T^{rr}(v)},\label{eq:1092}
\end{align}
where
\begin{align}
  \label{eq:1093}
  \jump{T^{\mu\nu}(v)}=T_+^{\mu\nu}(v)-T_-^{\mu\nu}(v)
\end{align}
and
\begin{align}
  \label{eq:1094}
  T_+^{\mu\nu}(v)&=T^{\mu\nu}(\alpha_+(v),\beta_+(v)),\\
  T_-^{\mu\nu}(v)&=T^{\mu\nu}(\alpha_-(v),\beta_-(v)),\label{eq:1095}
\end{align}
where
\begin{align}
  \label{eq:1096}
  \alpha_-(v)=\alpha^\ast(f(v),z(v)),\qquad \beta_-(v)=\beta^\ast(f(v),z(v))
\end{align}
the functions $\alpha^\ast(t,w)$, $\beta^\ast(t,w)$ in \eqref{eq:1096} being given by the solution in the maximal development. Using the initial condition for $\alpha$ as given by \eqref{eq:290}, we obtain from \eqref{eq:1048}, \eqref{eq:1049}
\begin{align}
  \label{eq:1097}
  \alpha_+(v)-\alpha_0&=\dot{\alpha}_0v+\left(\frac{\ddot{\alpha}_0}{2}+\hat{\alpha}_+(0)\right)v^2+\landau(v^2),\\
\beta_+(v)-\beta_0&=\hat{\beta}_+(0)v^2+\landau(v^2).\label{eq:1098}
\end{align}
It would actually suffice to assume $\hat{\alpha},\hat{\beta}\in C^0$. But since $\alpha_+$ and $\beta_+$ correspond to the solution of the fixed boundary problem, the assumption for $\hat{\alpha}$ and $\hat{\beta}$ to be continuously differentiable is consistent.

Expanding now $T^{\mu\nu}(\alpha,\beta)$ up to second order and substituting \eqref{eq:1097}, \eqref{eq:1098} we obtain
\begin{align}
  \label{eq:1099}
  T_+^{\mu\nu}(v)&=T_0^{\mu\nu}+\cp{\pp{T^{\mu\nu}}{\alpha}}\dot{\alpha}_0v+\Bigg\{\cp{\pp{T^{\mu\nu}}{\alpha}}\left(\frac{\ddot{\alpha}_0}{2}+\hat{\alpha}_+(0)\right)\notag\\
&\hspace{50mm}+\cp{\pp{T^{\mu\nu}}{\beta}}\hat{\beta}_+(0)+\cp{\ppp{T^{\mu\nu}}{\alpha}}\frac{\dot{\alpha}_0^2}{2}\Bigg\}v^2+\landau(v^2).
\end{align}
Expanding $\alpha^\ast(t,w)$, $\beta^\ast(t,w)$ up to second order and substituting $t=v^2\hat{f}(v)$, $w=vy(v)$ we obtain
\begin{align}
  \label{eq:1100}
  \alpha_-(v)-\alpha_0&=\dot{\alpha}_0y(v)v+\left\{\cp{\pp{\alpha^\ast}{t}}\hat{f}(0)+\frac{\ddot{\alpha}_0y(0)}{2}^2\right\}v^2+\landau(v^2),\\
\beta_+(v)-\beta_0&=\cp{\pp{\beta^\ast}{t}}\hat{f}(0)v^2+\landau(v^2),\label{eq:1101}
\end{align}
where for the second one we made use of \eqref{eq:257}.
Expanding now $T^{\mu\nu}(\alpha,\beta)$ up to second order and substituting \eqref{eq:1100}, \eqref{eq:1101} we obtain
\begin{align}
  \label{eq:1102}
  T_-^{\mu\nu}(v)&=T_0^{\mu\nu}+\cp{\pp{T^{\mu\nu}}{\alpha}}\dot{\alpha}_0y(v)v+\Bigg\{\cp{\pp{T^{\mu\nu}}{\alpha}}\left(\cp{\pp{\alpha^\ast}{t}}\hat{f}(0)+\frac{\ddot{\alpha}_0y(0)^2}{2}\right)\notag\\
&\hspace{52mm}+\cp{\pp{T^{\mu\nu}}{\beta}}\cp{\pp{\beta^\ast}{t}}\hat{f}(0)+\cp{\ppp{T^{\mu\nu}}{\alpha}}\frac{\dot{\alpha}_0^2y(0)^2}{2}\Bigg\}v^2+\landau(v^2).
\end{align}
Using the definitions
\begin{align}
  \label{eq:1103}
  m_-\coloneqq 2\cp{\pp{\alpha^\ast}{t}}\hat{f}(0),\qquad n_-\coloneqq 2\cp{\pp{\beta^\ast}{t}}\hat{f}(0),
\end{align}
this becomes
\begin{align}
  \label{eq:1104}
  T_-^{\mu\nu}(v)&=T_0^{\mu\nu}+\cp{\pp{T^{\mu\nu}}{\alpha}}\dot{\alpha}_0y(v)v+\frac{1}{2}\Bigg\{\cp{\pp{T^{\mu\nu}}{\alpha}}\left(m_-+\ddot{\alpha}_0y(0)^2\right)\notag\\
&\hspace{55mm}+\cp{\pp{T^{\mu\nu}}{\beta}}n_-+\cp{\ppp{T^{\mu\nu}}{\alpha}}\dot{\alpha}_0^2y(0)^2\Bigg\}v^2+\landau(v^2).
\end{align}
Therefore, from \eqref{eq:1099}, \eqref{eq:1104},
\begin{align}
  \label{eq:1105}
  \jump{T^{\mu\nu}(v)}&=\cp{\pp{T^{\mu\nu}}{\alpha}}\dot{\alpha}_0(1-y(v))v\notag\\
&\qquad+\frac{1}{2}\Bigg\{\cp{\pp{T^{\mu\nu}}{\alpha}}\left(\ddot{\alpha}_0(1-y(0)^2)+2\hat{\alpha}_+(0)-m_-\right)\notag\\
&\hspace{30mm}+\cp{\pp{T^{\mu\nu}}{\beta}}\left(2\hat{\beta}_+(0)-n_-\right)+\cp{\ppp{T^{\mu\nu}}{\alpha}}\dot{\alpha}_0^2(1-y(0)^2)\Bigg\}v^2+\landau(v^2).
\end{align}

We now substitute \eqref{eq:1105} together with \eqref{eq:1090} into \eqref{eq:1091}, \eqref{eq:1092}. Making use of (see \eqref{eq:184})
\begin{align}
  \label{eq:1106}
  \pp{T^{tr}}{\alpha}=c_+\pp{T^{tt}}{\alpha}=\frac{1}{c_+}\pp{T^{rr}}{\alpha},
\end{align}
dividing by $v^2$ and taking the limit as $v\rightarrow 0$ we arrive at
\begin{align}
  \label{eq:1107}
  &m\dot{\alpha}_0\cp{\pp{T^{tt}}{\alpha}}(1-y(0))+\dot{\alpha}_0^2\left\{c_{+0}\cp{\ppp{T^{tt}}{\alpha}}-\cp{\ppp{T^{tr}}{\alpha}}\right\}(1-y(0)^2)\notag\\
&\hspace{70mm}+\left\{c_{+0}\cp{\pp{T^{tt}}{\beta}}-\cp{\pp{T^{tr}}{\beta}}\right\}l_+=0,\\
&m\dot{\alpha}_0\cp{\pp{T^{tr}}{\alpha}}(1-y(0))+\dot{\alpha}_0^2\left\{c_{+0}\cp{\ppp{T^{tr}}{\alpha}}-\cp{\ppp{T^{rr}}{\alpha}}\right\}(1-y(0)^2)\notag\\
&\hspace{70mm}+\left\{c_{+0}\cp{\pp{T^{tr}}{\beta}}-\cp{\pp{T^{rr}}{\beta}}\right\}l_+=0,\label{eq:1108}
\end{align}
where we used the definition
\begin{align}
  \label{eq:1109}
  l_+\coloneqq 2\hat{\beta}_+(0)-n_-.
\end{align}
Let us define
\begin{alignat}{3}
  \label{eq:1110}
  c_{01}&\coloneqq c_{+0}\cp{\pp{T^{tt}}{\beta}}-\cp{\pp{T^{tr}}{\beta}},&\qquad c_{02}&\coloneqq c_{+0}\cp{\pp{T^{tr}}{\beta}}-\cp{\pp{T^{rr}}{\beta}},\\
  \label{eq:1111}
  c_{11}&\coloneqq \dot{\alpha}_0^2\left\{c_{+0}\cp{\ppp{T^{tt}}{\alpha}}-\cp{\ppp{T^{tr}}{\alpha}}\right\},&\qquad c_{12}&\coloneqq \dot{\alpha}_0^2\left\{c_{+0}\cp{\ppp{T^{tr}}{\alpha}}-\cp{\ppp{T^{rr}}{\alpha}}\right\},\\
  \label{eq:1112}
  c_{21}&\coloneqq \dot{\alpha}_0\cp{\pp{T^{tt}}{\alpha}},&\qquad c_{22}&\coloneqq \dot{\alpha}_0\cp{\pp{T^{tr}}{\alpha}}.
\end{alignat}
Using these definitions together with $p=-y(0)$ in \eqref{eq:1107}, \eqref{eq:1108} we obtain
\begin{align}
  \label{eq:1113}
  c_{01}l_++c_{11}(1-p^2)+c_{21}m(1+p)&=0,\\
  c_{02}l_++c_{12}(1-p^2)+c_{22}m(1+p)&=0.\label{eq:1114}
\end{align}

We now solve the system of equations \eqref{eq:1086}, \eqref{eq:1113}, \eqref{eq:1114}. Solving \eqref{eq:1086} for $m$ yields
\begin{align}
  \label{eq:1115}
  m=-3\kappa p\frac{1-p^2}{1+p^3}.
\end{align}
Substituting this in \eqref{eq:1113}, \eqref{eq:1114} gives
\begin{align}
  \label{eq:1116}
  c_{01}l_+&=(1-p^2)\left(-c_{11}+3\kappa c_{21}\frac{p(1+p)}{1+p^3}\right),\\
  \label{eq:1117}
  c_{02}l_+&=(1-p^2)\left(-c_{12}+3\kappa c_{22}\frac{p(1+p)}{1+p^3}\right).
\end{align}
Multiplying \eqref{eq:1116} by $c_{02}$ and \eqref{eq:1117} by $c_{01}$ and then subtracting the resulting equations from each other yields
\begin{align}
  \label{eq:1118}
  0=(1-p^2)\left(-d_1+3\kappa d_2\frac{p(1+p)}{1+p^3}\right),
\end{align}
where we used the definitions
\begin{align}
  \label{eq:1119}
  d_1\coloneqq c_{02}c_{11}-c_{01}c_{12},\qquad d_2\coloneqq c_{02}c_{21}-c_{01}c_{22}.
\end{align}
If $d_1$, $d_2$ have opposite sign (recall that $\kappa>0$) then $p=1$ is the only root of \eqref{eq:1118} (recall the requirement $p>0$ for a physical solution). We look at $d_1/d_2$. From \eqref{eq:1106} we deduce
\begin{align}
  \label{eq:1120}
  c_+\ppp{T^{tt}}{\alpha}-\ppp{T^{tr}}{\alpha}=-\pp{c_+}{\alpha}\pp{T^{tt}}{\alpha},\qquad c_+\ppp{T^{tr}}{\alpha}-\ppp{T^{rr}}{\alpha}=-\pp{c_+}{\alpha}\pp{T^{tr}}{\alpha}.
\end{align}
Therefore
\begin{align}
  \label{eq:1121}
  c_{11}=-\dot{\alpha}_0^2\cp{\pp{c_+}{\alpha}}\left(\pp{T^{tt}}{\alpha}\right)_0=-\kappa c_{21},\qquad c_{12}=-\dot{\alpha}_0^2\cp{\pp{c_+}{\alpha}}\left(\pp{T^{rr}}{\alpha}\right)_0=-\kappa c_{22},
\end{align}
where we made use of $(\partial c_+/\partial\alpha)_0=\kappa/\dot{\alpha}_0$ (see \eqref{eq:2027}). Hence,
\begin{align}
  \label{eq:1122}
  \frac{d_1}{d_2}=\frac{c_{02}c_{11}-c_{01}c_{12}}{c_{02}c_{21}-c_{01}c_{22}}=-\kappa<0
\end{align}
and we conclude that $p=1$ is the only root of \eqref{eq:1118}. From \eqref{eq:1115} we deduce $m=0$. \eqref{eq:1113} then yields $l_+=0$. Therefore, from \eqref{eq:1060}, \eqref{eq:1085} and the second of \eqref{eq:1103} together with \eqref{eq:1109} we obtain
\begin{align}
  \label{eq:1123}
  y(0)=-1,\qquad \hat{f}(0)=\frac{\lambda}{6\kappa^2},\qquad \hat{g}(0)=\frac{\lambda c_{+0}}{6\kappa^2},\qquad \hat{\beta}_+(0)=\frac{\lambda}{6\kappa^2}\cp{\pp{\beta^\ast}{t}}.
\end{align}
From \eqref{eq:1080} we obtain
\begin{align}
  \label{eq:1124}
  \frac{d\hat{f}}{dv}(0)=\frac{1}{c_{+0}}\frac{d\hat{g}}{dv}(0).
\end{align}

We now turn to $\hat{\alpha}_+(0)$. We recall
\begin{align}
  \label{eq:1125}
  \alpha_+(v)=\alpha(v,v)=\alpha_i(v)+\int_0^v\left(\pp{t}{v}\tilde{A}(\alpha,\beta,r)\right)(v,v')dv'.
\end{align}
Therefore, (this is \eqref{eq:700})
\begin{align}
  \label{eq:1126}
  \frac{d\alpha_+}{dv}(v)&=\frac{d\alpha_i}{dv}(v)+\left(\pp{t}{v}\tilde{A}(\alpha,\beta,r)\right)(v,v)\notag\\
&\qquad+\int_0^v\Bigg\{\left(\frac{\partial^2t}{\partial u\partial v}\tilde{A}(\alpha,\beta,r)\right)\notag\\
&\hspace{20mm}+\pp{t}{v}\left(\pp{\tilde{A}}{\alpha}(\alpha,\beta,r)\pp{\alpha}{u}+\pp{\tilde{A}}{\beta}(\alpha,\beta,r)\pp{\beta}{u}+\pp{\tilde{A}}{r}(\alpha,\beta,r)c_-\pp{t}{u}\right)\Bigg\}(v,v')dv'.
\end{align}
Using \eqref{eq:1075}, \eqref{eq:1078} we obtain
\begin{align}
  \label{eq:1127}
  \frac{d\alpha_+}{dv}(v)=\frac{d\alpha_i}{dv}(v)+2\hat{f}(0)\tilde{A}_0v+\landau(v),
\end{align}
which implies (using the second of \eqref{eq:1123})
\begin{align}
  \label{eq:1128}
  \alpha_+(v)-\alpha_i(v)=\frac{\lambda\tilde{A}_0}{6\kappa^2}v^2+\landau(v^2).
\end{align}
Therefore,
\begin{align}
  \label{eq:1129}
  \hat{\alpha}_+(0)=\frac{\lambda\tilde{A}_0}{6\kappa^2}.
\end{align}
This concludes the proof of the proposition.
\end{proof}

\subsection{Uniqueness}
\begin{theorem}\label{uniqueness_theorem}
  Let $(t',r',\alpha',\beta')$, $(t'',r'',\alpha'',\beta'')$ be two continuously differentiable solutions to the free boundary problem and let $z'(v)$, $z''(v)$ be the corresponding solutions of the identification equation. Let $z'(v)$, $z''(v)$ and
\begin{align}
  \label{eq:1130}
  f'(v)&\coloneqq t'(v,v),& g'(v)&\coloneqq r'(v,v)-r_0,& \alpha_+'(v)&\coloneqq \alpha'(v,v),& \beta_+'(v)&\coloneqq \beta'(v,v),\\
f''(v)&\coloneqq t''(v,v),& g''(v)&\coloneqq r''(v,v)-r_0,& \alpha_+''(v)&\coloneqq \alpha''(v,v),& \beta_+''(v)&\coloneqq \beta''(v,v),
\end{align}
be given by
\begin{alignat}{3}
  \label{eq:1131}
  z'(v)&=vy'(v),&\qquad  z''(v)&=vy''(v),\\
  f'(v)&=v^2\hat{f}'(v),&\qquad  f''(v)&=v^2\hat{f}''(v),\label{eq:1132}\\
  g'(v)&=v^2\hat{g}'(v),&\qquad  g''(v)&=v^2\hat{g}''(v),\label{eq:1133}\\
  \alpha_+'(v)-\alpha_i(v)&=v^2\hat{\alpha}_+'(v),&\qquad  \alpha_+''(v)-\alpha_i(v)&=v^2\hat{\alpha}_+''(v),\label{eq:1134}\\
  \beta_+'(v)-\beta_0&=v^2\hat{\beta}_+'(v),&\qquad  \beta_+''(v)-\beta_0&=v^2\hat{\beta}_+''(v),\label{eq:1135}
\end{alignat}
with
\begin{align}
  \label{eq:1136}
  y',y'',\hat{f}',\hat{f}'',\hat{g}',\hat{g}'',\hat{\alpha}_+',\hat{\alpha}_+'',\hat{\beta}_+',\hat{\beta}_+''\in C^1[0,\varepsilon].
\end{align}
Let $\hat{V}'(v)$, $\hat{V}''(v)$ be given by
\begin{alignat}{4}
  \label{eq:1238}
  V'(v)&=c_{+0}+\frac{\kappa}{2}(1+y'(v))v+v^2\hat{V}'(v),&\qquad\textrm{where}\qquad V'(v)&=\frac{\displaystyle{\frac{dg'}{dv}(v)}}{\displaystyle{\frac{df'}{dv}(v)}},\\
  \label{eq:1278}  V''(v)&=c_{+0}+\frac{\kappa}{2}(1+y''(v))v+v^2\hat{V}''(v),&\qquad\textrm{where}\qquad V''(v)&=\frac{\displaystyle{\frac{dg''}{dv}(v)}}{\displaystyle{\frac{df''}{dv}(v)}},
\end{alignat}
with
\begin{align}
  \label{eq:1242}
  \hat{V}',\hat{V}''\in C^0[0,\varepsilon].
\end{align}
Then it follows that for $\varepsilon$ sufficiently small, the two solutions coincide.
\end{theorem}

\begin{proof}
In the following, whenever there is no prime or no double prime on a function it is meant that the statement holds for both the primed as well as for the double primed function. Let us make the following definition for any function $f$
\begin{align}
  \label{eq:1137}
  \Delta f\coloneqq f''-f'.
\end{align}

We see that the assumptions from proposition \ref{prop_ass_form} are satisfied. Therefore, we are able to make use of the statement and the proof of proposition \ref{prop_ass_form}. From proposition \ref{prop_ass_form} we have
\begin{alignat}{3}
  \label{eq:1138}
  \frac{dz}{dv}(v)&=-1+\Landau(v),&\qquad z(v)&=-v+\Landau(v^2),\\
  \frac{df}{dv}(v)&=\frac{\lambda}{3\kappa^2}v+\Landau(v^2),&\qquad f(v)&=\frac{\lambda}{6\kappa^2}v^2+\Landau(v^3),\label{eq:1139}\\
  \frac{dg}{dv}(v)&=\frac{\lambda c_{+0}}{3\kappa^2}v+\Landau(v^2),&\qquad g(v)&=\frac{\lambda c_{+0}}{6\kappa^2}v^2+\Landau(v^3),\label{eq:1140}\\
  \frac{d\alpha_+}{dv}(v)&=\frac{d\alpha_i}{dv}(v)+\frac{\lambda\tilde{A}_0}{3\kappa^2}v+\Landau(v^2),&\qquad \alpha_+(v)&=\alpha_i(v)+\frac{\lambda\tilde{A}_0}{6\kappa^2}v^2+\Landau(v^3),\label{eq:1141}\\
  \frac{d\beta_+}{dv}(v)&=\frac{\lambda}{3\kappa^2}\cp{\pp{\beta^\ast}{t}}v+\Landau(v^2),&\qquad \beta_+(v)&=\beta_0+\frac{\lambda}{6\kappa^2}\cp{\pp{\beta^\ast}{t}}v^2+\Landau(v^3).\label{eq:1142}
\end{alignat}
Using the asymptotic forms \eqref{eq:1138}, \eqref{eq:1139} in
\begin{align}
  \label{eq:1143}
    \alpha_-(v)=\alpha^\ast(f(v),z(v)),\qquad\beta_-(v)=\beta^\ast(f(v),z(v)),
\end{align}
we obtain, in the same way as we arrived at \eqref{eq:690}, \eqref{eq:691}, \eqref{eq:696}, \eqref{eq:697},
\begin{alignat}{3}
  \label{eq:1144}
  \frac{d\alpha_-}{dv}(v)&=\dot{\alpha}_0\frac{dz}{dv}(v)+\Landau(v),&\qquad \alpha_-(v)&=\alpha_0+\dot{\alpha}_0z(v)+\Landau(v^2),\\
\frac{d\beta_-}{dv}(v)&=\frac{\lambda}{3\kappa^3}\cp{\pp{\beta^\ast}{t}}v+\Landau(v^2),&\qquad\beta_-(v)&=\beta_0+\frac{\lambda}{6\kappa^2}\cp{\pp{\beta^\ast}{t}}v^2+\Landau(v^3).
\end{alignat}
Substituting \eqref{eq:1138} in \eqref{eq:1144} we obtain
\begin{align}
  \label{eq:1250}
  \frac{d\alpha_-}{dv}(v)=-\dot{\alpha}_0+\Landau(v),\qquad \alpha_-(v)=\alpha_0-\dot{\alpha}_0v+\Landau(v^2).
\end{align}
The second of \eqref{eq:1141} together with the second of \eqref{eq:1250} imply
\begin{align}
  \label{eq:1145}
  \jump{\alpha(v)}=2\dot{\alpha}_0v+\Landau(v^2).
\end{align}
Looking at the proof of lemma \ref{lemma_est_jc} we see that the above asymptotic forms constitute all the necessary requirements for this proof to hold. Therefore, we have the following estimates
\begin{align}
  \label{eq:1146}
  \left|\frac{d\Delta\beta_+}{dv}(v)\right|&\leq C\left\{v^2\nd{\Delta z}+\nd{\Delta f}+v^2\nd{\Delta\alpha_+}\right\},\\
 \label{eq:1147}
  \left|\Delta V(v)-\frac{\kappa}{2}\Delta z(v)\right|&\leq C\left\{v^2\nd{\Delta z}+v\nd{\Delta f}+v\nd{\Delta\alpha_+}\right\},
\end{align}
where the $X$ norm is defined by \eqref{eq:746}.

We now define
\begin{align}
  \label{eq:1148}
  \delta(v)\coloneqq g(v)-c_{+0}f(v).
\end{align}
From \eqref{eq:1132}, \eqref{eq:1133} we have
\begin{align}
  \label{eq:1149}
  \delta(v)=v^2(\hat{g}(v)-c_{+0}\hat{f}(v)).
\end{align}
Furthermore, from $\hat{f}, \hat{g}\in C^1[0,\varepsilon]$ together with \eqref{eq:1052} we have
\begin{align}
  \label{eq:1150}
  \hat{g}(v)-c_{+0}\hat{f}(v)=\landau(v).
\end{align}
Therefore, we can write
\begin{align}
  \label{eq:1151}
  \delta(v)=v^3\hat{\delta}(v),
\end{align}
with
\begin{align}
  \label{eq:1305}
  \hat{\delta}\in C^1(0,\varepsilon]\cap C^0[0,\varepsilon],\qquad \hat{\delta}(0)=0.
\end{align}

We now show that $d\hat{\delta}/dv$ extends continuously to $0$. From the definition \eqref{eq:1148} we obtain
\begin{align}
  \label{eq:1256}
  \frac{d\delta}{dv}(v)=(V(v)-c_{+0})\frac{df}{dv}(v),
\end{align}
where (see \eqref{eq:1238}, \eqref{eq:1278})
\begin{align}
  \label{eq:1283}
  V(v)=c_{+0}+\frac{\kappa}{2}(1+y(v))v+v^2\hat{V}(v).
\end{align}
We have
\begin{align}
  \label{eq:1259}
  \frac{1+y(v)}{v}=\frac{1}{v}\int_0^v\frac{dy}{dv}(v')dv'.
\end{align}
Therefore,
\begin{align}
  \label{eq:1277}
  \lim_{v\rightarrow 0}\frac{1+y(v)}{v}=\frac{dy}{dv}(0)
\end{align}
and we obtain
\begin{align}
  \label{eq:1286}
  \lim_{v\rightarrow 0}\frac{V(v)-c_{+0}}{v^2}=V_2,
\end{align}
where
\begin{align}
  \label{eq:1289}
  V_2\coloneqq \frac{\kappa}{2}\frac{dy}{dv}(0)+\hat{V}(0).
\end{align}
I.e.
\begin{align}
  \label{eq:1290}
  V(v)-c_{+0}=V_2v^2+\landau(v^2).
\end{align}
It follows that
\begin{align}
  \label{eq:1291}
  \lim_{v\rightarrow 0}\frac{1}{v^3}\frac{d\delta}{dv}=\frac{V_2\lambda}{3\kappa^2},
\end{align}
where we used the first of \eqref{eq:1139}. Taking into account $\delta(0)=0$, this implies
\begin{align}
  \label{eq:1294}
  \lim_{v\rightarrow 0}\frac{1}{v^4}\delta=\frac{V_2\lambda}{12\kappa^2}.
\end{align}
Hence,
\begin{align}
  \label{eq:1302}
  \lim_{v\rightarrow 0}\frac{d\hat{\delta}}{dv}=\lim_{v\rightarrow 0}\left(\frac{1}{v^3}\frac{d\delta}{dv}-\frac{3\delta}{v^4}\right)=\frac{V_2\lambda}{12\kappa^2}.
\end{align}
Thus $\hat{\delta}$ extends to a $C^1$ function on $[0,\varepsilon]$, i.e.~we have
\begin{align}
  \label{eq:1152}
  \hat{\delta}\in C^1[0,\varepsilon].
\end{align}

Looking now at the proof of lemma \ref{lemma_est_ie} we see that the above asymptotic forms constitute all the necessary requirements for this proof to hold. Therefore, we have
\begin{align}
  \label{eq:1153}
  \frac{d\Delta y}{dv}(v)= \frac{3\kappa}{\lambda}\left\{\kappa\frac{d\Delta\hat{f}}{dv}(v)+\frac{d\Delta\hat{\delta}}{dv}(v)\right\}+\Landau(v\nd{\Delta\hat{f}})+\Landau(v\nd{\Delta y}).
\end{align}

Now we look at the partial derivatives of $t$. We recall \eqref{eq:1053}:
\begin{align}
  \label{eq:1156}
  \pppp{t}{u}{v}=\frac{1}{c_+-c_-}\left(\pp{c_-}{v}\pp{t}{u}-\pp{c_+}{u}\pp{t}{v}\right).
\end{align}
Integrating with respect to $v$ from $v=0$ yields
\begin{align}
  \label{eq:1157}
  \frac{\p t}{\p u}(u,v)=e^{-K(u,v)}\left\{h'(u)-\int_0^ve^{K(u,v')}\left(\mu\frac{\p t}{\p v}\right)(u,v')dv'\right\},
\end{align}
where we recall
\begin{align}
  \label{eq:1158}
  K(u,v)=\int_0^v(-\nu)(u,v')dv',\qquad\mu=\frac{1}{c_+-c_-}\pp{c_+}{u}
\end{align}
and
\begin{align}
  \label{eq:1164}
  \nu=\frac{1}{c_+-c_-}\pp{c_-}{v}.
\end{align}
We also recall from the proof of proposition \ref{prop_ass_form} the expressions \eqref{eq:1072}, \eqref{eq:1075}:
\begin{align}
  \pp{t}{u}(u,v)&=\landau_u(u)\label{eq:1160},\\
  \pp{t}{v}(u,v)&=2v\hat{f}(0)+u\landau_v(1)\label{eq:1159}.
\end{align}
From the second of \eqref{eq:1051} we obtain
\begin{align}
  \label{eq:1161}
  \pp{t}{v}(u,v)=\frac{\lambda}{3\kappa^2}v+u\landau_v(1).
\end{align}
From \eqref{eq:1076} we have
\begin{align}
  \label{eq:1162}
  \mu(u,v)=\frac{\kappa}{c_{+0}-c_{-0}}+\landau_u(1).
\end{align}
From \eqref{eq:1073} we have
\begin{align}
  \label{eq:1306}
  \nu(u,v)=\landau_v(1).
\end{align}
Using this together with \eqref{eq:1159} in \eqref{eq:1157} and recalling \eqref{eq:1054} and the second of \eqref{eq:1051}, we obtain
\begin{align}
  \label{eq:1163}
  \pp{t}{u}(u,v)=\frac{\lambda(3u^2-v^2)}{6\kappa(c_{+0}-c_{-0})}+\landau_u(u^2).
\end{align}

We now improve the expressions for $\mu$ and $\nu$. From
\begin{align}
  \label{eq:1165}
   \pp{\beta}{u}=\pp{t}{u}\tilde{B}
\end{align}
we have
\begin{align}
  \label{eq:1166}
  \pp{\beta}{u}=\Landau(u^2).
\end{align}
From
\begin{align}
  \label{eq:1167}
  \pp{\alpha}{v}=\pp{t}{v}\tilde{A}
\end{align}
we have
\begin{align}
  \label{eq:1168}
  \alpha(u,v)=\alpha_i(u)+\int_0^v\left(\pp{t}{v}\tilde{A}\right)(u,v')dv'.
\end{align}
Therefore,
\begin{align}
  \label{eq:1169}
  \pp{\alpha}{u}(u,v)=\frac{d\alpha_i}{du}(u)+\int_0^v\left(\pppp{t}{u}{v}\tilde{A}+\pp{t}{v}\pp{\tilde{A}}{u}\right)(u,v')dv',
\end{align}
which implies
\begin{align}
  \label{eq:1170}
  \pp{\alpha}{u}(u,v)=\dot{\alpha}_0+\Landau(u).
\end{align}

In view of \eqref{eq:1066} and the second of \eqref{eq:1158} we have (recall that $\cp{\partial c_+/\partial\alpha}\dot{\alpha}_0=\kappa$)
\begin{align}
  \label{eq:1171}
  \mu(u,v)=\frac{\kappa}{c_{+0}-c_{-0}}+\Landau(u).
\end{align}
From \eqref{eq:1161} and \eqref{eq:1167} we have
\begin{align}
  \label{eq:1172}
  \pp{\alpha}{v}(u,v)=\Landau(u).
\end{align}
From \eqref{eq:1165} we have
\begin{align}
  \label{eq:1173}
  \beta(u,v)=\beta_+(v)+\int_v^u\left(\pp{t}{u}\tilde{B}\right)(u',v)du'.
\end{align}
Therefore
\begin{align}
  \label{eq:1174}
  \pp{\beta}{v}(u,v)=\frac{d\beta_+}{dv}(v)+\int_v^u\left(\pppp{t}{u}{v}\tilde{B}+\pp{t}{u}\pp{\tilde{B}}{v}\right)(u',v)du',
\end{align}
which implies
\begin{align}
  \label{eq:1175}
  \pp{\beta}{v}(u,v)=\Landau(u).
\end{align}
In view of \eqref{eq:1069} we have
\begin{align}
  \label{eq:1176}
  \nu(u,v)=\Landau(u).
\end{align}

Now we look at $1/\gamma(v)$ given by (see \eqref{eq:299})
\begin{align}
  \label{eq:1177}
  \gamma(v)\coloneqq \frac{\bar{c}_+(v)-V(v)}{V(v)-\bar{c}_-(v)},
\end{align}
where
\begin{align}
  \label{eq:1178}
  \bar{c}_\pm(v)\coloneqq c_\pm(\alpha_+(v),\beta_+(v)).
\end{align}
Using the above asymptotic forms for $\alpha_+$, $\beta_+$ and $V$ we obtain, in exactly the same way as we did in the part on the fixed boundary problem,
\begin{align}
  \label{eq:1179}
  \frac{1}{\gamma(v)}=\frac{c_{+0}-c_{-0}}{\kappa v}(1+\rho(v))
\end{align}
with
\begin{align}
  \label{eq:1180}
  \rho(v)=\rho_0(v)+\Landau(v),\qquad \rho_0(v)=\frac{\frac{1}{2}(y(v)+1)}{1-\frac{1}{2}(y(v)+1)}.
\end{align}

The above asymptotic forms now allow us to follow the proof of lemma \ref{lemma_est_fxdbp}. % with only very few modification. These modifications are just bounds which have to be changed from being in terms of powers of $v$ to being in terms of powers of $u$. The overall argument is not affected since we always take the supremum over $T_u$ anyway.
We thus obtain
\begin{align}
  \label{eq:1181}
  \bigg|\frac{d\Delta\hat{f}}{dv}(v)\bigg|&\leq \frac{\lambda}{24\kappa^2}\nd{\Delta y}+C\left\{\sup_{[0,v]}|\Delta \hat{V}|+v\nd{\Delta\hat{\beta}_+}+v\nd{\Delta y}\right\},\\
  \left|\frac{d\Delta \hat{\delta}}{dv}(v)\right|&\leq\frac{\lambda}{24\kappa}\nd{\Delta y}+C\left\{\sup_{[0,v]}|\Delta \hat{V}|+v\nd{\Delta \hat{\beta}_+}+v\nd{\Delta y}\right\},\label{eq:1182}\\
  \nd{\Delta \alpha_+}&\leq Cv^2\left\{\sup_{[0,v]}|\Delta \hat{V}|+v\nd{\Delta \hat{\beta}_+}+\nd{\Delta y}\right\}.\label{eq:1183}
\end{align}

We now combine \eqref{eq:1146}, \eqref{eq:1147}, \eqref{eq:1153}, \eqref{eq:1181}, \eqref{eq:1182}, \eqref{eq:1183} in the same way as we did in the proof of proposition \ref{prop_est_combination} and find, for $\varepsilon$ sufficiently small,
\begin{align}
  \label{eq:1184}
  \nd{\Delta y}&\leq \frac{1}{3}\nd{\Delta y}+C\left\{v\nd{\Delta \hat{\beta}_+}+\sup_{[0,v]}|\Delta \hat{V}|\right\},\notag\\
\nd{\Delta \hat{\beta}_+}&\leq C\left\{\nd{\Delta y}+\sup_{[0,v]}|\Delta \hat{V}|+v\nd{\Delta \hat{\beta}_+}\right\},\notag\\
\sup_{[0,v]}|\Delta \hat{V}|&\leq C\left\{v\nd{\Delta y}+v\sup_{[0,v]}|\Delta \hat{V}|+v^2\nd{\Delta \hat{\beta}_+}\right\}.
\end{align}
These imply, for $\varepsilon$ sufficiently small,
\begin{align}
  \label{eq:1185}
  \nd{\Delta y}&\leq C\left\{v\nd{\Delta\hat{\beta}_+}+\sup_{[0,\varepsilon]}|\Delta \hat{V}|\right\},\\
  \nd{\Delta\hat{\beta}_+}&\leq C\left\{\nd{\Delta y}+\sup_{[0,\varepsilon]}|\Delta\hat{V}|\right\},\label{eq:1186}\\
  \sup_{[0,\varepsilon]}|\Delta\hat{V}|&\leq C\left\{v\nd{\Delta y}+v^2\nd{\Delta\hat{\beta}_+}\right\}.\label{eq:1187}
\end{align}
Substituting \eqref{eq:1185} in \eqref{eq:1186} and \eqref{eq:1187} yields
\begin{align}
  \label{eq:1188}
  \nd{\Delta\hat{\beta}_+}&\leq C\sup_{[0,\varepsilon]}|\Delta\hat{V}|,\\
  \sup_{[0,\varepsilon]}|\Delta\hat{V}|&\leq Cv^2\nd{\Delta\hat{\beta}_+},\label{eq:1189}
\end{align}
for $\varepsilon$ sufficiently small. Substituting \eqref{eq:1188} in \eqref{eq:1189} gives, for $\varepsilon$ sufficiently small,
\begin{align}
  \label{eq:1190}
  \Delta\hat{V}=0,
\end{align}
which gives
\begin{align}
  \label{eq:1191}
  \nd{\Delta\hat{\beta}_+}=0,
\end{align}
which gives
\begin{align}
  \label{eq:1192}
  \nd{\Delta y}=0.
\end{align}
In view of \eqref{eq:1181}, \eqref{eq:1182} and \eqref{eq:1183}, the vanishing of these differences implies that also the differences of $f$, $\delta$ and $\alpha_+$ vanish. Now we make use of estimates appearing in the proof of lemma \ref{lemma_est_fxdbp}. In all these estimates there appear no indices in the present context. $F(u)$ given by \eqref{eq:857} vanishes. Therefore, in view of \eqref{eq:866}, \eqref{eq:867} the differences of the partial derivatives of $t$ vanish. In view of \eqref{eq:774}, \eqref{eq:775} the differences of $\alpha$ and $\beta$ vanish, therefore also the differences of $\mu$ and $\nu$ vanish. In view of \eqref{eq:769} also the difference of $r$ vanishes. In view of \eqref{eq:806}, \eqref{eq:813} the differences of $\partial\alpha/\partial u$ and $\partial\beta/\partial v$ vanish. In view of the characteristic system the differences of $\partial\alpha/\partial v$ and $\partial\beta/\partial u$ vanish. In view of the Hodograph system also the differences of the partial derivatives of $r$ vanish. Therefore, the two solutions (prime and double prime) coincide. This completes the uniqueness proof.
\end{proof}

\subsection{Continuity of $L_+\alpha$ and $L_+\beta$ across the Incoming Characteristic originating at the Cusp Point}
In the present subsection we carry the argument of the above proof further. In particular we will first improve the estimates \eqref{eq:1161}, \eqref{eq:1172} and \eqref{eq:1175}. Then on the basis of these improved estimates we will show the continuity of $L_+\alpha$ and $L_+\beta$ across $\underline{C}$.
\begin{proposition}
  $L_+\alpha$ and $L_+\beta$ are continuous across $\underline{C}$.
\end{proposition}
\begin{proof}
Let us consider $\partial t/\partial v$, $\partial\beta/\partial v$ along $\mathcal{K}$. We have, along $\mathcal{K}$,
\begin{align}
  \label{eq:86}
  \pp{t}{v}=\frac{df}{dv}-\pp{t}{u},\qquad \pp{\beta}{v}=\frac{d\beta_+}{dv}-\pp{\beta}{u}
\end{align}
and, by proposition \ref{prop_ass_form},
\begin{align}
  \label{eq:93}
  \frac{df}{dv}(v)&=\frac{\lambda}{3\kappa^2}v+\Landau(v^2),\\
  \label{eq:1155}
  \frac{d\beta_+}{dv}(v)&=\frac{\lambda}{3\kappa^2}\cp{\pp{\beta^\ast}{t}}v+\Landau(v^2).
\end{align}
Evaluating \eqref{eq:1163} at $u=v$ we obtain
\begin{align}
  \label{eq:733}
  \pp{t}{u}(v,v)=\frac{\lambda}{3\kappa(c_{+0}-c_{-0})}v^2+\landau(v^2).
\end{align}
From \eqref{eq:93} and \eqref{eq:733} we obtain, through the first of \eqref{eq:86},
\begin{align}
  \label{eq:1061}
  \pp{t}{v}(v,v)=\frac{\lambda}{3\kappa^2}v+\Landau(v^2).
\end{align}
By \eqref{eq:1165} and \eqref{eq:733} we get
\begin{align}
  \label{eq:1154}
  \pp{\beta}{u}(v,v)=\Landau(v^2).
\end{align}
From \eqref{eq:1155} and \eqref{eq:1154} we obtain, through the second of \eqref{eq:86},
\begin{align}
  \label{eq:1215}
  \pp{\beta}{v}(v,v)=\frac{\lambda}{3\kappa^2}\cp{\pp{\beta^\ast}{t}}v+\Landau(v^2).
\end{align}

Let us then consider the system \eqref{eq:307}, \eqref{eq:308} along \textit{any} incoming characteristic. It is a system of the form \eqref{eq:309}:
\begin{align}
  \label{eq:1217}
  \frac{d}{du}
\left(\begin{array}{c}
      \partial\beta/\partial v\\
      \partial t/\partial v
    \end{array}\right)
=
\left(\begin{array}{cc}
  a_{11} & a_{12}\\
  a_{21} & a_{22}
\end{array}\right)
\left(\begin{array}{c}
  \partial\beta/\partial v\\
  \partial t/\partial v
\end{array}\right).
\end{align}
This is a linear homogeneous system with a coefficient matrix
\begin{align}
  \label{eq:1218}
  a=\left(\begin{array}{cc}
  a_{11} & a_{12}\\
  a_{21} & a_{22}
\end{array}\right)
\end{align}
which is continuous on $T_\varepsilon$. The initial data are on $\mathcal{K}$ and given by \eqref{eq:1061}, \eqref{eq:1215}. Let the matrix $m$ be the solution of
\begin{align}
  \label{eq:1220}
  \frac{dm}{du}=am,\qquad \left.m\right|_{\mathcal{K}}=\textrm{id}.
\end{align}
Then the solution of \eqref{eq:1217} is
\begin{align}
  \label{eq:1221}
  \left(\begin{array}{c}
      \partial\beta/\partial v\\
      \partial t/\partial v
    \end{array}\right)
=m\left.\left(\begin{array}{c}
      \partial\beta/\partial v\\
      \partial t/\partial v
    \end{array}\right)\right|_\mathcal{K}.
\end{align}
Since
\begin{align}
  \label{eq:1224}
  m-\textrm{id}=\Landau(u-v),
\end{align}
and by \eqref{eq:1061}, \eqref{eq:1215},
\begin{align}
  \label{eq:1226}
  \left.\left(\begin{array}{c}
      \partial\beta/\partial v\\
      \partial t/\partial v
    \end{array}\right)\right|_\mathcal{K}
=\Landau(v),
\end{align}
it follows that
\begin{align}
  \label{eq:1244}
  \left(\begin{array}{c}
      \partial\beta/\partial v\\
      \partial t/\partial v
    \end{array}\right)
=\left.\left(\begin{array}{c}
      \partial\beta/\partial v\\
      \partial t/\partial v
    \end{array}\right)\right|_\mathcal{K}+\Landau(uv),
\end{align}
i.e.~by \eqref{eq:1061}, \eqref{eq:1215},
\begin{align}
  \label{eq:1246}
  \pp{t}{v}(u,v)&=\frac{\lambda}{3\kappa^2}v+\Landau(uv),\\
  \label{eq:1312}
  \pp{\beta}{v}(u,v)&=\frac{\lambda}{3\kappa^2}\cp{\pp{\beta^\ast}{t}}v+\Landau(uv),
\end{align}
which improve \eqref{eq:1161} and \eqref{eq:1175} respectively. Also, \eqref{eq:1246} implies, through \eqref{eq:1167},
\begin{align}
  \label{eq:1319}
  \pp{\alpha}{v}(u,v)=\frac{\lambda\tilde{A}_0}{3\kappa^2}v+\Landau(uv),
\end{align}
which improves \eqref{eq:1172}.

Since $\alpha$ and $t$ are by construction continuous across $\underline{C}$ while $\beta$ and $r$ satisfy along $\underline{C}$ the o.d.e.~system
\begin{align}
  \label{eq:1321}
  \frac{d\beta}{du}&=\pp{t}{u}\tilde{B}(\alpha,\beta,r),\\
  \frac{dr}{du}&=\pp{t}{u}c_-(\alpha,\beta),
\end{align}
while
\begin{align}
  \label{eq:1324}
  \beta(0,0)=\beta^\ast(0,0)=\beta_0,\qquad r(0,0)=r^\ast(0,0)=r_0,
\end{align}
at the cusp point, it follows that $r$ and $\beta$ are continuous across $\underline{C}$ as well. Then,
\begin{align}
  \label{eq:1325}
  L_+\alpha=\frac{\partial\alpha/\partial v}{\partial t/\partial v}=\tilde{A}(\alpha,\beta,r)
\end{align}
is also continuous across $\underline{C}$.

Let us consider
\begin{align}
  \label{eq:1326}
  L_+\beta=\frac{\partial \beta/\partial v}{\partial t/\partial v}.
\end{align}
From \eqref{eq:307}, \eqref{eq:308} we have
\begin{align}
  \label{eq:1348}
  a_{11}&=\frac{1}{c_+-c_-}\pp{t}{u}\pp{c_-}{\beta}\tilde{B}+\pp{t}{u}\pp{\tilde{B}}{\beta},\\
  a_{12}&=-\frac{1}{c_+-c_-}\left(\pp{c_+}{u}-\pp{t}{u}\pp{c_-}{\alpha}\tilde{A}\right)\tilde{B}+\pp{t}{u}\left(\pp{\tilde{B}}{\alpha}\tilde{A}+\pp{\tilde{B}}{r}c_+\right),\\
  a_{21}&=\frac{1}{c_+-c_-}\pp{t}{u}\pp{c_-}{\beta},\\
  a_{22}&=-\frac{1}{c_+-c_-}\left(\pp{c_+}{u}-\pp{t}{u}\pp{c_-}{\alpha}\tilde{A}\right).
\end{align}
We see that the matrix $a$ is continuous across $\underline{C}$. This implies that the matrix $m$ is continuous across $\underline{C}$. By \eqref{eq:1221},
\begin{align}
  \label{eq:1349}
  L_+\beta=\frac{\left.m_{11}L_+\beta\right|_\mathcal{K}+m_{12}}{\left.m_{21}L_+\beta\right|_\mathcal{K}+m_{22}},
\end{align}
while by \eqref{eq:1061} and \eqref{eq:1215},
\begin{align}
  \label{eq:1372}
  \cp{L_+\beta}=\lim_{v\rightarrow 0}\left.L_+\beta\right|_\mathcal{K}=\lim_{v\rightarrow 0}\left(\frac{\partial\beta/\partial v}{\partial t/\partial v}\right)(v,v)=\cp{\pp{\beta^\ast}{t}}=\cp{L_+\beta^\ast}.
\end{align}
Hence $L_+\beta$ is continuous across $\underline{C}$ as well.

\end{proof}

%%% Local Variables: 
%%% mode: latex
%%% TeX-master: "./master"
%%% End: 

\section{Higher Regularity}

In the following we denote by $\bar{P}_{m,n}$ a polynomial in $v$ of degree $m$ starting with an $n$'th order term. We denote by $P_{m,n}(v)$ a sum of $\bar{P}_{m,n}$ and a function of $\Landau(v^{m+1})$. We then define $P_m(v)\coloneqq P_{m,0}(v)$. We also denote by $Q_{m,n}(u,v)$ a sum of a polynomial in $u$ and $v$ of degree $m$ starting with an $n$'th order term and a function of $\Landau(u^{m+1})$ (we recall that in the domain in question, i.e.~in $T_\varepsilon$, we have $0\leq v\leq u\leq\varepsilon$). We then define $Q_m\coloneqq Q_{m,0}$. We extend the meaning of $P_{m,n}$ and $Q_{m,n}$ to the case $n=m+1$ by
\begin{align}
  \label{eq:1425}
  P_{m,m+1}\coloneqq\Landau(v^{m+1}),\qquad Q_{m,m+1}\coloneqq\Landau(u^{m+1}).
\end{align}
Furthermore we will use the definitions
\begin{align}
  \label{eq:1219}
  I_{m,n}(v)\coloneqq\int_0^vv'^n\frac{d^my}{dv^m}(v')dv',\qquad \hat{I}_{m,n}(v)\coloneqq\frac{1}{v^n}I_{m,n}(v)
\end{align}
and
\begin{align}
  \label{eq:1927}
  Y_m\coloneqq\sup_{[0,\varepsilon]}\left|\frac{d^my}{dv^m}\right|,\qquad F_m\coloneqq\sup_{[0,\varepsilon]}\bigg|\frac{d^m\hat{f}}{dv^m}\bigg|.
\end{align}
In the following we prove that the solution established in the existence theorem is smooth. We do this by induction, showing that all derivatives of $y$, $\hat{f}$ are bounded and all derivatives of $t$, $r$, $\alpha$, $\beta$ are in $C^1$.

\subsection{Inductive Hypothesis}
We make the following inductive hypotheses: We assume that we have bounds for $Y_m$, $F_m$ for $m=1,\ldots, n-1$, i.e.
\begin{align}
  \label{eq:1197}
  Y_1,\ldots,Y_{n-1}&\leq C,\tag{$Y_{n-1}$}\\
  \label{eq:1299}
  F_1,\ldots,F_{n-1}&\leq C\tag{$F_{n-1}$}.
\end{align}
For the function $\alpha_+(v)=\alpha(v,v)$ we assume
\begin{align}
  \label{eq:1198}
  \frac{d^{n-1}\alpha_+}{dv^{n-1}}=P_1.\tag{$\alpha_{+,n-1}$}
\end{align}
For the function $t(u,v)$ we assume, for $n\geq 3$,
\begin{align}
  \label{eq:1233}
  \frac{\partial^{k-2}t}{\partial u^{k-2}}=Q_2,:k\leq n,\qquad \frac{\partial^{k-2}t}{\partial v^{k-2}}&=Q_1,:k\leq n,\tag{$t_{p,n-1}$}\\
  \label{eq:1637}
  \frac{\partial^{i+j}}{\partial u^i\partial v^j}\left(\pppp{t}{u}{v}\right)&=Q_1:i+j\leq n-3.\tag{$t_{m,n-1}$}
\end{align}
With the indices $p$ and $m$ we indicate that we refer to pure and mixed derivatives.

In the case $n=2$, \eqref{eq:1637} is not present and \eqref{eq:1233} is
\begin{align}
  \label{eq:2007}
  t=Q_2.\tag{$t_{p,1}$}
\end{align}
For the functions $\alpha$ and $\beta$ we assume
\begin{align}
  \label{eq:1235}
  \frac{\partial^{k-2}\alpha}{\partial u^{k-2}}=Q_2,:k\leq n,\qquad \frac{\partial^{k-2}\alpha}{\partial v^{k-2}}&=Q_1,:k\leq n,\tag{$\alpha_{p,n-1}$}\\
\label{eq:1641}
\frac{\partial^{i+j}}{\partial u^i\partial v^j}\left(\pppp{\alpha}{u}{v}\right)&=Q_1:i+j\leq n-3,\tag{$\alpha_{m,n-1}$}\\
  \label{eq:1292}
  \frac{\partial^{k-2}\beta}{\partial u^{k-2}}=Q_2,:k\leq n,\qquad\frac{\partial^{k-2}\beta}{\partial v^{k-2}}&=Q_1,:k\leq n,\tag{$\beta_{p,n-1}$}\\
  \label{eq:1639}
  \frac{\partial^{i+j}}{\partial u^i\partial v^j}\left(\pppp{\beta}{u}{v}\right)&=Q_1:i+j\leq n-3,\tag{$\beta_{m,n-1}$}
\end{align}
and
\begin{alignat}{3}
  \label{eq:1234}
  \frac{\partial^{n-1}\alpha}{\partial u^{n-1}}&=Q_0,\qquad& \frac{\partial^{n-1}\alpha}{\partial v^{n-1}}&=Q_0,\tag{$\alpha_{0,n-1}$}\\
  \label{eq:1293}
  \frac{\partial^{n-1}\beta}{\partial u^{n-1}}&=Q_0,& \frac{\partial^{n-1}\beta}{\partial v^{n-1}}&=Q_0.\tag{$\beta_{0,n-1}$}
\end{alignat}
For $n=2$ the properties \eqref{eq:1641}, \eqref{eq:1639} are not present and \eqref{eq:1235}, \eqref{eq:1292} are
\begin{align}
  \label{eq:2009}
  \alpha&=Q_2,\tag{$\alpha_{p,1}$}\\
  \label{eq:2010}
  \beta&=Q_2.\tag{$\beta_{p,1}$}
\end{align}

\subsection{Base Case $n=2$}
We show that the inductive hypothesis holds for $n=2$. We are going to use estimates established during the existence proof. The index on functions in estimates in the existence proof was in order to label the iterates. These estimates also hold in the limit hence without any indices. Also the dependencies on $\delta_2$ of the bounds are not present anymore since $\delta_2$ has been chosen appropriately.

Since $Y_1=Y$, where $Y$ was defined in \eqref{eq:293}, we have from \eqref{eq:745},
\begin{align}
  \label{eq:2004}
  Y_1\leq C.\tag{$Y_1$}
\end{align}
From \eqref{eq:614} we have
\begin{align}
  \label{eq:2005}
  F_1\leq C,\tag{$F_1$}
\end{align}
while from \eqref{eq:709} we have
\begin{align}
  \label{eq:2006}
  \frac{d\alpha_+}{dv}=P_1.\tag{$\alpha_{+,1}$}
\end{align}
Using now the first of \eqref{eq:701} in
\begin{align}
  \label{eq:2011}
  t(u,v)=t(u,0)+\int_0^v\pp{t}{v}(u,v')dv'
\end{align}
together with the fact that $t(u,0)=u^3\hat{h}(u)$, where $\hat{h}$ is a smooth function, we obtain \eqref{eq:2007}. From \eqref{eq:1166}, \eqref{eq:1170}, \eqref{eq:1172}, \eqref{eq:1175} we have
\begin{alignat}{2}
  \label{eq:2014}
  \pp{\alpha}{u}&=Q_0,&\qquad \pp{\alpha}{v}&=Q_{0,1},\\
  \label{eq:2018}
  \pp{\beta}{u}&=Q_{1,2},& \pp{\beta}{v}&=Q_{0,1}.
\end{alignat}
Therefore, $(\alpha_{0,1})$ and $(\beta_{0,1})$ hold.

We have
\begin{align}
  \label{eq:2015}
  \alpha(u,v)=\alpha_i(u)+\int_0^v\pp{\alpha}{v}(u,v')dv'.
\end{align}
From
\begin{align}
  \label{eq:2019}
  \pp{\alpha}{v}=\pp{t}{v}\tilde{A}(\alpha,\beta,r),
\end{align}
together with the first of \eqref{eq:701} we obtain
\begin{align}
  \label{eq:673}
  \pp{\alpha}{v}=Q_{1,1}.
\end{align}
Using this together with the fact that $\alpha_i$ is a smooth function in \eqref{eq:2015}, we deduce that $(\alpha_{p,1})$ holds.

We have
\begin{align}
  \label{eq:2016}
  \beta(u,v)=\beta_+(v)+\int_v^u\pp{\beta}{u}(u',v)du'.
\end{align}
From
\begin{align}
  \label{eq:694}
  \pp{\beta}{u}=\pp{t}{u}\tilde{B},
\end{align}
together with the second of \eqref{eq:701} we obtain
\begin{align}
  \label{eq:693}
  \pp{\beta}{u}=Q_{1,1}.
\end{align}
From the second of \eqref{eq:1142} we have
\begin{align}
  \label{eq:2017}
  \beta_+=P_2.
\end{align}
Using this together with  \eqref{eq:693} in \eqref{eq:2016} we deduce that $(\beta_{p,1})$ holds. We conclude that the inductive hypothesis holds in the case $n=2$.

\subsection{Inductive Step}
We now show the inductive step, i.e.~we show that \eqref{eq:1197}, \ldots, \eqref{eq:1293} hold with $n$ in the role of $n-1$. Once this is proved, we have proven the following regularity theorem.
\begin{theorem}
  The solution whose existence is the content of theorem \ref{existence_theorem} and whose uniqueness is the content of theorem \ref{uniqueness_theorem} is actually smooth.
\end{theorem}
We remark that to see that also the function $r(u,v)$, which is not present in the inductive hypothesis, is a smooth function, we appeal to the Hodograph system, i.e.~to \eqref{eq:159}.

We write
\begin{align}
  \label{eq:1328}
  \frac{d^{k-1}y}{dv^{k-1}}=\cp{\frac{d^{k-1}y}{dv^{k-1}}}+\int_0^v\frac{d^{k}y}{dv^{k}}(v')dv',
\end{align}
where we recall the notation $\cp{\cdot}$ for evaluation at the cusp point, i.e.~for functions of $v$ evaluation at $v=0$. Since $z=vy$, we have
\begin{align}
  \label{eq:1327}
  \frac{d^{i}z}{dv^{i}}=v\frac{d^{i}y}{dv^{i}}+i\frac{d^{i-1}y}{dv^{i-1}}.
\end{align}
Setting $i=n-1$ and using \eqref{eq:1328}, we deduce from the inductive hypothesis \eqref{eq:1197}
\begin{align}
  \label{eq:2050}
  \frac{d^{n-1}z}{dv^{n-1}}=P_{0}.
\end{align}
This implies, through integration,
\begin{align}
  \label{eq:1298}
  \frac{d^mz}{dv^m}=\left\{
    \begin{array}{ll}
      P_{n-1,1} & m=0,\\
      P_{n-m-1} & 1\leq m\leq n-1.
    \end{array}\right.
\end{align}

Recalling $f(v)=v^2\hat{f}(v)$ and making use of the assumption \eqref{eq:1299} we obtain, analogous to the way we arrived at \eqref{eq:1298},
\begin{align}
  \label{eq:1213}
  \frac{d^mf}{dv^m}=\left\{
    \begin{array}{ll}
      P_{n,2} & m=0,\\
      P_{n-1,1} & m=1,\\
      P_{n-m} & 2\leq m\leq n-1.
    \end{array}\right.
\end{align}
Making use of assumption \eqref{eq:1198} we obtain, through integration,
\begin{align}
  \label{eq:1300}
  \frac{d^m\alpha_+}{dv^m}=P_{n-m}:0\leq m\leq n-1.
\end{align}

We now look at the behavior of $\alpha_-(v)$ and $\beta_-(v)$. Recalling
\begin{align}
  \label{eq:1295}
  \alpha_-(v)=\alpha^\ast(f(v),z(v)),\qquad \beta_-(v)=\beta^\ast(f(v),z(v)),
\end{align}
where $\alpha^\ast(t,w)$, $\beta^\ast(t,w)$ denotes the solution in the state ahead and $t(v,v)=f(v)$, $w=z(v)$ are substituted, we have
\begin{align}
  \label{eq:1296}
  \frac{d\alpha_-}{dv}=\pp{\alpha^\ast}{t}(f,z)\frac{df}{dv}+\pp{\alpha^\ast}{w}(f,z)\frac{dz}{dv},\qquad\frac{d\beta_-}{dv}=\pp{\beta^\ast}{t}(f,z)\frac{df}{dv}+\pp{\beta^\ast}{w}(f,z)\frac{dz}{dv}.
\end{align}
Now, by \eqref{eq:1298} and \eqref{eq:1213} we have
\begin{align}
  \label{eq:1297}
  \pp{\alpha^\ast}{t}(f,z)=P_{n-1},\qquad\pp{\alpha^\ast}{w}(f,z)=P_{n-1}
\end{align}
and for higher order derivatives we have
\begin{align}
  \label{eq:1441}
  \frac{\partial^{i+j}\alpha^\ast}{\partial t^i\partial w^j}(f,z)=P_{n-1}:i+j\leq m.
\end{align}
Analogous, but now taking into account the fact that $\cp{\partial\beta^\ast/\partial w}=0$, we have
\begin{align}
  \label{eq:1301}
  \pp{\beta^\ast}{t}(f,z)=P_{n-1},\qquad\pp{\beta^\ast}{w}(f,z)=P_{n-1,1}
\end{align}
and for higher order derivatives we have
\begin{align}
  \label{eq:1442}
  \frac{\partial^{i+j}\beta^\ast}{\partial t^i\partial w^j}(f,z)=P_{n-1}:i+j\leq m.
\end{align}
Taking $m-1$ derivatives of the first of \eqref{eq:1296} and making use of \eqref{eq:1298}, \eqref{eq:1213}, \eqref{eq:1297} and \eqref{eq:1441} we obtain
\begin{align}
  \label{eq:1303}
  \frac{d^m\alpha_-}{dv^m}=P_{n-m-1}:0\leq m\leq n-1.
\end{align}
Taking $m-1$ derivatives of the second of \eqref{eq:1296} and making use of \eqref{eq:1298}, \eqref{eq:1213}, \eqref{eq:1301} and \eqref{eq:1442} we obtain
\begin{align}
  \label{eq:1304}
  \frac{d^m\beta_-}{dv^m}=\left\{
    \begin{array}{ll}
      P_{n-1,1} & m=1,\\
      P_{n-m} & 2\leq m\leq n-1.
    \end{array}\right.
\end{align}

Now we look at the behavior of $\beta_+(v)$. For this we recall $\jump{\beta}=\jump{\alpha}^3G(\alpha_+,\alpha_-,\beta_-)$. We have
\begin{align}
  \label{eq:1307}
  \frac{d^m\jump{\beta}}{dv^m}=\sum_{i=0}^m\binom{m}{i}\frac{d^i\jump{\alpha}^3}{dv^i}\frac{d^{m-i}G}{dv^{m-i}}.
\end{align}
From \eqref{eq:1300}, \eqref{eq:1303} and taking into account $\alpha_-(0)=\alpha_+(0)$, we obtain
\begin{align}
  \label{eq:1308}
  \jump{\alpha}=P_{n-1,1},\qquad \jump{\alpha}^2=P_{n,2},\qquad \jump{\alpha}^3=P_{n+1,3},
\end{align}
and
\begin{align}
  \label{eq:1460}
  \frac{d^m\jump{\alpha}}{dv^m}=P_{n-m-1}:1\leq m\leq n-1
\end{align}
and similarly
\begin{align}
  \label{eq:1309}
  \frac{d^m\jump{\alpha}^3}{dv^m}=\left\{
    \begin{array}{ll}
      P_{n,2} & m=1,\\
      P_{n-1,1} & m=2,\\
      P_{n-m+1} & 3\leq m\leq n-1.
    \end{array}\right.
\end{align}
In view of \eqref{eq:1300}, \eqref{eq:1303}, \eqref{eq:1304} we have
\begin{align}
  \label{eq:1310}
  \frac{d^{m-i}G}{dv^{m-i}}=P_{n-m+i-1}.
\end{align}
Using \eqref{eq:1308}, \eqref{eq:1309} and \eqref{eq:1310} in \eqref{eq:1307} we obtain
\begin{align}
  \label{eq:1311}
  \frac{d^m\jump{\beta}}{dv^m}=\left\{
    \begin{array}{ll}
      P_{n+1,3} & m=0,\\
      P_{n,2} & m=1,\\
      P_{n-1,1} & m=2,\\
      P_{n-m+1} & 3\leq m\leq n-1.
    \end{array}\right.
\end{align}
Therefore, from \eqref{eq:1304} we obtain
\begin{align}
  \label{eq:1313}
  \frac{d^m\beta_+}{dv^m}=\left\{
    \begin{array}{ll}
      P_{n-1,1} & m=1,\\
      P_{n-m} & 2\leq m\leq n-1.
    \end{array}\right.
\end{align}
We note that \eqref{eq:1300}, \eqref{eq:1303}, \eqref{eq:1304} and \eqref{eq:1313} constitute the behaviors of $\alpha_\pm(v)$, $\beta_\pm(v)$.

\subsubsection{Estimate for $d^{n-1}V/dv^{n-1}$}
We turn to estimating $d^{n-1}V/dv^{n-1}$. We recall
\begin{align}
  \label{eq:1200}
  V=\frac{\jump{T^{tr}}}{\jump{T^{tt}}},
\end{align}
which implies
\begin{align}
  \label{eq:1201}
  \frac{dV}{dv}=\frac{1}{\jump{T^{tt}}}\left\{\frac{d\jump{T^{tr}}}{dv}-V\frac{d\jump{T^{tt}}}{dv}\right\}.
\end{align}
We use the notation
\begin{align}
  \label{eq:1202}
  \bar{c}_\pm=c_\pm(\alpha_+,\beta_+),\qquad \ccirc_\pm=c_\pm(\alpha_-,\beta_-).
\end{align}
and we recall (see \eqref{eq:184}, \eqref{eq:732})
\begin{align}
  \label{eq:1314}
  \pp{T^{tr}}{\alpha}=c_+\pp{T^{tt}}{\alpha},\qquad \pp{T^{tr}}{\beta}=c_-\pp{T^{tt}}{\beta}.
\end{align}
\eqref{eq:1201} becomes
\begin{align}
  \label{eq:1315}
  \frac{dV}{dv}&=\frac{1}{\jump{T^{tt}}}\bigg\{(\bar{c}_+-V)\pp{T^{tt}}{\alpha}(\alpha_+,\beta_+)\frac{d\alpha_+}{dv}+(\bar{c}_--V)\pp{T^{tt}}{\beta}(\alpha_+,\beta_+)\frac{d\beta_+}{dv}\notag\\
&\qquad\qquad-(\ccirc_+-V)\pp{T^{tt}}{\alpha}(\alpha_-,\beta_-)\frac{d\alpha_-}{dv}-(\ccirc_--V)\pp{T^{tt}}{\beta}(\alpha_-,\beta_-)\frac{d\beta_-}{dv}\bigg\}.
\end{align}

Defining now $a(v)$ and $b(v)$ by
\begin{align}
  \label{eq:1517}
  a&\coloneqq-\frac{1}{\jump{T^{tt}}}\frac{d}{dv}\jump{T^{tt}},\\
  \label{eq:1928}
  b&\coloneqq\frac{1}{\jump{T^{tt}}}\left\{\frac{d}{dv}\jump{T^{tr}}-c_{+0}\frac{d}{dv}\jump{T^{tt}}\right\},
\end{align}
i.e.
\begin{align}
  \label{eq:1204}
  a&=-\frac{1}{\jump{T^{tt}}}\bigg\{\pp{T^{tt}}{\alpha}(\alpha_+,\beta_+)\frac{d\alpha_+}{dv}-\pp{T^{tt}}{\alpha}(\alpha_-,\beta_-)\frac{d\alpha_-}{dv}\notag\\
&\qquad\qquad\qquad +\pp{T^{tt}}{\beta}(\alpha_+,\beta_+)\frac{d\beta_+}{dv}-\pp{T^{tt}}{\beta}(\alpha_-,\beta_-)\frac{d\beta_-}{dv}\bigg\},\\
b&=\frac{1}{\jump{T^{tt}}}\bigg\{(\bar{c}_+-c_{+0})\pp{T^{tt}}{\alpha}(\alpha_+,\beta_+)\frac{d\alpha_+}{dv}-(\ccirc_+-c_{+0})\pp{T^{tt}}{\alpha}(\alpha_-,\beta_-)\frac{d\alpha_-}{dv}\notag\\
&\qquad\qquad\qquad +(\bar{c}_--c_{+0})\pp{T^{tt}}{\beta}(\alpha_+,\beta_+)\frac{d\beta_+}{dv}-(\ccirc_--c_{+0})\pp{T^{tt}}{\beta}(\alpha_-,\beta_-)\frac{d\beta_-}{dv}\bigg\},\label{eq:1211}
\end{align}
we can rewrite \eqref{eq:1315} as
\begin{align}
  \label{eq:1205}
  \frac{dV}{dv}=(V-c_{+0})a+b.
\end{align}

Now we define $\tilde{a}$ by
\begin{align}
  \label{eq:1206}
  \tilde{a}(v)\coloneqq a(v)+\frac{1}{v}.
\end{align}
With
\begin{align}
  \label{eq:1207}
  u\coloneqq V-c_{+0},\qquad \tilde{u}\coloneqq vu,\qquad \tilde{b}\coloneqq vb,
\end{align}
\eqref{eq:1205} becomes
\begin{align}
  \label{eq:1208}
  \frac{d\tilde{u}}{dv}=\tilde{a}\tilde{u}+\tilde{b}.
\end{align}
The $(i-1)$'th order derivative of this can be written as
\begin{align}
  \label{eq:1316}
  \frac{d^i\tilde{u}}{dv^i}=\tilde{a}_{i-1}\tilde{u}+\tilde{b}_{i-1}.
\end{align}
Differentiating this we obtain
\begin{align}
  \label{eq:1317}
  \frac{d^{i+1}\tilde{u}}{dv^{i+1}}=\left(\frac{d\tilde{a}_{i-1}}{dv}+\tilde{a}\tilde{a}_{i-1}\right)\tilde{u}+\frac{d\tilde{b}_{i-1}}{dv}+\tilde{a}_{i-1}\tilde{b},
\end{align}
which gives us the following recursion formulas
\begin{align}
  \label{eq:1318}
  \tilde{a}_i&=\left(\frac{d}{dv}+\tilde{a}\right)\tilde{a}_{i-1}, \qquad\tilde{a}_0=\tilde{a},\\
  \label{eq:1320}
  \tilde{b}_i&=\frac{d\tilde{b}_{i-1}}{dv}+\tilde{a}_{i-1}\tilde{b},\qquad\hspace{2mm} \tilde{b}_0=\tilde{b}.
\end{align}
Solving these gives
\begin{align}
  \label{eq:1210}
  \tilde{a}_i=\left(\frac{d}{dv}+\tilde{a}\right)^i\tilde{a},\qquad \tilde{b}_i=\frac{d^i\tilde{b}}{dv^i}+\sum_{m=0}^{i-1}\frac{d^{i-1-m}(\tilde{a}_m\tilde{b})}{dv^{i-1-m}}.
\end{align}
In the expression for the $n$'th derivative for $y$ there will be involved the $n$'th derivative of $f$ which in turn involves the $(n-1)$'th derivative of $V$. Therefore, we have to estimate
\begin{align}
  \label{eq:1209}
  \frac{d^{n-1}\tilde{u}}{dv^{n-1}}=\tilde{a}_{n-2}\tilde{u}+\tilde{b}_{n-2}.
\end{align}

To estimate $\tilde{a}_{n-2}$ we have to estimate the $n-2$ order derivative of $\tilde{a}$. To estimate $\tilde{b}_{n-2}$ we need to estimate the $n-2$ order derivative of $\tilde{b}$ and the $n-3$ order derivative of $\tilde{a}$. We consider first $\tilde{b}_{n-2}$. We derive expressions for $a$, $b$ to $\Landau(v^2)$. We start with $b$. Since $\jump{T^{tt}(0)}=0$, we have to estimate the $n-2$ order derivative of the numerator of $b$ to $\Landau(v^3)$. Let us denote by $T_i$ the $i$'th term in the curly bracket of \eqref{eq:1211} and let
\begin{align}
  \label{eq:1398}
  N\coloneqq -\sum_{i=1}^4T_i.
\end{align}
Then
\begin{align}
  \label{eq:1381}
  b=-\frac{N}{\jump{T^{tt}}}.
\end{align}
We have
\begin{align}
  \label{eq:1323}
  T_2&=-(\ccirc_+-c_{+0})\pp{T^{tt}}{\alpha}(\alpha_-,\beta_-)\frac{d\alpha_-}{dv}\notag\\
&=-(\ccirc_+-c_{+0})\pp{T^{tt}}{\alpha}(\alpha_-,\beta_-)\left\{\pp{\alpha^\ast}{t}(f,z)\frac{df}{dv}+\pp{\alpha^\ast}{w}(f,z)\frac{dz}{dv}\right\}.
\end{align}

Let us look at
\begin{align}
  \label{eq:1322}
  \frac{d^{n-2}}{dv^{n-2}}\left((\ccirc_+-c_{+0})\frac{dz}{dv}\right).
\end{align}
Using \eqref{eq:1327} with $n-1$ in the role of $i$ and in the resulting expression \eqref{eq:1328} with $n$ and $n-1$ in the role of $k$ we obtain
\begin{align}
  \label{eq:1329}
  \frac{d^{n-1}z}{dv^{n-1}}=\bar{P}_1+nv\int_0^v\frac{d^ny}{dv^n}(v')dv'-(n-1)\int_0^vv'\frac{d^ny}{dv^n}(v')dv'.
\end{align}
Using now \eqref{eq:1329} with $n-1$ in the role of $n$ and using \eqref{eq:1328} with $n$ in the role of $k$ for the resulting integrands yields
\begin{align}
  \label{eq:1330}
  \frac{d^{n-2}z}{dv^{n-2}}=\bar{P}_2+\frac{n}{2}v^2\int_0^v\frac{d^ny}{dv^n}(v')dv'-(n-1)v\int_0^vv'\frac{d^ny}{dv^n}(v')dv'+\frac{n-2}{2}\int_0^vv'^2\frac{d^ny}{dv^n}(v')dv'.
\end{align}
We rewrite \eqref{eq:1329} and \eqref{eq:1330} as
\begin{align}
  \label{eq:1331}
  \frac{d^{n-1}z}{dv^{n-1}}&=\bar{P}_1+nvI_{n,0}-(n-1)I_{n,1},\\
  \label{eq:1332}
  \frac{d^{n-2}z}{dv^{n-2}}&=\bar{P}_2+\frac{n}{2}v^2I_{n,0}-(n-1)vI_{n,1}+\frac{n-2}{2}I_{n,2},\\
  \label{eq:1333}
  \frac{d^{n-k}z}{dv^{n-k}}&=P_{k-1}:3\leq k\leq n-2,\\
  \label{eq:1337}
  \frac{dz}{dv}&=-1+P_{n-2,1},
\end{align}
where \eqref{eq:1333} follows directly from \eqref{eq:1298} and \eqref{eq:1337} follows from $z=vy$, $y(0)=-1$.

Since
\begin{align}
  \label{eq:1334}
  \ccirc_+(v)=c_+(\alpha_-(v),\beta_-(v)),
\end{align}
we have, in view of \eqref{eq:1303}, \eqref{eq:1304},
\begin{align}
  \label{eq:1335}
  \frac{d^m}{dv^m}(\ccirc_+-c_{+0})=P_{n-m-1}.
\end{align}
Taking the derivative of \eqref{eq:1334} we obtain
\begin{align}
  \label{eq:1338}
  \frac{d}{dv}(\ccirc_+-c_{+0})&=\left\{\pp{c_+}{\alpha}(\alpha_-,\beta_-)\pp{\alpha^\ast}{w}(f,z)+\pp{c_+}{\beta}(\alpha_-,\beta_-)\pp{\beta^\ast}{w}(f,z)\right\}\frac{dz}{dv}\notag\\
&\qquad+\left\{\pp{c_+}{\alpha}(\alpha_-,\beta_-)\pp{\alpha^\ast}{t}(f,z)+\pp{c_+}{\beta}(\alpha_-,\beta_-)\pp{\beta^\ast}{t}(f,z)\right\}\frac{df}{dv}.
\end{align}
Taking $n-3$ derivatives of this and making use of \eqref{eq:1298}, \eqref{eq:1213}, \eqref{eq:1303}, \eqref{eq:1304}, \eqref{eq:1337} and
\begin{align}
  \label{eq:1339}
  \cp{\pp{\beta^\ast}{w}}=0,\qquad \cp{\pp{c_+}{\alpha}}\cp{\pp{\alpha^\ast}{w}}=\kappa,
\end{align}
we obtain
\begin{align}
  \label{eq:1343}
  \frac{d^{n-2}}{dv^{n-2}}(\ccirc_+-c_{+0})=P_2+\kappa \frac{d^{n-2}z}{dv^{n-2}}.
\end{align}
Substituting \eqref{eq:1332} we obtain
\begin{align}
  \label{eq:1336}
  \frac{d^{n-2}}{dv^{n-2}}(\ccirc_+-c_{+0})&=P_2+\kappa \left\{\frac{n}{2}v^2I_{n,0}-(n-1)vI_{n,1}+\frac{n-2}{2}I_{n,2}\right\},\\
  \label{eq:1340}
  \frac{d^{n-k}}{dv^{n-k}}(\ccirc_+-c_{+0})&=P_{k-1}:3\leq k\leq n-2,\\
  \label{eq:1341}
  \frac{d}{dv}(\ccirc_+-c_{+0})&=-\kappa+P_{n-2,1},\\
  \label{eq:1342}
  \ccirc_+-c_{+0}&=-\kappa v+P_{n-1,2},
\end{align}
where \eqref{eq:1340} is \eqref{eq:1335} with $m=n-k$, \eqref{eq:1341} follows from \eqref{eq:1338}, \eqref{eq:1339} and \eqref{eq:1342} follows from $\ccirc_+(0)=c_{+0}$. We note that \eqref{eq:1331}, \ldots, \eqref{eq:1337} and \eqref{eq:1336}, \ldots, \eqref{eq:1342} express the derivatives of $z$ and $\ccirc_+-c_{+0}$ to $\Landau(v^3)$ for $n\geq 5$. In the case $n=4$ the derivatives of $z$ and $\ccirc_+-c_{+0}$ to $\Landau(v^3)$ are given by the same expressions with the exceptions of \eqref{eq:1333} and \eqref{eq:1340} which have to be excluded. In the case $n=3$ also \eqref{eq:1337} and \eqref{eq:1341} have to be excluded, i.e.~in the case $n=3$ the derivatives of $z$ and $\ccirc_+-c_{+0}$ to $\Landau(v^3)$ are given by \eqref{eq:1331}, \eqref{eq:1332}, \eqref{eq:1336}, \eqref{eq:1342}. Finally, in the case $n=2$ the derivative of $z$ to $\Landau(v^3)$ is given by \eqref{eq:1331} and $\ccirc_+-c_{+0}$ to $\Landau(v^3)$ is given by \eqref{eq:1336}. The fact that the latter is true is seen from the fact that \eqref{eq:1332} is also true for $n=2$ and then using this in \eqref{eq:1343} which in the case $n=2$ is the Taylor expansion of $\ccirc_+-c_{+0}$ to $\Landau(v^3)$.

Integrating by parts we have
\begin{align}
  \label{eq:1466}
  I_{n,m}=\int_0^vv'^m\frac{d^ny}{dv^n}(v')dv'=v^m\frac{d^{n-1}y}{dv^{n-1}}-\int_0^vmv'^{m-1}\frac{d^{n-1}y}{dv^{n-1}}(v')dv'.
\end{align}
By the inductive hypothesis this is $\Landau(v^m)$. In the case $m=0$ we have
\begin{align}
  \label{eq:1500}
  I_{n,0}=\int_0^v\frac{d^ny}{dv^n}(v')dv'=\frac{d^{n-1}y}{dv^{n-1}}-\cp{\frac{d^{n-1}y}{dv^{n-1}}},
\end{align}
which is $\Landau(1)$ by the inductive hypothesis. \eqref{eq:1466} and \eqref{eq:1500} imply
\begin{align}
  \label{eq:1346}
    v^kI_{n,m}=\Landau(v^{m+k}),\qquad v^k\hat{I}_{n,m}=\Landau(v^k).
\end{align}

Now, in the case $n\geq 5$ we have
\begin{align}
  \label{eq:1344}
  \frac{d^{n-2}}{dv^{n-2}}\left((\ccirc_+-c_{+0})\frac{dz}{dv}\right)&=\sum_{i=0}^{n-2}\binom{n-2}{i}\frac{d^i}{dv^i}(\ccirc_+-c_{+0})\frac{d^{n-1-i}z}{dv^{n-1-i}}\notag\\
&=\sum_{i=2}^{n-3}\binom{n-2}{i}\frac{d^i}{dv^i}(\ccirc_+-c_{+0})\frac{d^{n-1-i}z}{dv^{n-1-i}}\notag\\
&\qquad+(\ccirc_+-c_{+0})\frac{d^{n-1}z}{dv^{n-1}}+(n-2)\frac{d}{dv}(\ccirc_+-c_{+0})\frac{d^{n-2}z}{dv^{n-2}}+\frac{d^{n-2}}{dv^{n-2}}(\ccirc_+-c_{+0})\frac{dz}{dv}.
\end{align}
All terms in the sum are products of terms of the form \eqref{eq:1333} and \eqref{eq:1340} therefore they are all a $P_2$. For the second line in \eqref{eq:1344} we make use of \eqref{eq:1331}, \eqref{eq:1332}, \eqref{eq:1337} and \eqref{eq:1336}, \eqref{eq:1341}, \eqref{eq:1342} together with \eqref{eq:1346}. Therefore, we obtain in the case $n\geq 5$,
\begin{align}
  \label{eq:1345}
  \frac{d^{n-2}}{dv^{n-2}}\left((\ccirc_+-c_{+0})\frac{dz}{dv}\right)=P_2-\kappa J_n,
\end{align}
where we defined
\begin{align}
  \label{eq:1452}
  J_n\coloneqq\frac{n(n+1)}{2} v^2I_{n,0}-n(n-1) vI_{n,1}+\frac{(n-1)(n-2)}{2}I_{n,2}.
\end{align}

In the case $n=4$ the sum in \eqref{eq:1344} is not present and \eqref{eq:1333}, \eqref{eq:1340} are not needed. Therefore, \eqref{eq:1345} is valid in the case $n=4$ as well.

Let us look at the case $n=3$. From \eqref{eq:1332} and $\cp{dz/dv}=-1$ we have
\begin{align}
  \label{eq:1350}
  \frac{dz}{dv}=-1+\bar{P}_{2,1}+\frac{1}{2}\left\{3v^2I_{3,0}-4vI_{3,1}+I_{3,2}\right\}.
\end{align}
Using this together with \eqref{eq:1339} in \eqref{eq:1338} yields
\begin{align}
  \label{eq:1351}
  \frac{d}{dv}(\ccirc_+-c_{+0})=-\kappa+P_{2,1}+\frac{\kappa}{2}\left\{3v^2I_{3,0}-4vI_{3,1}+I_{3,2}\right\}.
\end{align}
Using now \eqref{eq:1331} in the case $n=3$ together with \eqref{eq:1342}, \eqref{eq:1350}, \eqref{eq:1351} we obtain
\begin{align}
  \label{eq:1352}
  \frac{d}{dv}\left((\ccirc_+-c_{+0})\frac{dz}{dv}\right)=P_2-\kappa J_3,
\end{align}
which is \eqref{eq:1345} with $3$ in the role of $n$. Therefore, \eqref{eq:1345} is valid in the case $n=3$ as well.

In the case $n=2$ we have from \eqref{eq:1331} together with $\cp{dz/dv}=-1$
\begin{align}
  \label{eq:1353}
  \frac{dz}{dv}=-1+\bar{P}_{1,1}+2vI_{2,0}-I_{2,1},
\end{align}
while from \eqref{eq:1336} in the case $n=2$ together with $\cp{\ccirc_+}=c_{+0}$, $\cp{d(\ccirc_+-c_{+0})/dv}=-\kappa$,
\begin{align}
  \label{eq:1354}
  \ccirc_+-c_{+0}=-\kappa v+P_{2,2}+\kappa\left\{v^2I_{2,0}-vI_{2,1}\right\}.
\end{align}
From \eqref{eq:1353} and \eqref{eq:1354} we obtain
\begin{align}
  \label{eq:1355}
  (\ccirc_+-c_{+0})\frac{dz}{dv}=\kappa v+P_{2,2}-\kappa J_2,
\end{align}
which agrees with \eqref{eq:1345} in the case $n=2$. We therefore that \eqref{eq:1345} is valid for $n\geq 2$.

Integrating \eqref{eq:1345} we obtain
\begin{align}
  \label{eq:1366}
  \frac{d^k}{dv^k}\left((\ccirc_+-c_{+0})\frac{dz}{dv}\right)&=P_{n-k-1}:k\leq n-3,\\
  \label{eq:1371}
  (\ccirc_+-c_{+0})\frac{dz}{dv}&=P_{n-1,1}.
\end{align}
We note that these are valid in the case $n\geq 4$. In the case $n=3$ we have to make use of \eqref{eq:1352} instead of \eqref{eq:1366} (\eqref{eq:1371} stays valid), while in the case $n=2$ we have to make use of \eqref{eq:1355} alone.

Now we look at
\begin{align}
  \label{eq:1356}
  \frac{d^{n-2}}{dv^{n-2}}\left((\ccirc_+-c_{+0})\frac{df}{dv}\right).
\end{align}
From \eqref{eq:1213} we have
\begin{align}
  \label{eq:1358}
  \frac{d^{n-1}f}{dv^{n-1}}&=P_1,\\
  \label{eq:1359}
  \frac{d^{n-2}f}{dv^{n-2}}&=P_2,\\
  \label{eq:1360}
  \frac{d^{n-k}f}{dv^{n-k}}&=P_k:3\leq k\leq n-2,\\
  \label{eq:1361}
  \frac{df}{dv}&=P_{n-1,1}.
\end{align}
We note that \eqref{eq:1358}, \ldots, \eqref{eq:1361} express derivatives of $f$ for $n\geq 5$. In the case $n=4$ the derivatives of $f$ are expressed by the same expressions with the exceptions of \eqref{eq:1360} which has to be excluded. In the case $n=3$ the second and first derivative of $f$ are given by \eqref{eq:1358} and \eqref{eq:1361} respectively. Finally in the the case $n=2$ the first derivative of $f$ is given by \eqref{eq:1361}.

We start with the case $n\geq 5$. We have
\begin{align}
  \label{eq:1357}
  \frac{d^{n-2}}{dv^{n-2}}\left((\ccirc_+-c_{+0})\frac{df}{dv}\right)&=\sum_{i=0}^{n-2}\binom{n-2}{i}\frac{d^i}{dv^i}(\ccirc_+-c_{+0})\frac{d^{n-1-i}f}{dv^{n-1-i}}\notag\\
&=\sum_{i=2}^{n-3}\binom{n-2}{i}\frac{d^i}{dv^i}(\ccirc_+-c_{+0})\frac{d^{n-1-i}f}{dv^{n-1-i}}\notag\\
&\qquad+(\ccirc_+-c_{+0})\frac{d^{n-1}f}{dv^{n-1}}+(n-2)\frac{d}{dv}(\ccirc_+-c_{+0})\frac{d^{n-2}f}{dv^{n-2}}+\frac{d^{n-2}}{dv^{n-2}}(\ccirc_+-c_{+0})\frac{df}{dv}.
\end{align}
All terms in the sum are products of terms of the form \eqref{eq:1340} and \eqref{eq:1360} therefore they are all a $P_2$. For the second line in \eqref{eq:1357} we make use of \eqref{eq:1336}, \eqref{eq:1341}, \eqref{eq:1342} and \eqref{eq:1358}, \eqref{eq:1359}, \eqref{eq:1361} together with \eqref{eq:1346}. We see that in the case $n\geq 5$ all terms in \eqref{eq:1357} are a $P_2$ and therefore so is \eqref{eq:1356}. In the case $n=4$ the sum in \eqref{eq:1357} is not present, hence \eqref{eq:1340} and \eqref{eq:1360} are not needed. Therefore, also in the case $n=4$ \eqref{eq:1356} is a $P_2$. In the case $n=3$ we use \eqref{eq:1351} together with \eqref{eq:1342}, \eqref{eq:1358}, \eqref{eq:1361}. We obtain that also in the case $n=3$ \eqref{eq:1356} is a $P_2$. Finally in the case $n=2$ we use \eqref{eq:1354} and \eqref{eq:1361}. We conclude that, for $n\geq 2$,
\begin{align}
  \label{eq:1362}
  \frac{d^{n-2}}{dv^{n-2}}\left((\ccirc_+-c_{+0})\frac{df}{dv}\right)=P_2.
\end{align}

By integration we obtain from \eqref{eq:1362}
\begin{align}
  \label{eq:1363}
  \frac{d^k}{dv^k}\left((\ccirc_+-c_{+0})\frac{df}{dv}\right)&=P_{n-k}:2\leq k\leq n-2,\\
  \label{eq:1364}
  \frac{d}{dv}\left((\ccirc_+-c_{+0})\frac{df}{dv}\right)&=P_{n-1,1},\\
  \label{eq:1365}
  (\ccirc_+-c_{+0})\frac{df}{dv}&=P_{n,2}.
\end{align}
Here \eqref{eq:1363} is only valid for $n\geq 4$ and \eqref{eq:1364} is only valid for $n\geq 3$.

Let us define
\begin{align}
  \label{eq:1367}
  F_1&\coloneqq\pp{T^{tt}}{\alpha}(\alpha_-(f,z),\beta_-(f,z))\pp{\alpha^\ast}{w}(f,z),\\
  \label{eq:1368}
  F_2&\coloneqq\pp{T^{tt}}{\alpha}(\alpha_-(f,z),\beta_-(f,z))\pp{\alpha^\ast}{t}(f,z).
\end{align}
From \eqref{eq:1298}, \eqref{eq:1213} we have
\begin{align}
  \label{eq:1369}
  \frac{d^kF}{dv^k}&=P_{n-k-1}:k\leq n-2 \quad\textrm{for}\quad F\in \{F_1,F_2\},\\
  \label{eq:1373}
  F_1&=\dot{\alpha}_0\cp{\pp{T^{tt}}{\alpha}}+P_{n-1,1},
\end{align}
where for the second we used $\cp{\partial \alpha^\ast/\partial w}=\dot{\alpha}_0$.

Now, in the case $n\geq 4$ we have
\begin{align}
  \label{eq:1370}
  \frac{d^{n-2}}{dv^{n-2}}\left(F_1(\ccirc_+-c_{+0})\frac{dz}{dv}\right)&=\sum_{i=0}^{n-2}\binom{n-2}{i}\frac{d^iF_1}{dv^i}\frac{d^{n-2-i}}{dv^{n-2-i}}\left((\ccirc_+-c_{+0})\frac{dz}{dv}\right)\notag\\
&=\sum_{i=1}^{n-3}\binom{n-2}{i}\frac{d^iF_1}{dv^i}\frac{d^{n-2-i}}{dv^{n-2-i}}\left((\ccirc_+-c_{+0})\frac{dz}{dv}\right)\notag\\
&\qquad +F_1\frac{d^{n-2}}{dv^{n-2}}\left((\ccirc_+-c_{+0})\frac{dz}{dv}\right)+\frac{d^{n-2}F_1}{dv^{n-2}}(\ccirc_+-c_{+0})\frac{dz}{dv}.
\end{align}
All terms in the sum are products of terms of the form \eqref{eq:1366} and \eqref{eq:1369} therefore they are all a $P_2$. For the first term in the second line of \eqref{eq:1370} we use \eqref{eq:1345} together with \eqref{eq:1373} and for the second term in the second line of \eqref{eq:1370} we use \eqref{eq:1371}. We find
\begin{align}
  \label{eq:1374}
  \frac{d^{n-2}}{dv^{n-2}}\left(F_1(\ccirc_+-c_{+0})\frac{dz}{dv}\right)=P_2-\dot{\alpha}_0\kappa\cp{\pp{T^{tt}}{\alpha}} J_n.
\end{align}

In the case $n=3$ we have
\begin{align}
  \label{eq:1375}
  \frac{d}{dv}\left(F_1(\ccirc_+-c_{+0})\frac{dz}{dv}\right)=\frac{dF_1}{dv}(\ccirc_+-c_{+0})\frac{dz}{dv}+F\frac{d}{dv}\left((\ccirc_+-c_{+0})\frac{dz}{dv}\right).
\end{align}
Making use of \eqref{eq:1352}, \eqref{eq:1371}, \eqref{eq:1369}, \eqref{eq:1373} we see that \eqref{eq:1374} is also valid in the case $n=3$. In the case $n=2$ we make use of \eqref{eq:1355}, \eqref{eq:1373} and see that \eqref{eq:1374} is valid in the case $n=2$ as well but we have in particular
\begin{align}
  \label{eq:1448}
  F_1(\ccirc_+-c_{+0})\frac{dz}{dv}=P_{2,2}+\dot{\alpha}_0\kappa\cp{\pp{T^{tt}}{\alpha}}(v-J_2).
\end{align}

Now, in the case $n\geq 4$ we have
\begin{align}
  \label{eq:1376}
  \frac{d^{n-2}}{dv^{n-2}}\left(F_2(\ccirc_+-c_{+0})\frac{df}{dv}\right)&=\sum_{i=0}^{n-2}\binom{n-2}{i}\frac{d^iF_2}{dv^i}\frac{d^{n-2-i}}{dv^{n-2-i}}\left((\ccirc_+-c_{+0})\frac{df}{dv}\right)\notag\\
&=\sum_{i=0}^{n-4}\binom{n-2}{i}\frac{d^iF_2}{dv^i}\frac{d^{n-2-i}}{dv^{n-2-i}}\left((\ccirc_+-c_{+0})\frac{df}{dv}\right)\notag\\
&\qquad+(n-2)\frac{d^{n-3}F}{dv^{n-3}}\frac{d}{dv}\left((\ccirc_+-c_{+0})\frac{df}{dv}\right)+\frac{d^{n-2}F_2}{dv^{n-2}}(\ccirc_+-c_{+0})\frac{df}{dv}.
\end{align}
All terms in the sum are products of terms of the form \eqref{eq:1363} and \eqref{eq:1369} therefore they are all at least a $P_2$. For the first term in the last line of \eqref{eq:1376} we use \eqref{eq:1364} together with \eqref{eq:1369}. For the second term in the last line in \eqref{eq:1376} we use \eqref{eq:1365} together with \eqref{eq:1369}. We find in the case $n\geq 4$,
\begin{align}
  \label{eq:1377}
  \frac{d^{n-2}}{dv^{n-2}}\left(F_2(\ccirc_+-c_{+0})\frac{df}{dv}\right)=P_2.
\end{align}

In the case $n=3$ the sum in \eqref{eq:1376} is not present and the second line in \eqref{eq:1376} is dealt with in the same way as in the case $n\geq 4$. We see that \eqref{eq:1377} is also valid in the case $n=3$. In the case $n=2$ we use again \eqref{eq:1365} together with \eqref{eq:1369}. We see that \eqref{eq:1377} is also valid in the case $n=2$ and we have in particular
\begin{align}
  \label{eq:1447}
  F_2(\ccirc_+-c_{+0})\frac{df}{dv}=P_{2,2}.
\end{align}

From \eqref{eq:1374}, \eqref{eq:1377}, in view of \eqref{eq:1323} and the definitions \eqref{eq:1367}, \eqref{eq:1368}, we deduce, for $n\geq 2$,
\begin{align}
  \label{eq:1380}
  \frac{d^{n-2}T_2}{dv^{n-2}}=P_2+\dot{\alpha}_0\kappa\cp{\pp{T^{tt}}{\alpha}} J_n.
\end{align}
In the case $n=2$ we have in particular, from \eqref{eq:1448} \eqref{eq:1447},
\begin{align}
  \label{eq:1449}
  T_2=P_{2,2}-\dot{\alpha}_0\kappa\cp{\pp{T^{tt}}{\alpha}}(v-J_2).
\end{align}

We turn to
\begin{align}
  \label{eq:1382}
  \frac{d^{n-2}T_1}{dv^{n-2}},
\end{align}
where (see \eqref{eq:1211}, \eqref{eq:1398}, \eqref{eq:1381})
\begin{align}
  \label{eq:1383}
  T_1=(\bar{c}_+-c_{+0})\pp{T^{tt}}{\alpha}(\alpha_+,\beta_+)\frac{d\alpha_+}{dv}=F\frac{d\alpha_+}{dv},
\end{align}
where we defined
\begin{align}
  \label{eq:1384}
  F\coloneqq(\bar{c}_+-c_{+0})\pp{T^{tt}}{\alpha}(\alpha_+,\beta_+).
\end{align}
Since $\bar{c}_+(v)=c_+(\alpha_+(v),\beta_+(v))$, in view of \eqref{eq:1300}, \eqref{eq:1313} we have
\begin{align}
  \label{eq:1385}
  F=P_{n,1},\qquad \frac{d^mF}{dv^m}=P_{n-m}:1\leq m\leq n-2.
\end{align}
From this together with \eqref{eq:1300} we obtain
\begin{align}
  \label{eq:1386}
  \frac{d^{n-2}T_1}{dv^{n-2}}=P_2.
\end{align}

We now look at the case $n=2$ in more detail. In the case $n=2$ we use
\begin{align}
  \label{eq:1438}
  \pp{T^{tt}}{\alpha}(\alpha_+,\beta_+)=\cp{\pp{T^{tt}}{\alpha}}+P_{2,1},
\end{align}
together with
\begin{align}
  \label{eq:1439}
  \bar{c}_+-c_{+0}=\kappa v+P_{2,2},
\end{align}
to deduce (in agreement with the first of \eqref{eq:1385})
\begin{align}
  \label{eq:1440}
  F=\cp{\pp{T^{tt}}{\alpha}}\kappa v+P_{2,2}.
\end{align}
From \eqref{eq:1300} we have
\begin{align}
  \label{eq:1421}
  \frac{d\alpha_+}{dv}=\dot{\alpha}_0+P_{1,1}.
\end{align}
Therefore, in the case $n=2$, we obtain
\begin{align}
  \label{eq:1443}
  T_1=\dot{\alpha}_0\kappa\cp{\pp{T^{tt}}{\alpha}} v+P_{2,2}.
\end{align}

We turn to $T_3+T_4$ (see \eqref{eq:1211}, \eqref{eq:1398}, \eqref{eq:1381}).
\begin{align}
  \label{eq:1387}
  T_3+T_4&=(\bar{c}_--c_{+0})\pp{T^{tt}}{\beta}(\alpha_+,\beta_+)\frac{d\beta_+}{dv}-(\ccirc_--c_{+0})\pp{T^{tt}}{\beta}(\alpha_-,\beta_-)\frac{d\beta_-}{dv}\notag\\
&=\left\{(\bar{c}_--c_{+0})\pp{T^{tt}}{\beta}(\alpha_+,\beta_+)-(\ccirc_--c_{+0})\pp{T^{tt}}{\beta}(\alpha_-,\beta_-)\right\}\frac{d\beta_-}{dv}\notag\\
&\qquad +(\bar{c}_--c_{+0})\pp{T^{tt}}{\beta}(\alpha_+,\beta_+)\frac{d\jump{\beta}}{dv}.
\end{align}
Defining
\begin{align}
  F&\coloneqq(c_--c_{+0})\pp{T^{tt}}{\beta},\\
  \label{eq:1388}
  \bar{F}&\coloneqq F(\alpha_+,\beta_+),\\
  \label{eq:1929}
  \Fcirc&\coloneqq F(\alpha_-,\beta_-),
\end{align}
\eqref{eq:1387} becomes
\begin{align}
  \label{eq:1389}
  T_3+T_4=(\bar{F}-\Fcirc)\frac{d\beta_-}{dv}+\bar{F}\frac{d\jump{\beta}}{dv}.
\end{align}
From \eqref{eq:1300}, \eqref{eq:1303}, \eqref{eq:1304}, \eqref{eq:1313} we have
\begin{align}
  \label{eq:1390}
  (\bar{F}-\Fcirc)=P_{n-1,1},\qquad \frac{d^m}{dv^m}(\bar{F}-\Fcirc)=P_{n-m-1}:1\leq m\leq n-2.
\end{align}

Now, in the case $n\geq 4$,
\begin{align}
  \label{eq:1391}
  \frac{d^{n-2}}{dv^{n-2}}\left((\bar{F}-\Fcirc)\frac{d\beta_-}{dv}\right)&=\sum_{i=0}^{n-2}\binom{n-2}{i}\frac{d^i}{dv^i}(\bar{F}-\Fcirc)\frac{d^{n-1-i}\beta_-}{dv^{n-1-i}}\notag\\
&=\sum_{i=1}^{n-3}\binom{n-2}{i}\frac{d^i}{dv^i}(\bar{F}-\Fcirc)\frac{d^{n-1-i}\beta_-}{dv^{n-1-i}}\notag\\
&\qquad +(\bar{F}-\Fcirc)\frac{d^{n-1}\beta_-}{dv^{n-1}}+\frac{d^{n-2}}{dv^{n-2}}(\bar{F}-\Fcirc)\frac{d\beta_-}{dv}.
\end{align}
All terms in the sum are products of terms of the form of the second case of \eqref{eq:1304} and the second of \eqref{eq:1390}. For the second line in \eqref{eq:1391} we use the first case of \eqref{eq:1304} and the first of \eqref{eq:1390}. We find that in the case $n\geq 4$,
\begin{align}
  \label{eq:1392}
  \frac{d^{n-2}}{dv^{n-2}}\left((\bar{F}-\Fcirc)\frac{d\beta_-}{dv}\right)=P_2.
\end{align}

In the case $n=3$ the sum in \eqref{eq:1391} is not present. For the remaining terms we argue as in the case $n\geq 4$ and find again \eqref{eq:1392}. In the case $n=2$ only the product of the first case of \eqref{eq:1304} and the first of \eqref{eq:1390} is present and we find
\begin{align}
  \label{eq:1444}
  (\bar{F}-\Fcirc)\frac{d\beta_-}{dv}=P_{2,2},
\end{align}
in agreement with \eqref{eq:1392}. In view of \eqref{eq:1300}, \eqref{eq:1313}, we have
\begin{align}
  \label{eq:1394}
  \frac{d^m\bar{F}}{dv^m}=P_{n-m}.
\end{align}
Using this together with \eqref{eq:1311} we obtain, for $n\geq 3$,
\begin{align}
  \label{eq:1395}
  \frac{d^{n-2}}{dv^{n-2}}\left(\bar{F}\frac{d\jump{\beta}}{dv}\right)&=\sum_{i=0}^{n-2}\binom{n-2}{i}\frac{d^i\bar{F}}{dv^i}\frac{d^{n-1-i}\jump{\beta}}{dv^{n-1-i}}\notag\\
&=\sum_{i=0}^{n-2}P_{n-i}P_{2+i}=P_2.
\end{align}
In the case $n=2$ we have, using the second case of \eqref{eq:1311} together with \eqref{eq:1394},
\begin{align}
  \label{eq:1445}
  \bar{F}\frac{d\jump{\beta}}{dv}=P_{2,2}.
\end{align}

From \eqref{eq:1392}, \eqref{eq:1444}, \eqref{eq:1395}, \eqref{eq:1445} we obtain that for $n\geq 2$,
\begin{align}
  \label{eq:1396}
  \frac{d^{n-2}}{dv^{n-2}}(T_3+T_4)=P_2,
\end{align}
and, in particular in the case $n=2$,
\begin{align}
  \label{eq:1446}
  T_3+T_4=P_{2,2}.
\end{align}

In view of \eqref{eq:1398}, we deduce from \eqref{eq:1380}, \eqref{eq:1386} and \eqref{eq:1396}
\begin{align}
  \label{eq:1399}
  \frac{d^{n-2}N}{dv^{n-2}}=P_2-\dot{\alpha}_0\kappa\cp{\pp{T^{tt}}{\alpha}} J_n.
\end{align}
In the case $n=2$ we have in particular, from \eqref{eq:1449}, \eqref{eq:1443} and \eqref{eq:1446},
\begin{align}
  \label{eq:1450}
  N=P_{2,2}-\dot{\alpha}_0\kappa\cp{\pp{T^{tt}}{\alpha}}J_2.
\end{align}

Now we look at
\begin{align}
  \label{eq:1400}
  \frac{d^{n-2}}{dv^{n-2}}\jump{T^{tt}}.
\end{align}
We first restrict ourselves to the case $n\geq 3$. Defining
\begin{align}
  \label{eq:1401}
  F_\alpha&\coloneqq\pp{T^{tt}}{\alpha}(\alpha_+,\beta_+),\\
  \label{eq:1930}
  F_\beta&\coloneqq\pp{T^{tt}}{\beta}(\alpha_+,\beta_+),\\
  \label{eq:1931}
  F_t&\coloneqq F_t'(f,z),\qquad\textrm{where}\qquad F_t'\coloneqq\pp{T^{tt}}{\alpha}(\alpha^\ast,\beta^\ast)\pp{\alpha^\ast}{t}+\pp{T^{tt}}{\beta}(\alpha^\ast,\beta^\ast)\pp{\beta^\ast}{t},\\
  \label{eq:1932}
  F_w&\coloneqq F_w'(f,z),\qquad\textrm{where}\qquad F_w'\coloneqq\pp{T^{tt}}{\alpha}(\alpha^\ast,\beta^\ast)\pp{\alpha^\ast}{w}+\pp{T^{tt}}{\beta}(\alpha^\ast,\beta^\ast)\pp{\beta^\ast}{w},
\end{align}
we have
\begin{align}
  \label{eq:1402}
  \frac{d}{dv}\jump{T^{tt}}=F_\alpha\frac{d\alpha_+}{dv}+F_\beta\frac{d\beta_+}{dv}-F_t\frac{df}{dv}-F_w\frac{dz}{dv}.
\end{align}
From \eqref{eq:1298}, \eqref{eq:1213} we obtain
\begin{align}
  \label{eq:1403}
  \frac{d^mF_w}{dv^m}=P_{n-m-1},\qquad\frac{d^mF_t}{dv^m}=P_{n-m-1}.
\end{align}
In particular, from $\cp{\partial\alpha/\partial w}=\dot{\alpha}_0$ we have
\begin{align}
  \label{eq:1404}
  F_w=\dot{\alpha}_0\cp{\pp{T^{tt}}{\alpha}}+P_{n-1,1}.
\end{align}

Now, in the case $n\geq 4$ we have
\begin{align}
  \label{eq:1405}
  \frac{d^{n-3}}{dv^{n-3}}\left(F_w\frac{dz}{dv}\right)&=\sum_{i=0}^{n-3}\binom{n-3}{i}\frac{d^iF_w}{dv^i}\frac{d^{n-2-i}z}{dv^{n-2-i}}\notag\\
&=\sum_{i=1}^{n-3}\binom{n-3}{i}\frac{d^iF_w}{dv^i}\frac{d^{n-2-i}z}{dv^{n-2-i}}+F_w\frac{d^{n-2}z}{dv^{n-2}}.
\end{align}
In view of \eqref{eq:1298} and the first of \eqref{eq:1403} we see that all terms in the sum are a $P_2$. For the remaining term we use \eqref{eq:1332} together with \eqref{eq:1404}. Therefore, in the case $n\geq 4$,
\begin{align}
  \label{eq:1406}
  \frac{d^{n-3}}{dv^{n-3}}\left(F_w\frac{dz}{dv}\right)=P_2+\dot{\alpha}_0\cp{\pp{T^{tt}}{\alpha}}\left\{\frac{n}{2}v^2I_{n,0}-(n-1)vI_{n,1}+\frac{n-2}{2}I_{n,2}\right\}.
\end{align}

In the case $n=3$ we only have the last term in \eqref{eq:1405}. Arguing as in the case $n\geq 4$ we see that \eqref{eq:1406} is also valid in the case $n=3$. From \eqref{eq:1213} together with the second of \eqref{eq:1403} we obtain, in the case $n\geq 3$,
\begin{align}
  \label{eq:1408}
  \frac{d^{n-3}}{dv^{n-3}}\left(F_t\frac{df}{dv}\right)=\sum_{i=0}^{n-3}\binom{n-3}{i}\frac{d^iF_t}{dv^i}\frac{d^{n-2-i}f}{dv^{n-2-i}}=P_2.
\end{align}
In view of \eqref{eq:1300}, \eqref{eq:1313} we have
\begin{align}
  \label{eq:1409}
  \frac{d^mF_\alpha}{dv^m}=P_{n-m},\qquad\frac{d^mF_\beta}{dv^m}=P_{n-m}.
\end{align}
Therefore, in the case $n\geq 3$,
\begin{align}
  \label{eq:1410}
  \frac{d^{n-3}}{dv^{n-3}}\left(F_\alpha\frac{d\alpha_+}{dv}\right)&=\sum_{i=0}^{n-3}\binom{n-3}{i}\frac{d^iF_\alpha}{dv^i}\frac{d^{n-2-i}\alpha_+}{dv^{n-2-i}}=P_2,\\
  \label{eq:1411}
  \frac{d^{n-3}}{dv^{n-3}}\left(F_\beta\frac{d\beta_+}{dv}\right)&=\sum_{i=0}^{n-3}\binom{n-3}{i}\frac{d^iF_\beta}{dv^i}\frac{d^{n-2-i}\beta_+}{dv^{n-2-i}}=P_2.
\end{align}
From \eqref{eq:1406}, \eqref{eq:1408}, \eqref{eq:1410} and \eqref{eq:1411} we deduce, in the case $n\geq 3$,
\begin{align}
  \label{eq:1412}
  \frac{d^{n-2}}{dv^{n-2}}\jump{T^{tt}}=P_2-\dot{\alpha}_0\cp{\pp{T^{tt}}{\alpha}}K_n,
\end{align}
where
\begin{align}
  \label{eq:1454}
  K_n\coloneqq\frac{n}{2}v^2I_{n,0}-(n-1)vI_{n,1}+\frac{n-2}{2}I_{n,2}.
\end{align}

Now we look at the case $n=2$. From \eqref{eq:1401} and the first of \eqref{eq:1409} we have
\begin{align}
  \label{eq:1422}
  F_\alpha=\cp{\pp{T^{tt}}{\alpha}}+P_{1,1}.
\end{align}
Therefore, together with \eqref{eq:1421} and using also \eqref{eq:1313} and the second of \eqref{eq:1409}, we obtain
\begin{align}
  \label{eq:1413}
  F_\alpha\frac{d\alpha_+}{dv}=\dot{\alpha}_0\cp{\pp{T^{tt}}{\alpha}}+P_{1,1},\qquad F_\beta\frac{d\beta_+}{dv}=P_{1,1}.
\end{align}
From \eqref{eq:1213} and the second of \eqref{eq:1403} we have
\begin{align}
  \label{eq:1414}
  F_t\frac{df}{dv}=P_{1,1}.
\end{align}
For the derivative of $z$ in the case $n=2$ we use \eqref{eq:1331}. Since $y(0)=-1$, \eqref{eq:1331} in the case $n=2$ is
\begin{align}
  \label{eq:1423}
  \frac{dz}{dv}=-1+\bar{P}_{1,1}+2vI_{2,0}-I_{2,1}.
\end{align}
Together with \eqref{eq:1404} we obtain
\begin{align}
  \label{eq:1415}
  F_w\frac{dz}{dv}=-\dot{\alpha}_0\cp{\pp{T^{tt}}{\alpha}}+P_{1,1}+\dot{\alpha}_0\cp{\pp{T^{tt}}{\alpha}}(2vI_{2,0}-I_{2,1}).
\end{align}
Using \eqref{eq:1413}, \eqref{eq:1414} and \eqref{eq:1415} in \eqref{eq:1402} we obtain
\begin{align}
  \label{eq:1417}
  \frac{d\jump{T^{tt}}}{dv}=2\dot{\alpha}_0\cp{\pp{T^{tt}}{\alpha}}+P_{1,1}-\dot{\alpha}_0\cp{\pp{T^{tt}}{\alpha}}(2vI_{2,0}-I_{2,1}).
\end{align}
Since
\begin{align}
  \label{eq:1416}
  2\int_0^vv'I_{2,0}(v')dv'&=v^2I_{2,0}-I_{2,2},\\
  \label{eq:1419}
  \int_0^vI_{2,1}(v')dv'&=vI_{2,1}-I_{2,2},
\end{align}
we obtain from \eqref{eq:1417},
\begin{align}
  \label{eq:1418}
  \jump{T^{tt}}=P_{2,2}+\dot{\alpha}_0\cp{\pp{T^{tt}}{\alpha}}(2v-K_2),
\end{align}
where we made use of $\jump{T^{tt}(0)}=0$. From this we see that \eqref{eq:1412} is also valid for $n\geq 2$. In addition we see that $\jump{T^{tt}}$ is divisible by $v$.

We now integrate \eqref{eq:1399} $n-2$ times to obtain $N$. If we integrate $n-2$ times a $P_2$ we obtain a $P_n$. In view of the definition \eqref{eq:1452} we must calculate the $l$-fold iterated integral of
\begin{align}
  \label{eq:1457}
  v^{2-i}I_{n,i}:i=0,1,2.
\end{align}
Let
\begin{align}
  \label{eq:1455}
  f_l\coloneqq v^lf_0,
\end{align}
for an integrable function $f_0$ (see \eqref{eq:1459} below). We define
\begin{align}
  \label{eq:1456}
  g_0(v)&\coloneqq\int_0^vf_0(v')dv',\\
  \label{eq:1462}
  g_l(v)&\coloneqq\int_0^vf_l(v')dv'.
\end{align}
We integrate $l$ times the function $v^kg_0(v)$ and denote the result by $G_{k,l}$. We claim
\begin{align}
  \label{eq:1458}
  G_{k,l}=\frac{k!}{(k+l)!}v^{k+l}g_0+\sum_{m=1}^l\frac{(-1)^m}{(m-1)!(l-m)!}\frac{v^{l-m}}{(k+m)}g_{k+m}.
\end{align}

We prove the claim in \eqref{eq:1458} by induction. We start with $l=1$. Integrating by parts and making use of \eqref{eq:1456}, \eqref{eq:1462}, we obtain
\begin{align}
  \label{eq:1461}
  G_{k,1}&=\int_0^vv'^kg_0(v')dv'\notag\\
&=\frac{1}{k+1}\left(v^{k+1}g_0-\int_0^vv'^{k+1}\frac{dg_0}{dv}(v')dv'\right)\notag\\
&=\frac{1}{k+1}\left(v^{k+1}g_0-g_{k+1}\right),
\end{align}
which is \eqref{eq:1458} in the case $l=1$. Let then \eqref{eq:1458} hold for $l=1, \ldots, l$. We have
\begin{align}
  \label{eq:1463}
  \int_0^vv'^{k+l}g_0(v')dv'&=\frac{1}{k+l+1}\left(v^{k+l+1}g_0-g_{k+l+1}\right),\\
  \label{eq:1933}
  \int_0^vv'^{l-m}g_{k+m}(v')dv'&=\frac{1}{l-m+1}\left(v^{l-m+1}g_{k+m}-g_{k+l+1}\right).
\end{align}
Using the above to integrate \eqref{eq:1458} we obtain
\begin{align}
  \label{eq:1464}
  G_{k,l+1}&=\frac{k!}{(k+l+1)!}v^{k+l+1}g_0+\sum_{m=1}^l\frac{(-1)^m}{(m-1)!(l+1-m)!}\frac{v^{l+1-m}}{k+m}g_{k+m}\notag\\
&\qquad -\left\{\frac{k!}{(k+l+1)!}+\sum_{m=1}^l\frac{(-1)^m}{(m-1)!(l+1-m)!}\frac{1}{k+m}\right\}g_{k+l+1}.
\end{align}
Since the first line agrees already with the right hand side with $l+1$ in the role of $l$, except for the last term in the sum, the thing left to show is
\begin{align}
  \label{eq:1465}
  -\frac{k!}{(k+l+1)!}-\sum_{m=1}^l\frac{(-1)^m}{(m-1)!(l+1-m)!}\frac{1}{k+m}=\frac{(-1)^{l+1}}{l!}\frac{1}{k+l+1}.
\end{align}
We have
\begin{align}
  \label{eq:1467}
  -l!\left\{\sum_{m=1}^l\frac{(-1)^m}{(m-1)!(l+1-m)!}\frac{1}{k+m}+\frac{k!}{(k+l+1)!}\right\}&=-\sum_{m=1}^l\binom{l}{m-1}\frac{(-1)^m}{k+m}-\frac{l!k!}{(k+l+1)!}\notag\\
&=-\frac{k!}{(k+l+1)!}P_l,
\end{align}
where
\begin{align}
  \label{eq:1468}
  P_l\coloneqq l!+\sum_{m=1}^l\binom{l}{m-1}(-1)^m(k+1)\cdots {\widetilde{(k+m)}}\cdots(k+l+1),
\end{align}
where the tilde denotes omission. $P_l$ is a polynomial in $k$ of degree $l$. The coefficient of $k^l$ is
\begin{align}
  \label{eq:1469}
  \sum_{m=1}^l\binom{l}{m-1}(-1)^m=(-1)^l.
\end{align}
Since if $n\in \{1,\ldots,l\}$
\begin{align}
  \label{eq:1470}
  (-n+1)\cdots\widetilde{(-n+m)}\cdots(-n+l+1)=\left\{\begin{array}{ll} 0 & m\neq n,\\ (n-1)!(-1)^{n-1}(l-n+1)! & m=n,\end{array}\right.
\end{align}
we have
\begin{align}
  \label{eq:1471}
  \sum_{m=1}^l\binom{l}{m-1}(-1)^m(-n+1)\cdots\widetilde{(-n+m)}\cdots(-n+l+1)=-l!.
\end{align}
Therefore,
\begin{align}
  \label{eq:1472}
  P_l(-n)=0:n=1,\ldots,l,
\end{align}
i.e.~the roots of $P_l(k)$ are $k=-1,\ldots,-l$, which implies, together with the fact that the coefficient of $k^l$ in $P_l(k)$ is given by \eqref{eq:1469},
\begin{align}
  \label{eq:1473}
  P_l(k)=(-1)^l(k+1)\cdots(k+l).
\end{align}
Using this in \eqref{eq:1467} yields \eqref{eq:1465}. This completes the proof of \eqref{eq:1458}. We remark that since \eqref{eq:1458} was proved by induction on $l$, a positive integer, it also holds when $k$ is a negative integer, as long as $k+l<0$ and for $k<0$ the first factor in the first term of \eqref{eq:1458} is interpreted as
\begin{align}
  \label{eq:1725}
  \frac{k!}{(k+l)!}=\frac{1}{(k+l)\cdots(k+1)}=\frac{(-1)^l}{(-k-l)\cdots(-k-1)}=(-1)^l\frac{(-k-l-1)!}{(-k-1)!}.
\end{align}

We now set
\begin{align}
  \label{eq:1459}
  f_0=v^i\frac{d^ny}{dv^n}:i=0,1,2,
\end{align}
such that (see \eqref{eq:1462})
\begin{align}
  \label{eq:1474}
  g_0=I_{n,i}.
\end{align}
We also set $k=2-i$. Then $G_{k,l}$ given by \eqref{eq:1458} is the $l$-fold iterated integral of $v^kg_0=v^{2-i}I_{n,i}$ (cf.~\eqref{eq:1457}) and
\begin{align}
  \label{eq:1476}
  g_j(v)&=\int_0^vv'^{i+j}\frac{d^ny}{dv^n}(v')dv'\notag\\
&=I_{n,i+j}.
\end{align}
Therefore, the $l$-fold iterated integral of $v^{2-i}I_{n,i}$ is
\begin{align}
  \label{eq:1475}
  \frac{(2-i)!}{(2-i+l)!}v^{2-i+l}I_{n,i}+\sum_{m=1}^l\frac{(-1)^m}{(m-1)!(l-m)!}\frac{v^{l-m}}{(2-i+m)}I_{n,m+2}.
\end{align}
Using this in $J_n$, given by \eqref{eq:1452}, we obtain, after a straightforward computation, that the $l$-fold iterated integral of $J_n$ is
\begin{align}
  \label{eq:1477}
  v^{l+2}\sum_{m=0}^{l+2}\frac{(-1)^m(n-m)(n+1-m)}{m!(l+2-m)!}\hat{I}_{n,m},
\end{align}
where we recall the notation $\hat{I}_{n,m}=(1/v^m)I_{n,m}$. From \eqref{eq:1399} this gives $d^{n-l-2}N/dv^{n-l-2}$. Setting then $j=l+2$ we obtain
\begin{align}
  \label{eq:1478}
  \frac{d^{n-j}N}{dv^{n-j}}=P_j-\dot{\alpha}_0\kappa\cp{\pp{T^{tt}}{\alpha}}v^j\sum_{m=0}^j\frac{(-1)^m(n-m)(n+1-m)}{m!(j-m)!}\hat{I}_{n,m}.
\end{align}
In particular, setting $j=n$ yields (see \eqref{eq:1450})
\begin{align}
  \label{eq:1479}
  N=P_{n,2}-\dot{\alpha}_0\kappa \cp{\pp{T^{tt}}{\alpha}}v^n\sum_{m=0}^{n-1}\frac{(-1)^m(n+1-m)}{m!(n-1-m)!}\hat{I}_{n,m}.
\end{align}

Making use of \eqref{eq:1475}, we obtain, after a straightforward computation, that the $l$-fold iterated integral of $K_n$ (see the definition \eqref{eq:1454}) is
\begin{align}
  \label{eq:1492}
  v^{l+2}\sum_{m=0}^{l+2}\frac{(-1)^m(n-m)}{m!(l+2-m)!}\hat{I}_{n,m}.
\end{align}
Together with \eqref{eq:1412} this gives $d^{n-l-2}\jump{T^{tt}}/dv^{n-l-2}$. Setting then $j=l+2$ we obtain
\begin{align}
  \label{eq:1493}
  \frac{d^{n-j}}{dv^{n-j}}\jump{T^{tt}}=P_j-\dot{\alpha}_0\cp{\pp{T^{tt}}{\alpha}}v^j\sum_{m=0}^j\frac{(-1)^m(n-m)}{m!(j-m)!}\hat{I}_{n,m}.
\end{align}
In particular, setting $j=n$ yields (see \eqref{eq:1418})
\begin{align}
  \label{eq:1494}
  \jump{T^{tt}}=P_{n,2}+\dot{\alpha}_0\cp{\pp{T^{tt}}{\alpha}}\left(2v-v^n\sum_{m=0}^{n-1}\frac{(-1)^m}{m!(n-1-m)!}\hat{I}_{n,m}\right).
\end{align}

Since $\jump{T^{tt}}$ and $N$ are divisible by $v$, we define
\begin{align}
  \label{eq:1420}
  \widehat{\jump{T^{tt}}}\coloneqq\frac{\jump{T^{tt}}}{v},\qquad \hat{N}\coloneqq\frac{N}{v}.
\end{align}
Furthermore, we note that, in view of \eqref{eq:1450}, \eqref{eq:1418} and using (see \eqref{eq:1346})
\begin{align}
  \label{eq:1453}
  |vI_{2,0}|\leq Cv,\qquad |I_{2,1}|\leq Cv,
\end{align}
we have
\begin{align}
  \label{eq:1451}
  \hat{N}(0)=0,\qquad \widehat{\jump{T^{tt}}}(0)=2\dot{\alpha}_0\cp{\pp{T^{tt}}{\alpha}}.
\end{align}
We then have (see \eqref{eq:1381})
\begin{align}
  \label{eq:1426}
  b=-\frac{\hat{N}}{\widehat{\jump{T^{tt}}}}.
\end{align}

Let now $f$ be a $C^k$ function of $v$ such that $f(0)=0$. We define
\begin{align}
  \label{eq:1427}
  \hat{f}\coloneqq\frac{f}{v}.
\end{align}
We have
\begin{align}
  \label{eq:1428}
  \frac{d^k\hat{f}}{dv^k}=\sum_{i=0}^k\binom{k}{i}\frac{d^{k-i}}{dv^{k-i}}\left(\frac{1}{v}\right)\frac{d^if}{dv^i}.
\end{align}
Since
\begin{align}
  \label{eq:1429}
  \frac{d^m}{dv^m}\left(\frac{1}{v}\right)=\frac{(-1)^mm!}{v^{m+1}},
\end{align}
\eqref{eq:1428} is
\begin{align}
  \label{eq:1430}
  \frac{d^k\hat{f}}{dv^k}=\frac{k!}{v^{k+1}}A_kf,
\end{align}
where $A_k$ is the operator
\begin{align}
  \label{eq:1431}
  A_k\coloneqq\sum_{i=0}^k\frac{(-1)^{k-i}}{i!}v^i\frac{d^i}{dv^i},
\end{align}
which is homogeneous of degree zero relative to scaling.

In view of
\begin{align}
  \label{eq:1432}
  \frac{d^mv^k}{dv^m}=\left\{\begin{array}{ll}\frac{k!}{(k-m)!}v^{k-m}& m\leq k,\\0 & m>k,\end{array}\right.
\end{align}
we have
\begin{align}
  \label{eq:1433}
  A_kv^m=a_{k,m}v^m,
\end{align}
where
\begin{align}
  \label{eq:1434}
  a_{k,m}\coloneqq\sum_{i=0}^{\min\{m,k\}}\binom{m}{i}(-1)^{k-i}.
\end{align}
We see that $a_{k,m}=0$ for $m\leq k$, unless $m=0$. Therefore,
\begin{align}
  \label{eq:1435}
  A_k\bar{P}_{k,1}=0,
\end{align}
i.e.~the null space of $A_k$ is the space of polynomials of degree $k$ with no constant term, a $k$-dimensional space. In view of the above, applying $d^{n-2}/dv^{n-2}$ to $\hat{N}$, $\widehat{\jump{T^{tt}}}$, we obtain
\begin{align}
  \label{eq:1436}
  \frac{d^{n-2}\hat{N}}{dv^{n-2}}&=\frac{(n-2)!}{v^{n-1}}A_{n-2}N,\\
  \label{eq:1437}
  \frac{d^{n-2}}{dv^{n-2}}\widehat{\jump{T^{tt}}}&=\frac{(n-2)!}{v^{n-1}}A_{n-2}\jump{T^{tt}}.
\end{align}

Setting $k=n-2$ and $j=n-i$ in \eqref{eq:1431} we obtain
\begin{align}
  \label{eq:1480}
  A_{n-2}N=\sum_{j=2}^n\frac{(-1)^j}{(n-j)!}v^{n-j}\frac{d^{n-j}N}{dv^{n-j}}.
\end{align}
Since the null space of $A_{n-2}$ is the space of polynomials of degree $n-2$ with no constant term, only the terms corresponding to the powers $v^{n-1}$, $v^n$ in $P_{n,2}$ survive when we apply $A_{n-2}$ to $N$. Therefore, substituting \eqref{eq:1478} into \eqref{eq:1480}, we obtain
\begin{align}
  \label{eq:1481}
  A_{n-2}N=P_{n,n-1}-\dot{\alpha}_0\kappa\cp{\pp{T^{tt}}{\alpha}}v^n\sum_{j=2}^n\frac{(-1)^j}{(n-j)!}\sum_{m=0}^j\frac{(-1)^m(n-m)(n+1-m)}{m!(j-m)!}\hat{I}_{n,m}.
\end{align}
We rewrite the double sum as
\begin{align}
  \label{eq:1482}
  \sum_{m=0}^na_{n,m}\frac{(-1)^m(n-m)(n+1-m)}{m!}\hat{I}_{n,m},
\end{align}
where
\begin{align}
  \label{eq:1483}
  a_{n,m}\coloneqq\sum_{j=\max\{2,m\}}^n\frac{(-1)^j}{(n-j)!(j-m)!}.
\end{align}

Let us consider
\begin{align}
  \label{eq:1484}
  \tilde{a}_{n,m}=\sum_{j=m}^n\frac{(-1)^j}{(n-j)!(j-m)!}.
\end{align}
For $m\geq 2$ we have $\tilde{a}_{n,m}=a_{n,m}$, but for $m=0,1$ we have
\begin{align}
  \label{eq:1485}
  \tilde{a}_{n,0}&=a_{n,0}-\frac{1}{n(n-2)!},\\
  \label{eq:1776}
  \tilde{a}_{n,1}&=a_{n,1}-\frac{1}{(n-1)!}.
\end{align}
We set $k=j-m$ in \eqref{eq:1484} to obtain
\begin{align}
  \label{eq:1486}
  \tilde{a}_{n,m}&=\frac{(-1)^m}{(n-m)!}\sum_{k=0}^{n-m}\binom{n-m}{k}(-1)^k\notag\\
&=\frac{(-1)^m}{(n-m)!}(1-1)^{n-m}.
\end{align}
This vanishes except in the case $n=m$ and we have $\tilde{a}_{n,n}=(-1)^n$. I.e.
\begin{align}
  \label{eq:1603}
  \tilde{a}_{n,m}=\left\{
  \begin{array}{ll}
    0 & m\neq n,\\
    (-1)^n & m=n.
  \end{array}\right.
\end{align}
We conclude
\begin{align}
  \label{eq:1487}
  a_{n,m}=\left\{\begin{array}{ll}\frac{1}{n(n-2)!} & m=0, \\ \frac{1}{(n-1)!} & m=1, \\ (-1)^n & m=n, \\ 0 & m\neq 0,1,n.\end{array}\right.
\end{align}

Substituting \eqref{eq:1487} in \eqref{eq:1482} we conclude that the double sum in \eqref{eq:1481} is
\begin{align}
  \label{eq:1488}
  \frac{(n+1)}{(n-2)!}\hat{I}_{n,0}-\frac{n}{(n-2)!}\hat{I}_{n,1}.
\end{align}
Therefore, from \eqref{eq:1436},
\begin{align}
  \label{eq:1489}
  \frac{d^{n-2}\hat{N}}{dv^{n-2}}=P_1-\dot{\alpha}_0\kappa\cp{\pp{T^{tt}}{\alpha}}v\left\{(n+1)\hat{I}_{n,0}-n\hat{I}_{n,1}\right\}.
\end{align}
For $n=2$ we have from \eqref{eq:1479},
\begin{align}
  \label{eq:1501}
  \hat{N}=P_{1,1}-\dot{\alpha}_0\kappa\cp{\pp{T^{tt}}{\alpha}}v\left\{3\hat{I}_{2,0}-2\hat{I}_{2,1}\right\}.
\end{align}

Using the inductive hypothesis \eqref{eq:1197} we have
\begin{align}
  \label{eq:1490}
  \frac{d^{n-2}\hat{N}}{dv^{n-2}}=P_{0}.
\end{align}
Integrating this yields
\begin{align}
  \label{eq:1491}
  \frac{d^m\hat{N}}{dv^m}=P_{n-m-2}:1\leq m\leq n-3,\qquad \hat{N}=P_{n-2,1},
\end{align}
where for the second we used \eqref{eq:1479}. We note that the first of \eqref{eq:1491} is only valid for $n\geq 4$.

Setting $k=n-2$ and $j=n-i$ in \eqref{eq:1431} we obtain
\begin{align}
  \label{eq:1495}
  A_{n-2}\jump{T^{tt}}=\sum_{j=2}^n\frac{(-1)^j}{(n-j)!}v^{n-j}\frac{d^{n-j}}{dv^{n-j}}\jump{T^{tt}}.
\end{align}
Substituting \eqref{eq:1493} yields
\begin{align}
  \label{eq:1496}
  A_{n-2}\jump{T^{tt}}=P_{n,n-1}-\dot{\alpha}_0\cp{\pp{T^{tt}}{\alpha}}v^n\sum_{j=2}^n\frac{(-1)^j}{(n-j)!}\sum_{m=0}^j\frac{(-1)^m(n-m)}{m!(j-m)!}\hat{I}_{n,m}.
\end{align}
We rewrite the double sum as
\begin{align}
  \label{eq:1497}
  \sum_{m=0}^na_{n,m}\frac{(-1)^m(n-m)}{m!}\hat{I}_{n,m},
\end{align}
where the coefficients $a_{n,m}$ are defined by \eqref{eq:1483} and given by \eqref{eq:1487}. We find that the double sum in \eqref{eq:1496} is
\begin{align}
  \label{eq:1498}
  \frac{1}{(n-2)!}\left(\hat{I}_{n,0}-\hat{I}_{n,1}\right).
\end{align}
Therefore, from \eqref{eq:1437}
\begin{align}
  \label{eq:1499}
  \frac{d^{n-2}}{dv^{n-2}}\widehat{\jump{T^{tt}}}=P_1-\dot{\alpha}_0\cp{\pp{T^{tt}}{\alpha}}v\left(\hat{I}_{n,0}-\hat{I}_{n,1}\right).
\end{align}
For $n=2$ we have, in view of \eqref{eq:1418},
\begin{align}
  \label{eq:1504} 
  \widehat{\jump{T^{tt}}}=P_{1,1}+\dot{\alpha}_0\cp{\pp{T^{tt}}{\alpha}}\left\{2-v\left(\hat{I}_{2,0}-\hat{I}_{2,1}\right)\right\}.
\end{align}

Using the inductive hypothesis \eqref{eq:1197} we have
\begin{align}
  \label{eq:1502}
  \frac{d^{n-2}}{dv^{n-2}}\widehat{\jump{T^{tt}}}=P_0.
\end{align}
Integrating this yields
\begin{align}
  \label{eq:1503}
  \frac{d^{m}}{dv^{m}}\widehat{\jump{T^{tt}}}=P_{n-m-2}:1\leq m\leq n-3,\qquad \widehat{\jump{T^{tt}}}=P_{n-2,1}+2\dot{\alpha}_0\cp{\pp{T^{tt}}{\alpha}},
\end{align}
where for the second we used the second of \eqref{eq:1494}. We note that the first of \eqref{eq:1503} is only valid for $n\geq 4$.

We now go back to \eqref{eq:1381} and calculate $d^{n-2}b/dv^{n-2}$ using \eqref{eq:1489} and \eqref{eq:1499}. In the case $n\geq 4$ we have
\begin{align}
  \label{eq:1505}
  \frac{d^{n-2}b}{dv^{n-2}}&=-\sum_{m=0}^{n-2}\binom{n-2}{m}\frac{d^{n-2-m}}{dv^{n-2-m}}\left(\frac{1}{\widehat{\jump{T^{tt}}}}\right)\frac{d^m\hat{N}}{dv^m}\notag\\
&=-\sum_{m=1}^{n-3}\binom{n-2}{m}\frac{d^{n-2-m}}{dv^{n-2-m}}\left(\frac{1}{\widehat{\jump{T^{tt}}}}\right)\frac{d^m\hat{N}}{dv^m}\notag\\
&\qquad -\frac{d^{n-2}}{dv^{n-2}}\left(\frac{1}{\widehat{\jump{T^{tt}}}}\right)\hat{N}-\frac{1}{\widehat{\jump{T^{tt}}}}\frac{d^{n-2}\hat{N}}{dv^{n-2}}.
\end{align}
Each of the terms in the sum involves products of terms of the form given by the first of \eqref{eq:1491} and the first of \eqref{eq:1503}, therefore they are all a $P_1$. Making use of
\begin{align}
  \label{eq:1506}
  -\frac{d^{n-2}}{dv^{n-2}}\left(\frac{1}{\widehat{\jump{T^{tt}}}}\right)&=\frac{1}{\widehat{\jump{T^{tt}}}^2}\frac{d^{n-2}}{dv^{n-2}}\widehat{\jump{T^{tt}}}+P_1\notag\\
&=P_{1}-\frac{1}{4\dot{\alpha}_0\cp{\pp{T^{tt}}{\alpha}}}v\left(\hat{I}_{n,0}-\hat{I}_{n,1}\right),
\end{align}
where for the second equality we used \eqref{eq:1499} and the second of \eqref{eq:1503}, together with the second of \eqref{eq:1491}, we see that the first term in the second line of \eqref{eq:1505} is a $P_1$. For the second term in the second line of \eqref{eq:1505} we use \eqref{eq:1489} and the second of \eqref{eq:1503}. In the case $n=3$ the sum in \eqref{eq:1505} is not present. With the second line in this equation we deal in the same way as in the case $n\geq 4$. In the case $n=2$ we have
\begin{align}
  \label{eq:1509}
  b=-\frac{\hat{N}}{\widehat{\jump{T^{tt}}}}
\end{align}
and we make use of \eqref{eq:1501}, \eqref{eq:1504}. We conclude
\begin{align}
  \label{eq:1508}
  \frac{d^{n-2}b}{dv^{n-2}}=\frac{\kappa}{2}v\left((n+1)\hat{I}_{n,0}-n\hat{I}_{n,1}\right)+\left\{
        \begin{array}{ll}
          P_1 & n\geq 3,\\
          P_{1,1}& n=2.
        \end{array}
\right.
\end{align}
Integrating this we obtain
\begin{align}
  \label{eq:1541}
  \frac{d^mb}{dv^m}=P_{n-m-2}:m\leq n-3,\qquad b=P_{n-2,1},
\end{align}
where for the second we used \eqref{eq:1451}. The first is valid in the case $n\geq 3$.

We recall $\tilde{b}=vb$. We have
\begin{align}
  \label{eq:1511}
  \frac{d^{n-2}\tilde{b}}{dv^{n-2}}=v\frac{d^{n-2}b}{dv^{n-2}}+(n-2)\frac{d^{n-3}b}{dv^{n-3}}.
\end{align}
Since
\begin{align}
  \label{eq:1512}
  \int_0^vv'\hat{I}_{n,0}(v')dv'&=\frac{v^2}{2}\left(\hat{I}_{n,0}-\hat{I}_{n,2}\right),\\
  \label{eq:1934}
\int_0^vv'\hat{I}_{n,1}(v')dv'&=v^2\left(\hat{I}_{n,1}-\hat{I}_{n,2}\right),
\end{align}
we obtain from \eqref{eq:1508},
\begin{align}
  \label{eq:1513}
  \frac{d^{n-3}b}{dv^{n-3}}=\frac{\kappa}{2}v^2\left(\frac{n+1}{2}\hat{I}_{n,0}-n\hat{I}_{n,1}+\frac{n-1}{2}\hat{I}_{n,2}\right)+\left\{
  \begin{array}{ll}
    P_2 & n\geq 4,\\
    P_{2,1} & n =3.
  \end{array}
\right.
\end{align}
Substituting in \eqref{eq:1511} and using again \eqref{eq:1508} yields $d^{n-2}\tilde{b}/dv^{n-2}$ for $n\geq 3$. In the case $n=2$ we use directly \eqref{eq:1508}. We conclude
\begin{align}
  \label{eq:1558}
  \frac{d^{n-2}\tilde{b}}{dv^{n-2}}=\frac{\kappa}{2}J_n+\left\{
    \begin{array}{ll}
      P_2 & n\geq 4,\\
      P_{2,1} & n=3,\\
      P_{2,2} & n=2,
    \end{array}
\right.
\end{align}
(see \eqref{eq:1452} for the definition of $J_n$). Using the inductive hypothesis \eqref{eq:1197} we have
\begin{align}
  \label{eq:1515}
  \frac{d^{n-2}\tilde{b}}{dv^{n-2}}=P_{1}.
\end{align}
Therefore, by integration,
\begin{align}
  \label{eq:1516}
  \frac{d^m\tilde{b}}{dv^m}=P_{n-m-1}:1\leq m\leq n-3,\qquad \frac{d\tilde{b}}{dv}=P_{n-2,1},\qquad\tilde{b}=P_{n-1,2}.
\end{align}
The first one is valid in the case $n\geq 4$. The second and the third are valid in the case $n\geq 3$. We note  that in the case $n\geq 2$ we have
\begin{align}
  \label{eq:1559}
  \tilde{b}=\Landau(v^2).
\end{align}

We turn to $a$ given by \eqref{eq:1517}. We have
\begin{align}
  \label{eq:1510}
  a=-\frac{M}{\jump{T^{tt}}},
\end{align}
where
\begin{align}
  \label{eq:1524}
  M\coloneqq\frac{d}{dv}\jump{T^{tt}}.
\end{align}
In view of \eqref{eq:1402} we have
\begin{align}
  \label{eq:1519}
  \frac{d^{n-2}M}{dv^{n-2}}=\frac{d^{n-2}}{dv^{n-2}}\left\{F_\alpha\frac{d\alpha_+}{dv}+F_\beta\frac{d\beta_+}{dv}-F_t\frac{df}{dv}-F_w\frac{dz}{dv}\right\}.
\end{align}
As in \eqref{eq:1410}, \eqref{eq:1411} we have
\begin{align}
  \label{eq:1520}
  \frac{d^{n-2}}{dv^{n-2}}\left(F_\alpha\frac{d\alpha_+}{dv}\right)=P_1,\qquad \frac{d^{n-2}}{dv^{n-2}}\left(F_\beta\frac{d\beta_+}{dv}\right)=P_1.
\end{align}
As in \eqref{eq:1408} we have
\begin{align}
  \label{eq:1521}
  \frac{d^{n-2}}{dv^{n-2}}\left(F_t\frac{df}{dv}\right)=P_1.
\end{align}
Using the first of \eqref{eq:1403} together with \eqref{eq:1298} we obtain
\begin{align}
  \label{eq:1522}
  \frac{d^{n-2}}{dv^{n-2}}\left(F_w\frac{dz}{dv}\right)=P_0.
\end{align}
Therefore,
\begin{align}
  \label{eq:1523}
  \frac{d^{n-2}M}{dv^{n-2}}=P_0.
\end{align}

Now, from \eqref{eq:1417} we have
\begin{align}
  \label{eq:1525}
  M(0)=2\dot{\alpha}_0\cp{\pp{T^{tt}}{\alpha}}.
\end{align}
Recalling the second of \eqref{eq:1451}, we define
\begin{align}
  \label{eq:1526}
  M'&\coloneqq M-M(0),\\
  \label{eq:1935}
  \widehat{\jump{T^{tt}}}'&\coloneqq\widehat{\jump{T^{tt}}}-M(0).
\end{align}
We then have
\begin{align}
  \label{eq:1527}
  a=-\frac{M'+M(0)}{v\left(\widehat{\jump{T^{tt}}}'+M(0)\right)},
\end{align}
which, recalling \eqref{eq:1206}, implies
\begin{align}
  \label{eq:1528}
  \tilde{a}=\frac{\widehat{\jump{T^{tt}}}'-M'}{v\left(\widehat{\jump{T^{tt}}}'+M(0)\right)}.
\end{align}

Setting now
\begin{align}
  \label{eq:1529}
  \widehat{\widehat{\jump{T^{tt}}}}\coloneqq\frac{\widehat{\jump{T^{tt}}}'}{v},\qquad \hat{M}\coloneqq\frac{M'}{v},
\end{align}
we have
\begin{align}
  \label{eq:1530}
  \tilde{a}=\frac{\widehat{\widehat{\jump{T^{tt}}}}-\hat{M}}{\widehat{\jump{T^{tt}}}}.
\end{align}

Integrating \eqref{eq:1523} in view of \eqref{eq:1526} we obtain
\begin{align}
  \label{eq:1531}
  \frac{d^mM'}{dv^m}=P_{n-m-2},\qquad M'=P_{n-2,1}.
\end{align}
We now use \eqref{eq:1436} with $\hat{M}$, $M'$ in the role of $\hat{N}$, $N$ respectively. Since the null space of $A_{n-2}$ is the space of polynomials of degree $n-2$ with no constant term, we obtain
\begin{align}
  \label{eq:1532}
  \frac{d^{n-2}\hat{M}}{dv^{n-2}}=\Landau(1).
\end{align}
By \eqref{eq:1502} we have
\begin{align}
  \label{eq:1533}
  \frac{d^{n-2}}{dv^{n-2}}\widehat{\jump{T^{tt}}}'=P_0.
\end{align}
Integrating this yields
\begin{align}
  \label{eq:1534}
  \frac{d^m}{dv^{m}}\widehat{\jump{T^{tt}}}'=P_{n-m-2},\qquad \widehat{\jump{T^{tt}}}'=P_{n-2,1},
\end{align}
the second in view of the fact that $\widehat{\jump{T^{tt}}}'(0)=0$. Then, from \eqref{eq:1436} with $\widehat{\jump{T^{tt}}}'$, $\widehat{\widehat{\jump{T^{tt}}}}$ in the role of $N$, $\hat{N}$ respectively, we obtain
\begin{align}
  \label{eq:1535}
  \frac{d^{n-2}}{dv^{n-2}}\widehat{\widehat{\jump{T^{tt}}}}=\Landau(1).
\end{align}
The estimates \eqref{eq:1532}, \eqref{eq:1535} imply, through integration,
\begin{align}
  \label{eq:1547}
  \frac{d^{m}\hat{M}}{dv^{m}},\frac{d^{m}}{dv^{m}}\widehat{\widehat{\jump{T^{tt}}}}=P_{n-m-3}:m\leq n-3.
\end{align}

From \eqref{eq:1502}, \eqref{eq:1503} we have
\begin{align}
  \label{eq:1548}
  \frac{d^m}{dv^m}\widehat{\jump{T^{tt}}}=P_{n-m-2}:m\leq n-2.
\end{align}
\eqref{eq:1547}, \eqref{eq:1548} imply, through \eqref{eq:1530},
\begin{align}
  \label{eq:1536}
  \frac{d^{n-2}\tilde{a}}{dv^{n-2}}=\Landau(1).
\end{align}
Integrating this yields
\begin{align}
  \label{eq:1537}
  \frac{d^m\tilde{a}}{dv^m}=P_{n-m-3}:m\leq n-3.
\end{align}

In order to estimate $d^{n-1}\tilde{u}/dv^{n-1}$ we have to estimate $\tilde{a}_{n-2}$, $\tilde{b}_{n-2}$ (see \eqref{eq:1209}). From the first of \eqref{eq:1210} we have
\begin{align}
  \label{eq:1549}
  \tilde{a}_{n-2}=\left(\frac{d}{dv}+\tilde{a}\right)^{n-2}\tilde{a}.
\end{align}
Using now \eqref{eq:1536}, \eqref{eq:1537} we obtain
\begin{align}
  \label{eq:1550}
  \tilde{a}_{n-2}=\Landau(1).
\end{align}

We turn to $\tilde{b}_{n-2}$. We restrict ourselves first to the case $n\geq 3$. From the second of \eqref{eq:1210} we have
\begin{align}
  \label{eq:1538}
  \tilde{b}_{n-2}=\frac{d^{n-2}\tilde{b}}{dv^{n-2}}+\sum_{m=0}^{n-3}\sum_{l=0}^{n-3-m}\binom{n-3-m}{l}\frac{d^{n-3-m-l}\tilde{b}}{dv^{n-3-m-l}}\frac{d^l\tilde{a}_m}{dv^l}.
\end{align}
For the first term we use \eqref{eq:1558}. The double sum we rewrite as
\begin{align}
  \label{eq:1539}
  \sum_{j=0}^{n-3}\sum_{l=0}^j\binom{n-3-j+l}{l}\frac{d^{n-3-j}\tilde{b}}{dv^{n-3-j}}\frac{d^l\tilde{a}_{j-l}}{dv^l}.
\end{align}
Now, in view of the first of \eqref{eq:1210}, the factor $d^l\tilde{a}_{j-l}/dv^l$ involves $d^j\tilde{a}/dv^j$. Therefore, each of the terms in the inner sum in \eqref{eq:1539} involves
\begin{align}
  \label{eq:1542}
  \frac{d^{n-3-j}\tilde{b}}{dv^{n-3-j}}\frac{d^j\tilde{a}}{dv^j}.
\end{align}
Here $0\leq j\leq n-3$. In the case $j=n-3$ we use the third of \eqref{eq:1516} together with \eqref{eq:1537}. We obtain
\begin{align}
  \label{eq:1543}
  \tilde{b}\frac{d^{n-3}\tilde{a}}{dv^{n-3}}=P_{2,2}.
\end{align}
In the case $j=n-4$ we use the second of \eqref{eq:1516} together with \eqref{eq:1537}. We note that this case only shows up for $n\geq 4$. We obtain
\begin{align}
  \label{eq:1544}
  \frac{d\tilde{b}}{dv}\frac{d^{n-4}\tilde{a}}{dv^{n-4}}=P_{2,1}.
\end{align}
In the case $0\leq j\leq n-5$ we use the first of \eqref{eq:1516} together with \eqref{eq:1537}. We note that this case only shows up for $n\geq 5$. We obtain
\begin{align}
  \label{eq:1545}
  \frac{d^{n-3-j}\tilde{b}}{dv^{n-3-j}}\frac{d^j\tilde{a}}{dv^j}=P_2.
\end{align}
In the case $n=2$ we use $\tilde{b}_{n-2}=\tilde{b}_0$ and \eqref{eq:1558} in the case $n=2$. In view of the above we conclude
\begin{align}
  \label{eq:1561}
    \tilde{b}_{n-2}=\frac{\kappa}{2}J_n+\left\{
    \begin{array}{ll}
      P_2 & n\geq 4,\\
      P_{2,1} & n=3,\\
      P_{2,2} & n=2.
    \end{array}
\right.
\end{align}

We turn to $d^{n-1}u/dv^{n-1}$. Let us investigate first the case $n=2$. To find an estimate for $\tilde{u}$ we integrate \eqref{eq:1208} and obtain
\begin{align}
  \label{eq:1553}
  \tilde{u}=\int_0^ve^{\int_{v'}^{v}\tilde{a}_0(v'')dv''}\tilde{b}_0(v')dv'.
\end{align}
From \eqref{eq:1550} and \eqref{eq:1561} we have
\begin{align}
  \label{eq:1554}
  \tilde{a}_0=\Landau(1),\qquad \tilde{b}_0=\frac{\kappa}{2}J_2+P_{2,2}=\Landau(v^2).
\end{align}
Therefore,
\begin{align}
  \label{eq:1565}
  \left|\int_0^v\left(e^{\int_{v'}^v\tilde{a}_0(v'')dv''}-1\right)\tilde{b}_0(v')dv'\right|\leq Cv^4
\end{align}
and we obtain, through \eqref{eq:1553},
\begin{align}
  \label{eq:1566}
  \tilde{u}=P_{3,3}+\frac{\kappa}{2}\int_0^vJ_2(v')dv'.
\end{align}
Using now
\begin{align}
  \label{eq:1567}
  \int_0^vv'^2I_{2,0}(v')dv'&=\frac{1}{3}\left(v^3I_{2,0}-I_{2,3}\right),\\
  \label{eq:1936}
  \int_0^vv'I_{2,1}(v')dv'&=\frac{1}{2}\left(v^2I_{2,1}-I_{2,3}\right),
\end{align}
we obtain
\begin{align}
  \label{eq:1568}
  \tilde{u}=P_{3,3}+\frac{\kappa}{2}\left(v^3I_{2,0}-v^2I_{2,1}\right).
\end{align}
We note that $\tilde{u}=\Landau(v^3)$. From this together with \eqref{eq:1550} and \eqref{eq:1561}, in view of \eqref{eq:1208}, we obtain
\begin{align}
  \label{eq:1583}
  \frac{d\tilde{u}}{dv}=P_{2,2}+\frac{\kappa}{2}J_2.
\end{align}
Substituting this together with \eqref{eq:1568} into
\begin{align}
  \label{eq:1584}
  \frac{du}{dv}=\frac{1}{v}\frac{d\tilde{u}}{dv}-\frac{\tilde{u}}{v^2},
\end{align}
we find
\begin{align}
  \label{eq:1585}
  \frac{du}{dv}=P_{1,1}+\frac{\kappa}{2}v\left(2\hat{I}_{2,0}-\hat{I}_{2,1}\right).
\end{align}

Now we turn to the case $n\geq 3$. From $\tilde{u}=\Landau(v^3)$ together with \eqref{eq:1550} and \eqref{eq:1561} in \eqref{eq:1209} we obtain
\begin{align}
  \label{eq:1571}
  \frac{d^{n-1}\tilde{u}}{dv^{n-1}}=\frac{\kappa}{2}J_n+\left\{
    \begin{array}{ll}
      P_2 & n\geq 4,\\
      P_{2,1} & n=3.
    \end{array}
\right.
\end{align}
The $l$-fold iterated integral of $J_n$ is given by \eqref{eq:1477}. From \eqref{eq:1571} this gives $d^{n-1-l}\tilde{u}/dv^{n-1-l}$. Setting then $j=l+1$ we obtain
\begin{align}
  \label{eq:1572}
  \frac{d^{n-j}\tilde{u}}{dv^{n-j}}=\frac{\kappa}{2}v^{j+1}\sum_{m=0}^{j+1}\frac{(-1)^m(n-m)(n+1-m)}{m!(j+1-m)!}\hat{I}_{n,m}+\left\{
  \begin{array}{ll}
    P_{n+1,3} & j=n,\\
    P_{n,2} & j=n-1,\\
    P_{n-1,1} & j=n-2\\
    P_{j+1} & j\leq n-3.
  \end{array}\right.
\end{align}
The different behavior of the polynomial part is explained by the Taylor expansion of $\tilde{u}$ beginning with a cubic term. Now we use \eqref{eq:1436} with $\tilde{u}$, $u$ in the role of $N$, $\hat{N}$, respectively, and with $n+1$ in the role of $n$ (recall that $v\hat{N}=N$). I.e.~we use
\begin{align}
  \label{eq:1573}
  \frac{d^{n-1}u}{dv^{n-1}}=\frac{(n-1)!}{v^{n}}A_{n-1}\tilde{u}.
\end{align}
Setting $k=n-1$ and $j=n-i$ in \eqref{eq:1431} we have
\begin{align}
  \label{eq:1574}
  A_{n-1}\tilde{u}=\sum_{j=1}^n\frac{(-1)^{j-1}}{(n-j)!}v^{n-j}\frac{d^{n-j}\tilde{u}}{dv^{n-j}}.
\end{align}
Recalling that the null space of $A_{n-1}$ consists of all polynomials of degree $n-1$ with no constant term, we obtain, substituting \eqref{eq:1572},
\begin{align}
  \label{eq:1575}
  A_{n-1}\tilde{u}=P_{n+1,n}+\frac{\kappa}{2}v^{n+1}\sum_{j=1}^n\sum_{m=0}^{j+1}\frac{(-1)^{j+m-1}(n-m)(n+1-m)}{m!(n-j)!(j+1-m)!}\hat{I}_{n,m}.
\end{align}

We rewrite the double sum in \eqref{eq:1575} as
\begin{align}
  \label{eq:1576}
  \sum_{m=0}^{n+1}b_{n,m}\frac{(-1)^m(n-m)(n+1-m)}{m!}\hat{I}_{n,m},
\end{align}
where
\begin{align}
  \label{eq:1577}
  b_{n,m}\coloneqq\sum_{j=\max\{1,m-1\}}^n\frac{(-1)^{j-1}}{(n-j)!(j+1-m)!}.
\end{align}
Recalling $\tilde{a}_{n,m}$ given by \eqref{eq:1484}, we see that for $m\geq 2$ we have
\begin{align}
  \label{eq:1578}
  b_{n,m}=-\tilde{a}_{n,m-1}=\left\{
  \begin{array}{ll}
    0 & 2\leq m\leq n,\\
    (-1)^{n-1} & m=n+1.
  \end{array}\right.
\end{align}
On the other hand,
\begin{align}
  \label{eq:1579}
  b_{n,1}=-\tilde{a}_{n,0}+\frac{1}{n!}=\frac{1}{n!}
\end{align}
and
\begin{align}
  \label{eq:1580}
  b_{n,0}=\sum_{j=1}^n\frac{(-1)^{j+1}}{(n-j)!(j+1)!}=\frac{1}{(n+1)(n-1)!}.
\end{align}
Then \eqref{eq:1576} is
\begin{align}
  \label{eq:1581}
  \frac{n}{(n-1)!}\hat{I}_{n,0}-\frac{n-1}{(n-1)!}\hat{I}_{n,1}
\end{align}
and we conclude from \eqref{eq:1573},
\begin{align}
  \label{eq:1582}
  \frac{d^{n-1}u}{dv^{n-1}}=P_{1,0}+\frac{\kappa}{2}v\left\{n\hat{I}_{n,0}-(n-1)\hat{I}_{n,1}\right\}.
\end{align}

Now, from \eqref{eq:1585}, \eqref{eq:1582} we conclude, in view of $u=V-c_{+0}$,
\begin{align}
  \label{eq:1586}
  \frac{d^{n-1}V}{dv^{n-1}}=\frac{\kappa}{2}v\left(n\hat{I}_{n,0}-(n-1)\hat{I}_{n,1}\right)+\left\{
  \begin{array}{ll}
    P_{1,1} & n=2,\\
    P_{1} & n\geq 3.
  \end{array}\right.
\end{align}

\subsubsection{Estimate for $d^{n-1}\rho/dv^{n-1}$}
We recall the function $\rho$ given by
\begin{align}
  \label{eq:1222}
  \frac{1}{\gamma(v)}=\frac{c_{+0}-c_{-0}}{\kappa v}(1+\rho(v)),
\end{align}
where
\begin{align}
  \label{eq:1225}
  \gamma(v)=\frac{\bar{c}_+(v)-V(v)}{V(v)-\bar{c}_-(v)}.
\end{align}
Using also $u=V-c_{+0}$ and defining
\begin{align}
  \label{eq:1588}
  \nu\coloneqq\frac{\tilde{\nu}}{v},\qquad \textrm{where}\qquad\tilde{\nu}\coloneqq\bar{c}_+-c_{+0}-u,
\end{align}
we obtain
\begin{align}
  \label{eq:1587}
  \rho=\frac{\kappa}{c_{+0}-c_{-0}}\frac{u+c_{+0}-\bar{c}_-}{\nu}-1.
\end{align}

We have (recall $\bar{c}_+(v)=c_+(\alpha_+(v),\beta_+(v))$)
\begin{align}
  \label{eq:1590}
  \frac{d\tilde{\nu}}{dv}=\pp{c_+}{\alpha}(\alpha_+,\beta_+)\frac{d\alpha_+}{dv}+\pp{c_+}{\beta}(\alpha_+,\beta_+)\frac{d\beta_+}{dv}-\frac{du}{dv}.
\end{align}
Recalling
\begin{align}
  \label{eq:1592}
  \frac{d\alpha_+}{dv}(0)=\dot{\alpha}_0,\qquad \frac{d\beta_+}{dv}(0)=0,\qquad \cp{\pp{c_+}{\alpha}}\dot{\alpha}_0=\kappa
\end{align}
and \eqref{eq:1585} we obtain
\begin{align}
  \label{eq:1593}
  \frac{d\tilde{\nu}}{dv}(0)=\kappa.
\end{align}
Therefore,
\begin{align}
  \label{eq:1610}
  \nu(0)=\kappa,\qquad \rho(0)=0.
\end{align}

From \eqref{eq:1300}, \eqref{eq:1313} we have
\begin{align}
  \label{eq:1595}
  \frac{d^{n-1}\bar{c}_+}{dv^{n-1}}=P_1.
\end{align}
Together with \eqref{eq:1582} we obtain
\begin{align}
  \label{eq:1596}
  \frac{d^{n-1}\tilde{\nu}}{dv^{n-1}}=-\frac{\kappa}{2}v\left(n\hat{I}_{n,0}-(n-1)\hat{I}_{n,1}\right)+\left\{
  \begin{array}{ll}
    P_{1,1}+\kappa & n=2,\\
    P_{1} & n\geq 3.
  \end{array}\right.
\end{align}
We now use \eqref{eq:1456}, \eqref{eq:1462} and \eqref{eq:1458} to compute the $l$-fold iterated integral of this. In the role of $f_0$ we have $d^ny/dv^n$ and $vd^ny/dv^n$. We obtain
\begin{align}
  \label{eq:1597}
  \frac{d^{n-1-l}\tilde{\nu}}{dv^{n-1-l}}&=P_{l+1}-\frac{\kappa}{2}v^{l+1}\bigg\{\frac{n}{(l+1)!}\hat{I}_{n,0}+n\sum_{m=1}^l\frac{(-1)^m}{(m-1)!(l-m)!}\frac{1}{(m+1)}\hat{I}_{n,m+1}\notag\\
&\hspace{30mm} -\frac{n-1}{l!}\hat{I}_{n,1}-(n-1)\sum_{m=1}^l\frac{(-1)^m}{(m-1)!(l-m)!}\frac{1}{m}\hat{I}_{n,m+1}\bigg\}.
\end{align}
Setting then $j=l+1$, we obtain, after a straightforward computation,
\begin{align}
  \label{eq:1598}
  \frac{d^{n-j}\tilde{\nu}}{dv^{n-j}}=-\frac{\kappa}{2}v^j\sum_{m=0}^j\frac{(-1)^m(n-m)}{m!(j-m)!}\hat{I}_{n,m}+\left\{
  \begin{array}{ll}
    P_{n,1} & j=n,\\
    P_j & j\leq n-1,
  \end{array}\right.
\end{align}
where we also used the fact that $\tilde{\nu}(0)=0$.

We now apply \eqref{eq:1573}, \eqref{eq:1574} with $\nu$, $\tilde{\nu}$ in the role of $u$, $\tilde{u}$ respectively. Since the null space of $A_{n-1}$ consists of all polynomials of degree $n-1$ with no constant term, we obtain
\begin{align}
  \label{eq:1599}
  A_{n-1}\tilde{\nu}=P_{n,n}+\frac{\kappa}{2}v^n\sum_{j=1}^n\sum_{m=0}^j\frac{(-1)^{j+m}(n-m)}{m!(n-j)!(j-m)!}\hat{I}_{n,m}.
\end{align}
We rewrite the double sum as
\begin{align}
  \label{eq:1600}
  \sum_{m=0}^n\frac{(-1)^m(n-m)}{m!}a'_{n,m}\hat{I}_{n,m},
\end{align}
where
\begin{align}
  \label{eq:1601}
  a'_{n,m}\coloneqq\sum_{j=\max\{1,m\}}^n\frac{(-1)^j}{(n-j)!(j-m)!}.
\end{align}
Comparing with the coefficients $\tilde{a}_{n,m}$ given by \eqref{eq:1484} we see that for $m\geq 1$ we have $a'_{n,m}=\tilde{a}_{n,m}$, but for $m=0$ we have
\begin{align}
  \label{eq:1602}
  a'_{n,0}=\tilde{a}_{n,0}-\frac{1}{n!}.
\end{align}
From \eqref{eq:1603} we obtain
\begin{align}
  \label{eq:1604}
  a'_{n,m}=\left\{
    \begin{array}{ll}
      -\frac{1}{n!} & m=0,\\
      (-1)^n & m=n,\\
      0 & m\neq 0,n.
    \end{array}\right.
\end{align}
Therefore, \eqref{eq:1600} reduces to $-\hat{I}_{n,0}/(n-1)!$ and we conclude from \eqref{eq:1599},
\begin{align}
  \label{eq:1605}
  A_{n-1}\tilde{\nu}=P_{n,n}-\frac{\kappa v^n}{2(n-1)!}\hat{I}_{n,0}.
\end{align}
Hence, from \eqref{eq:1573} with $\nu$, $\tilde{\nu}$ in the role of $u$, $\tilde{u}$ respectively,
\begin{align}
  \label{eq:1606}
  \frac{d^{n-1}\nu}{dv^{n-1}}=P_0-\frac{\kappa}{2}\hat{I}_{n,0}.
\end{align}

Using the inductive hypothesis \eqref{eq:1197} in \eqref{eq:1582} and \eqref{eq:1606} we obtain
\begin{align}
  \label{eq:1607}
  \frac{d^{n-1}u}{dv^{n-1}}=P_0,\qquad \frac{d^{n-1}\nu}{dv^{n-1}}=\Landau(1).
\end{align}
From \eqref{eq:1300}, \eqref{eq:1313} we obtain, in view of $\bar{c}_-=c_-(\alpha_+,\beta_+)$,
\begin{align}
  \label{eq:1608}
  \frac{d^{n-1}\bar{c}_-}{dv^{n-1}}=P_1.
\end{align}
Integrating \eqref{eq:1607}, \eqref{eq:1608}, we obtain for $m\leq n-2$
\begin{align}
  \label{eq:1609}
  \frac{d^mu}{dv^m}=P_{n-m-1},\qquad \frac{d^m\nu}{dv^m}=P_{n-m-2},\qquad \frac{d^m\bar{c}_-}{dv^m}=P_{n-m}.
\end{align}

By \eqref{eq:1610}, \eqref{eq:1606}, \eqref{eq:1609} we obtain (see \eqref{eq:1587})
\begin{align}
  \label{eq:1223}
  \frac{d^{n-1}\rho}{dv^{n-1}}&=P_0-\frac{\kappa}{c_{+0}-c_{-0}}\frac{u+c_{+0}-\bar{c}_-}{\nu^2}\frac{d^{n-1}\nu}{dv^{n-1}}\notag\\
&=P_0+\frac{1}{2}\hat{I}_{n,0}.
\end{align}
Setting $f_0=\frac{d^ny}{dv^n}$ in \eqref{eq:1455} so that $g_0=\hat{I}_{n,0}$ (see \eqref{eq:1456}), the $l$-fold iterated integral of $\hat{I}_{n,0}$ is given by \eqref{eq:1458} with $0$ in the role of $k$, i.e.~it is
\begin{align}
  \label{eq:1946}
  v^l\left\{\frac{\hat{I}_{n,0}}{l!}+\sum_{m=1}^l\frac{(-1)^m}{m!(l-m)!}\hat{I}_{n,m}\right\}.
\end{align}
From this we obtain $d^{n-1-l}\rho/dv^{n-1-l}$. Setting $j=l+1$ and taking into account that $\rho(0)=0$ we find
\begin{align}
  \label{eq:1612}
  \frac{d^{n-j}\rho}{dv^{n-j}}=\frac{1}{2}v^{j-1}\sum_{m=0}^{j-1}\frac{(-1)^m}{m!(j-1-m)!}\hat{I}_{n,m}+\left\{
   \begin{array}{ll}
     P_{n-1,1} & j=n,\\
     P_{j-1} & j \leq n-1.
   \end{array}\right.
\end{align}

Let $\hat{\rho}\coloneqq\rho/v$. We now apply \eqref{eq:1573}, \eqref{eq:1574} with $\hat{\rho}$, $\rho$ in the role of $u$, $\tilde{u}$ respectively. Since the null space of $A_{n-1}$ consists of all polynomials of degree $n-1$ with no constant term, we obtain
\begin{align}
  \label{eq:1613}
  A_{n-1}\rho=\Landau(v^n)+\frac{1}{2}v^{n-1}\sum_{j=1}^n\sum_{m=0}^{j-1}\frac{(-1)^{j+m-1}}{m!(n-j)!(j-1-m)!}\hat{I}_{n,m}.
\end{align}
We rewrite the double sum as
\begin{align}
  \label{eq:1614}
  \sum_{m=0}^{n-1}\frac{(-1)^m}{m!}c_{n,m}\hat{I}_{n,m},
\end{align}
where
\begin{align}
  \label{eq:1615}
  c_{n,m}\coloneqq\sum_{j=m+1}^n\frac{(-1)^{j-1}}{(n-j)!(j-1-m)!}.
\end{align}
Comparing with the coefficients $\tilde{a}_{n,m}$ given by \eqref{eq:1484} we see that $c_{n,m}=-\tilde{a}_{n,m+1}$. Then, from \eqref{eq:1603},
\begin{align}
  \label{eq:1617}
  c_{n,m}=\left\{
    \begin{array}{ll}
      0 & 0\leq m \leq n-2,\\
      (-1)^{n-1} & m=n-1.
    \end{array}\right.
\end{align}
Hence \eqref{eq:1614} is $\hat{I}_{n,n-1}/(n-1)!$ and \eqref{eq:1613} is
\begin{align}
  \label{eq:1618}
  A_{n-1}\rho=\Landau(v^n)+\frac{1}{2}\frac{v^{n-1}}{(n-1)!}\hat{I}_{n,n-1}.
\end{align}
Then \eqref{eq:1573} with $\hat{\rho}$, $\rho$ in the role of $u$, $\tilde{u}$ respectively, yields
\begin{align}
  \label{eq:1619}
  \frac{d^{n-1}\hat{\rho}}{dv^{n-1}}=\Landau(1)+\frac{1}{2v}\hat{I}_{n,n-1}.
\end{align}

Now we find the $l$-fold iterated integral of \eqref{eq:1619}. Since
\begin{align}
  \label{eq:1726}
  \frac{1}{v}\hat{I}_{n,n-1}=\frac{1}{v^n}I_{n,n-1},
\end{align}
we use \eqref{eq:1458} with $f_0=v^{n-1}\frac{d^ny}{dv^n}$ which then implies $g_0=I_{n,n-1}$ and $k=-n$, so $g_{k+m}=I_{n,m-1}$ (see \eqref{eq:1456}, \eqref{eq:1462}). Since $k<0$ we use \eqref{eq:1725} for the first term in \eqref{eq:1458}. Therefore, the $l$-fold iterated integral of \eqref{eq:1619} is
\begin{align}
  \label{eq:1727}
  \frac{d^{n-1-l}\hat{\rho}}{dv^{n-1-l}}=P_{l-1}+\frac{v^{l-1}}{2}\left\{(-1)^l\frac{(n-l-1)!}{(n-1)!}\hat{I}_{n,n-1}-\sum_{m=1}^l\frac{(-1)^m}{(m-1)!(l-m)!(n-m)}\hat{I}_{n,m-1}\right\}.
\end{align}
From the inductive hypothesis \eqref{eq:1197} (see also \eqref{eq:1346}) we get
\begin{align}
  \label{eq:1728}
  \frac{d^m\hat{\rho}}{dv^m}=\left\{
    \begin{array}{ll}
      \Landau(1) & m=n-2,\\
      P_{n-m-3} & 0\leq m\leq n-3.
    \end{array}\right.
\end{align}

\subsubsection{Inductive Step for Derivatives of $t$, $\alpha$, $\beta$ Part One}
In the following, using the assumptions \eqref{eq:1233}, \ldots, \eqref{eq:1293}, we prove
\begin{align}
  \label{eq:1236}
  \frac{\partial^{n-1}t}{\partial u^{n-1}}=Q_2,\qquad \frac{\partial^{n-1}t}{\partial v^{n-1}}&=Q_1,\\
  \label{eq:1638}
  \frac{\partial^{i+j}}{\partial u^i\partial v^j}\left(\pppp{t}{u}{v}\right)&=Q_1:i+j=n-2.
\end{align}
This will then establish ($t_{p,n}$), ($t_{m,n}$). We recall the equation for $t(u,v)$ satisfied for $(u,v)\in T_\varepsilon$
\begin{align}
  \label{eq:1622}
  \pppp{t}{u}{v}+\mu\pp{t}{v}-\nu\pp{t}{u}=0,
\end{align}
where
\begin{align}
  \label{eq:1229}
  \mu&=\frac{1}{c_+-c_-}\pp{c_+}{u}=\frac{1}{c_+-c_-}\left(\pp{c_+}{\alpha}\pp{\alpha}{u}+\pp{c_+}{\beta}\pp{\beta}{u}\right),\\
  \label{eq:1642}
  \nu&=\frac{1}{c_+-c_-}\pp{c_-}{v}=\frac{1}{c_+-c_-}\left(\pp{c_-}{\alpha}\pp{\alpha}{v}+\pp{c_-}{\beta}\pp{\beta}{v}\right).
\end{align}

We first prove \eqref{eq:1236}, \eqref{eq:1638} for $n=2$. For \eqref{eq:1236} in the case $n=2$ we make use of what has already been established in the existence proof (see the proposition in the end of the chapter dealing with the solution of the fixed boundary problem), i.e.~we have
\begin{align}
  \label{eq:1626}
  \pp{t}{u}=Q_{2,2},\qquad \pp{t}{v}=Q_{1,1}.
\end{align}
Using the inductive hypotheses \eqref{eq:1234}, \eqref{eq:1293} with $n=2$ together with \eqref{eq:1626} in \eqref{eq:1622} we obtain
\begin{align}
  \label{eq:1627}
  \pppp{t}{u}{v}=Q_{1,1},
\end{align}
which is \eqref{eq:1638} in the case $n=2$.

We now consider the case $n\geq 3$. We have, for $i+j= n-2$,
\begin{align}
  \label{eq:1237}
  \frac{\partial^{i+j}}{\partial u^i\partial v^j}\left(\pppp{t}{u}{v}\right)=\frac{\partial^{i+j}}{\partial u^i\partial v^j}\left(-\mu\pp{t}{v}+\nu\pp{t}{u}\right).
\end{align}
Let us first consider the case $i,j\geq 1$. We note that this case only shows up for $n\geq 4$. The case $n=3$ will be contained in the proof of the cases $i=0$ and $j=0$. From \eqref{eq:1641}, \eqref{eq:1639}, \eqref{eq:1235}, \eqref{eq:1292} together with \eqref{eq:1626} we obtain
\begin{align}
  \label{eq:1623}
  \frac{\partial^{i+j}\mu}{\partial u^i\partial v^j}\pp{t}{v}=Q_{1,1},\qquad\frac{\partial^{i+j}\nu}{\partial u^i\partial v^j}\pp{t}{u}=Q_{1,1}.
\end{align}
From \eqref{eq:1637} we have
\begin{align}
  \label{eq:1624}
  \frac{\partial^{i+j+1}t}{\partial u^i\partial v^{j+1}},\frac{\partial^{i+j+1}t}{\partial u^{i+1}\partial v^{j}}=Q_1.
\end{align}
All other derivatives of $\alpha$, $\beta$ and $t$ appearing in \eqref{eq:1237} are of the order less than $n-1$ and are, by the inductive hypothesis, all $Q_1$. Therefore,
\begin{align}
  \label{eq:1239}
  \frac{\partial^{i+j}}{\partial u^i\partial v^j}\left(\pppp{t}{u}{v}\right)=Q_1:i+j= n-2,i,j\geq 1.
\end{align}
This is \eqref{eq:1638} in the case $i,j\geq 1$.

We now study \eqref{eq:1237} in the case $i=0$. We note that in the case $n=3$ only this case (or the other case, namely $j=0$) shows up and not the case $i,j\geq 1$. We define
\begin{align}
  \label{eq:1635}
  R_{v,n-1}\coloneqq\frac{\partial^{n-2}}{\partial v^{n-2}}\left(-\mu\pp{t}{v}+\nu\pp{t}{u}\right)+\mu\frac{\partial^{n-1}t}{\partial v^{n-1}}
\end{align}
and rewrite \eqref{eq:1237} in the case $i=0$ as
\begin{align}
  \label{eq:1240}
  \frac{\partial}{\partial u}\left(\frac{\partial^{n-1}t}{\partial v^{n-1}}\right)+\mu \frac{\partial^{n-1}t}{\partial v^{n-1}}=R_{v,n-1},
\end{align}
which implies
\begin{align}
  \label{eq:1243}
  \frac{\partial^{n-1}t}{\partial v^{n-1}}(u,v)=e^{-L(u,v)}\left\{\frac{\partial^{n-1}t}{\partial v^{n-1}}(v,v)+\int_v^ue^{L(u',v)}R_{v,n-1}(u',v)du'\right\},
\end{align}
where we recall
\begin{align}
  \label{eq:1636}
  L(u,v)=\int_v^u\mu(u',v)du'.
\end{align}

We have
\begin{align}
  \label{eq:1629}
  \frac{\partial^{n-2}}{\partial v^{n-2}}\left(\mu\pp{t}{v}\right)=\mu\frac{\partial^{n-1}t}{\partial v^{n-1}}+\frac{\partial^{n-2}\mu}{\partial v^{n-2}}\pp{t}{v}+\sum_{m=1}^{n-3}\binom{n-2}{m}\frac{\partial^{n-2-m}\mu}{\partial v^{n-2-m}}\frac{\partial^{m+1}t}{\partial v^{m+1}}.
\end{align}
The terms in the sum involve derivatives of $t$ w.r.t.~$v$ of order at most $n-2$ which, by the inductive hypotheses \eqref{eq:1233}, are all $Q_1$. The terms in the sum also involve derivatives of $\mu$ w.r.t.~$v$ of order at most $n-3$ which in turn, through \eqref{eq:1229}, involve derivatives of $\alpha$, $\beta$ of order at most $n-2$. By the inductive hypothesis \eqref{eq:1235}, \eqref{eq:1641}, \eqref{eq:1292}, \eqref{eq:1639}, each of these terms is a $Q_1$. Therefore, the sum is a $Q_1$. The second term in \eqref{eq:1629} involves mixed derivatives of $\alpha$ and $\beta$ of order $n-1$ and pure derivatives of order $n-2$ which, by the inductive hypothesis \eqref{eq:1235}, \eqref{eq:1641}, \eqref{eq:1292}, \eqref{eq:1639}, are a $Q_1$. Therefore,
\begin{align}
  \label{eq:1630}
  \frac{\partial^{n-2}}{\partial v^{n-2}}\left(\mu\pp{t}{v}\right)=\mu\frac{\partial^{n-1}t}{\partial v^{n-1}}+Q_1.
\end{align}
We also have
\begin{align}
  \label{eq:1631}
  \frac{\partial^{n-2}}{\partial v^{n-2}}\left(\nu\pp{t}{u}\right)=\frac{\partial^{n-2}\nu}{\partial v^{n-2}}\pp{t}{u}+\sum_{m=1}^{n-2}\binom{n-2}{m}\frac{\partial^{n-2-m}\nu}{\partial v^{n-2-m}}\frac{\partial^{m+1}t}{\partial v^{m}\partial u}.
\end{align}
The first term on the right involves pure derivatives of $\alpha$, $\beta$ w.r.t.~$v$ of order at most $n-1$. These derivatives are, by the inductive hypothesis \eqref{eq:1293}, all $Q_0$. Making use of the first of \eqref{eq:1626} we see that the first term is a $Q_1$. In view of the inductive hypothesis each of the terms in the sum is also a $Q_1$. Therefore,
\begin{align}
  \label{eq:1632}
  \frac{\partial^{n-2}}{\partial v^{n-2}}\left(\nu\pp{t}{u}\right)=Q_1.
\end{align}
From \eqref{eq:1635}, \eqref{eq:1630}, \eqref{eq:1632} we obtain
\begin{align}
  \label{eq:1241}
  R_{v,n-1}=Q_1.
\end{align}

We recall
\begin{align}
  \label{eq:1643}
  \pp{t}{u}(v,v)=\gamma(v)\pp{t}{v}(v,v),
\end{align}
where $\gamma$ is given by \eqref{eq:1225}. Differentiating $n-2$ times we obtain
\begin{align}
  \label{eq:1644}
  \left\{\left(\pp{}{v}+\pp{}{u}\right)^{n-2}\pp{t}{u}\right\}(v,v)=\sum_{l=0}^{n-2}\binom{n-2}{l}\frac{d^l\gamma}{dv^l}(v)\left\{\left(\pp{}{v}+\pp{}{u}\right)^{n-2-l}\pp{t}{v}\right\}(v,v).
\end{align}
Defining now
\begin{align}
  \label{eq:1645}
  a_k(v)&\coloneqq\left\{\left(\pp{}{v}+\pp{}{u}\right)^{k-1}\pp{t}{v}\right\}(v,v),\\
  \label{eq:1649}
  b_k(v)&\coloneqq\left\{\left(\pp{}{v}+\pp{}{u}\right)^{k-1}\pp{t}{u}\right\}(v,v),
\end{align}
we have
\begin{align}
  \label{eq:1646}
  a_k(v)+b_k(v)&=\left\{\left(\pp{}{v}+\pp{}{u}\right)^kt\right\}(v,v)\notag\\
&=\frac{d^kf}{dv^k}(v).
\end{align}
Also, \eqref{eq:1644} reads
\begin{align}
  \label{eq:1647}
  b_{n-1}&=\sum_{l=0}^{n-2}\binom{n-2}{l}\frac{d^l\gamma}{dv^l}a_{n-1-l}\notag\\
&=\gamma a_{n-1}+\sum_{l=1}^{n-2}\binom{n-2}{l}\frac{d^l\gamma}{dv^l}a_{n-1-l}.
\end{align}

Using now the relation \eqref{eq:1646} with $n-1$ in the role of $k$, substituting \eqref{eq:1647} and solving for $a_{n-1}$ yields
\begin{align}
  \label{eq:1245}
  a_{n-1}=\frac{1}{1+\gamma}\left\{\frac{d^{n-1}f}{dv^{n-1}}-\sum_{l=1}^{n-2}\binom{n-2}{l}\frac{d^l\gamma}{dv^l}a_{n-1-l}\right\}.
\end{align}
By the third case of \eqref{eq:1213} the first term in the curly bracket is a $P_1$. From \eqref{eq:1233}, \eqref{eq:1637} we have
\begin{align}
  \label{eq:1650}
  a_k=P_1:k\leq n-2.
\end{align}
From \eqref{eq:1222} we have
\begin{align}
  \label{eq:1652}
  \gamma=\frac{\kappa}{c_{+0}-c_{-0}}\,\frac{v}{1+\rho}.
\end{align}
From \eqref{eq:1223} we have
\begin{align}
  \label{eq:1651}
  \frac{d^{n-1}\rho}{dv^{n-1}}=\Landau(1).
\end{align}
We therefore find (recall that $\rho(0)=0$)
\begin{align}
  \label{eq:1653}
  \frac{d^{n-1}\gamma}{dv^{n-1}}=P_0.
\end{align}
This implies
\begin{align}
  \label{eq:1654}
  \frac{d^m\gamma}{dv^m}=P_{n-m-1}.
\end{align}
From the above we deduce
\begin{align}
  \label{eq:1706}
  a_{n-1}=P_1.
\end{align}

From \eqref{eq:1645} with $n-1$ in the role of $k$ we have
\begin{align}
  \label{eq:1655}
  \frac{\partial^{n-1}t}{\partial v^{n-1}}(v,v)=a_{n-1}(v)-\sum_{l=1}^{n-2}\binom{n-2}{l}\frac{\partial^{n-1}t}{\partial v^{n-1-l}\partial u^l}(v,v).
\end{align}
By \eqref{eq:1637} all the terms in the sum are a $P_1$. Therefore,
\begin{align}
  \label{eq:1247}
  \frac{\partial^{n-1}t}{\partial v^{n-1}}(v,v)=P_1(v),
\end{align}
which, together with \eqref{eq:1241}, through \eqref{eq:1243}, implies
\begin{align}
  \label{eq:1248}
  \frac{\partial^{n-1}t}{\partial v^{n-1}}(u,v)=Q_1.
\end{align}
This is the second of \eqref{eq:1236}. Using this in \eqref{eq:1240} we obtain \eqref{eq:1638} in the case $i=0$.

We now study \eqref{eq:1237} in the case $j=0$. We define
\begin{align}
  \label{eq:1657}
  R_{u,n-1}\coloneqq\frac{\partial ^{n-2}}{\partial u^{n-2}}\left(-\mu\pp{t}{v}+\nu\pp{t}{u}\right)-\nu\frac{\partial^{n-1}t}{\partial u^{n-1}}
\end{align}
and rewrite \eqref{eq:1237} in the case $j=0$ as
\begin{align}
  \label{eq:1658}
  \pp{}{v}\left(\frac{\partial^{n-1}t}{\partial u^{n-1}}\right)-\nu \frac{\partial^{n-1}t}{\partial u^{n-1}}=R_{u,n-1},
\end{align}
which implies
\begin{align}
  \label{eq:1659}
  \frac{\partial^{n-1}t}{\partial u^{n-1}}(u,v)=e^{-K(u,v)}\left\{\frac{d^{n-1}h}{du^{n-1}}(u)-\int_0^ve^{K(u,v')}R_{u,n-1}(u,v')dv'\right\},
\end{align}
where we recall
\begin{align}
  \label{eq:1660}
  K(u,v)=\int_0^v(-\nu)(u,v')dv'
\end{align}
and the initial condition $t(u,0)=h(u)$. We assume $h$ to be a smooth function. Analogous to the treatment of $R_{v,n-1}$ (see \eqref{eq:1629}, \ldots, \eqref{eq:1632}) we find
\begin{align}
  \label{eq:1661}
  R_{u,n-1}=Q_1,
\end{align}
which, through \eqref{eq:1659}, implies
\begin{align}
  \label{eq:1249}
  \frac{\partial^{n-1}t}{\partial u^{n-1}}(u,v)=Q_2.
\end{align}
This is the first of \eqref{eq:1236}. Using this in \eqref{eq:1658} we obtain \eqref{eq:1638} in the case $j=0$. We have thus shown \eqref{eq:1236}, \eqref{eq:1638}, i.e.~we have established ($t_{p,n}$), ($t_{m,n}$).

Now we turn to show
\begin{align}
  \label{eq:1663}
  \frac{\partial^{n-1}\alpha}{\partial u^{n-1}}=Q_2,\qquad \frac{\partial^{n-1}\alpha}{\partial v^{n-1}}&=Q_1,\\
    \label{eq:1664}
\frac{\partial^{i+j}}{\partial u^i\partial v^j}\left(\pppp{\alpha}{u}{v}\right)&=Q_1:i+j= n-2,\\
      \label{eq:1665}
  \frac{\partial^{n-1}\beta}{\partial u^{n-1}}=Q_2,\qquad\frac{\partial^{n-1}\beta}{\partial v^{n-1}}&=Q_1,\\
        \label{eq:1666}
  \frac{\partial^{i+j}}{\partial u^i\partial v^j}\left(\pppp{\beta}{u}{v}\right)&=Q_1:i+j= n-2.
\end{align}
This will then establish ($\alpha_{p,n}$), ($\alpha_{m,n}$), ($\beta_{p,n}$) and ($\beta_{m,n}$).

Let us recall the system of equations for $\alpha$, $\beta$
\begin{align}
  \label{eq:1251}
  \pp{\alpha}{v}&=\pp{t}{v}\tilde{A}(\alpha,\beta,r),\\
  \label{eq:1252}
  \pp{\beta}{u}&=\pp{t}{u}\tilde{B}(\alpha,\beta,r),
\end{align}
which implies
\begin{align}
  \label{eq:1253}
  \alpha(u,v)&=\alpha_i(u)+\int_0^v\left\{\pp{t}{v}\tilde{A}(\alpha,\beta,r)\right\}(u,v')dv',\\
  \beta(u,v)&=\beta_+(v)+\int_v^u\left\{\pp{t}{u}\tilde{B}(\alpha,\beta,r)\right\}(u',v)du'.\label{eq:1254}
\end{align}
Here $\alpha_i$ is given by the initial conditions on $\underline{C}$ and we assume $\alpha_i$ to be smooth. $\beta_+$ is given by the jump condition and we recall $\jump{\beta}=\jump{\alpha}^3G(\alpha_+,\alpha_-,\beta_-)$.

From \eqref{eq:1253} we obtain
\begin{align}
  \label{eq:1255}
  \frac{\partial^{n-1}\alpha}{\partial u^{n-1}}(u,v)=\frac{d^{n-1}\alpha_i}{dv^{n-1}}(u)+\int_0^v\left\{\sum_{i=0}^{n-1}\binom{n-1}{i}\frac{\partial^{i+1}t}{\partial u^i\partial v}\frac{\partial^{n-1-i}\tilde{A}}{\partial u^{n-1-i}}\right\}(u,v')dv'.
\end{align}
Here we use the abbreviation $\tilde{A}$ for $\tilde{A}(\alpha,\beta,r)$. We split the sum into
\begin{align}
  \label{eq:1676}
  \sum_{i=1}^{n-2}\binom{n-1}{i}\frac{\partial^{i+1}t}{\partial u^i\partial v}\frac{\partial^{n-1-i}\tilde{A}}{\partial u^{n-1-i}}+\pp{t}{v}\frac{\partial^{n-1}\tilde{A}}{\partial u^{n-1}}+\frac{\partial^nt}{\partial u^{n-1}\partial v}\tilde{A}.
\end{align}
From
\begin{align}
  \label{eq:1677}
  \tilde{A}(u,v)=\tilde{A}(\alpha(u,v),\beta(u,v),r(u,v))
\end{align}
and recalling the second of the Hodograph system
\begin{align}
  \label{eq:1678}
  \pp{r}{v}&=c_+(\alpha,\beta)\pp{t}{v},\\
  \label{eq:1680}
  \pp{r}{u}&=c_-(\alpha,\beta)\pp{t}{u},
\end{align}
we see that in the second term of \eqref{eq:1676} there are involved the partial derivatives of $\alpha$, $\beta$ and $t$ w.r.t.~$u$ of order at most $n-1$. Using now the assumption \eqref{eq:1234} for the partial derivatives of $\alpha$ and $\beta$ together with \eqref{eq:1249} we obtain
\begin{align}
  \label{eq:1679}
  \frac{\partial^{n-1}\tilde{A}}{\partial u^{n-1}}=Q_0.
\end{align}
Together with the second of \eqref{eq:1626} we find that the second term in \eqref{eq:1676} is a $Q_{1,1}$. From \eqref{eq:1638} in the case $j=0$, which was established above, we find that the third term in \eqref{eq:1676} is a $Q_1$. Each of the terms in the sum in \eqref{eq:1676} involves derivatives of $\tilde{A}$ of order less than $n-1$ and mixed derivatives of $t$ of order less than $n$. These terms are being taken care of by the assumptions \eqref{eq:1233}, \eqref{eq:1637}, \eqref{eq:1235}, \eqref{eq:1641}, \eqref{eq:1292}, \eqref{eq:1639}. Therefore, taking into account $\int_0^vQ_1(u,v')dv'=Q_2$, we find
\begin{align}
  \label{eq:1257}
  \frac{\partial^{n-1}\alpha}{\partial u^{n-1}}=Q_2.
\end{align}
We note that in the case $n=2$ the sum in \eqref{eq:1676} is not present and instead of using \eqref{eq:1638} we use \eqref{eq:1627} to deal with the last term in \eqref{eq:1676}.

From \eqref{eq:1252} we obtain
\begin{align}
  \label{eq:1258}
  \frac{\partial^{n-1}\beta}{\partial u^{n-1}}=\sum_{i=0}^{n-2}\binom{n-2}{i}\frac{\partial^{i+1}t}{\partial u^{i+1}}\frac{\partial^{n-2-i}\tilde{B}}{\partial u^{n-2-i}}.
\end{align}
We split the sum into
\begin{align}
  \label{eq:1682}
  \sum_{i=0}^{n-3}\binom{n-2}{i}\frac{\partial^{i+1}t}{\partial u^{i+1}}\frac{\partial^{n-2-i}\tilde{B}}{\partial u^{n-2-i}}+\frac{\partial^{n-1}t}{\partial u^{n-1}}\tilde{B}.
\end{align}
The first ones of \eqref{eq:1233}, \eqref{eq:1235}, \eqref{eq:1292} in conjunction with \eqref{eq:1680} imply
\begin{align}
  \label{eq:1683}
  \frac{\partial^{n-2-i}\tilde{B}}{\partial u^{n-2-i}}=Q_2:0\leq i\leq n-3.
\end{align}
Using the first of \eqref{eq:1233} we find that the sum in \eqref{eq:1682} is a $Q_2$. Now, $\alpha$, $\beta$ and $r$ being continuously differentiable as established in the existence proof, we obtain from \eqref{eq:1253}, \eqref{eq:1254} together with \eqref{eq:1626} and \eqref{eq:1313} that $\alpha=Q_2$, $\beta=Q_2$, which implies
\begin{align}
  \label{eq:1260}
  \tilde{A},\tilde{B}=Q_2.
\end{align}
From the second of this together with \eqref{eq:1249} we obtain that also the second term in \eqref{eq:1682} is a $Q_2$. Therefore,
\begin{align}
  \label{eq:1261}
  \frac{\partial^{n-1}\beta}{\partial u^{n-1}}=Q_2.
\end{align}
We note that in the case $n=2$ the sum in \eqref{eq:1682} is not present and for the second term we use the first of \eqref{eq:1626}.

From \eqref{eq:1251} we obtain
\begin{align}
  \label{eq:1266}
  \frac{\partial^{n-1}\alpha}{\partial v^{n-1}}=\sum_{i=0}^{n-2}\binom{n-2}{i}\frac{\partial^{i+1}t}{\partial v^{i+1}}\frac{\partial^{n-2-i}\tilde{A}}{\partial v^{n-2-i}}.
\end{align}
We split the sum into
\begin{align}
  \label{eq:1684}
  \sum_{i=0}^{n-3}\binom{n-2}{i}\frac{\partial^{i+1}t}{\partial v^{i+1}}\frac{\partial^{n-2-i}\tilde{A}}{\partial v^{n-2-i}}+\frac{\partial^{n-1}t}{\partial v^{n-1}}\tilde{A}.
\end{align}
Each of the second ones of \eqref{eq:1233}, \eqref{eq:1235}, \eqref{eq:1292} in conjunction with \eqref{eq:1678} imply
\begin{align}
  \label{eq:1685}
  \frac{\partial^{n-2-i}\tilde{A}}{\partial v^{n-2-i}}=Q_1:0\leq i\leq n-3.
\end{align}
Using the second of \eqref{eq:1233} we find that the sum in \eqref{eq:1684} is a $Q_1$. From the first of \eqref{eq:1260} together with \eqref{eq:1248} we obtain that also the second term in \eqref{eq:1684} is a $Q_1$. Therefore,
\begin{align}
  \label{eq:1262}
  \frac{\partial^{n-1}\alpha}{\partial v^{n-1}}=Q_1.
\end{align}
We note that in the case $n=2$ the sum in \eqref{eq:1684} is not present and for the second term we use the second of \eqref{eq:1626}.

For $f=f(u,v)$ we define
\begin{align}
  \label{eq:1667}
  F(a,v)\coloneqq\int_a^vf(v,v')dv'.
\end{align}
We claim
\begin{align}
  \label{eq:1668}
  \frac{\partial^kF}{\partial v^k}(a,v)=\int_a^v\frac{\partial^kf}{\partial u^k}(v,v')dv'+\sum_{l=0}^{k-1}\binom{k}{l+1}\frac{\partial^{k-1}f}{\partial u^{k-1-l}\partial v^l}(v,v).
\end{align}
We prove this claim by induction. It is satisfied for $k=1$. Let it now hold for $k$. Then
\begin{align}
  \label{eq:1670}
  \frac{\partial^{k+1}F}{\partial v^{k+1}}(a,v)&=\int_a^v\frac{\partial^{k+1}f}{\partial u^{k+1}}(v,v')dv'+\frac{\partial^kf}{\partial u^k}(v,v)\notag\\
&\qquad +\sum_{l=0}^{k-1}\binom{k}{l+1}\frac{\partial^{k}f}{\partial u^{k-l}\partial v^l}(v,v)+\sum_{l=0}^{k-1}\binom{k}{l+1}\frac{\partial^{k}f}{\partial u^{k-1-l}\partial v^{l+1}}(v,v).
\end{align}
Rewriting the second sum as
\begin{align}
  \label{eq:1671}
  \sum_{l=1}^k\binom{k}{l}\frac{\partial^kf}{\partial u^{k-l}\partial v^l}(v,v),
\end{align}
the sum of the two sums is
\begin{align}
  \label{eq:1672}
  \sum_{l=0}^{k-1}\left\{\binom{k}{l}+\binom{k}{l+1}\right\}\frac{\partial^kf}{\partial u^{k-l}\partial v^l}(v,v)-\frac{\partial^kf}{\partial u^k}(v,v)+\frac{\partial^kf}{\partial v^k}(v,v).
\end{align}
Making use of
\begin{align}
  \label{eq:1673}
  \binom{k}{l}+\binom{k}{l+1}=\binom{k+1}{l+1},
\end{align}
\eqref{eq:1672} becomes
\begin{align}
  \label{eq:1674}
  \sum_{l=0}^k\binom{k+1}{l+1}\frac{\partial^kf}{\partial u^{k-l}\partial v^l}(v,v)-\frac{\partial^kf}{\partial u^k}(v,v).
\end{align}
Using this in \eqref{eq:1670} we obtain
\begin{align}
  \label{eq:1675}
  \frac{\partial^{k+1}F}{\partial v^{k+1}}(a,v)=\int_a^v\frac{\partial^{k+1}f}{\partial u^{k+1}}(v,v')dv'+\sum_{l=0}^k\binom{k+1}{l+1}\frac{\partial^kf}{\partial u^{k-l}\partial v^l}(v,v),
\end{align}
which is \eqref{eq:1668} with $k+1$ in the role of $k$. This proves the claim.

In view of \eqref{eq:1254} we now set $f(u,v)=-B(v,u)$ and $a=u$ in \eqref{eq:1667}, where we recall
\begin{align}
  \label{eq:1686}
  B=\pp{t}{u}\tilde{B}(\alpha,\beta,r).
\end{align}
From \eqref{eq:1668} we then obtain
\begin{align}
  \label{eq:1264}
  \frac{\partial^{n-1}\beta}{\partial v^{n-1}}(u,v)&=\frac{d^{n-1}\beta_+}{dv^{n-1}}(v)+\int_v^u\left\{\sum_{i=0}^{n-1}\binom{n-1}{i}\frac{\partial^{i+1}t}{\partial v^i\partial u}\frac{\partial^{n-1-i}\tilde{B}}{\partial v^{n-1-i}}\right\}(u',v)du'\notag\\
&\qquad -\sum_{l=0}^{n-2}\binom{n-1}{l+1}\left\{\frac{\partial^{n-2}}{\partial u^l\partial v^{n-2-l}}\left(\pp{t}{u}\tilde{B}\right)\right\}(v,v).
\end{align}
For the first term we use \eqref{eq:1313} in the case $m=n-1$. We split the sum in the integral into
\begin{align}
  \label{eq:1687}
  \sum_{i=1}^{n-2}\binom{n-1}{i}\frac{\partial^{i+1}t}{\partial v^i\partial u}\frac{\partial^{n-1-i}\tilde{B}}{\partial v^{n-1-i}}+\pp{t}{u}\frac{\partial^{n-1}\tilde{B}}{\partial v^{n-1}}+\frac{\partial^nt}{\partial v^{n-1}\partial u}\tilde{B}.
\end{align}
We see that in the second term there are involved the partial derivatives of $\alpha$, $\beta$ and $r$ w.r.t.~$v$ of order at most $n-1$. From the assumptions \eqref{eq:1234}, \eqref{eq:1293} together with \eqref{eq:1248} we obtain
\begin{align}
  \label{eq:1688}
  \frac{\partial^{n-1}\tilde{B}}{\partial v^{n-1}}=Q_0.
\end{align}
Together with the first of \eqref{eq:1626} we find that the second term in \eqref{eq:1687} is a $Q_1$. From \eqref{eq:1638} in the case $i=0$, which was established above, we find that the third term in \eqref{eq:1687} is a $Q_1$. All terms in the sum in \eqref{eq:1687} involve derivatives of $\tilde{B}$ of order less than $n-1$ and mixed derivatives of $t$ of order less than $n$. These terms are being taken care of by the assumptions \eqref{eq:1233}, \ldots, \eqref{eq:1639}. Therefore, the second term in \eqref{eq:1264} is a $Q_1$.

We split the sum in the second line of \eqref{eq:1264} into
\begin{align}
  \label{eq:1689}
  \sum_{l=0}^{n-3}\binom{n-1}{l+1}\frac{\partial^{n-2}}{\partial u^l\partial v^{n-2-l}}\left(\pp{t}{u}\tilde{B}\right)+\frac{\partial^{n-2}}{\partial u^{n-2}}\left(\pp{t}{u}\tilde{B}\right).
\end{align}
The second term involves a derivative of $t$ with respect to $u$ of order $n-1$ which, by \eqref{eq:1249} is a $Q_2$. All other terms appearing in \eqref{eq:1689} involve mixed derivatives of $t$ of order at most $n-1$ which, by the assumption \eqref{eq:1637}, are a $Q_1$. Furthermore, these terms involve derivatives of $t$, $\alpha$, $\beta$ and $r$ of order at most $n-2$. By the assumptions \eqref{eq:1233}, \ldots, \eqref{eq:1293} and the Hodograph system \eqref{eq:1678}, \eqref{eq:1680}, all of them are a $Q_1$. Therefore, the third term in \eqref{eq:1264} is a $Q_1$. We conclude,
\begin{align}
  \label{eq:1265}
  \frac{\partial^{n-1}\beta}{\partial v^{n-1}}=Q_1.
\end{align}
We note that in the case $n=2$ the sum in \eqref{eq:1687} is not present and instead of using \eqref{eq:1638} we use \eqref{eq:1627} to deal with the last term in \eqref{eq:1687}.

In view of the system \eqref{eq:1251}, \eqref{eq:1252} and the Hodograph system \eqref{eq:1678}, \eqref{eq:1680}, a mixed derivative of $\alpha$ or $\beta$ of order $n$ is given in terms of a mixed derivative of $t$ of order at most $n$ and derivatives of $\alpha$, $\beta$ and $t$ of order at most $n-1$. By the assumptions \eqref{eq:1233}, \ldots, \eqref{eq:1293} together with \eqref{eq:1236}, \eqref{eq:1638}, \eqref{eq:1257}, \eqref{eq:1261}, \eqref{eq:1262}, \eqref{eq:1265} we conclude that each of the mixed derivatives of $\alpha$ and $\beta$ of order $n$ is a $Q_1$, i.e.
\begin{align}
  \label{eq:1263}
  \frac{\partial^{i+j}}{\partial u^i\partial v^j}\left(\pppp{\alpha}{u}{v}\right)&=Q_1:i+j=n-2,\\
  \frac{\partial^{i+j}}{\partial u^i\partial v^j}\left(\pppp{\beta}{u}{v}\right)&=Q_1:i+j=n-2.
\end{align}
From \eqref{eq:1236}, \eqref{eq:1638}, \eqref{eq:1663}, \eqref{eq:1664}, \eqref{eq:1665}, \eqref{eq:1666}  we conclude that ($t_{p,n}$), ($t_{m,n}$), ($\alpha_{p,n}$), ($\alpha_{m,n}$), ($\beta_{p,n}$), ($\beta_{m,n}$) hold.

In the following we prove 
\begin{align}
  \label{eq:1690}
  \frac{\partial^n\alpha}{\partial u^n}=Q_0,\qquad\frac{\partial^n\beta}{\partial u^n}=Q_0.
\end{align}
Putting $n$ in the role of $n-1$ in the equations \eqref{eq:1659}, \eqref{eq:1255} and \eqref{eq:1258} we have
\begin{align}
  \label{eq:1692}
  \frac{\partial^{n}t}{\partial u^{n}}(u,v)&=e^{-K(u,v)}\left\{\frac{d^{n}h}{du^{n}}(u)-\int_0^ve^{K(u,v')}R_{u,n}(u,v')dv'\right\},\\
  \label{eq:1695}
  \frac{\partial^{n}\alpha}{\partial u^{n}}(u,v)&=\frac{d^{n}\alpha_i}{dv^{n}}(u)+\int_0^v\left\{\sum_{i=0}^{n}\binom{n}{i}\frac{\partial^{i+1}t}{\partial u^i\partial v}\frac{\partial^{n-i}\tilde{A}}{\partial u^{n-i}}\right\}(u,v')dv',\\
  \label{eq:1694}
  \frac{\partial^{n}\beta}{\partial u^{n}}(u,v)&=\left\{\sum_{i=0}^{n-1}\binom{n-1}{i}\frac{\partial^{i+1}t}{\partial u^{i+1}}\frac{\partial^{n-1-i}\tilde{B}}{\partial u^{n-1-i}}\right\}(u,v).
\end{align}
From \eqref{eq:1657} with $n$ in the role of $n-1$, together with ($t_{p,n}$), ($t_{m,n}$), ($\alpha_{p,n}$), ($\alpha_{m,n}$), ($\beta_{p,n}$), ($\beta_{m,n}$), we have
\begin{align}
  \label{eq:1693}
  R_{u,n}=Q_1+Q_1'\frac{\partial^n\alpha}{\partial u^n}+Q_1''\frac{\partial^n\beta}{\partial u^n},
\end{align}
which implies
\begin{align}
  \label{eq:1268}
\frac{\partial^nt}{\partial u^n}(u,v)=Q_1+\int_0^v\left\{Q_1'\frac{\partial^n\alpha}{\partial u^n}+Q_1''\frac{\partial^n\beta}{\partial u^n}\right\}(u,v')dv'.
\end{align}

We split the sum in \eqref{eq:1695} into
\begin{align}
  \label{eq:1697}
  \sum_{i=1}^{n-1}\binom{n}{i}\frac{\partial^{i+1}t}{\partial u^i\partial v}\frac{\partial^{n-i}\tilde{A}}{\partial u^{n-i}}+\pp{t}{v}\frac{\partial^n\tilde{A}}{\partial u^n}+\frac{\partial^{n+1}t}{\partial u^n\partial v}\tilde{A}.
\end{align}
In view of ($t_{p,n}$), ($t_{m,n}$), ($\alpha_{p,n}$), ($\alpha_{m,n}$), ($\beta_{p,n}$), ($\beta_{m,n}$) and the Hodograph system, each of the terms in the sum is a $Q_1$ and for the second term in \eqref{eq:1697} we have
\begin{align}
  \label{eq:1698}
  \frac{\partial^n\tilde{A}}{\partial u^n}=Q_1+Q_1'\frac{\partial^n\alpha}{\partial u^n}+Q_1''\frac{\partial^n\beta}{\partial u^n}+Q_1'''\frac{\partial^nt}{\partial u^n}.
\end{align}
From \eqref{eq:1237} with $i=n-1$, $j=0$ we obtain, using again the results for the partial derivatives of $\alpha$, $\beta$ and $t$,
\begin{align}
  \label{eq:1696}
  \frac{\partial^{n+1}t}{\partial u^n\partial v}=Q_1+Q_1'\frac{\partial^n\alpha}{\partial u^n}+Q_1''\frac{\partial^n\beta}{\partial u^n}+Q_1'''\frac{\partial^n t}{\partial u^n}.
\end{align}
Substituting now \eqref{eq:1698}, \eqref{eq:1696} in \eqref{eq:1697} and the resulting expression in \eqref{eq:1695} we obtain
\begin{align}
  \label{eq:1267}
  \frac{\partial^n\alpha}{\partial u^n}(u,v)=Q_1+\int_0^v\left\{Q_1'\frac{\partial^n\alpha}{\partial u^n}+Q_1''\frac{\partial^n\beta}{\partial u^n}+Q_1'''\frac{\partial^nt}{\partial u^n}\right\}(u,v')dv'.
\end{align}
For \eqref{eq:1694} we make use of \eqref{eq:1268}. We obtain
\begin{align}
  \label{eq:1269}
  \frac{\partial^n\beta}{\partial u^n}(u,v)=Q_1+\int_0^v\left\{Q_1'\frac{\partial^n\alpha}{\partial u^n}+Q_1''\frac{\partial^n\beta}{\partial u^n}\right\}(u,v')dv'.
\end{align}

Defining
\begin{align}
  \label{eq:1270}
  F\coloneqq\left|\frac{\partial^n\alpha}{\partial u^n}\right|+\left|\frac{\partial^n\beta}{\partial u^n}\right|+\left|\frac{\partial^nt}{\partial u^n}\right|,
\end{align}
and taking the sum of the absolute values of \eqref{eq:1268}, \eqref{eq:1267} and \eqref{eq:1269} we obtain
\begin{align}
  \label{eq:1271}
  F(u,v)\leq C+C'\int_0^vF(u,v')dv',
\end{align}
which implies
\begin{align}
  \label{eq:1272}
  F(u,v)\leq C,
\end{align}
which in turn implies
\begin{align}
  \label{eq:1273}
  \left|\frac{\partial^n\alpha}{\partial u^n}\right|,\left|\frac{\partial^n\beta}{\partial u^n}\right|,\left|\frac{\partial^nt}{\partial u^n}\right|\leq C.
\end{align}
Therefore, using this in \eqref{eq:1268}, \eqref{eq:1267} and \eqref{eq:1269}, we obtain
\begin{align}
  \label{eq:1274}
  \frac{\partial^n\alpha}{\partial u^n},\frac{\partial^n\beta}{\partial u^n},\frac{\partial^nt}{\partial u^n}=Q_0,
\end{align}
the first and the second of which are \eqref{eq:1690}. For the analogous expression for derivatives with respect to $v$ see \ref{section_part_2} beginning on page \pageref{section_part_2}.

\subsubsection{Estimate for $d^n\hat{f}/dv^n$}
We recall the function $A$ given by
\begin{align}
  \label{eq:1620}
  A(v)=e^{-K(v,v)}\left(\hat{\rho}(v)+\frac{\kappa}{c_{+0}-c_{-0}}\right)-\frac{1}{v}\left(1-e^{-K(v,v)}\right),
\end{align}
where we recall
\begin{align}
  \label{eq:1721}
  K(u,v)=\int_0^v(-\nu)(u,v')dv'.
\end{align}
We set $f(u,v)=-\nu(u,v)$ and $a=0$ in \eqref{eq:1667} which implies $F(0,v)=K(v,v)$. We obtain from \eqref{eq:1668}
\begin{align}
  \label{eq:1722}
  \frac{d^{n-1}K}{dv^{n-1}}(v,v)=-\int_0^v\frac{\partial^{n-1}\nu}{\partial u^{n-1}}(v,v')dv'-\sum_{l=0}^{n-2}\binom{n-1}{l+1}\frac{\partial^{n-2}\nu}{\partial u^{n-2-l}\partial v^l}(v,v).
\end{align}
The integrand involves partial derivatives of $\alpha$, $\beta$ of order at most $n$, where the pure derivatives with respect to $v$ of order $n$ do not show up. By the above results for the partial derivatives of $\alpha$, $\beta$ these are all a $Q_0$ which implies that the first term is a $P_1$. The second term involves derivatives of $\alpha$, $\beta$ of order at most $n-1$. Again by the results for the partial derivatives of $\alpha$, $\beta$ these are all a $P_1$. We conclude
\begin{align}
  \label{eq:1723}
  \frac{d^{n-1}K}{dv^{n-1}}(v,v)=P_1(v),
\end{align}
which implies, by integration,
\begin{align}
  \label{eq:1724}
  \frac{d^mK}{dv^m}(v,v)=P_{n-m}(v):1\leq m\leq n-1,\qquad K(v,v)=P_{n,1}(v),
\end{align}
where for the second we took into account $K(0,0)=0$.

We apply $d^{n-1}/dv^{n-1}$ to \eqref{eq:1620}. Using \eqref{eq:1619}, \eqref{eq:1728} and \eqref{eq:1724} we obtain
\begin{align}
  \label{eq:1729}
  \frac{d^{n-1}A}{dv^{n-1}}=\Landau(1)+\frac{1}{2v}\hat{I}_{n,n-1}.
\end{align}
Since the expression for $d^{n-1}A/dv^{n-1}$ is formally identical to the one for $d^{n-1}\hat{\rho}/dv^{n-1}$ given by the right hand side of \eqref{eq:1619}, the $l$-fold iterated integral of \eqref{eq:1729} is given by the right hand side of \eqref{eq:1727}, i.e.
\begin{align}
  \label{eq:1730}
  \frac{d^{n-1-l}A}{dv^{n-1-l}}=P_{l-1}+\frac{v^{l-1}}{2}\left\{(-1)^l\frac{(n-l-1)!}{(n-1)!}\hat{I}_{n,n-1}-\sum_{m=1}^l\frac{(-1)^m}{(m-1)!(l-m)!(n-m)}\hat{I}_{n,m-1}\right\}.
\end{align}

Let us recall (here $f(v)=t(v,v)$)
\begin{align}
  \label{eq:1731}
  \frac{df}{dv}=\frac{\lambda}{2\kappa^2}M+N,
\end{align}
and
\begin{align}
  \label{eq:1733}
  M(v)=M_{0}(v)+M_{1}(v)+M_{2}(v),\qquad N(v)=N_{0}(v)+N_{1}(v),
\end{align}
as well as
\begin{align}
  \label{eq:1736}
  M_{0}(v)&\coloneqq v-\frac{1}{v^2}\int_0^vv'^2dv'=\frac{2}{3}v,\\
  \label{eq:1737}
  M_{1}(v)&\coloneqq\frac{1}{v^2}\int_0^v\left(1-e^{-\int_{v'}^vA(v'')dv''}\right)v'^2dv',\\
  \label{eq:1738}
  M_{2}(v)&\coloneqq -\frac{A(v)}{v}\int_0^ve^{-\int_{v'}^vA(v'')dv''}v'^2dv'
\end{align}
and
\begin{align}
  \label{eq:1735}
  N_{i}(v)\coloneqq v\hat{B}_{i}(v)-\frac{1}{v^2}(1+vA(v))\int_0^ve^{-\int_{v'}^vA(v'')dv''}v'^2\hat{B}_{i}(v')dv',\qquad \textrm{$i=0,1$}.
\end{align}
We also recall
\begin{align}
  \label{eq:1739}
  \hat{B}_{0}(v)&\coloneqq\frac{\lambda}{2\kappa^2}\left\{e^{-K(v,v)}\left(\rho(v)+\frac{\kappa v}{c_{+0}-c_{-0}}\right)-\left(1-e^{-K(v,v)}\right)\right\}\notag\\
  &\qquad +\frac{1}{v^2}e^{-K(v,v)}\left(1+\rho(v)+\frac{\kappa v}{c_{+0}-c_{-0}}\right)\left(\frac{c_{+0}-c_{-0}}{\kappa}h'(v)-\frac{\lambda}{2\kappa^2}v^2+h(v)\right),\\
  \label{eq:1740}
  \hat{B}_{1}(v)&\coloneqq -\frac{1}{v^2}e^{-K(v,v)}\left(1+\rho(v)+\frac{\kappa v}{c_{+0}-c_{-0}}\right)I(v).
\end{align}
Differentiating \eqref{eq:1731} $n-1$ times we obtain
\begin{align}
  \label{eq:1732}
  \frac{d^{n}f}{dv^n}=\frac{\lambda}{2\kappa^2}\frac{d^{n-1}M}{dv^{n-1}}+\frac{d^{n-1}N}{dv^{n-1}}.
\end{align}

Setting $a=0$ and
\begin{align}
  \label{eq:1741}
  f(u,v)\coloneqq\left(1-e^{-\int_v^uA(v'')dv''}\right)v^2
\end{align}
in \eqref{eq:1667}, and letting $\tilde{F}(v)\coloneqq F(0,v)$, we have
\begin{align}
  \label{eq:1742}
  M_1(v)=\frac{\tilde{F}(v)}{v^2}=\frac{1}{v^2}\int_0^vf(v,v')dv'
\end{align}
and
\begin{align}
  \label{eq:1743}
  \frac{d^k\tilde{F}}{dv^k}(v)=\frac{\partial^kF}{\partial v^k}(0,v),
\end{align}
where the right hand side is given by \eqref{eq:1668}. Since
\begin{align}
  \label{eq:1762}
  \tilde{F}(0)=\frac{d\tilde{F}}{dv}(0)=0,
\end{align}
the Taylor expansion of $\tilde{F}$ begins with quadratic terms. Since
\begin{align}
  \label{eq:1744}
  \frac{d^k}{dv^k}\left(\frac{1}{v^2}\right)=\frac{(-1)^k(k+1)!}{v^{k+2}},
\end{align}
we have
\begin{align}
  \label{eq:1745}
  \frac{d^{n-1}M_1}{dv^{n-1}}&=\frac{d^{n-1}}{dv^{n-1}}\bigg(\frac{\tilde{F}}{v^2}\bigg)\notag\\
&=\frac{(n-1)!}{v^{n+1}}B_{n-1}\tilde{F},
\end{align}
where $B_k$ is the linear $k$th order operator
\begin{align}
  \label{eq:1761}
  B_k=\sum_{m=0}^k\frac{(-1)^{k-m}(k+1-m)}{m!}v^m\frac{d^m}{dv^m}.
\end{align}
This operator is homogeneous w.r.t.~scaling. Hence $B_k$ takes a polynomial to a polynomial of the same degree. Let $G$ be a polynomial which begins with quadratic terms. Then $M\coloneqq G/v^2$ is analytic, hence so is $v^{-k-2}B_kG$. This follows from \eqref{eq:1745} with $G$ in the role of $\tilde{F}$, $M$ in the role of $M_1$ and $k+1$ in the role of $n$. It follows that the polynomial $B_kG$ begins with terms of degree $k+2$. We conclude that the null space of $B_k$ consists of all polynomials of degree $k+1$ which begin with quadratic terms, i.e.
\begin{align}
  \label{eq:1763}
  B_k\bar{P}_{k+1,2}=0.
\end{align}
This is a $k$-dimensional space.

Since $f(v,v)=0$, we have
\begin{align}
  \label{eq:1746}
  \sum_{l=0}^{k-1}\binom{k-1}{l}\frac{\partial^{k-1}f}{\partial u^{k-1-l}\partial v^l}(v,v)=0,
\end{align}
which implies
\begin{align}
  \label{eq:1747}
  \frac{\partial^{k-1}f}{\partial v^{k-1}}(v,v)+\sum_{l=0}^{k-2}\binom{k-1}{l}\frac{\partial^{k-1}f}{\partial u^{k-1-l}\partial v^l}(v,v)=0.
\end{align}
Substituting this into \eqref{eq:1668} and making use of (note that this is only valid for $l\leq k-2$)
\begin{align}
  \label{eq:1748}
  \binom{k}{l+1}-\binom{k-1}{l}=\binom{k-1}{l+1},
\end{align}
we obtain
\begin{align}
  \label{eq:1749}
  \frac{\partial^kF}{\partial v^k}(0,v)=\int_0^v\frac{\partial^kf}{\partial u^k}(v,v')dv'+\sum_{l=0}^{k-2}\binom{k-1}{l+1}\frac{\partial^{k-1}f}{\partial u^{k-1-l}\partial v^l}(v,v).
\end{align}

To deal with the first term in \eqref{eq:1749} we have to study $\partial^kf/\partial u^k$.
%Using the inductive hypothesis \eqref{eq:1197}  we obtain from \eqref{eq:1730}
%\begin{align}
%  \label{eq:1751}
%  \frac{d^kA}{dv^k}=P_{n-k-3}.
%\end{align}
From \eqref{eq:1730} together with the inductive hypothesis \eqref{eq:1197} (see also the second of \eqref{eq:1346}) we obtain
\begin{align}
  \label{eq:1788}
  \frac{d^{n-l-1}A}{dv^{n-l-1}}=\left\{
  \begin{array}{ll}
    P_{l-2} & 2\leq l\leq n-1,\\
    \Landau(1) & l=1,\\
    \Landau(v^{-1}) & l=0.
  \end{array}\right.
\end{align}
Using this we deduce from \eqref{eq:1741}
\begin{align}
  \label{eq:1752}
  \frac{\partial^kf}{\partial u^k}(u,v)=v^2e^{-\int_v^uA(v'')dv''}\left\{\frac{d^{k-1}A}{dv^{k-1}}(u)+P_{n-k-1}\right\},
\end{align}
which implies
\begin{align}
  \label{eq:1753}
  \int_0^v\frac{\partial^kf}{\partial u^k}(v,v')dv'=\frac{d^{k-1}A}{dv^{k-1}}(v)\int_0^vv'^2e^{-\int_{v'}^vA(v'')dv''}dv'+P_{n-k+2,3}.
\end{align}
Setting now $l=n-k$ in \eqref{eq:1730}, equation \eqref{eq:1753} becomes
\begin{align}
  \label{eq:1754}
  \int_0^v\frac{\partial^kf}{\partial u^k}(v,v')dv'&=P_{n-k+2,3}+\frac{v^{n-k+2}}{6}\Bigg\{(-1)^{n-k}\frac{(k-1)!}{(n-1)!}\hat{I}_{n,n-1}\notag\\
&\hspace{43mm}-\sum_{l=1}^{n-k}\frac{(-1)^l}{(l-1)!(n-k-l)!(n-l)}\hat{I}_{n,l-1}\Bigg\}.
\end{align}

We now look at the second term in \eqref{eq:1749}. From \eqref{eq:1741} we deduce, using \eqref{eq:1788},
\begin{align}
  \label{eq:1756}
  \frac{\partial^{k-1}f}{\partial u^{k-1}}(v,v)&=v^2\frac{d^{k-2}A}{dv^{k-2}}(v)+P_{n-k+2,2},\\
  \label{eq:1758}
  \frac{\partial^{k-1}f}{\partial u^{k-2}\partial v}(v,v)&=2v\frac{d^{k-3}A}{dv^{k-3}}(v)+P_{n-k+2,1},\\
  \label{eq:1757}
\frac{\partial^{k-1}f}{\partial u^{k-3}\partial v^2}(v,v)&=2\frac{d^{k-4}A}{dv^{k-4}}(v)+P_{n-k+2},\\
  \label{eq:1759}
\frac{\partial^{k-1}f}{\partial u^{k-1-l}\partial v^l}(v,v)&=P_{n-k+3}:3\leq l\leq k-2.
\end{align}
Therefore,
\begin{align}
  \label{eq:1755}
  \sum_{l=0}^{k-2}\binom{k-1}{l+1}\frac{\partial^{k-1}f}{\partial u^{k-1-l}\partial v^l}=(k-1)v^2\frac{d^{k-2}A}{dv^{k-2}}&+(k-1)(k-2)v\frac{d^{k-3}A}{dv^{k-3}}\notag\\
&+\frac{(k-1)(k-2)(k-3)}{3}\frac{d^{k-4}A}{dv^{k-4}}+P_{n-k+2}.
\end{align}
Substituting now \eqref{eq:1730} with $n-m+1$, $n-m+2$, $n-m+3$ in the role of $l$ into the first, second and third term on the right hand side, respectively, and using the resulting expression together with \eqref{eq:1754} in \eqref{eq:1749}, we obtain, after a straightforward computation,
\begin{align}
  \label{eq:1760}
  \frac{d^k\tilde{F}}{dv^k}=P_{n-k+2}-\frac{v^{n+2-k}}{6}\sum_{l=1}^{n+3-k}\frac{(-1)^l(n+1-l)(n+2-l)}{(l-1)!(n+3-k-l)!}\hat{I}_{n,l-1}.
\end{align}
Using this now in \eqref{eq:1745} (see \eqref{eq:1761} with $n-1$ in the role of $k$ for the operator $B_{n-1}$), we obtain
\begin{align}
  \label{eq:1764}
  \frac{d^{n-1}M_1}{dv^{n-1}}=P_1-\frac{v}{6}(n-1)!\sum_{m=0}^{n-1}\sum_{l=1}^{n+3-m}\frac{(-1)^{n+1-l+m}(n-m)(n+1-l)(n+2-l)}{m!(l-1)!(n+3-m-l)!}\hat{I}_{n,l-1}.
\end{align}

We rewrite the double sum in \eqref{eq:1764} as
\begin{align}
  \label{eq:1765}
  \sum_{l=1}^{n+3}c_{n,l}\frac{(-1)^{n+1-l}(n+1-l)(n+2-l)}{(l-1)!}\hat{I}_{n,l-1},
\end{align}
where
\begin{align}
  \label{eq:1766}
  c_{n,l}\coloneqq\sum_{m=0}^{\min\{n-1,n+3-l\}}\frac{(-1)^m(n-m)}{m!(n+3-l-m)!}.
\end{align}
We consider
\begin{align}
  \label{eq:1767}
  \tilde{c}_{n,l}\coloneqq\sum_{m=0}^{n+3-l}\frac{(-1)^m(n-m)}{m!(n+3-l-m)!}.
\end{align}
For $l\geq 3$ we have $\tilde{c}_{n,l}=c_{n,l}$. For $l=2$ we have
\begin{align}
  \label{eq:1768}
  \tilde{c}_{n,2}=c_{n,2}+\frac{(-1)^n}{(n+1)!},
\end{align}
while for $l=1$ we have
\begin{align}
  \label{eq:1769}
  \tilde{c}_{n,1}=c_{n,1}+\frac{(-1)^n}{(n+1)!}-\frac{2(-1)^n}{(n+2)!}.
\end{align}

To compute $\tilde{c}_{n,l}$ we express it as
\begin{align}
  \label{eq:1770}
  \tilde{c}_{n,l}=na_{n,l}-b_{n,l},
\end{align}
where
\begin{align}
  \label{eq:1771}
  a_{n,l}&=\sum_{m=0}^{n+3-l}\frac{(-1)^m}{m!(n+3-l-m)!}\notag\\
&=\frac{1}{(n+3-l)!}\sum_{m=0}^{n+3-l}\binom{n+3-l}{m}(-1)^m\notag\\
&=\frac{1}{(n+3-l)!}(1-1)^{n+3-l}\notag\\
&=\left\{
\begin{array}{ll}
  0 & l<n+3,\\
  1 & l=n+3
\end{array}\right.
\end{align}
and
\begin{align}
  \label{eq:1772}
  b_{n,l}&=\sum_{m=0}^{n+3-l}\frac{m(-1)^m}{m!(n+3-l-m)!}\notag\\
&=\frac{1}{(n+3-l)!}\left.\frac{d}{dx}\left\{\sum_{m=0}^{n+3-l}\binom{n+3-l}{m}(-x)^m\right\}\right|_{x=1}\notag\\
&=\frac{1}{(n+3-l)!}\left.\frac{d}{dx}\left\{(1-x)^{n+3-l}\right\}\right|_{x=1}\notag\\
&=\left\{
\begin{array}{ll}
  0 & l\neq n+2,\\
  -1 & l=n+2.
\end{array}\right.
\end{align}
Hence
\begin{align}
  \label{eq:1773}
  \tilde{c}_{n,l}=\left\{
    \begin{array}{ll}
      0 & 1\leq l\leq n+1,\\
      1 & l=n+2,\\
      n & l=n+3.
    \end{array}\right.
\end{align}

We obtain from \eqref{eq:1768}, \eqref{eq:1769} and \eqref{eq:1773}
\begin{align}
  \label{eq:1774}
  c_{n,l}=\left\{
    \begin{array}{ll}
      -\frac{n(-1)^n}{(n+2)!} & l=1,\\
      -\frac{(-1)^n}{(n+1)!} & l=2,\\
      0 & 3\leq l\leq n+1,\\
      1 & l=n+2,\\
      n & l=n+3.
    \end{array}\right.
\end{align}
Substituting in \eqref{eq:1765} and the resulting expression in \eqref{eq:1764} we obtain
\begin{align}
  \label{eq:1775}
  \frac{d^{n-1}M_1}{dv^{n-1}}=\frac{v}{6}\left\{\frac{n}{n+2}\hat{I}_{n,0}-\frac{n-1}{n+1}\hat{I}_{n,1}-\frac{2}{(n+1)(n+2)}\hat{I}_{n,n+2}\right\}+\left\{
  \begin{array}{ll}
    P_{1,1} & n=2,\\
    P_1 & n\geq 3.
  \end{array}\right.
\end{align}

To see that the polynomial part in \eqref{eq:1775} has no term of order zero in the case $n=2$ we consider
\begin{align}
  \label{eq:1777}
  \frac{dM_1}{dv}=-\frac{2}{v^3}\int_0^v\left(1-e^{-\int_{v'}^vA(v'')dv''}\right)v'^2dv'-\frac{A(v)}{v^2}\int_0^ve^{-\int_{v'}^vA(v'')dv''}v'^2dv'.
\end{align}
Since \eqref{eq:1730} in the case $n=2$, $l=1$ (this is \eqref{eq:1729} in the case $n=2$ integrated once) is
\begin{align}
  \label{eq:1778}
  A=P_0-\frac{1}{2}\hat{I}_{2,1}+\frac{1}{2}\hat{I}_{2,0},
\end{align}
we have from the inductive hypothesis \eqref{eq:1197}
\begin{align}
  \label{eq:1779}
  A=\Landau(1).
\end{align}
Using this in \eqref{eq:1777} we obtain
\begin{align}
  \label{eq:1780}
  \frac{dM_1}{dv}=\Landau(v).
\end{align}

We turn to $d^{n-1}M_2/dv^{n-1}$. Recalling \eqref{eq:1737}, \eqref{eq:1738} we can write
\begin{align}
  \label{eq:1781}
  M_2=-\frac{1}{3}v^2A+vM_1A.
\end{align}
Applying $d^{n-1}/dv^{n-1}$ to the first term we obtain
\begin{align}
  \label{eq:1782}
  -\frac{d^{n-1}}{dv^{n-1}}\left(\frac{1}{3}v^2A\right)=-\frac{1}{3}v^2\frac{d^{n-1}A}{dv^{n-1}}-\frac{2}{3}(n-1)v\frac{d^{n-2}A}{dv^{n-2}}-\frac{1}{3}(n-1)(n-2)\frac{d^{n-3}A}{dv^{n-3}}.
\end{align}
For the first term we use \eqref{eq:1729} while for the second and third we use \eqref{eq:1730} with $l=1$, $l=2$ respectively. We arrive at
\begin{align}
  \label{eq:1783}
  -\frac{d^{n-1}}{dv^{n-1}}\left(\frac{1}{3}v^2A\right)=\frac{v}{6}\left\{-n\hat{I}_{n,0}+(n-1)\hat{I}_{n,1}\right\}+\left\{
    \begin{array}{ll}
      P_{1,1} & n=2,\\
      P_1 & n\geq 3.
    \end{array}\right.
\end{align}

Applying $d^{n-1}/dv^{n-1}$ to the second term in \eqref{eq:1781} we obtain
\begin{align}
  \label{eq:1784}
  \frac{d^{n-1}}{dv^{n-1}}(vM_1A)=(n-1)\frac{d^{n-2}}{dv^{n-2}}(M_1A)+v\frac{d^{n-1}}{dv^{n-1}}(M_1A).
\end{align}
From \eqref{eq:1775} together with the inductive hypothesis \eqref{eq:1197} we have
\begin{align}
  \label{eq:1785}
  \frac{d^{n-1}M_1}{dv^{n-1}}=\left\{
    \begin{array}{ll}
      \Landau(v) & n=2,\\
      P_0 & n\geq 3,
    \end{array}\right.
\end{align}
which, through integration and taking into account that the Taylor expansion of $M_1$ starts with quadratic terms, implies
\begin{align}
  \label{eq:1786}
  \frac{d^mM_1}{dv^m}=\left\{
    \begin{array}{ll}
      P_{n-1,2} & m=0,\\
      P_{n-2,1} & m=1,\\
      P_{n-m-1} & 2\leq m\leq n-2.
    \end{array}\right.
\end{align}
Using \eqref{eq:1788} and \eqref{eq:1786} in \eqref{eq:1784} yields
\begin{align}
  \label{eq:1790}
  \frac{d^{n-1}}{dv^{n-1}}(vM_1A)=\left\{
    \begin{array}{ll}
      P_{1,2} & n=2,\\
      P_{1,1} & n= 3,\\
      P_1 & n\geq 4,
    \end{array}\right.
\end{align}
which, together with \eqref{eq:1783}, implies
\begin{align}
  \label{eq:1791}
  \frac{d^{n-1}M_2}{dv^{n-1}}=\frac{v}{6}\left\{-n\hat{I}_{n,0}+(n-1)\hat{I}_{n,1}\right\}+\left\{
    \begin{array}{ll}
      P_{1,1} & n=2,\\
      P_1 & n\geq 3.
    \end{array}\right.
\end{align}
Noting that only in the case $n=2$ we have the contribution $dM_0/dv=2/3$ (see \eqref{eq:1736}), we obtain from \eqref{eq:1775}, \eqref{eq:1791}
\begin{align}
  \label{eq:1792}
  \frac{d^{n-1}M}{dv^{n-1}}=\frac{v}{6}\left\{-\frac{n(n+1)}{n+2}\hat{I}_{n,0}+\frac{n(n-1)}{n+1}\hat{I}_{n,1}-\frac{2}{(n+1)(n+2)}\hat{I}_{n,n+2}\right\}+\left\{
  \begin{array}{ll}
    \frac{2}{3}+P_{1,1} & n=2,\\
    P_1 & n\geq 3.
  \end{array}\right.
\end{align}

We turn to $d^{n-1}N/dv^{n-1}$. We look at \eqref{eq:1735} and write
\begin{align}
  \label{eq:1793}
  N_i=N'_i+H_{i1}+H_{i2},\qquad i=0,1,
\end{align}
where
\begin{alignat}{4}
  \label{eq:1794}
  N'_i(v)&\coloneqq v\hat{B}_i(v)-\frac{1}{v^2}\int_0^vv'^2\hat{B}_i(v')dv',&\qquad & i=0,1,\\
  \label{eq:1937}
  H_{i1}(v)&\coloneqq\frac{1}{v^2}\int_0^vv'^2\hat{B}_i(v')\left(1-e^{-\int_{v'}^vA(v'')dv''}\right)dv',&\qquad & i=0,1,\\
  \label{eq:1938}
  H_{i2}(v)&\coloneqq -\frac{A(v)}{v}\int_0^vv'^2\hat{B}_i(v')e^{-\int_{v'}^vA(v'')dv''}dv',&\qquad & i=0,1.
\end{alignat}
We first establish estimates for the derivatives of $\hat{B}_0$. In view of \eqref{eq:1739} this involves estimates for the derivatives of $K(v,v)$ and $\rho(v)$. The derivatives of $K(v,v)$ are given by \eqref{eq:1724}. We recall that the second bracket in the second line of \eqref{eq:1739} is a smooth function whose Taylor expansion begins with cubic terms. This implies
\begin{align}
  \label{eq:1796}
  \frac{d^j\hat{B}_0}{dv^j}=\frac{\lambda}{2\kappa^2}\frac{d^j\rho}{dv^j}+\left\{
  \begin{array}{ll}
    P_{n-1,1} & j=0,\\
    P_{n-j-1} & 1\leq j\leq n-1.
  \end{array}\right.
\end{align}
From \eqref{eq:1651}, through integration, we obtain
\begin{align}
  \label{eq:1797}
  \frac{d^j\rho}{dv^j}=\left\{
    \begin{array}{ll}
      P_{n-2,1}& j=0,\\
      P_{n-j-2} & 1\leq j\leq n-2,\\
      \Landau(1) & j=n-1.
    \end{array}\right.
\end{align}
where we made use of $\rho(0)=0$ (see the second of \eqref{eq:1610}). Therefore,
\begin{align}
  \label{eq:1798}
  \frac{d^j\hat{B}_0}{dv^j}=\left\{
    \begin{array}{ll}
      P_{n-2,1} & j=0,\\
      P_{n-j-2} & 1\leq j\leq n-2,\\
      \Landau(1) & j=n-1.
    \end{array}\right.
\end{align}
Using this together with \eqref{eq:1788} and making use of arguments analogous to the ones we used for derivatives of $\hat{f}=f/v$ (see \eqref{eq:1427}) and $M_1=\tilde{F}/v^2$ (see \eqref{eq:1745}) to take care of the terms involving prefactors of $1/v$ and $1/v^2$ respectively, we arrive at
\begin{align}
  \label{eq:1799}
  \frac{d^{n-1}H_{0i}}{dv^{n-1}}=\left\{
    \begin{array}{ll}
      P_{1,1} & n=2,\\
      P_1 & n\geq 3,
    \end{array}\right.\qquad\qquad \textrm{$i=1,2$}.
\end{align}

We now turn to the principal term in \eqref{eq:1793} in the case $i=0$, which is $N_0'$. We define
\begin{align}
  \label{eq:1800}
  G_0(v)\coloneqq v^3\hat{B}_0(v)-\int_0^vv'^2\hat{B}_0(v')dv'
\end{align}
and rewrite
\begin{align}
  \label{eq:1801}
  N'_0=\frac{G_0}{v^2}.
\end{align}
Since the Taylor expansion of $G_0$ begins with quartic terms, we can apply \eqref{eq:1745} with $N_0'$, $G_0$ in the roles of $M_1$, $\tilde{F}$, respectively, i.e.
\begin{align}
  \label{eq:1802}
  \frac{d^{n-1}N'_0}{dv^{n-1}}=\frac{(n-1)!}{v^{n+1}}B_{n-1}G_0,
\end{align}
where $B_k$ is the operator \eqref{eq:1761}. We claim
\begin{align}
  \label{eq:1803}
  \frac{d^mG_0}{dv^m}&=v^3\frac{d^m\hat{B}_0}{dv^m}+(3m-1)v^2\frac{d^{m-1}\hat{B}_0}{dv^{m-1}}+(3m-2)(m-1)v\frac{d^{m-2}\hat{B}_0}{dv^{m-2}}\notag\\
&\qquad+(m-1)^2(m-2)\frac{d^{m-3}\hat{B}_0}{dv^{m-3}}.
\end{align}

For the proof of the claim in \eqref{eq:1803} we suppress the index $0$. For $m=1,2,3$ it is true as can be seen by direct computation. For $m=3$ we have
\begin{align}
  \label{eq:1809}
  \frac{d^3G}{dv^3}=v^3\frac{d^3\hat{B}}{dv^3}+8v^2\frac{d^2\hat{B}}{dv^2}+14v\frac{d\hat{B}}{dv}+4\hat{B}.
\end{align}
For $m\geq 3$ we write
\begin{align}
  \label{eq:1804}
  \frac{d^mG}{dv^m}=v^3\frac{d^m\hat{B}}{dv^m}+c_{2,m}v^2\frac{d^{m-1}\hat{B}}{dv^{m-1}}+c_{1,m}v\frac{d^{m-2}\hat{B}}{dv^{m-2}}+c_{0,m}\frac{d^{m-3}\hat{B}}{dv^{m-3}}.
\end{align}
From \eqref{eq:1809} we have
\begin{align}
  \label{eq:1810}
  c_{2,3}=8,\qquad c_{1,3}=14,\qquad c_{0,3}=4.
\end{align}

Differentiating \eqref{eq:1804}
\begin{align}
  \label{eq:1805}
  \frac{d^{m+1}G}{dv^{m+1}}&=v^3\frac{d^{m+1}\hat{B}}{dv^{m+1}}+(c_{2,m}+3)v^2\frac{d^{m}\hat{B}}{dv^{m}}\notag\\
&\qquad+(c_{1,m}+2c_{2,m})v\frac{d^{m-1}\hat{B}}{dv^{m-1}}+(c_{0,m}+c_{1,m})\frac{d^{m-2}\hat{B}}{dv^{m-2}},
\end{align}
gives us the following recursion formulas
\begin{align}
  \label{eq:1806}
  c_{2,m+1}&=c_{2,m}+3,\\
  \label{eq:1807}
  c_{1,m+1}&=c_{1,m}+2c_{2,m},\\
  \label{eq:1808}
  c_{0,m+1}&=c_{0,m}+c_{1,m}.
\end{align}
From the first of \eqref{eq:1810} and \eqref{eq:1806},
\begin{align}
  \label{eq:1811}
  c_{2,m}=3m-1.
\end{align}
Substituting in \eqref{eq:1807} gives
\begin{align}
  \label{eq:1812}
  c_{1,m+1}=c_{1,m}+6m-2.
\end{align}
Therefore, since $c_{1,1}=0$,
\begin{align}
  \label{eq:1813}
  c_{1,m}=\sum_{j=1}^{m-1}(6j-2)=(3m-2)(m-1).
\end{align}
Substituting in \eqref{eq:1808} gives
\begin{align}
  \label{eq:1814}
  c_{0,m+1}=c_{0,m}+(3m-2)(m-1).
\end{align}
Therefore, since $c_{0,2}=0$,
\begin{align}
  \label{eq:1815}
  c_{0,m}=\sum_{j=2}^{m-1}(3j-2)(j-1)=(m-1)^2(m-2).
\end{align}
In view of \eqref{eq:1811}, \eqref{eq:1813} and \eqref{eq:1815} the claim \eqref{eq:1803} is proven.

We substitute \eqref{eq:1796} with $m$, $m-1$, $m-2$, $m-3$ in the role of $j$ into \eqref{eq:1803}. For the expressions involving derivatives of $\rho$ we use \eqref{eq:1612}. We obtain, after a straightforward computation,
\begin{align}
  \label{eq:1816}
  \frac{d^mG_0}{dv^m}=P_{n-m+2}-\frac{\lambda v^{n-m+2}}{4\kappa^2}\sum_{l=1}^{n+3-m}\frac{(-1)^l(n+1-l)(n+2-l)^2}{(l-1)!(n+3-m-l)!}\hat{I}_{n,l-1}.
\end{align}

Now we go back to \eqref{eq:1802}. We must calculate $B_{n-1}G_0$. From \eqref{eq:1761},
\begin{align}
  \label{eq:1817}
  B_{n-1}G_0=\sum_{m=0}^{n-1}\frac{(-1)^{n-m-1}(n-m)}{m!}v^m\frac{d^mG_0}{dv^m}.
\end{align}
Since the null space of $B_{n-1}$ consists of all polynomials of degree $n$ which begin with quadratic terms, $B_{n-1}$ applied to the polynomial part of $G_0$ gives a $\bar{P}_{n+2,n+1}$. Then
\begin{align}
  \label{eq:1818}
  B_{n-1}G_0=P_{n+2,n+1}-\frac{\lambda v^{n+2}}{4\kappa^2}\sum_{m=0}^{n-1}\sum_{l=1}^{n+3-m}\frac{(-1)^{n-m-1+l}(n-m)(n+1-l)(n+2-l)^2}{(l-1)!m!(n+3-m-l)!}\hat{I}_{n,l-1}.
\end{align}
The double sum is
\begin{align}
  \label{eq:1819}
  \sum_{l=1}^{n+3}c_{n,l}\frac{(-1)^{n+1-l}(n+1-l)(n+2-l)^2}{(l-1)!}\hat{I}_{n,l-1},
\end{align}
where the coefficients $c_{n,l}$ are given by \eqref{eq:1766}. Using \eqref{eq:1774} we obtain, after a straightforward computation,
\begin{align}
  \label{eq:1821}
  B_{n-1}G_0=P_{n+2,n+1}+\frac{\lambda}{4\kappa^2}\frac{v^{n+2}}{(n-1)!}\left\{\frac{n(n+1)}{n+2}\hat{I}_{n,0}-\frac{n(n-1)}{n+1}\hat{I}_{n,1}+\frac{2}{(n+2)(n+1)}\hat{I}_{n,n+2}\right\}.
\end{align}
Substituting in \eqref{eq:1802} and using the resulting expression together with \eqref{eq:1799} in the $(n-1)$'th order derivative of \eqref{eq:1793} with $i=0$, we arrive at
\begin{align}
  \label{eq:1822}
  \frac{d^{n-1}N_0}{dv^{n-1}}=\frac{\lambda}{4\kappa^2}v\left\{\frac{n(n+1)}{n+2}\hat{I}_{n,0}-\frac{n(n-1)}{n+1}\hat{I}_{n,1}+\frac{2}{(n+2)(n+1)}\hat{I}_{n,n+2}\right\}+\left\{
    \begin{array}{ll}
      P_{1,1} & n=2,\\
      P_1 & n\geq 3.
    \end{array}\right.
\end{align}

To see that the polynomial part in \eqref{eq:1822} has no term of order zero in the case $n=2$ we take the derivative of \eqref{eq:1735}
\begin{align}
  \label{eq:1823}
  \frac{dN_i}{dv}=v\frac{d\hat{B}_i}{dv}-vA\hat{B}_i+\left\{\frac{2}{v^3}+\frac{2A}{v^2}-\frac{1}{v}\frac{dA}{dv}+\frac{A}{v^2}\right\}\int_0^vv'^2\hat{B}_i(v')e^{-\int_{v'}^vA(v'')dv''}dv'.
\end{align}
From \eqref{eq:1788}, \eqref{eq:1798} we have, in the case $n=2$,
\begin{align}
  \label{eq:1824}
  A=\Landau(1), \qquad \frac{dA}{dv}=\Landau(v^{-1}),\qquad \hat{B}_0=\Landau(v),\qquad \frac{d\hat{B}_0}{dv}=\Landau(1).
\end{align}
Using these in \eqref{eq:1823} with $i=0$ we obtain
\begin{align}
  \label{eq:1825}
  \frac{dN_0}{dv}=\Landau(v).
\end{align}

We turn to derive an estimate for $d^{n-1}N_1/dv^{n-1}$. In view of \eqref{eq:1740} we need an estimate for $d^{n-1}I/dv^{n-1}$, where we recall
\begin{align}
  \label{eq:1227}
  I(v)=\int_0^v\left\{e^{K(v,v')}-1+e^{K(v,v')}\tau(v,v')\right\}\pp{t}{v}(v,v')dv',
\end{align}
where
\begin{align}
  \label{eq:1228}
  K(u,v)=\int_0^v(-\nu)(u,v')dv',\qquad \tau(u,v)=\frac{c_{+0}-c_{-0}}{\kappa}\mu(u,v)-1.
\end{align}

Defining
\begin{align}
  \label{eq:1230}
  f(u,v)\coloneqq\left\{e^{K(u,v)}-1+e^{K(u,v)}\tau(u,v)\right\}\pp{t}{v}(u,v)
\end{align}
we have
\begin{align}
  \label{eq:1231}
  I(v)=\int_0^vf(v,v')dv'.
\end{align}
Making now use of \eqref{eq:1667}, \eqref{eq:1668}, we obtain
\begin{align}
  \label{eq:1232}
  \frac{d^{n-1}I}{dv^{n-1}}(v)=\int_0^v\frac{\partial^{n-1}f}{\partial u^{n-1}}(v,v')dv'+\sum_{l=0}^{n-2}\binom{n-1}{l+1}\frac{\partial^{n-2}f}{\partial u^{n-2-l}\partial v^l}(v,v).
\end{align}
We see that we need estimates for mixed derivatives of $f$ of order at most $n-2$, for derivatives of $f$ with respect to $u$ of order at most $n-1$ and for derivatives of $f$ with respect to $v$ of order at most $n-2$. Therefore we need expressions for derivatives of $\alpha$ and $\beta$ with respect to $u$ of order at most $n$, expressions for mixed derivatives of $\alpha$ and $\beta$ of order at most  $n$ and expressions for derivatives of $\alpha$ and $\beta$ with respect to $v$ of order at most $n-1$. We also need expressions for mixed derivatives of $t$ of order at most $n$ and expressions for derivatives of $t$ with respect to $v$ of order at most $n-1$. We define
\begin{align}
  \label{eq:1826}
  H(u,v)\coloneqq e^{K(u,v)}-1+e^{K(u,v)}\tau(u,v).
\end{align}

For the first term in \eqref{eq:1232} we have
\begin{align}
  \label{eq:1828}
  \frac{\partial^{n-1}f}{\partial u^{n-1}}=\frac{\partial^{n-1}}{\partial u^{n-1}}\left(H\pp{t}{v}\right)&=\sum_{i=0}^{n-1}\binom{n-1}{i}\frac{\partial^iH}{\partial u^i}\frac{\partial^{n-i}t}{\partial u^{n-1-i}\partial v}\notag\\
&=\sum_{i=0}^{n-2}\binom{n-1}{i}\frac{\partial^iH}{\partial u^i}\frac{\partial^{n-i}t}{\partial u^{n-1-i}\partial v}+\frac{\partial^{n-1}H}{\partial u^{n-1}}\pp{t}{v}.
\end{align}
By the results for the partial derivatives of $\alpha$, $\beta$ and $t$ established above we see that each of the terms in the sum is a $Q_1$. By the same results we see that
\begin{align}
  \label{eq:1829}
  \frac{\partial^{n-1}H}{\partial u^{n-1}}=Q_0.
\end{align}
Together with the second of \eqref{eq:1626} we see that the second term in the second line in \eqref{eq:1828} is a $Q_{1,1}$. Therefore,
\begin{align}
  \label{eq:1830}
  \int_0^v\frac{\partial^{n-1}f}{\partial u^{n-1}}(v,v')dv'=P_{2,1}(v).
\end{align}

For the second term in \eqref{eq:1232} we write
\begin{align}
  \label{eq:1838}
  \sum_{l=0}^{n-2}\binom{n-1}{l+1}\frac{\partial^{n-2}f}{\partial u^{n-2-l}\partial v^l}&=\sum_{l=0}^{n-2}\binom{n-1}{l+1}\frac{\partial^{n-2}}{\partial u^{n-2-l}\partial v^l}\left(H\pp{t}{v}\right)\notag\\
&=\sum_{l=0}^{n-2}\sum_{i=0}^l\binom{n-1}{l+1}\binom{l}{i}\frac{\partial^{n-2-l}}{\partial u^{n-2-l}}\left(\frac{\partial^iH}{\partial v^i}\frac{\partial^{l+1-i}t}{\partial v^{l+1-i}}\right)\notag\\
&=\sum_{l=0}^{n-2}\sum_{i=0}^l\sum_{j=0}^{n-2-l}\binom{n-1}{l+1}\binom{l}{i}\binom{n-2-l}{j}\frac{\partial^{i+j}H}{\partial u^j\partial v^i}\frac{\partial^{n-1-i-j}t}{\partial u^{n-2-l-j}\partial v^{l+1-i}}.
\end{align}
We first look at the case $i=j=0$. From \eqref{eq:1300}, \eqref{eq:1313} we have $\alpha_+,\beta_+=P_2$. Therefore, $\bar{c}_\pm=P_2$. Using this together with the first ones of \eqref{eq:1235}, \eqref{eq:1292} in \eqref{eq:1229} we obtain
\begin{align}
  \label{eq:1841}
  \mu(v,v)=P_2(v),
\end{align}
which, together with
\begin{align}
  \label{eq:1843}
  \cp{\pp{c_+}{\alpha}}\cp{\pp{\alpha}{u}}=\kappa,\qquad\cp{\pp{\beta}{u}}=0,
\end{align}
implies
\begin{align}
  \label{eq:1842}
  \tau(v,v)=P_{2,1}(v).
\end{align}
Therefore, together with \eqref{eq:1723},
\begin{align}
  \label{eq:1844}
  H(v,v)=P_{2,1}(v).
\end{align}
From ($t_{p,n}$), ($t_{m,n}$) we have
\begin{align}
  \label{eq:1845}
  \frac{\partial^{n-1}t}{\partial u^{n-2-l}\partial v^{l+1}}=Q_1:0\leq l\leq n-2.
\end{align}
From \eqref{eq:1844}, \eqref{eq:1845} we deduce that each of the terms in \eqref{eq:1838} with $i=j=0$ is a $P_{2,1}$.

Now we look at the case $i+j=n-2$. The only term in \eqref{eq:1838} satisfying this condition is
\begin{align}
  \label{eq:1848}
  \frac{\partial^{n-2}H}{\partial u^j\partial v^i}\pp{t}{v}.
\end{align}
The first factor involves derivatives of $\alpha$ and $\beta$ of order at most $n-1$. Therefore, by ($\alpha_{p,n}$), ($\alpha_{m,n}$), ($\beta_{p,n}$), ($\beta_{m,n}$) the first factor in \eqref{eq:1848} is a $Q_1$. By \eqref{eq:1213} and the first of \eqref{eq:1626} we have for $n\geq 3$
\begin{align}
  \label{eq:1849}
  \pp{t}{v}(v,v)=\frac{df}{dv}(v)-\pp{t}{u}(v,v)=P_{2,1}.
\end{align}
Therefore, for $n\geq 3$,
\begin{align}
  \label{eq:1850}
  \left(\frac{\partial^{n-2}H}{\partial u^j\partial v^i}\pp{t}{v}\right)(v,v)=P_{2,1}(v).
\end{align}
In the case $n=2$ \eqref{eq:1848} is $H(\partial t/\partial v)$. Using the second of \eqref{eq:1626} and \eqref{eq:1844} we obtain that \eqref{eq:1850} is also valid in the case $n=2$. We note that for $n=2$ this is the only non-vanishing term in \eqref{eq:1838}. For $n=3$ the two cases $i=j=0$ and $i+j=n-2=1$ cover all the terms appearing in \eqref{eq:1838}.

Let now $n\geq 4$. For the terms in \eqref{eq:1838} with $1\leq i+j\leq n-3$ we need estimates for
\begin{align}
  \label{eq:1847}
  \qquad\frac{\partial^{k+1}t}{\partial u^{k-l}\partial v^{l+1}}:\left\{
  \begin{array}{l}
    1\leq k\leq n-3,\\
    0\leq l\leq k.
  \end{array}\right.
\end{align}
We set
\begin{align}
  \label{eq:1832}
  g(v)\coloneqq\frac{\partial^{k+1}t}{\partial u^{k-l}\partial v^{l+1}}(v,v).
\end{align}
Then
\begin{align}
  \label{eq:1833}
  \frac{d^pg}{dv^p}(v)&=\left\{\left(\pp{}{u}+\pp{}{v}\right)^p\frac{\partial^{k+1}t}{\partial u^{k-l}\partial v^{l+1}}\right\}(v,v)\notag\\
&=\sum_{q=0}^p\binom{p}{q}\frac{\partial^{p+k+1}t}{\partial u^{p-q+k-l}\partial v^{q+l+1}}(v,v).
\end{align}
Setting $p=n-k-2$ and $j=q+l$ implies, through \eqref{eq:1845},
\begin{align}
  \label{eq:1835}
  \frac{d^{n-k-2}g}{dv^{n-k-2}}=P_1.
\end{align}
This yields, through integration,
\begin{align}
  \label{eq:1836}
  g=P_{n-k-1},
\end{align}
i.e.,
\begin{align}
  \label{eq:1837}
  \frac{\partial^{k+1}t}{\partial u^{k-l}\partial v^{l+1}}(v,v)=P_{n-k-1}(v).
\end{align}
Hence,
\begin{align}
  \label{eq:1852}
  \frac{\partial^{k+1}t}{\partial u^{k-l}\partial v^{l+1}}(v,v)=P_2(v):\left\{
    \begin{array}{ll}
      1\leq k\leq n-3,\\
      0\leq l\leq k.
    \end{array}\right.
\end{align}

We also need estimates for
\begin{align}
  \label{eq:1860}
  \frac{\partial^kH}{\partial u^{k-l}\partial v^l}:\left\{
  \begin{array}{l}
    1\leq k\leq n-3,\\
    0\leq l\leq k.
  \end{array}\right.
\end{align}
These in turn require estimates for
\begin{align}
  \label{eq:1853}
  \frac{\partial^{k+1}\alpha}{\partial u^{k+1-l}\partial v^l}:\left\{
  \begin{array}{l}
    1\leq k\leq n-3,\\
    0\leq l\leq k.
  \end{array}\right.
\end{align}
We set
\begin{align}
  \label{eq:1854}
  h(v)\coloneqq\frac{\partial^{k+1}\alpha}{\partial u^{k+1-l}\partial v^l}(v,v).
\end{align}
Then
\begin{align}
  \label{eq:1855}
  \frac{d^ph}{dv^p}(v)&=\left\{\left(\pp{}{u}+\pp{}{v}\right)^p\frac{\partial^{k+1}\alpha}{\partial u^{k+1-l}\partial v^{l}}\right\}(v,v)\notag\\
&=\sum_{q=0}^p\binom{p}{q}\frac{\partial^{p+k+1}\alpha}{\partial u^{p-q+k-l+1}\partial v^{q+l}}(v,v).
\end{align}
Setting $p=n-k-2$ and $j=q+l$ and using ($\alpha_{p,n}$), ($\alpha_{m,n}$) implies
\begin{align}
  \label{eq:1856}
  \frac{d^{n-k-2}h}{dv^{n-k-2}}=P_1.
\end{align}
This yields, through integration,
\begin{align}
  \label{eq:1858}
  h=P_{n-k-1},
\end{align}
i.e.,
\begin{align}
  \label{eq:1857}
  \frac{\partial^{k+1}\alpha}{\partial u^{k+1-l}\partial v^l}(v,v)=P_{n-k-1}(v).
\end{align}
Hence,
\begin{align}
  \label{eq:1859}
  \frac{\partial^{k+1}\alpha}{\partial u^{k+1-l}\partial v^l}(v,v)=P_2(v):\left\{
  \begin{array}{l}
    1\leq k\leq n-3,\\
    0\leq l\leq k.
  \end{array}\right.
\end{align}
The same procedure applies to derivatives of $\beta$ using $(\beta_{p,n})$, $(\beta_{m,n})$. Therefore,
\begin{align}
  \label{eq:1861}
  \frac{\partial^kH}{\partial u^{k-l}\partial v^l}(v,v)=P_2(v):\left\{
  \begin{array}{l}
    1\leq k\leq n-3,\\
    0\leq l\leq k.
  \end{array}\right.
\end{align}

From \eqref{eq:1852} together with \eqref{eq:1861} we obtain that each of the terms in \eqref{eq:1838} with $1\leq i+j\leq n-3$ is a $P_2$. We conclude
\begin{align}
  \label{eq:1862}
  \sum_{l=0}^{n-2}\binom{n-1}{l+1}\frac{\partial^{n-2}f}{\partial u^{n-2-l}\partial v^l}(v,v)=P_2(v),
\end{align}
which, together with \eqref{eq:1830} implies
\begin{align}
  \label{eq:1276}
  \frac{d^{n-1}I}{dv^{n-1}}=P_2.
\end{align}

Making use of
\begin{align}
  \label{eq:1865}
  K(0,0)=\tau(0,0)=\pp{t}{v}(0,0)=\pppp{t}{u}{v}(0,0)=0,
\end{align}
we obtain
\begin{align}
  \label{eq:1864}
  I(0)=\frac{dI}{dv}(0)=\frac{d^2I}{dv^2}(0)=0.
\end{align}
Therefore, by integrating \eqref{eq:1276}, we deduce
\begin{align}
  \label{eq:1866}
  \frac{d^jI}{dv^j}=\left\{
    \begin{array}{ll}
      P_{n+1,3} & j=0,\\
      P_{n,2} & j=1,\\
      P_{n-1,1} & j=2,\\
      P_{n+1-j} & 3\leq j\leq n-1.
    \end{array}\right.
\end{align}

Now we apply \eqref{eq:1745} with $I$ in the role of $\tilde{F}$ (see \eqref{eq:1742} for the relation between $M_1$ and $\tilde{F}$).
\begin{align}
  \label{eq:1867}
  \frac{d^{n-1}}{dv^{n-1}}\left(\frac{I}{v^2}\right)=\frac{(n-1)!}{v^{n+1}}B_{n-1}I,
\end{align}
where the operator $B_k$ is given in \eqref{eq:1761}. We find
\begin{align}
  \label{eq:1868}
  \frac{d^{n-1}}{dv^{n-1}}\left(\frac{I}{v^2}\right)=P_0,
\end{align}
which, through integration, implies
\begin{align}
  \label{eq:1869}
  \frac{d^j}{dv^j}\left(\frac{I}{v^2}\right)=\left\{
    \begin{array}{ll}
      P_{n-1,1} & j=0,\\
      P_{n-1-j} & 1\leq j\leq n-1.
    \end{array}\right.
\end{align}

Now we apply $d^{n-1}/dv^{n-1}$ to $\hat{B}_1$. We rewrite
\begin{align}
  \label{eq:1870}
  \hat{B}_1=\Xi \frac{I}{v^2},
\end{align}
where
\begin{align}
  \label{eq:1871}
  \Xi(v)\coloneqq -e^{-K(v,v)}\left(1+\rho(v)+\frac{\kappa v}{c_{+0}-c_{-0}}\right).
\end{align}
Using \eqref{eq:1797} and \eqref{eq:1724} we obtain
\begin{align}
  \label{eq:1872}
  \frac{d^j\Xi}{dv^j}=P_{n-j-2}.
\end{align}

Now,
\begin{align}
  \label{eq:1873}
  \frac{d^{n-1}\hat{B}_1}{dv^{n-1}}=\sum_{i=1}^{n-1}\binom{n-1}{i}\frac{d^i}{dv^i}\left(\frac{I}{v^2}\right)\frac{d^{n-1-i}\Xi}{dv^{n-1-i}}+\frac{I}{v^2}\frac{d^{n-1}\Xi}{dv^{n-1}}.
\end{align}
Therefore, from \eqref{eq:1869} and \eqref{eq:1872},
\begin{align}
  \label{eq:1874}
  \frac{d^{n-1}\hat{B}_1}{dv^{n-1}}=P_0,
\end{align}
which, through integration, implies
\begin{align}
  \label{eq:1875}
  \frac{d^j\hat{B}_1}{dv^j}=\left\{
    \begin{array}{ll}
      P_{n-1,1} & j=0,\\
      P_{n-1-j} & 1\leq j\leq n-1.
    \end{array}\right.
\end{align}
We see that, in particular, \eqref{eq:1798} with $1$ in the role of $0$ holds. Therefore, also \eqref{eq:1799} with $1$ in the role of $0$ holds, i.e.~we have
\begin{align}
  \label{eq:1876}
    \frac{d^{n-1}H_{1i}}{dv^{n-1}}=\left\{
    \begin{array}{ll}
      P_{1,1} & n=2,\\
      P_1 & n\geq 3,
    \end{array}\right.\qquad\qquad \textrm{$i=1,2$}.
\end{align}

Now we define (this is \eqref{eq:1800} with $1$ in the role of $0$)
\begin{align}
  \label{eq:1878}
  G_1\coloneqq v^3\hat{B}_1-\int_0^vv'^2\hat{B}_1(v')dv'
\end{align}
and have
\begin{align}
  \label{eq:1879}
  N_1'=\frac{G_1}{v^2}.
\end{align}
We then have, as in \eqref{eq:1802} with $1$ in the role of $0$,
\begin{align}
  \label{eq:1880}
  \frac{d^{n-1}N_1'}{dv^{n-1}}=\frac{(n-1)!}{v^{n+1}}B_{n-1}G_1,
\end{align}
where (this is \eqref{eq:1817} with $1$ in the role of $0$)
\begin{align}
  \label{eq:1881}
  B_{n-1}G_1=\sum_{m=0}^{n-1}\frac{(-1)^{n-m-1}(n-m)}{m!}v^m\frac{d^mG_1}{dv^m}.
\end{align}

Now we use \eqref{eq:1875} in \eqref{eq:1803}, setting successively $j=m,m-1,m-2,m-3$. We obtain
\begin{align}
  \label{eq:1882}
  \frac{d^mG_1}{dv^m}=P_{n-m+2}.
\end{align}
Using this in \eqref{eq:1881} and taking into account that the null space of $B_{n-1}$ consists of all polynomials of degree $n$ which begin with quadratic terms, we obtain
\begin{align}
  \label{eq:1883}
  B_{n-1}G_1=P_{n+2,n+1}.
\end{align}
Substituting this in \eqref{eq:1880} yields
\begin{align}
  \label{eq:1885}
  \frac{d^{n-1}N_1'}{dv^{n-1}}=P_1,
\end{align}
which, together with \eqref{eq:1876}, yields
\begin{align}
  \label{eq:1884}
  \frac{d^{n-1}N_1}{dv^{n-1}}=\left\{
    \begin{array}{ll}
      P_{1,1} & n=2,\\
      P_1 & n\geq 3.
    \end{array}\right.
\end{align}
To see that the polynomial part has no term of order zero in the case $n=2$ we use \eqref{eq:1874}, \eqref{eq:1875} together with the first and the second of \eqref{eq:1824} in \eqref{eq:1823} with $i=1$.

Combining the estimates \eqref{eq:1822} and \eqref{eq:1884} we obtain
\begin{align}
  \label{eq:1886}
  \frac{d^{n-1}N}{dv^{n-1}}=\frac{\lambda}{4\kappa^2}v\left\{\frac{n(n+1)}{n+2}\hat{I}_{n,0}-\frac{n(n-1)}{n+1}\hat{I}_{n,1}+\frac{2}{(n+2)(n+1)}\hat{I}_{n,n+2}\right\}+\left\{
    \begin{array}{ll}
      P_{1,1} & n=2,\\
      P_1 & n\geq 3.
    \end{array}\right.
\end{align}
%Combining this with \eqref{eq:1792} we obtain (see \eqref{eq:1731})
%\begin{align}
%  \label{eq:1887}
%  \frac{d^nf}{dv^n}=\frac{\lambda}{6\kappa^2}v\left\{\frac{n(n+1)}{n+2}\hat{I}_{n,0}-\frac{n(n-1)}{n+1}\hat{I}_{n,1}+\frac{2}{(n+1)(n+2)}\hat{I}_{n,n+2}\right\}+\left\{
%  \begin{array}{ll}
%    \frac{\lambda}{3\kappa^2}+P_{1,1} & n=2,\\
%    P_1 & n\geq 3.
%  \end{array}\right.
%\end{align}

We introduce
\begin{align}
  \label{eq:1888}
  \Phi(v)\coloneqq f(v)-\frac{\lambda}{6\kappa^2}v^2.
\end{align}
We have
\begin{align}
  \label{eq:1889}
  \frac{d^{n}\Phi}{dv^{n}}=\frac{d^nf}{dv^n}-\left\{
    \begin{array}{ll}
      \frac{\lambda}{3\kappa^2} & n=2,\\
      0 & n\geq 3.
    \end{array}\right.
\end{align}
Combining now \eqref{eq:1792} with \eqref{eq:1886} we obtain (see \eqref{eq:1731})
\begin{align}
  \label{eq:1890}
  \frac{d^{n}\Phi}{dv^{n}}=\frac{\lambda}{6\kappa^2}v\left\{\frac{n(n+1)}{n+2}\hat{I}_{n,0}-\frac{n(n-1)}{n+1}\hat{I}_{n,1}+\frac{2}{(n+1)(n+2)}\hat{I}_{n,n+2}\right\}+\left\{
  \begin{array}{ll}
    P_{1,1} & n=2,\\
    P_1 & n\geq 3.
  \end{array}\right.
\end{align}
We note that since
\begin{align}
  \label{eq:1893}
  f(0)=0,\qquad\frac{df}{dv}(0)=0,\qquad \frac{d^2f}{dv^2}(0)=\frac{\lambda}{3\kappa^2},
\end{align}
we have
\begin{align}
  \label{eq:1894}
  \Phi(0)=\frac{d\Phi}{dv}(0)=\frac{d^2\Phi}{dv^2}(0)=0.
\end{align}

We recall $f=v^2\hat{f}$. Using \eqref{eq:1745} we obtain
\begin{align}
  \label{eq:1896}
  \frac{d^n\hat{f}}{dv^n}=\frac{d^n}{dv^n}\left(\frac{\Phi}{v^2}\right)=\frac{1}{v^{n+2}}M_n\Phi,
\end{align}
where (see \eqref{eq:1761}) the linear $n$'th order operator $M_n$ is given by
\begin{align}
  \label{eq:1897}
  M_n=n!B_n&=\sum_{m=0}^n\frac{(-1)^{n-m}n!(n+1-m)}{m!}v^m\frac{d^m}{dv^m}\notag\\
  &=\sum_{l=0}^n\frac{(-1)^{l}n!(l+1)}{(n-l)!}v^{n-l}\frac{d^{n-l}}{dv^{n-l}}.
\end{align}
We recall that the null space of $B_n$ (and hence the null space of $M_n$) is the space of all polynomials of degree $n+1$ which begin with quadratic terms.

Now we use \eqref{eq:1458} to compute, for $0\leq l\leq n$, the $l$-fold iterated integral of $v\hat{I}_{n,0}$, $v\hat{I}_{n,1}$, $v\hat{I}_{n,n+2}$ appearing in \eqref{eq:1890}. Setting $f_0=\frac{d^ny}{dv^n}$ in \eqref{eq:1456}, so that $g_0=I_{n,0}=\hat{I}_{n,0}$, and $k=1$ in \eqref{eq:1458}, we find that the $l$-fold iterated integral of $v\hat{I}_{n,0}$ is given by
\begin{align}
  \label{eq:1900}
  v^{l+1}\left\{\frac{\hat{I}_{n,0}}{(l+1)!}+\sum_{m=1}^l\frac{(-1)^m}{(m-1)!(l-m)!}\frac{\hat{I}_{n,m+1}}{(m+1)}\right\}.
\end{align}
Setting $f_0=v\frac{d^ny}{dv^n}$ in \eqref{eq:1456}, so that $g_0=I_{n,1}=v\hat{I}_{n,1}$, and $k=0$ in \eqref{eq:1458}, we find that the $l$-fold iterated integral of $v\hat{I}_{n,1}$ is given by
\begin{align}
  \label{eq:1901}
  v^{l+1}\left\{\frac{\hat{I}_{n,1}}{l!}+\sum_{m=1}^l\frac{(-1)^m}{(m-1)!(l-m)!}\frac{\hat{I}_{n,m+1}}{m}\right\}.
\end{align}
Setting $f_0=v^{n+2}\frac{d^ny}{dv^n}$ in \eqref{eq:1456}, so that $g_0=I_{n,n+2}=v^{n+2}\hat{I}_{n,n+2}$, and $k=-n-1$, we find that the $l$-fold iterated integral of $v\hat{I}_{n,n+2}$ is given by
\begin{align}
  \label{eq:1902}
  v^{l+1}\left\{\frac{(-1)^l(n-l)!\hat{I}_{n,n+2}}{n!}-\sum_{m=1}^l\frac{(-1)^m}{(m-1)!(l-m)!}\frac{\hat{I}_{n,m+1}}{(n+1-m)}\right\}.
\end{align}
We note that here we made use of \eqref{eq:1458} in the case $k<0$, which is valid since $k+l\leq -1$. Therefore we interpret the first factor in the first term of \eqref{eq:1458} as in \eqref{eq:1725}.

Using now \eqref{eq:1900}, \eqref{eq:1901} and \eqref{eq:1902} for the $l$-fold iterated integral of \eqref{eq:1890} in conjunction with \eqref{eq:1897}, yields
\begin{align}
  \label{eq:1903}
  M_n\Phi&=\frac{\lambda v^{n+1}}{6\kappa^2}\Bigg\{\frac{n(n+1)}{n+2}\hat{I}_{n,0}\sum_{l=0}^n\frac{(-1)^ln!}{(n-l)!l!}\notag\\
&\hspace{20mm}-\frac{n(n-1)}{n+1}\hat{I}_{n,1}\sum_{l=0}^n\frac{(-1)^ln!(l+1)}{(n-l)!l!}\notag\\
&\hspace{20mm}+\frac{2}{(n+1)(n+2)}\hat{I}_{n,n+2}\sum_{l=0}^n(l+1)\notag\\
&\hspace{20mm}-\sum_{l=1}^n\sum_{m=1}^l\frac{(-1)^{m+l}n!}{(m+1)!(l-m)!}\frac{(n-m)(n-1-m)(l+1)}{(n-l)!(n+1-m)}\hat{I}_{n,m+1}\Bigg\}+\Landau(v^{n+2}).
\end{align}
Here we made use of the fact that, when \eqref{eq:1890} is iteratively integrated $n$ times, the polynomial part of $\Phi$ is of degree $n+1$ and, in view of \eqref{eq:1894} begins with cubic terms. Therefore, the polynomial part of $\Phi$ is annihilated by $B_n$ hence it is annihilated by $M_n$. The first sum in \eqref{eq:1903} is
\begin{align}
  \label{eq:1904}
  \sum_{l=0}^n\binom{n}{l}(-1)^l=(1-1)^n=0,
\end{align}
since $n\geq 2$.

In view of this the second sum in \eqref{eq:1903} is
\begin{align}
  \label{eq:1905}
  \sum_{l=0}^n\binom{n}{l}(-1)^ll=\left.\frac{d}{dx}\left\{\sum_{l=0}^n\binom{n}{l}(-x)^l\right\}\right|_{x=1}=\frac{d}{dx}\left.(1-x)^n\right|_{x=1}=0,
\end{align}
since $n\geq 2$.

The third sum in \eqref{eq:1903} is
\begin{align}
  \label{eq:1906}
  \sum_{l=0}^n(l+1)=\frac{(n+1)(n+2)}{2}.
\end{align}

Finally, we rewrite the double sum in \eqref{eq:1903} as
\begin{align}
  \label{eq:1907}
  \sum_{m=1}^n\frac{(-1)^m(n-m)(n-1-m)n!}{(m+1)!(n+1-m)!}a_{n,m}\hat{I}_{n,m+1},
\end{align}
where
\begin{align}
  \label{eq:1908}
  a_{n,m}&\coloneqq\sum_{l=m}^n\frac{(-1)^l(l+1)}{(l-m)!(n-l)!}\notag\\
&=(-1)^m\sum_{i=0}^{n-m}\frac{(-1)^i(m+i+1)}{i!(n-m-i)!}.
\end{align}
We see that in \eqref{eq:1907} the terms with $m=n$ and $m=n-1$ vanish. Therefore, we can restrict to the case $1\leq m\leq n-2$. We have
\begin{align}
  \label{eq:1909}
  \sum_{i=0}^{n-m}\frac{(-1)^i}{i!(n-m-i)!}&=\frac{1}{(n-m)!}\sum_{i=0}^{n-m}\binom{n-m}{i}(-1)^i\notag\\
&=\frac{1}{(n-m)!}(1-1)^{n-m}\notag\\
&=0,
\end{align}
since $n-m\geq 1$. Also,
\begin{align}
  \label{eq:1910}
  \sum_{i=0}^{n-m}\frac{(-1)^ii}{i!(n-m-i)!}&=\frac{1}{(n-m)!}\sum_{i=0}^{n-m}\binom{n-m}{i}(-1)^ii\notag\\
&=\frac{1}{(n-m)!}\frac{d}{dx}\left.\left\{\sum_{i=0}^{n-m}\binom{n-m}{i}(-x)^i\right\}\right|_{x=1}\notag\\
&=\frac{1}{(n-m)!}\frac{d}{dx}\left.(1-x)^{n-m}\right|_{x=1}\notag\\
&=0,
\end{align}
since $n-m\geq 2$. We conclude that
\begin{align}
  \label{eq:1911}
  a_{n,m}=0:1\leq m\leq n-2.
\end{align}
Hence the double sum in \eqref{eq:1903} vanishes.

We deduce from the above
\begin{align}
  \label{eq:1912}
  M_n\Phi=\frac{\lambda v^{n+1}}{6\kappa^2}\hat{I}_{n,n+2}+\Landau(v^{n+2}).
\end{align}
Therefore, substituting in \eqref{eq:1896}, we conclude
\begin{align}
  \label{eq:1913}
  \frac{d^n\hat{f}}{dv^n}=\frac{\lambda}{6\kappa^2 v}\hat{I}_{n,n+2}+\Landau(1).
\end{align}

\subsubsection{Estimate for $d^n\hat{\delta}/dv^n$}
We recall the function $\delta$ given by
\begin{align}
  \label{eq:1914}
  \delta(v)=g(v)-c_{+0}f(v),
\end{align}
where
\begin{align}
  \label{eq:1915}
  g(v)=r(v,v)-r_0.
\end{align}
We also recall the function $\hat{\delta}$ given by $\delta(v)=v^3\hat{\delta}(v)$. Using
\begin{align}
  \label{eq:1917}
  \frac{d^k}{dv^k}\left(\frac{1}{v^3}\right)=\frac{(-1)^k(k+2)!}{2v^{k+3}},
\end{align}
we deduce
\begin{align}
  \label{eq:1918}
  \frac{d^n\hat{\delta}}{dv^n}&=\sum_{k=0}^n\binom{n}{k}\frac{d^k}{dv^k}\left(\frac{1}{v^3}\right)\frac{d^{n-k}\delta}{dv^{n-k}}\notag\\
&=\frac{1}{2}\sum_{k=0}^n\frac{n!(-1)^k(k+1)(k+2)}{(n-k)!v^{k+3}}\frac{d^{n-k}\delta}{dv^{n-k}}\notag\\
&=\frac{1}{v^{n+3}}L_n\delta,
\end{align}
where $L_n$ is the $n$'th order differential operator
\begin{align}
  \label{eq:1919}
  L_n=\frac{1}{2}\sum_{m=0}^n\frac{(-1)^{n-m}n!(n+2-m)(n+1-m)}{m!}v^m\frac{d^m}{dv^m},
\end{align}
which is homogeneous w.r.t.~scaling. Hence $L_n$ takes a polynomial to a polynomial of the same degree. Let $G$ be a polynomial which begins cubic terms. Then $M\coloneqq G/v^3$ is analytic, hence so is $v^{-k-3}L_nG$. This follows from \eqref{eq:1918} with $G$ in the role of $\delta$ and $M$ in the role of $\hat{\delta}$. It follows that the polynomial $L_nG$ begins with terms of degree $k+3$. We conclude that the null space of $L_n$ consists of all polynomials of degree $n+2$ which begin with cubic terms, i.e.
\begin{align}
  \label{eq:1920}
  L_n\bar{P}_{n+2,3}=0.
\end{align}
This is a $n$-dimensional space. %If this short argument is not satisfactory, then look at pages SF212, SF213.

We now estimate $d^m\delta/dv^m$. Let us recall the splitting of $\delta(v)$
\begin{align}
  \label{eq:1921}
  \delta(v)=\delta_0(v)+\delta_1(v),
\end{align}
where the functions $\delta_0$ and $\delta_1$ are given by
\begin{alignat}{3}
  \label{eq:1922}
  \frac{d\delta_0}{dv}(v)&=\frac{\lambda}{6\kappa}(1+y(v))v^2,&\qquad & \delta_0(0)=0,\\
  \label{eq:1925}
  \frac{d\delta_1}{dv}(v)&=\left(V(v)-c_{+0}-\frac{\kappa}{2}(1+y(v))v\right)\frac{\lambda}{3\kappa^2}v+(V(v)-c_{+0})\phi(v),& & \delta_1(0)=0,
\end{alignat}
where we recall the function $\phi$
\begin{align}
  \label{eq:1923}
  \phi(v)=\frac{df}{dv}(v)-\frac{\lambda}{3\kappa^2}v.
\end{align}
We note that
\begin{align}
  \label{eq:1924}
  \phi(v)=\frac{d\Phi}{dv}(v),
\end{align}
where $\Phi$ is given in \eqref{eq:1888}.

Now we apply $d^{m-1}/dv^{m-1}$ to $d\delta_0/dv$. We claim
\begin{align}
  \label{eq:1939}
  \frac{d^m\delta_0}{dv^m}=\frac{\lambda}{6\kappa}\left\{v^2\frac{d^{m-1}y}{dv^{m-1}}+2(m-1)v\frac{d^{m-2}y}{dv^{m-2}}+(m-1)(m-2)\frac{d^{m-3}y}{dv^{m-3}}\right\}+\bar{P},
\end{align}
where $\bar{P}$ is a generic polynomial and we interpret $d^ky/dv^k$ for $k<0$ as the $k$-fold iterated integral of $y$. For $m\geq 1$ this follows directly and we note that for $m\geq 3$ the polynomial $\bar{P}$ is the zero polynomial. For the case $m=0$ we define
\begin{align}
  \label{eq:1940}
  u_1(v)&\coloneqq\int_0^vy(v')dv',\\
  \label{eq:1941}
  u_2(v)&\coloneqq\int_0^vu_1(v')dv',\\
  \label{eq:1942}
  u_3(v)&\coloneqq\int_0^vu_2(v')dv'.
\end{align}
Then the right hand side in \eqref{eq:1939} becomes
\begin{align}
  \label{eq:1943}
  \frac{\lambda}{6\kappa}\left\{v^2u_1(v)-2vu_2(v)+2u_3(v)\right\}+\bar{P}(v),
\end{align}
while from \eqref{eq:1922} we have
\begin{align}
  \label{eq:1944}
  \delta_0(v)&=\frac{\lambda}{6\kappa}\int_0^vv'^2(y(v')+1)dv'\notag\\
  &=\frac{\lambda}{6\kappa}\int_0^vv'^2\left(1+\frac{du_1}{dv}(v')\right)dv'\notag\\
&=\frac{\lambda}{18\kappa}v^3+\frac{\lambda}{6\kappa}\left\{v^2u_1(v)-2\int_0^vv'u_1(v')dv'\right\}.
\end{align}
We write
\begin{align}
  \label{eq:1945}
  \int_0^vv'u_1(v')dv'&=\int_0^vv'\frac{du_2}{dv}(v')dv'\notag\\
&=vu_2(v)-\int_0^vu_2(v')dv'\notag\\
&=vu_2(v)-u_3(v).
\end{align}
Substituting in \eqref{eq:1944} we see that \eqref{eq:1944} coincides with \eqref{eq:1943} up to a polynomial. Therefore, \eqref{eq:1939} holds for $m=0$ also, hence it holds for $m\geq 0$.

From
\begin{align}
  \label{eq:1947}
  \frac{d^{n-1}y}{dv^{n-1}}(v)=\frac{d^{n-1}y}{dv^{n-1}}(0)+\hat{I}_{n,0}(v),
\end{align}
together with the expression for the $l$-fold iterated integral of $\hat{I}_{n,0}$ given by \eqref{eq:1946} we obtain $d^{n-1-l}y/dv^{n-1-l}$. Setting $m=n-1-l$ we obtain
\begin{align}
  \label{eq:1948}
  \frac{d^my}{dv^m}=\bar{P}_{n-1-m}+v^{n-1-m}\left\{\frac{\hat{I}_{n,0}}{(n-1-m)!}+\sum_{j=1}^{n-1-m}\frac{(-1)^j}{j!(n-1-m-j)!}\hat{I}_{n,j}\right\}.
\end{align}
Replacing $m$ by $m-1$, $m-2$, $m-3$ and substituting in \eqref{eq:1939} we obtain
\begin{align}
  \label{eq:1949}
  \frac{d^m\delta_0}{dv^m}&=\frac{\lambda}{6\kappa}v^{n+2-m}\left\{\frac{n(n+1)}{(n+2-m)!}\hat{I}_{n,0}+\sum_{j=1}^{n+2-m}\frac{(-1)^j(n-j)(n-j+1)}{j!(n+2-m-j)!}\hat{I}_{n,j}\right\}\notag\\
&\hspace{70mm}+\left\{
    \begin{array}{ll}
      P_{n+2-m,4-m} & 0\leq m\leq 3,\\
      P_{n+2-m} & m\geq 4.\\
    \end{array}\right.
\end{align}
The polynomial part follows from $\delta_0=\Landau(v^4)$, which in turn follows from $d\hat{\delta}_0/dv=\Landau(1)$.

Now we estimate $L_n\delta_0$, where the operator $L_n$ is given in \eqref{eq:1919}. Recalling that the null space of $L_n$ consists of all polynomials of degree $n+2$ which begin with cubic terms, we see that $L_n$ annihilates the polynomial part of $\delta_0$. Therefore,
\begin{align}
  \label{eq:1950}
  L_n\delta_0&=\frac{\lambda}{12\kappa}v^{n+2}\bigg\{n(n+1)(-1)^n\hat{I}_{n,0}\sum_{m=0}^n\binom{n}{m}(-1)^{m}\notag\\
&\qquad\qquad\qquad+\sum_{m=0}^n\sum_{j=1}^{n+2-m}\frac{(-1)^{n-m-j}n!(n+2-m)(n+1-m)(n-j)(n+1-j)}{m!j!(n+2-m-j)!}\hat{I}_{n,j}\bigg\}.
\end{align}
Since
\begin{align}
  \label{eq:1951}
  \sum_{m=0}^n\binom{n}{m}(-1)^m=(1-1)^n=0,
\end{align}
for $n\geq 1$, the first sum in \eqref{eq:1950} vanishes.

We rewrite the double sum in \eqref{eq:1950} as
\begin{align}
  \label{eq:1952}
  \sum_{j=1}^{n+2}\frac{(-1)^{n+j}n!(n-j)(n+1-j)}{j!}a_{n,j}\hat{I}_{n,j},
\end{align}
where
\begin{align}
  \label{eq:1953}
  a_{n,j}\coloneqq\sum_{m=0}^{n+2-j}\frac{(-1)^m(n+2-m)(n+1-m)}{m!(n+2-m-j)!}.
\end{align}
We were able to include the term $m=n+1$ for $j=1$ trivially because it vanishes. We have
\begin{align}
  \label{eq:1954}
  a_{n,j}&=\frac{1}{(n+2-j)!}\sum_{m=0}^{n+2-j}\binom{n+2-j}{m}(-1)^m(n+2-m)(n+1-m)\notag\\
&=\frac{1}{(n+2-j)!}\frac{d^2}{dx^2}\left.\left\{\sum_{m=0}^{n+2-j}\binom{n+2-j}{m}x^{n+2-m}(-1)^m\right\}\right|_{x=1}\notag\\
&=\frac{1}{(n+2-j)!}\frac{d^2}{dx^2}\left.\left\{x^j(x-1)^{n+2-j}\right\}\right|_{x=1}.
\end{align}
Therefore,
\begin{align}
  \label{eq:1955}
  a_{n,j}=\left\{
    \begin{array}{ll}
      (n+1)(n+2) & j=n+2,\\
      2(n+1) & j=n+1,\\
      1 & j=n,\\
      0 & j\leq n-1.
    \end{array}\right.
\end{align}
However, the terms with the coefficients $a_{n,n}$ and $a_{n,n+1}$ do not contribute in the sum in \eqref{eq:1952} and we obtain that \eqref{eq:1952} collapses to the term $j=n+2$, i.e.~\eqref{eq:1952} is equal to
\begin{align}
  \label{eq:1956}
  2\hat{I}_{n,n+2}.
\end{align}
Therefore,
\begin{align}
  \label{eq:1957}
  L_n\delta_0=\frac{\lambda}{6\kappa}v^{n+2}\hat{I}_{n,n+2},
\end{align}
which, when substituted in \eqref{eq:1918} with $\delta_0$ in the role of $\delta$, gives
\begin{align}
  \label{eq:1958}
  \frac{d^n\hat{\delta}_0}{dv^n}=\frac{\lambda}{6\kappa v}\hat{I}_{n,n+2}.
\end{align}

We turn to $d^n\hat{\delta}_1/dv^n$. We apply $d^{n-1}/dv^{n-1}$ to $d\delta_1/dv$. For $n\geq 3$ we obtain,
\begin{align}
  \label{eq:1959}
  \frac{d^n\delta_1}{dv^n}&=\frac{\lambda}{3\kappa^2}v\left\{\frac{d^{n-1}V}{dv^{n-1}}-\frac{\kappa v}{2}\frac{d^{n-1}y}{dv^{n-1}}-\frac{(n-1)\kappa}{2}\frac{d^{n-2}y}{dv^{n-2}}\right\}\notag\\
&\qquad+(n-1)\frac{\lambda}{3\kappa^2}\left\{\frac{d^{n-2}V}{dv^{n-2}}-\frac{\kappa v}{2}\frac{d^{n-2}y}{dv^{n-2}}-\frac{(n-2)\kappa}{2}\frac{d^{n-3}y}{dv^{n-3}}\right\}\notag\\
&\qquad+\sum_{l=0}^{n-1}\binom{n-1}{l}\frac{d^lV}{dv^l}\frac{d^{n-1-l}\phi}{dv^{n-1-l}}.
\end{align}
From \eqref{eq:1947} and \eqref{eq:1948} with $m=n-2$ we have, for $n\geq 3$,
\begin{align}
  \label{eq:1960}
  v\frac{d^{n-1}y}{dv^{n-1}}&=\bar{P}_{1,1}+v\hat{I}_{n,0},\\
  \label{eq:1961}
  \frac{d^{n-2}y}{dv^{n-2}}&=\bar{P}_{1}+v\hat{I}_{n,0}-v\hat{I}_{n,1}.
\end{align}
Using \eqref{eq:1586} we obtain
\begin{align}
  \label{eq:1962}
  \frac{d^{n-1}V}{dv^{n-1}}-\frac{\kappa v}{2}\frac{d^{n-1}y}{dv^{n-1}}-\frac{(n-1)\kappa}{2}\frac{d^{n-2}y}{dv^{n-2}}=P_2,
\end{align}
which implies that the first line in \eqref{eq:1959} is a $P_{2,1}$.

From \eqref{eq:1948} with $m=n-3$ we have
\begin{align}
  \label{eq:1963}
  \frac{d^{n-3}y}{dv^{n-3}}=\bar{P}_2+\frac{v^2}{2}\left\{\hat{I}_{n,0}-2\hat{I}_{n,1}+\hat{I}_{n,2}\right\}.
\end{align}
Using
\begin{align}
  \label{eq:1964}
  \int_0^vv'\hat{I}_{n,0}(v')dv'&=\frac{v^2}{2}\left(\hat{I}_{n,0}-\hat{I}_{n,2}\right),\\
  \label{eq:1973}
\int_0^vv'\hat{I}_{n,1}(v')dv'&=v^2\left(\hat{I}_{n,1}-\hat{I}_{n,2}\right),
\end{align}
we obtain, through integration of \eqref{eq:1586},
\begin{align}
  \label{eq:1965}
  \frac{d^{n-2}V}{dv^{n-2}}=\frac{\kappa v^2}{4}\left\{n\hat{I}_{n,0}-2(n-1)\hat{I}_{n,1}+(n-2)\hat{I}_{n,2}\right\}.
\end{align}
From \eqref{eq:1961}, \eqref{eq:1963} and \eqref{eq:1965} we obtain
\begin{align}
  \label{eq:1966}
  \frac{d^{n-2}V}{dv^{n-2}}-\frac{\kappa v}{2}\frac{d^{n-2}y}{dv^{n-2}}-\frac{(n-2)\kappa}{2}\frac{d^{n-3}y}{dv^{n-3}}=P_2,
\end{align}
i.e.~the second line in \eqref{eq:1959} is a $P_2$.

From \eqref{eq:1586} together with the fact that $V(v)-c_{+0}=\Landau(v^2)$ we have
\begin{align}
  \label{eq:1967}
  \frac{d^mV}{dv^m}=\left\{
    \begin{array}{ll}
      c_{+0}+P_{n-1,2} & m=0,\\
      P_{n-2,1} & m=1,\\
      P_{n-m-1} & 2\leq m\leq n-1.
    \end{array}\right.
\end{align}
In view of \eqref{eq:1924} we have from \eqref{eq:1890}, \eqref{eq:1894},
\begin{align}
  \label{eq:1968}
  \frac{d^m\phi}{dv^m}=\left\{
    \begin{array}{ll}
      P_{n-1,2} & m=0,\\
      P_{n-2,1} & m=1,\\
      P_{n-m-1} & 2\leq m\leq n-1.
    \end{array}\right.
\end{align}
Using now \eqref{eq:1967} and \eqref{eq:1968} we deduce that each term in the sum in \eqref{eq:1959} is a $P_2$. We conclude that for $n\geq 3$, \eqref{eq:1959} is a $P_2$.

In the case $n=2$ in place of \eqref{eq:1959} we have
\begin{align}
  \label{eq:1969}
  \frac{d^2\delta_1}{dv^2}&=\frac{\lambda}{3\kappa^2}v\left(\frac{dV}{dv}-\frac{\kappa}{2}(1+y)-\frac{\kappa v}{2}\frac{dy}{dv}\right)+\frac{\lambda}{3\kappa^2}\left(V-c_{+0}-\frac{\kappa}{2}(1+y)v\right)\notag\\
&\qquad+(V-c_{+0})\frac{d\phi}{dv}+\frac{dV}{dv}\phi.
\end{align}
Using
\begin{align}
  \label{eq:1970}
  \int_0^v\hat{I}_{2,0}(v')dv'=v\left(\hat{I}_{2,0}-\hat{I}_{2,1}\right)
\end{align}
in \eqref{eq:1947} with $n=2$ we obtain (recall that $y(0)=-1$)
\begin{align}
  \label{eq:1974}
  1+y=\int_0^v\frac{dy}{dv}(v')dv'=\bar{P}_{1,1}+v\left(\hat{I}_{2,0}-\hat{I}_{2,1}\right).
\end{align}
Together with \eqref{eq:1586} and \eqref{eq:1947}, both with $n=2$, we obtain
\begin{align}
  \label{eq:1971}
  \frac{dV}{dv}-\frac{\kappa}{2}(1+y)-\frac{\kappa v}{2}\frac{dy}{dv}=P_{1,1}.
\end{align}
Using \eqref{eq:1964}, \eqref{eq:1973}, both with $n=2$, in \eqref{eq:1586} with $n=2$, we deduce
\begin{align}
  \label{eq:1972}
  V-c_{+0}=\int_0^v\frac{dV}{dv}(v')dv'=\frac{\kappa v^2}{2}\left(\hat{I}_{2,0}-\hat{I}_{2,1}\right)+P_{2,2}.
\end{align}
This together with \eqref{eq:1974} implies
\begin{align}
  \label{eq:1975}
  V-c_{+0}-\frac{\kappa}{2}(1+y)v=P_{2,2}.
\end{align}
In view of \eqref{eq:1967}, \eqref{eq:1968},
\begin{align}
  \label{eq:1976}
  (V-c_{+0})\frac{d\phi}{dv},\frac{dV}{dv}\phi=\Landau(v^3).
\end{align}
From \eqref{eq:1971}, \eqref{eq:1975} and \eqref{eq:1976} we obtain
\begin{align}
  \label{eq:1977}
  \frac{d^2\delta_1}{dv^2}=P_{2,2}.
\end{align}
Together with the above conclusion for $n\geq 3$ we conclude
\begin{align}
  \label{eq:1978}
  \frac{d^n\delta_1}{dv^n}=P_2.
\end{align}

From \eqref{eq:1978} we have
\begin{align}
  \label{eq:1979}
  \frac{d^m\delta_1}{dv^m}=P_{n+2-m}.
\end{align}
In view also of \eqref{eq:626} we have, in particular,
\begin{align}
  \label{eq:1980}
  \delta_1=P_{n+2,4}.
\end{align}
It follows that $L_n$ annihilates the polynomial part of $\delta_1$. From \eqref{eq:1918} (see also \eqref{eq:1919}) with $\delta_1$ in the role of $\delta$ we obtain
\begin{align}
  \label{eq:1981}
  \frac{d^n\hat{\delta}_1}{dv^n}=\Landau(1).
\end{align}
Combining with \eqref{eq:1958} we conclude
\begin{align}
  \label{eq:1982}
  \frac{d^n\hat{\delta}}{dv^n}=\frac{\lambda}{6\kappa v}\hat{I}_{n,n+2}+\Landau(1).
\end{align}

\subsubsection{Estimate for $d^ny/dv^n$}
We recall
\begin{align}
  \label{eq:1983}
  \frac{dy}{dv}(v)=-\frac{(\partial\hat{F})/\partial v)(v,y(v))}{(\partial\hat{F}/\partial y)(v,y(v))},
\end{align}
where
\begin{align}
  \label{eq:1984}
  \pp{\hat{F}}{y}(v,y)&=\frac{\lambda}{2\kappa}y^2-\kappa\hat{f}(v)+v\pp{R}{y}(v,y),\\
  \label{eq:1986}
\pp{\hat{F}}{v}(v,y)&=-\kappa y\frac{d\hat{f}}{dv}(v)+\frac{d\hat{\delta}}{dv}(v)+R(v,y)+v\pp{R}{v}(v,y).
\end{align}
We recall that $R$ is given by
\begin{align}
  \label{eq:1985}
  R(v,y)\coloneqq -\left(\frac{\partial^2r^\ast}{\partial t^2}\right)_0\left(\hat{f}(v)\right)^2-\left(\frac{\partial^4r^\ast}{\partial w^4}\right)_0\frac{y^4}{24}-\left(\frac{\partial^3r^\ast}{\partial t\partial w^2}\right)_0\frac{y^2\hat{f}(v)}{2}-vH(\hat{f}(v),y),
\end{align}
where $H$ is a smooth function of its arguments.

Setting $f_0=v^{n+2}\frac{d^ny}{dv^n}$ in \eqref{eq:1456} so that $g_0=I_{n,n+2}$ and using \eqref{eq:1458} with $k=-n-3$ we obtain the $l$-fold iterated integral of $\frac{1}{v}\hat{I}_{n,n+2}$. Using this in \eqref{eq:1913}, \eqref{eq:1982} we obtain
\begin{align}
  \label{eq:1987}
  \frac{d^m\hat{f}}{dv^m}&=\left\{
    \begin{array}{ll}
      P_{n-m-2} & 0\leq m\leq n-2,\\
      \Landau(1) & m=n-1,
    \end{array}\right.\\
  \label{eq:1988}
  \frac{d^m\hat{\delta}}{dv^m}&=\left\{
    \begin{array}{ll}
      P_{n-m-2} & 0\leq m\leq n-2,\\
      \Landau(1) & m=n-1.
    \end{array}\right.
\end{align}
From \eqref{eq:1948} we have
\begin{align}
  \label{eq:1990}
  \frac{d^my}{dv^m}=\left\{
    \begin{array}{ll}
      P_{n-m-2} & 0\leq m\leq n-2,\\
      \Landau(1) & m=n-1.
    \end{array}\right.
\end{align}
Recalling $(\partial\hat{F}/\partial y)(0,-1)=\frac{\lambda}{3\kappa}$, we deduce from \eqref{eq:1987}, \eqref{eq:1988} and \eqref{eq:1990}
\begin{align}
  \label{eq:1989}
  \frac{d^m}{dv^m}\left\{\pp{\hat{F}}{y}(v,y(v))\right\}&=\left\{
    \begin{array}{ll}
      \frac{\lambda}{3\kappa}+\Landau(v) & m=0,\\
      P_{n-m-2} & 1\leq m\leq n-2,\\
      \Landau(1) & m=n-1,
    \end{array}\right.\\
  \label{eq:1991}
  \frac{d^m}{dv^m}\left\{\pp{\hat{F}}{v}(v,y(v))\right\}&=\kappa\frac{d^{m+1}\hat{f}}{dv^{m+1}}+\frac{d^{m+1}\hat{\delta}}{dv^{m+1}}+\left\{
    \begin{array}{ll}
      P_{n-m-2} & 0\leq m\leq n-2,\\
      \Landau(1) & m=n-1.
    \end{array}\right.
\end{align}

We apply $d^{n-1}/dv^{n-1}$ to \eqref{eq:1983}. Using \eqref{eq:1913}, \eqref{eq:1982}, \eqref{eq:1989} and \eqref{eq:1991} we obtain
\begin{align}
  \label{eq:1992}
  \frac{d^ny}{dv^n}=\frac{1}{v}\hat{I}_{n,n+2}+\Landau(1),
\end{align}
i.e.
\begin{align}
  \label{eq:1993}
  \frac{d^ny}{dv^n}(v)=\frac{1}{v^{n+3}}\int_0^vv'^{n+2}\frac{d^ny}{dv^n}(v')dv'+\Landau(1).
\end{align}
Setting
\begin{align}
  \label{eq:1994}
  Z_n(v)\coloneqq\int_0^vv'^{n+2}\frac{d^ny}{dv^n}(v')dv',
\end{align}
we have
\begin{align}
  \label{eq:1995}
  \frac{d}{dv}\left(\frac{Z_n}{v}\right)=\Landau(v^{n+1}).
\end{align}
Integrating gives
\begin{align}
  \label{eq:1996}
  \frac{Z_n}{v}=\Landau(v^{n+2}).
\end{align}
Therefore,
\begin{align}
  \label{eq:1997}
  Z_n=\Landau(v^{n+3}),
\end{align}
which, when substituted in \eqref{eq:1993}, yields
\begin{align}
  \label{eq:1998}
  \frac{d^ny}{dv^n}=\Landau(1),
\end{align}
i.e.~$Y_n$ is bounded, hence $(Y_n)$ holds and the inductive step for the derivatives of the function $y$ is complete.

\subsubsection{Bound for $d^n\hat{f}/dv^n$ and $d^n\alpha_+/dv^n=P_1$}
We recall \eqref{eq:1913}
\begin{align}
  \label{eq:1999}
  \frac{d^n\hat{f}}{dv^n}(v)=\frac{\lambda}{6\kappa^2v^{n+3}}\int_0^vv'^{n+2}\frac{d^ny}{dv^n}(v')dv'+\Landau(1).
\end{align}
In view of the bound on $Y_n$ we have
\begin{align}
  \label{eq:2000}
  \frac{d^n\hat{f}}{dv^n}=\Landau(1).
\end{align}
Therefore, $F_n$ is bounded. Hence $(F_n)$ holds and the inductive step for the derivatives of the function $\hat{f}$ is complete. We note that \eqref{eq:2000} implies (recall $f=v^2\hat{f}$)
\begin{align}
  \label{eq:2001}
  \frac{d^nf}{dv^n}=P_1.
\end{align}

We turn to $d^n\alpha_+/dv^n$. From \eqref{eq:1253} we obtain
\begin{align}
  \label{eq:1279}
  \frac{d^n\alpha_+}{dv^n}(v)&=\frac{d^n\alpha_i}{du^n}(v)+\int_0^v\left\{\sum_{i=0}^{n}\binom{n}{i}\frac{\partial^{i+1}t}{\partial u^i\partial v}\frac{\partial^{n-i}\tilde{A}}{\partial u^{n-i}}\right\}(v,v')dv'\notag\\
&\qquad+\sum_{l=0}^{n-1}\binom{n-2}{l+1}\left\{\frac{\partial^{n-1}}{\partial u^{n-1-l}\partial v^l}\left(\pp{t}{v}\tilde{A}\right)\right\}(v,v).
\end{align}
The first term is taken care of by the assumption on the initial data. In the following we will make use of $(t_{p,m})$, $(t_{m,n})$, $(\alpha_{p,n})$, $(\alpha_{m,n})$, $(\beta_{p,n})$, $(\beta_{m,n})$ and \eqref{eq:1274} without any further reference. Each of the terms in the sum of the second term is at least a $Q_0$. Therefore, the second term is a $P_{1,1}$. We split the third term according to
\begin{align}
  \label{eq:1280}
  \sum_{l=0}^{n-1}\binom{n-2}{l+1}\frac{\partial^{n-1}}{\partial u^{n-1-l}\partial v^l}\left(\pp{t}{v}\tilde{A}\right)=\sum_{l=0}^{n-2}\binom{n-2}{l+1}\frac{\partial^{n-1}}{\partial u^{n-1-l}\partial v^l}\left(\pp{t}{v}\tilde{A}\right)+\binom{n-2}{n}\frac{\partial^{n-1}}{\partial v^{n-1}}\left(\pp{t}{v}\tilde{A}\right).
\end{align}
The first term involves mixed derivatives of $t$ of order $n$ and less and derivatives of $\alpha$, $\beta$ and $t$ of order $n-1$ and less which are all $Q_1$. Therefore, the first term in \eqref{eq:1280} is a $P_1$. For the second term in \eqref{eq:1280} we have
\begin{align}
  \label{eq:1281}
  \left\{\frac{\partial^{n-1}}{\partial v^{n-1}}\left(\pp{t}{v}\tilde{A}\right)\right\}(v,v)=P_1+\frac{\partial^nt}{\partial v^n}(v,v)\tilde{A}(v,v),
\end{align}
where for the first term we reason as above. To deal with the second term in \eqref{eq:1281} we use \eqref{eq:1645} with $n$ in the role of $k$ and \eqref{eq:1245} with $n$ in the role of $n-1$ which is
\begin{align}
  \label{eq:1282}
  a_{n}=\frac{1}{1+\gamma}\left\{\frac{d^{n}f}{dv^{n}}-\sum_{l=1}^{n-1}\binom{n-1}{l}\frac{d^l\gamma}{dv^l}a_{n-l}\right\}.
\end{align}
For the first term in the bracket we use \eqref{eq:2001}. For the sum we note that each of the $a_{n-l}$ involves derivatives of $t$ of order $n-1$ and less and is therefore a $P_1$. From \eqref{eq:1998} we have
\begin{align}
  \label{eq:1284}
  \hat{I}_{n,0}=\int_0^v\frac{d^ny}{dv^n}(v')dv'=\Landau(v).
\end{align}
Therefore, from \eqref{eq:1223},
\begin{align}
  \label{eq:2002}
  \frac{d^{n-1}\rho}{dv^{n-1}}=P_0,
\end{align}
which, through \eqref{eq:1652}, implies, recalling that $\rho(0)=0$,
\begin{align}
  \label{eq:1285}
  \frac{d^{n-1}\gamma}{dv^{n-1}}=P_1.
\end{align}
Therefore,
\begin{align}
  \label{eq:2003}
 a_n=P_1. 
\end{align}
Therefore, from \eqref{eq:1645} with $n$ in the role of $k$ we find
\begin{align}
  \label{eq:1287}
  \frac{\partial^nt}{\partial v^n}(v,v)=P_1(v).
\end{align}
Using this in \eqref{eq:1281} we see that the second term in \eqref{eq:1280} is a $P_1$ and therefore the whole sum in \eqref{eq:1280} is a $P_1$ which in turn implies that the third term in \eqref{eq:1279} is a $P_1$. We conclude
\begin{align}
  \label{eq:1288}
  \frac{d^n\alpha_+}{dv^n}=P_1,
\end{align}
i.e.~$(\alpha_{+,n})$ holds.

\subsubsection{Inductive Step for Derivatives of $t$, $\alpha$, $\beta$ Part Two\label{section_part_2}}
In the following we prove
\begin{align}
  \label{eq:1407}
    \frac{\partial^n\alpha}{\partial v^n}=Q_0,\qquad\frac{\partial^n\beta}{\partial v^n}=Q_0.
\end{align}
Together with \eqref{eq:1690} this will then establish \eqref{eq:1234}, \eqref{eq:1293} with $n$ in the role of $n-1$, i.e.~it will establish ($\alpha_{0,n}$), ($\beta_{0,n}$). With $n$ in the role of $n-1$ in the equations \eqref{eq:1243}, \eqref{eq:1266}, \eqref{eq:1264} we have
\begin{align}
  \label{eq:1702}
  \frac{\partial^{n}t}{\partial v^{n}}(u,v)&=e^{-L(u,v)}\left\{\frac{\partial^{n}t}{\partial v^{n}}(v,v)+\int_v^ue^{L(u',v)}R_{v,n}(u',v)du'\right\},\\
  \label{eq:1703}
\frac{\partial^{n}\alpha}{\partial v^{n}}(u,v)&=\left\{\sum_{i=0}^{n-1}\binom{n-1}{i}\frac{\partial^{i+1}t}{\partial v^{i+1}}\frac{\partial^{n-1-i}\tilde{A}}{\partial v^{n-1-i}}\right\}(u,v),\\
  \label{eq:1704}
\frac{\partial^{n}\beta}{\partial v^{n}}(u,v)&=\frac{d^{n}\beta_+}{dv^{n}}(v)+\int_v^u\left\{\sum_{i=0}^{n}\binom{n}{i}\frac{\partial^{i+1}t}{\partial v^i\partial u}\frac{\partial^{n-i}\tilde{B}}{\partial v^{n-i}}\right\}(u',v)du'\notag\\
&\qquad -\sum_{l=0}^{n-1}\binom{n}{l+1}\left\{\frac{\partial^{n-1}}{\partial u^l\partial v^{n-1-l}}\left(\pp{t}{u}\tilde{B}\right)\right\}(v,v).
\end{align}

Analogous to \eqref{eq:1693} we find
\begin{align}
  \label{eq:1705}
  R_{v,n}=Q_1+Q_1'\frac{\partial^n\alpha}{\partial v^n}+Q_1''\frac{\partial^n\beta}{\partial v^n}.
\end{align}
Therefore, taking into account \eqref{eq:1287},
\begin{align}
  \label{eq:1709}
  \frac{\partial^nt}{\partial v^n}(u,v)=Q_1+\int_v^u\left\{Q_1'\frac{\partial^n\alpha}{\partial v^n}+Q_1''\frac{\partial^n\beta}{\partial v^n}\right\}(u',v)du'.
\end{align}

For \eqref{eq:1703} we make use of \eqref{eq:1709} together with ($t_{p,n}$), ($t_{m,n}$), ($\alpha_{p,n}$), ($\alpha_{m,n}$), ($\beta_{p,n}$), ($\beta_{m,n}$). We find
\begin{align}
  \label{eq:1710}
  \frac{\partial^n\alpha}{\partial v^n}(u,v)=Q_1+\int_v^u\left\{Q_1'\frac{\partial^n\alpha}{\partial v^n}+Q_1''\frac{\partial^n\beta}{\partial v^n}\right\}(u',v)du'.
\end{align}

Now, ($Y_n$) implies that \eqref{eq:1298} holds with $n$ in the role of $n-1$. By \eqref{eq:2001} and \eqref{eq:1288} also \eqref{eq:1213} and \eqref{eq:1300} hold with $n$ in the role of $n-1$. Therefore, we obtain, in the same way as we derived \eqref{eq:1313}, that \eqref{eq:1313} holds with $n$ in the role of $n-1$. Therefore, the first term in \eqref{eq:1704} is a $P_1$.

We split the sum in the integral of \eqref{eq:1704} into
\begin{align}
  \label{eq:1711}
  \sum_{i=1}^{n-1}\binom{n}{i}\frac{\partial^{i+1}t}{\partial v^i\partial u}\frac{\partial^{n-i}\tilde{B}}{\partial v^{n-i}}+\pp{t}{u}\frac{\partial^n\tilde{B}}{\partial v^n}+\frac{\partial^{n+1}t}{\partial u\partial v^n}\tilde{B}.
\end{align}
In view of ($t_{p,n}$), ($t_{m,n}$), ($\alpha_{p,n}$), ($\alpha_{m,n}$), ($\beta_{p,n}$), ($\beta_{m,n}$) and the Hodograph system, each of the terms in the sum is a $Q_1$ and for the second term in \eqref{eq:1711} we have
\begin{align}
  \label{eq:1712}
  \frac{\partial^n\tilde{B}}{\partial v^n}=Q_1+Q_1'\frac{\partial^n\alpha}{\partial v^n}+Q_1''\frac{\partial^n\beta}{\partial v^n}+Q_1'''\frac{\partial^nt}{\partial v^n}.
\end{align}
From \eqref{eq:1237} with $i=0$, $j=n-1$ we obtain, using again ($t_{p,n}$), ($t_{m,n}$), ($\alpha_{p,n}$), ($\alpha_{m,n}$), ($\beta_{p,n}$), ($\beta_{m,n}$),
\begin{align}
  \label{eq:1713}
  \frac{\partial^{n+1}t}{\partial u\partial v^n}=Q_1+Q_1'\frac{\partial^n\alpha}{\partial v^n}+Q_1''\frac{\partial^n\beta}{\partial v^n}+Q_1'''\frac{\partial^n t}{\partial v^n}.
\end{align}

We now look at the last term in \eqref{eq:1704}. We split the sum into
\begin{align}
  \label{eq:1714}
  \sum_{l=0}^{n-2}\binom{n}{l+1}\frac{\partial^{n-1}}{\partial u^l\partial v^{n-1-l}}\left(\pp{t}{u}\tilde{B}\right)+\frac{\partial^{n-1}}{\partial u^{n-1}}\left(\pp{t}{u}\tilde{B}\right).
\end{align}
In view of ($t_{p,n}$), ($t_{m,n}$), ($\alpha_{p,n}$), ($\alpha_{m,n}$), ($\beta_{p,n}$), ($\beta_{m,n}$) and the Hodograph system, each of the terms in the sum is a $Q_1$ and for the second term we use \eqref{eq:1274} and conclude that this term is a $Q_0$. Using now \eqref{eq:1712}, \eqref{eq:1713} in \eqref{eq:1711} and the resulting expression in \eqref{eq:1704} we find
\begin{align}
  \label{eq:1715}
  \frac{\partial^n\beta}{\partial v^n}(u,v)=Q_0+\int_v^u\left\{Q_1\frac{\partial^n\alpha}{\partial v^n}+Q_1'\frac{\partial^n\beta}{\partial v^n}+Q_1'''\frac{\partial^nt}{\partial v^n}\right\}(u',v)du'.
\end{align}

Defining
\begin{align}
  \label{eq:1716}
  G\coloneqq\left|\frac{\partial^n\alpha}{\partial v^n}\right|+\left|\frac{\partial^n\beta}{\partial v^n}\right|+\left|\frac{\partial^nt}{\partial v^n}\right|,
\end{align}
and taking the sum of the absolute values of \eqref{eq:1709}, \eqref{eq:1710} and \eqref{eq:1715}, we obtain
\begin{align}
  \label{eq:1717}
  G(u,v)\leq C+C'\int_v^uG(u',v)du',
\end{align}
which implies
\begin{align}
  \label{eq:1718}
  G(u,v)\leq C,
\end{align}
which in turn implies
\begin{align}
  \label{eq:1719}
  \left|\frac{\partial^n\alpha}{\partial v^n}\right|,\left|\frac{\partial^n\beta}{\partial v^n}\right|,\left|\frac{\partial^nt}{\partial v^n}\right|\leq C.
\end{align}
Therefore, using this in \eqref{eq:1709}, \eqref{eq:1710}, \eqref{eq:1715} we obtain
\begin{align}
  \label{eq:1275}
  \frac{\partial^n\alpha}{\partial v^n},\frac{\partial^n\beta}{\partial v^n},\frac{\partial^nt}{\partial v^n}=Q_0,
\end{align}
the first and second of which are \eqref{eq:1407}. From \eqref{eq:1690} and \eqref{eq:1407} we conclude that ($\alpha_{0,n}$), ($\beta_{0,n}$) hold. This completes the proof of the inductive step for the derivatives of $\alpha$, $\beta$ and $t$. Therefore, the inductive step is complete.

\subsection{Blowup on the Incoming Characteristic originating at the Cusp Point}
We recall the following asymptotic forms
\begin{align}
  \label{eq:1378}
  \pp{\alpha}{v}(u,v)&=\frac{\lambda\tilde{A}_0}{3\kappa^2}v+\Landau(uv),\\
  \label{eq:1393}
  \pp{\beta}{v}(u,v)&=\frac{\lambda}{3\kappa^2}\cp{\pp{\beta^\ast}{t}}v+\Landau(uv),\\
  \label{eq:1379}
  \pp{t}{v}(u,v)&=\frac{\lambda}{3\kappa^2}v+\Landau(uv).
\end{align}
As established above, $\alpha$, $\beta$ and $t$ are smooth functions of $u$ and $v$ in the state behind the shock. Let us now consider an outgoing characteristic originating at a point on $\underline{C}$ corresponding to the coordinates $(u,0)$. According to \eqref{eq:1378}, \eqref{eq:1393} and \eqref{eq:1379}, along this outgoing characteristic the Taylor expansions in $v$ of $\alpha$, $\beta$ as well as $t$ do not contain linear terms but do contain odd powers beginning with the third. Therefore, $\alpha$ and $\beta$ hence also the $\psi_\mu$ are smooth functions not of the parameter $t$ but rather of the parameter
\begin{align}
  \label{eq:1397}
  \sqrt{t-t_0},\qquad \textrm{where}\qquad t_0=t(u,0).
\end{align}
Therefore, the derivatives of the $\psi_\mu$ with respect to $L_+$ of order greater than the first blow up as we approach $\underline{C}$ from the state behind the shock (recall \eqref{eq:123} for $L_+$).

%\subsection*{Remarks}
%Check the sign of $\gamma$ in relation to the determinism condition and how it works therefore (hopefully) in \eqref{eq:1245}.

%%% Local Variables: 
%%% mode: latex
%%% TeX-master: "./master"
%%% End: 

\bibliographystyle{alpha}
\bibliography{bibliography}

\end{document}